%% file: main.tex
\pdfoutput=1
\documentclass[10pt,leqno]{amsart}
\usepackage{graphicx}
\usepackage{indentfirst,csquotes}

\usepackage{amsmath}

\DeclareMathOperator*{\argmax}{arg\,max}
\DeclareMathOperator*{\argmin}{arg\,min}

\usepackage{amssymb,amsthm,amsmath}
\usepackage{xcolor,paralist,hyperref,fancyhdr,etoolbox}
\newtheorem{theorem}{Theorem}[]
\newtheorem{definition}[theorem]{Definition}

\newtheorem{lemma}[theorem]{Lemma}
\newtheorem{proposition}[theorem]{Proposition}

\newtheorem{assumption}{Assumption}%
\newtheorem{remark}{Remark}%

\numberwithin{equation}{section}

\usepackage{subcaption}
\usepackage{caption}
\usepackage{float}

\usepackage{microtype}

\usepackage{pgfplots}
\usepackage{fancyvrb} 
\usepackage{lipsum}
\usepackage{amsfonts}
\usepackage{graphicx}
\usepackage{epstopdf}
\usepackage{bbm}

\hypersetup{ colorlinks=true, linkcolor=black, filecolor=black, urlcolor=black }

\usepackage{lipsum}

\DeclareUnicodeCharacter{2212}{−}
\usepgfplotslibrary{groupplots,dateplot}
\usetikzlibrary{patterns,shapes.arrows}
\pgfplotsset{compat=newest}

\usepackage{algorithm,algpseudocode}

\ifpdf
    \DeclareGraphicsExtensions{.eps,.pdf,.png,.jpg}
\else
  \DeclareGraphicsExtensions{.eps}
\fi

\usepackage{amsaddr}

\begin{document}

\title[CURVATURE ALIGNED SIMPLEX GRADIENT]{CURVATURE ALIGNED SIMPLEX GRADIENT: PRINCIPLED
SAMPLE SET CONSTRUCTION FOR NUMERICAL DIFFERENTIATION}

\author{Daniel Lengyel, Panos Parpas}
\email[Daniel Lengyel]{d.lengyel19@imperial.ac.uk}
\address{Department of Computing, Imperial College London}

\author{Nikolas Kantas}
\address{Department of Mathematics, Imperial College London}

\author{Nicholas R. Jennings}
\address{Loughborough University}

\date{\today}

\maketitle

\begin{abstract}
The simplex gradient, a popular numerical differentiation method due to its flexibility, lacks a principled method by which to construct the sample set, specifically the location of function evaluations. Such evaluations, especially from real-world systems, are often noisy and expensive to obtain, making it essential that each evaluation is carefully chosen to reduce cost and increase accuracy. This paper introduces the curvature aligned simplex gradient (CASG), which provably selects the optimal sample set under a mean squared error objective. As CASG requires function-dependent information often not available in practice, we additionally introduce a framework which exploits a history of function evaluations often present in practical applications. Our numerical results, focusing on applications in sensitivity analysis and derivative free optimization, show that our methodology significantly outperforms or matches the performance of the benchmark gradient estimator given by forward differences (FD) which is given exact function-dependent information that is not available in practice. Furthermore, our methodology is comparable to the performance of central differences (CD) that requires twice the number of function evaluations.
\end{abstract}

\section{Introduction}

Gradient estimation is crucial in applications ranging from chemical engineering \cite{barton1992computing}, financial markets \cite{gourieroux2000sensitivity} to medical and biological research \cite{qian2020sensitivity, johnston2007examples}. When it is possible to modify a system, then methods such as automatic differentiation \cite{griewank1989automatic, baydin2018automatic, schulman2015gradient}, infinitesimal perturbation analysis  \cite{suri1987infinitesimal, wardi2018perturbation}, score function estimators \cite{williams1992simple} or weak derivative estimators \cite{bhatt2019policy} are the first choice for gradient estimation. However, many applications involve expensive physical processes, simulations or legacy code which cannot be easily modified. This is generally the domain of numerical differentiation, where one only has access to potentially corrupted system function evaluations which may be expensive to obtain. The challenge in such circumstances is determining the optimal points for function evaluation to refine the gradient estimate.

A strong numerical differentiation method that significantly benefits from a careful choice of the sample set, i.e., locations of evaluations, is the simplex gradient. The simplex gradient places limited restrictions on the sample set, which has led to its gain in popularity \cite{hare2020calculus, hare2022error, regis2015calculus, berahas2019linear, coope2019efficient}, especially in the derivative-free optimization community \cite{conn2008geometry, conn2008geometry2, custodio2007using, custodio2008using}. However, there is only a limited understanding of which sample sets may lead to good gradient estimates and how to construct them in practice \cite{hare2022error, hare2020calculus, regis2015calculus}. In this paper, we present substantial advancements addressing this open question.

\subsection{Related Works}

Besides a general sample set, there are various other benefits to the use of the simplex gradient. The simplex gradient uses a linear model and hence only requires $d + 1$ function evaluations, with $d$ being the number of input variables. A linear model is often preferable as higher order models can lead to instability \cite{cheney2009course, dahlquist2003numerical}, which becomes amplified when noise is introduced. Unlike the simplex gradient, such methods then often require a form of explicit regularization \cite{hanke2001inverse, ahnert2007numerical, cullum1971numerical, knowles2014methods, wei2007numerical, van2020numerical} which is difficult to choose in practice. The number of function evaluations is also a strong benefit of the simplex gradient. Higher-order methods or methods such as Richardson Extrapolation \cite{burden2015numerical, bruno2012numerical, knowles2014methods, brekelmans2008gradient, fornberg1988generation} often require far more evaluations which may be infeasible. Conversely, using fewer points would make the linear fit under-determined and can only be applied in cases where the gradient is sparse \cite{borkar2017gradient}. Inexact methods which require fewer than $d+1$ function evaluations are often stochastic in nature, for example simultaneous perturbation stochastic approximation \cite{spall1992multivariate, spall1998overview, s2013stochastic} and smoothing methods \cite{berahas2019linear}. However, such methods are often far too inaccurate for gradient estimation  \cite{blakney2019comparison} and have been shown to have poorer performance compared to a deterministic counterpart \cite{scheinberg2022finite, berahas2019linear, berahas2022theoretical}. 

If a gradient estimator is allowed to use only $d+1$ function evaluations, the baseline methods are often variations of the simplex gradient, with the variation being in the choice of the sample set. While simple in nature, forward differences has remained the go-to simplex gradient based method as it provides a closed form solution for its optimal sample set \cite{brekelmans2005gradient, berahas2022theoretical} and outperforms other comparable methods \cite{brekelmans2005gradient}. However the conditions on the sample set are very restrictive and may hinder performance. Other methods which have allowed for a general sample set are often heuristically derived from general principles from function approximation \cite{dahlquist2003numerical, burden2015numerical} or have been presented in terms of Design of Experiment (DoE) \cite{zhang1991design, brekelmans2005gradient}. DoE was pioneered by \cite{fisher1936design} to assess the factors influencing the response of a system by laying out principles by which to assess the individual and combined impact of changing parameters to the system in an efficient manner. These methods have resulted in various schemes, such as the fractional factorial design  \cite{montgomery2017design} and the Plackett–Burman design \cite{plackett1946design}. Yet, there is no guarantee or systematic understanding for when these designs will provide good gradient estimates. Additionally their performance is generally comparable to or worse than forward differences \cite{brekelmans2005gradient}, making finite difference so far the preferable choice.

To choose a sample set in practice, there is some function-dependent information which is needed as otherwise a sample set cannot be adapted to a specific problem. However, such information is not readily available in practice, and hence many methods have focused on developing such estimators for the use in finite difference schemes \cite{gill2019practical, gill1983intervals, shi2022adaptive, more2012estimating, barton1992computing}. While these methods attempt to reduce the number of function evaluations by estimating the higher-order information only periodically, the procedure remains expensive as more function evaluations than strictly needed for gradient estimation are required.

\subsection{Contributions}
Against this background we present the following contributions.

\begin{itemize}
    \item We introduce the Curvature Aligned Simplex Gradient (CASG), a simplex gradient method based on principled sample set construction. CASG minimizes a mean squared error using a Taylor approximation of the true gradient for problems with dimensions that are powers of two. Theorem \ref{thm:SVD_Opt} proves the optimality of CASG and Proposition \ref{prop:sig_alg} provides an efficient construction algorithm.
    
    \item We introduce the extended curvature aligned simplex gradient (eCASG) for use in arbitrary dimensions. eCASG heuristically partitions the space and Proposition \ref{prop:eCASG} proves that under this partition eCASG is the optimal simplex gradient based estimator. 
    
    \item For the construction of our gradient estimators, we introduce a global model framework to estimate any necessary function-dependent information. Instead of requiring further sampling, the framework reuses a history of function evaluations to build a global model which allows us to construct any function-dependent information needed for CASG and eCASG. 
    
    \item We numerically evaluate our methodology by considering applications in sensitivity analysis and derivative free optimization. 
    \begin{itemize}
        \item Our methodology significantly outperforms or matches the performance of a best-case forward difference estimator which is given all required function-dependent information. As current gradient estimators with $d+1$-function evaluations are upper-bounded by the performance of  this best-case FD estimator, our proposed method can be used to significantly improve on such methods. 
        \item Our methodology often matches the performance of central differences, generally considered a far superior gradient estimator but requiring twice the number of function evaluations per gradient estimate. This finding implies that in many situations we can replace central difference methods with our methodology with limited loss in performance while reducing the function evaluations by a factor of two. 
    \end{itemize}
\end{itemize}


\subsection{Outline}
Section \ref{sec:setup} introduces the relevant concepts. In Section \ref{sec:optMotivation} we motivate and derive a Taylor-expansion based objective function for the optimization problem related to finding the optimal sampling set for the simplex gradient and derive some crucial simplifications. In Section \ref{sec:mainThm} we present the main results and provide the proof and algorithms for CASG. We additionally provide some intuition behind CASG and an initial comparison to forward differences. In Section \ref{sec:subspace} we then present the extended curvature aligned simplex gradient. We provide the global model framework algorithm and the empirical results in Section \ref{sec:Numerics}. In Section \ref{sec:Numerics} we also discuss in further detail the use of the global model as a separate gradient estimation method. We also consider the cost of the global model framework when no history of function evaluation is available.

\subsection{Notation}
$\Vert \cdot \Vert_2$ is the $L_2$ norm and $\Vert \cdot \Vert_F$ the Frobenius norm. $\text{tr}(\cdot)$ is the trace operator. The general linear group of degree $d$ is $GL_d$ and represents the space of all invertible $d \times d$-matrices. Given a vector $v \in \mathbb{R}^d$, $diag(v) \in \mathbb{R}^{d \times d}$ is the diagonal matrix with $v$ along the diagonal. $v \succ 0$ denotes all entries greater than zero and $v \succeq 0$ greater or equal to zero. Given a matrix $B \in \mathbb{R}^{d \times d}$, $diag(B) \in \mathbb{R}^d$ is the diagonal of the matrix $B$. $B \succ 0$ denotes $B$ positive definite and $B \succeq 0$ positive semi-definite. $\mathbb{R}_+$ is the set of all strictly positive real values and $\mathbb{Z}_+$ is the space of all positive integers. Given a set $A$, $\mathcal{P}(A)$ is the power-set of $A$ and $\vert A \vert$ the size of $A$. When clear from context, $H$ denotes the second derivative, i.e., the Hessian, of a function. Let the eigen-decomposition be given by $H = R D R^T$, with $R$ orthogonal and $D$ diagonal. By nomenclature we use curvature to refer to the second order behavior of a function. To disambiguate between different Big-O notations, we use $\mathcal{O}$ for runtime and $O$ for bounding error terms from Taylor's-Theorem. The Hadamard matrix $M_k$ is an orthogonal matrix with the property $(M_k)_{ij}^2 = \frac{1}{d}$ and all positive entries in the $k$th column.

\section{Problem Statement and Set-Up} \label{sec:setup}

The goal is to find the gradient of a differentiable function $f : \mathcal{D} \rightarrow \mathbb{R}$ where $\mathcal{D} \subset \mathbb{R}^d$ has a non-empty interior, while only having access to a noisy black-box. Since function evaluations may be expensive, the function can only be evaluated $d+1$ times per gradient estimation. The metric by which the quality of the gradient is assessed is the mean squared error (MSE).

\subsection{Black Box Evaluation and History}
We define the black-box function as a random process to formalize the notion of sequentially evaluating a black-box. The black-box function is given by $\tilde{f} : \mathcal{D} \times \mathbb{Z}_+$ where  
\begin{align*}
    \tilde{f}(x; i) = f(x) + \epsilon(i).
\end{align*}
 When clear from context, we refer to the black-box function as $\tilde{f}(x)$ and the noise as $\epsilon(x)$. We also define the \textit{evaluation history} as the set 
\begin{align*}
    \mathcal{E}(I) = \{(x(i), \tilde{f}(x(i); i)\}_{1 \leq i \leq I}.
\end{align*}
We write $x(i)$ to make it explicit that the choice of where $x$ is evaluated at the $i$th time-step may depend on the past. 
 
 We now introduce the following assumption on the noise. 
\begin{assumption}\label{as:sig}
$\epsilon(i)$ is assumed to have zero mean and constant and bounded second moment. We additionally we assume that for $i \neq j$, $\epsilon(i)$ is independent of $\epsilon(j)$ and that $x(i)$ is independent of $\epsilon(i)$. 
\end{assumption}

The assumption of zero mean noise and constant and bounded second moment are common in the literature. We additionally make an implicit assumption explicit, namely that subsequent function evaluations have independent noise. The last assumption that $\epsilon(i)$ is independent of $x(i)$ makes explicit that we can not choose $x$ by predicting the noise.

\subsection{The Simplex Gradient}\label{subsec:simplex}
The simplex gradient set-up and relevant terminology is in line with other works \cite{hare2020calculus, regis2015calculus}.
Let $x'$ be the point at which we want to estimate the gradient. We define the \textbf{sample set} as an ordered set $\mathcal{X} = \{x_i\}_{0 \leq i \leq d}$, with $x_0 = x'$. The \textbf{function vector} is the random vector $\mathbf{\tilde{f}}(\mathcal{X})$ such that $\mathbf{\tilde{f}}(\mathcal{X})_i = \tilde{f}(x_i)$. The \textbf{difference matrix} $S(\mathcal{X}) \in \mathbb{R}^{d \times d}$ is defined as $S_{ij} = (x_j - x_0)_i$. We will often just refer to it as $S$ when clear from context. We refer to the columns as the \textbf{difference vectors} and write $s_i = x_j - x_0$. Lastly, we define the \textbf{function difference vector} as $\delta \mathbf{\tilde{f}}_{i}(\mathcal{X}) = \tilde{f}(x_i) - \tilde{f}(x_0)$.
\begin{lemma}
    Let $\mathcal{X}$ be a sample set such that the size is $d + 1$, $x_0 = x'$ and $span(S(\mathcal{X})) = \mathbb{R}^d$. Then the unique affine function that interpolates the set of points $\{(x_i, \tilde{f}(x_i))\}_{0 \leq i \leq d}$ is given by
    \begin{align*}
        \eta^*(x) &= \mathbf{\tilde{f}}(x_0) + (\nabla_S \tilde{f}(x_0))^T (x - x_0)
    \end{align*}
    where $\nabla_S \tilde{f}(x_0)$ is the \textbf{simplex gradient} and is given by
    \begin{align*}
        \nabla_S \tilde{f}(x_0) = S^{-T} \delta \mathbf{\tilde{f}}
    \end{align*}
    with $S^{-T} = (S^T)^{-1}$.
\end{lemma}
\begin{remark}
The simplex gradient is a random variable as it depends on $\tilde{f}$. Therefore, the mean squared error of the gradient estimator is taken with respect to the probability space $(\Sigma, \mathcal{F}, \mathbb{P})$. 
\end{remark}

Forward differences can be seen as an instance of the simplex gradient where the sample set is given by $\{x'\} \cup \{h_i e_i\}_{1 \leq i \leq d}$. Central difference can also be seen as an instance of an over-determined simplex gradient \cite{hare2020calculus}.

\section{Optimization Problem for Optimal Sample Set}\label{sec:optMotivation}
In this section we motivate and define the objective function which is used to assess the goodness of a sample set for numerical differentiation under the simplex gradient. The objective function is based on the mean squared error and is derived via a Taylor-expansion argument as is common for finite difference schemes \cite{burden2015numerical}. 

To better understand the mean squared error, we consider its decomposition into the approximation and noise error via the following Proposition. 
\begin{proposition}[MSE Decomposition]\label{prop:mse_noise}
Let the sample set $\mathcal{X}$ and the black-box function $\tilde{f}$ be given and given Assumption \ref{as:sig} on the noise. Let $S$ be the difference set corresponding to the sample set. Then the mean squared error is given by 
\begin{align*}
    &\mathbb{E}[\Vert \nabla_{S} \tilde{f}(x_0) - \nabla f(x_0) \Vert_2^2] \\
    &=  \underbrace{\Vert \nabla_{S} f(x_0) - \nabla f(x_0) \Vert_2^2}_{\text{Approximation Error}} 
    + \underbrace{\sigma^2 \Vert S^{-1} \Vert_F^2 + \sigma^2 \Vert S^{-T} \mathbf{1} \Vert_2^2}_{\text{Noise Error}}. 
\end{align*}
\end{proposition}
\begin{proof}\label{prf:mse_noise}
We expand the means squared error after which we proceed term by term
\begin{align*}
    \mathbb{E}[\Vert \nabla_{S} \tilde{f}(x_0) - \nabla f(x_0) \Vert_2^2] &= \mathbb{E}[\Vert \nabla_{S} \tilde{f}(x_0) - \nabla_S f(x_0) + \nabla_S f(x_0) - \nabla f(x_0) \Vert_2^2] \\
    &= \mathbb{E}[\Vert \nabla_{S} \tilde{f}(x_0) - \nabla_S f(x_0)  \Vert_2^2] + \Vert \nabla_S f(x_0) - \nabla f(x_0) \Vert_2^2 \\
    &\quad + 2\mathbb{E}[\Big( \nabla_{S} \tilde{f}(x_0) - \nabla_S f(x_0) \Big)^T \Big( \nabla_S f(x_0) - \nabla f(x_0)\Big) ]. 
\end{align*}
By linearity of the simplex gradient we have that $\mathbb{E}[\Vert \nabla_{S} \tilde{f}(x_0) - \nabla_S f(x_0)  \Vert_2^2] = \mathbb{E}[\Vert \nabla_S \epsilon(0) \Vert_2^2 ]$ where we write $\nabla_S \epsilon(0) = S^{-T} \delta \mathbf{\epsilon}(\mathcal{X})$ with $\delta \mathbf{\epsilon}(\mathcal{X}) = [\epsilon(1) - \epsilon(0), \dots, \epsilon(d) - \epsilon(0)]^T$. We then have for the first term
\begin{align*} 
    \mathbb{E}[\Vert \nabla_S \epsilon \Vert_2^2] &= \mathbb{E}[\delta \mathbf{\epsilon}(\mathcal{X})^T S^{-1} S^{-T} \delta \mathbf{\epsilon}(\mathcal{X})] \\
    &=  \text{tr}\left(S^{-1}S^{-T} (\sigma^2 I + \sigma^2 \mathbf{1} \mathbf{1}^T)\right) \\
    &= \sigma^2 \Vert S^{-T} \Vert_F^2 + \sigma^2 \Vert S^{-T} \mathbf{1} \Vert_2^2.    
\end{align*}
where we used the result from \cite{kendrick2005stochastic} that for any random vector $\mathbf{X}$, with mean $\mu$ and covariance matrix $\Sigma$, and symmetric matrix $A$ we have $\mathbb{E}[\mathbf{X}^T A \mathbf{X}] = \mu^T A \mu + \text{tr}(A \Sigma)$.
The second term $\Vert \nabla_S f(x_0) - \nabla f(x_0) \Vert_2^2$ is already in the correct form. Lastly, the cross term $\mathbb{E}[\Big( \nabla_{S} \tilde{f}(x_0) - \nabla_S f(x_0) \Big)^T \Big( \nabla_S f(x_0) - \nabla f(x_0)\Big)]$ is easily seen to equal zero since $\mathbb{E}[\nabla_{S} \tilde{f}(x_0)] = \nabla_S f(x_0)$. 
\end{proof}

While the noise error only depends on the noise level and distribution of points, the approximation error also depends on the function structure. Since we do not have access to the true function $f$, we need use some approximation of $f$. However, even with an approximation of $f$, the MSE may be difficult to minimize directly. Instead, we introduce a bound on the approximation error.

\subsection{Approximation Error Bound}
While there exist multiple methods by which to estimate the approximation error, we proceed via a Taylor argument as done in many other works \cite{hare2022error, burden2015numerical, bruno2012numerical}. The crucial difference is that previous works upper bound the second order term, while we retain the second order term exactly and only bound the higher order terms.
\begin{lemma}\label{lma:usg_error}
Let $f \in \mathcal{C}^2$ and let $\Vert S \Vert_2 \leq h$ for some $h \in \mathbb{R}_+$. Then we have 
\begin{align*}
    \Vert\nabla_S f(x_0)  - \nabla f(x_0) \Vert_2^2 &= \frac{1}{4}  \Vert S^{-T} [s_1 \nabla^2 f(x_0) s_1, \dots, s_d \nabla^2 f(x_0) s_d]^T \Vert_2^2 + O(h^3).
\end{align*}
\end{lemma}
The proof is in Appendix \ref{app:taylor_bound} and is a simple application of Taylor's Theorem.

\subsection{Objective Function}
When using the bound on the approximation error and the noise error, the only variables are the noise level $\sigma$, the given Hessian matrix $H$, and the sample difference set $S$. To control the error due to higher-order effects, the hyperparameter $h$ is introduced. The optimization problem is then given by the following.
\begin{definition}\label{defn:optimization_problem}
    Let $H \in \mathbb{R}^{d \times d}$ be symmetric, $\sigma \in \mathbb{R}_+$ and $h \in \mathbb{R}_+$. Let
    \begin{align*}
        \textbf{AE} = \frac{1}{4} \Vert S^{-T} [s_1^T H s_1, \dots, s_d^T H s_d]^T \Vert_2^2
    \end{align*} be the approximation error and 
    \begin{align*}
        \textbf{NE} = \sigma^2 \Vert S^{-1} \Vert_2^2 + \sigma^2 \Vert S^{-T} \mathbf{1} \Vert_2^2
    \end{align*}
    be the noise error. We then define the function $l_{H, \sigma, h} : \mathbb{R}^{d \times d} \rightarrow [0, \infty]$ as
    \begin{align*}
        l_{H, \sigma, h}(S) &= 
        \begin{cases}
                 \textbf{AE} + \textbf{NE} & \text{if } h \geq \Vert S \Vert_2 \text{ and } S \in GL_d\\
                \infty & \text{, else.}
        \end{cases}
    \end{align*}
and its reparametrization in SVD form
\begin{align*}
        l_{H, \sigma, h}(U, \Sigma, V) &= 
        \begin{cases}
                l_{H, \sigma, h}(U \Sigma V^T) & \text{if } U, V \text{ orthogonal, $D$ diagonal, and $D \succeq 0$.} \\
                \infty & \text{, else.}
        \end{cases}
    \end{align*}
\end{definition}
When clear from context, we drop the $H, \sigma, h$ subscript and write $l(S)$ or $l(U, \Sigma, V)$.

\begin{remark}
Defining the objective function on the extended real number line and setting it to infinity outside of the valid input makes the optimization problem well defined. Specifically, a minimizer exists as the objective function is continuous and defined on a compact domain.
\end{remark}

The following Proposition significantly simplifies the objective function and allows us solve the optimization problem in the basis in which $H$ is diagonal.
\begin{proposition}\label{prop:rot_invariant}
Let $H$ be a symmetric matrix and let $R$ be an orthogonal matrix such that $R^T H R = D$. Defining $S' = R^T S$ then $l_{H, \sigma}(S) = l_{D, \sigma}(S')$. Furthermore, $l_{D, \sigma}(S) = l_{-D, \sigma}(S)$.
\end{proposition}
\begin{proof}\label{proof:rot_invariant}
Note that $\Vert S^{-1} \Vert_F = \text{tr}( R^T S'^{-T} S'^{-1} R)$ and by the cyclic property of trace $\Vert S^{-1} \Vert_F = \Vert S'^{-1} \Vert_F$. Furthermore, since orthogonal operations do not change norms $\Vert S^{-T} \mathbf{1} \Vert_2^2 = \Vert S'^{-T} \mathbf{1} \Vert_2^2$. For the approximation error note that $s_i^T H s_i = s_i'^T D s_i'$. As established, orthogonal matrices do not change norms and hence $\Vert S^{-T} [s_1^T D s_1, \dots, s_d^T D s_d]^T \Vert_2^2 = \Vert S'^{-T} [s_1'^T D s_1', \dots, s_d'^T D s_d']^T \Vert_2^2$. The invariance to negating the diagonal is due to the absolute homogeneity property of the norm. 
\end{proof}

While more manageable, the objective function is not trivial to minimize due its remaining non-convexity and non-linearity. This is unlike standard forward difference where the difference matrix has orthogonal columns which allows the optimization problem to separate into a sum of one-dimensional convex functions which are easy to solve. 

\section{Curvature Aligned Simplex Gradient}\label{sec:mainThm}
We now present the results and methods to compute 
\begin{align*}
    S^* = \argmin_{S \in \mathbb{R}^{d \times d}} \ l_{H, \sigma, h}(S)
\end{align*}
where $l_{H, \sigma, h}(S)$ is given in Definition \ref{defn:optimization_problem}. We call the simplex gradient which uses $S^*$ the \textit{Curvature Aligned Simplex Gradient} (CASG), since our method naturally aligns difference vectors along directions of low curvature. 

The results will be purely theoretical, as we assume a Hessian $H$ and noise-level $\sigma$ to be given. In Section \ref{sec:Numerics} we show how to construct CASG for real-world applications using our global model framework. We make this precise with the following assumption on the parameters.
\begin{assumption}\label{as:exact_D_sig}
The matrix $H \in \mathbb{R}^{d \times d}$ is symmetric, and $\sigma, h \in \mathbb{R}_+$.
\end{assumption}

Recall that for symmetric matrices we may write $H = R D R^T$ with orthogonal matrix $R$ and diagonal matrix $D$. To simplify the presentation of the problem, we use Proposition \ref{prop:rot_invariant} and general properties of diagonalizable matrices to make the following assumption on $H$ and $D$ without any loss of generality.
\begin{assumption}\label{as:D}
$H$ has positive trace and $D$ is in increasing order.
\end{assumption}

We now introduce the final assumption we will be using. 
\begin{assumption}\label{as:pow}
The dimension of the problem is a power of two. 
\end{assumption}
While Assumption \ref{as:pow} may seem artificial, it appears naturally when deriving the analytic solution. We discuss in Section \ref{sec:subspace} how to generalize the results from this section to arbitrary dimensions.

We now present the results used to construct a minimizer $S^*$. 
\begin{theorem}\label{thm:SVD_Opt}
Assume that for parameters $H, \sigma$ and $h$ Assumptions \ref{as:exact_D_sig}-\ref{as:pow} hold and recall the eigen-decomposition $H = R D R^T$. The global minimum of $l_{H, \sigma, h}$, $S^*$, is given by $R U^* \Sigma^* (V^*)^T$ where
\begin{itemize}
\item $U^* = I$
\item $\Sigma^*$ is the solution to the following optimization problem
\begin{align*}
        \Sigma^* = \argmin_{\Sigma}& 
        \begin{cases}
            \frac{(\sum_{i = 1}^d D_i \Sigma^2_i)^2}{4 d \Sigma_{max}^2} + \sigma^2 \sum_{i =0}^d \frac{1}{\Sigma^2_i} + \sigma^2 \frac{1}{\Sigma_{max}^2} d & \text{if } h I \succeq \Sigma \succ 0 \\
            \infty & \text{otherwise.}
        \end{cases}
    \end{align*}
\item $V^* = M_{k}$ where $k = \argmax_{k} \Sigma_k$. Recall that $M_k$ is the Hadamard matrix with the entries of the $k$th row being $1/\sqrt{d}$.
\end{itemize}
\end{theorem}

The optimization problem for $\Sigma^*$ can be solved without the need of numerical methods. Algorithm \ref{alg:get_sigma_star} provides the construction of $\Sigma^*$.
\begin{proposition}\label{prop:sig_alg}
    The optimization problem for $\Sigma^*$ as given in Theorem \ref{thm:SVD_Opt} is well defined, has a unique solution and can be computed in $\mathcal{O}(d)$ time. 
\end{proposition}

\subsection{Algorithm and Runtime}
\begin{algorithm}[ht]
\caption{CASG}\label{alg:optS}
\begin{algorithmic}[1]
\Procedure{Compute $S^*$}{$H, \sigma, h$}
\State $d \gets \text{Dimension of } H$
\State $R, D \gets \text{Diagonalize H; D a vector sorted in increasing order}$
\If{$\sum_{i = 1}^d D_i < 0$}
    \State \Return \text{Compute $S^*(-H, \sigma, h)$} \Comment{Call function with negated Hessian.} 
\EndIf
\State $\Sigma^* \gets \text{Get}\Sigma^*(D, \sigma, h)$ \Comment{In descending order.}
\State $V^* \gets \text{Hadamard matrix with first row all ones}$
\State $S_D^* \gets \Sigma^* (V^*)^T$ 
\State \Return $R S_D^*$
\EndProcedure
\end{algorithmic}
\end{algorithm}

Algorithm \ref{alg:optS} describes the procedure to construct the minimzer of the objective function where the problem follows Assumptions \ref{as:exact_D_sig} and \ref{as:pow}. The first six lines transform the input to align with Assumption \ref{as:D}. By Theorem \ref{thm:SVD_Opt} we have that $U^* =I$ leaving only the computation of $\Sigma^*$ and $V^*$. We obtain $\Sigma^*$ via the function call $\text{Get}\Sigma^*(D, \sigma, h)$, given by Algorithm \ref{alg:get_sigma_star}, to obtain $\Sigma^*$. To transform the solution to the original space of $H$ we use Proposition \ref{prop:rot_invariant} to perform a change of basis on $S_D^*$ which yields $S^*$.

The runtime of this algorithm is dominated by the eigendecomposition of $H$, which has runtime $\mathcal{O}(d^3)$. The computation of $\Sigma^*$ is done in linear time by Proposition \ref{prop:sig_alg}. The Hadamard matrix is computed via the Sylvester's Construction which has runtime $\mathcal{O}(d^2)$. 

While the runtime of CASG is larger than the linear runtime of many finite difference methods, we note that for numerical differentiation the assumption is that function evaluations are expensive. Hence, we trade computational speed for accuracy in gradient estimates and hence fewer function evaluations.

\subsection{Intuition and Initial Numerics}\label{sec:intuition}
\begin{figure}
\begin{minipage}[c]{0.32\textwidth}
    \centering
    \resizebox{\textwidth}{!}{\input{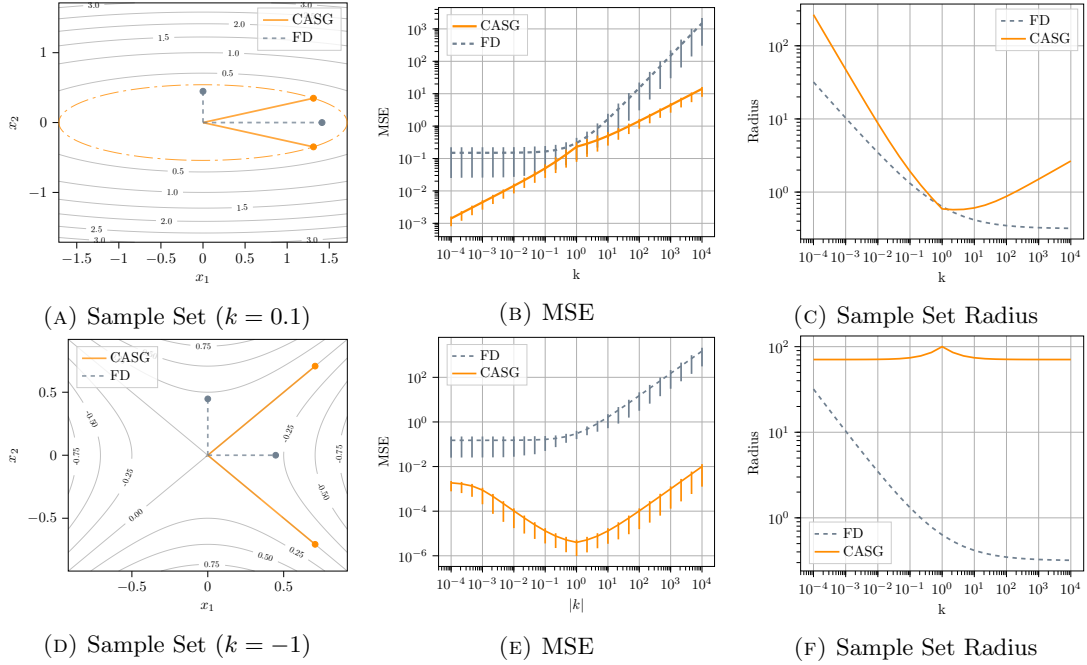}}
    \subcaption{Sample Set ($k=0.1$)}
    \label{fig:vis_2d_0.1}
\end{minipage}\hfill
\begin{minipage}[c]{0.32\textwidth}
    \centering
    \resizebox{\textwidth}{!}{\input{figs/Intro/sig_0.1_num_1000_2d_Error}}
    \subcaption{MSE}
    \label{fig:err_2d}
\end{minipage}\hfill
\begin{minipage}[c]{0.32\textwidth}
    \centering
    \resizebox{\textwidth}{!}{\input{figs/Intro/sig_0.1_num_1_2d_S_scale}}
    \subcaption{Sample Set Radius}
    \label{fig:radius_2d}
\end{minipage}\hfill

\begin{minipage}[c]{0.32\textwidth}
    \centering
    \resizebox{\textwidth}{!}{\input{figs/Intro/neg_a_1_sig_0.1_2d}}
    \subcaption{Sample Set ($k=-1$)}
    \label{fig:vis_2d_neg_1}
\end{minipage}\hfill
\begin{minipage}[c]{0.32\textwidth}
    \centering
    \resizebox{\textwidth}{!}{\input{figs/Intro/sig_0.1_num_1000_2d_Error_negative_h_100.0}}
    \subcaption{MSE}
    \label{fig:err_2d_neg_a}
\end{minipage}\hfill
\begin{minipage}[c]{0.32\textwidth}
    \centering
    \resizebox{\textwidth}{!}{\input{figs/Intro/sig_0.1_num_1_2d_S_scale_negative_h_100.0}}
    \subcaption{Sample Set Radius}
    \label{fig:radius_2d_neg_a}
\end{minipage}
\caption{Comparing forward differences (FD) with optimal difference vector lengths \cite{burden2015numerical} to the curvature aligned simplex gradient (CASG). The first row corresponds to a positive definite Hessian with $k \in [10^{-4}, 10^4]$. The bottom row represents an indefinite Hessian with $k \in [-10^4, -10^{-4}]$ and with $h=10^2$.}
\label{fig:intro_demo}
\end{figure}

We now present some initial numerics to demonstrate the benefit of CASG compared to forward differences on a toy example. To compare the methods, we estimate the gradient at the origin of a quadratic function corrupted by additive noise. Specifically, we consider the function $\tilde{f}_k(x, y) = k x^2 + y^2 + \epsilon$, where $k \in \mathbb{R}$ is used to control the conditioning of the Hessian and $\epsilon \sim \mathcal{N}(0, 0.01^2)$. We provide the exact Hessian and noise level to both CASG and FD. The mean squared error is then exactly given by the objective value.

To obtain the error bars for the mean squared error plots in Figures \ref{fig:err_2d} and \ref{fig:err_2d_neg_a} we perform 1,000 gradient estimates to then plot the 25th and 75th percentile values. The contour lines of the quadratic form given by the Hessian of the noiseless quadratic are depicted in Figures \ref{fig:vis_2d_0.1} and \ref{fig:vis_2d_neg_1} and represent the error of a directional gradient estimate using finite differences. In Figures \ref{fig:radius_2d} and \ref{fig:radius_2d_neg_a} we plot the average of the lengths of the two difference vectors.

When the Hessian is positive definite, we see in Figure \ref{fig:vis_2d_0.1} that CASG aligns the difference vectors along low curvature directions to achieve an approximation error comparable to FD while increasing the difference vector length to decrease the noise error. We confirm that this happens consistently for $H$'s with different condition numbers in Figure \ref{fig:radius_2d}. The overall improvement is apparent in Figure \ref{fig:err_2d} where the mean squared error is magnitudes lower whenever the Hessian is ill-conditioned and the method can exploit low curvature directions. 

When the Hessian is indefinite, in addition to exploiting low curvature directions we observe that when aligning the difference vectors close to the zero-set lines, as seen in Figure \ref{fig:vis_2d_neg_1}, the approximation error is close to zero. Increasing the length of the difference vectors will then decrease the noise error without increasing the approximation error drastically. In Figure \ref{fig:err_2d_neg_a} we see that the mean squared error is lowest when $k = -1$, which happens when the difference vectors are fully aligned with the zero-set lines. This is because there is no approximation error and the orthogonal layout of the difference vectors is optimal for the noise error. In the case $k=-1$, CASG also has the same performance as central differences, as neither will suffer from the approximation error.

\subsection{Proof of Theorem \ref{thm:SVD_Opt}}
By Proposition \ref{prop:rot_invariant}, it suffices to rotate a minimizer of $l_{D, \sigma, h}$ by $R$ to obtain a minimizer for $l_{H, \sigma, h}$. It remains to prove the result for diagonal matricies $D$. 

For the proof of Theorem \ref{thm:SVD_Opt} we need the following Propositions. We defer the proof of those to Appendix \ref{app:prop_main_thm_proof}. For the following results let Assumptions \ref{as:exact_D_sig}-\ref{as:pow} hold.

\begin{proposition}\label{prop:obj_lower_bound}
    For a diagonal matrix $D \in \mathbb{R}^{d \times d}$ and positive real numbers $\sigma, h \in \mathbb{R}_+$ define the function 
    \begin{align}
        \underline{l}_{D, \sigma, h}(U, \Sigma) =
    \begin{cases}
    \frac{a_{D, \sigma}(U, \Sigma)^2}{4d \Sigma_{max}^2} + \sigma^2 \sum_{i = 1}^d \frac{1}{\Sigma_i^2} + \sigma^2 \frac{d}{\Sigma_{max}^2} & \text{if } h I \succeq \Sigma \succ 0 \\
    \infty & \text{otherwise,}
    \end{cases}
    \end{align}
    with orthogonal $U \in \mathbb{R}^{d \times d}$, diagonal $\Sigma \in \mathbb{R}^{d \times d}$, and $a_{D, \sigma}(U, \Sigma) = \text{tr}(\Sigma U^T D U \Sigma)$.
    Then we have 
    \begin{align*}
        \underline{l}_{D, \sigma, h}(U, \Sigma) \leq \min_{V} l_{D, \sigma, h}(U, \Sigma, V).
    \end{align*}
\end{proposition}

\begin{proposition}\label{prop:Uopt}
The global minimizers $U'$ and $\Sigma'$ of $\underline{l}_{D, \sigma, h}(U, \Sigma)$ are given by $U'$ being the identity matrix and $\Sigma'$ being the solution to $\min_{S} \underline{l}_{D, \sigma, h}(I, S)$. The minimization problem for $S'$ is equivalent to the one given in Theorem \ref{thm:SVD_Opt}.
\end{proposition}

\begin{proposition}\label{prop:optV}
    Let $U'$ and $\Sigma'$ be the minimizers of $\underline{l}_{D, \sigma, h}(U, \Sigma)$. Letting $V'$ be the Hadamard matrix $M_{k}$, where $k = \argmax_{i} \Sigma_i$, we have that $l_{D, \sigma, h}(U', \Sigma') = l_{D, \sigma, h}(U', \Sigma', V')$.
\end{proposition}

\begin{proof}[Proof of Theorem \ref{thm:SVD_Opt}]
Notice that $U', \Sigma'$ given by Proposition \ref{prop:Uopt} and $V'$ given by Proposition \ref{prop:optV} have the same form as $U^*, \Sigma^*$ and $V^*$ in Theorem \ref{thm:SVD_Opt}. Hence, it remains to show that they are in fact the global minimizers of $l_{D, \sigma, h}$. Consider the following chain of inequalities. 
\begin{align*}
    l_{D, \sigma, h}(U', \Sigma', V') &\geq l_{D, \sigma, h}(U^*, \Sigma^*, V^*) \\
    &= \min_{U, \Sigma} \min_{V} l_{D, \sigma, h}(U, \Sigma, V) \\
    &\geq \min_{U, \Sigma} \underline{l}_{D, \sigma, h}(U, \Sigma) \quad \text{by Proposition \ref{prop:obj_lower_bound}}\\
    &= \underline{l}_{D, \sigma, h}(U', \Sigma').
\end{align*}
Due to Proposition \ref{prop:optV} we know $l_{D, \sigma, h}(U', \Sigma', V') = \underline{l}_{D, \sigma, h}(U', \Sigma')$. Hence the above chain of inequalities becomes a chain of equalities and specifically $l_{D, \sigma, h}(U', \Sigma', V') = l_{D, \sigma, h}(U^*, \Sigma^*, V^*)$. Hence $U', \Sigma'$ and $V'$ are also the global minimizers of $l_{D, \sigma, h}$. Rotating the result by $R$ completes the proof and provides a minimizer of $l_{H, \sigma, h}$.
\end{proof}

\subsection{Proof of Proposition \ref{prop:sig_alg}}
We now prove Proposition \ref{prop:sig_alg} and Algorithm \ref{alg:get_sigma_star} which computes a global minimizer $\Sigma^*$. To keep notation tight, we let $\lambda \in \mathbb{R}^d$ represent the diagonal of the square of $\Sigma$, i.e. $\lambda = diag(\Sigma^2)$. 

To develop an efficient method to solve for a global minimum we note the following useful properties. 
\begin{proposition}\label{prop:sig_dec}
    $\lambda^*$ is in decreasing order, i.e. $\lambda^*_i \leq \lambda^*_j$ whenever $i \leq j$. Furthermore, $\lambda^*_i = h^2$ for all $1 \leq i \leq d$ such that $D_i \leq 0$.  
\end{proposition}
The proof is in Appendix \ref{app:sig_dec}. 

\begin{remark}
Due to $\lambda^*$ being in decreasing order, we know that $\lambda_{max}^* = \lambda_1^*$ which will allow us to avoid any issues of non-differentiability of the maximum operator.
\end{remark}

We now proceed by presenting a closed-form solution for $\lambda^*$ with known run-time.

\subsubsection{Finding all Candidate Solutions}\label{sec:cand_sol}
We solve the optimization problem by constructing a set of candidate solutions. The candidate solutions compromise the extrema of a set of unconstrained problems which we can solve analytically. We now formalize this. All proofs for this section are in Appendix \ref{app:active_sets}.

Let $A \in \mathcal{P}(\{1, \dots, d\})$ be an \textbf{active set}. $A$ will contain all the indices of $\lambda$ which will be set to $h^2$. We then write the 
\begin{align*}
    \lambda^{(A)}_i &=
    \begin{cases}
        h^2 & \text{if $i \in A$,} \\
        \lambda_i & \text{otherwise.}
    \end{cases}
\end{align*}
Using Proposition \ref{prop:sig_dec} to replace $\lambda_{max}$ with $\lambda_1$ we let the unconstrained function $l_{D, \sigma, h, A} : \mathbb{R}_+^d \rightarrow \mathbb{R}$ be
\begin{align*}
    l_{D, \sigma, h, A}(\lambda) := 
        \frac{(\sum_{i = 1}^d D_i \lambda^{(A)}_i)^2}{4 d \lambda^{(A)}_{1}} + \sigma^2 \sum_{i =0}^d \frac{1}{\lambda^{(A)}_i} + \sigma^2 \frac{1}{\lambda^{(A)}_{1}} d.
\end{align*}
We write $l_A(\lambda)$ instead of $l_{D, \sigma, h, A}(\lambda)$ when clear from context.

Since the function $l_A$ is differentiable on its domain, the stationary points are well defined. We denote the set of stationary points associated with the active set $A$ by $\lambda^*_{A} = \{\lambda^{(A)} : \nabla l_A(\lambda) = 0\}$. 

We now show that given the above reformulation there exists an active set $A$ such that the global solution $\lambda^*$ is contained in $\lambda_A^*$. Additionally, instead of computing solutions for $2^d$ possible active sets, we need to check at most $d$ active sets due to Proposition \ref{prop:sig_dec}. 
\begin{proposition}\label{prop:active_sets_solution_exists}
Let $K = \{i \vert D_i \leq 0\}$ and $A_J = \{1, \dots, J\}$. Then there exists $J \in \{\vert K \vert, \dots, d\}$ such that $\lambda^* \in \lambda^*_{A_J}$.
\end{proposition}

We now need to show that finding $\lambda^*_A$ can be done efficiently. For this we first present the following. 
\begin{proposition}\label{prop:FONC}
    Let $K = \{i \vert D_i \leq 0\}$ and $A$ be some active set such that $K \subseteq A$. Then the set $\lambda_A^*$ consists of all $\lambda$ which satisfy
    \begin{align*}
    \lambda_{1} &= 
    \begin{cases}
        h^2 & \text{if } 1 \in A, \\
        \frac{2d}{a D_{1}}\left(\frac{a^2}{4d} + \sigma^2 (d + 1) \right) & \text{otherwise;}\\
    \end{cases}\\
    \lambda_i &= 
    \begin{cases} 
        h^2 & \text{if } i \in A, \\
        \sigma \sqrt{\frac{2 d \lambda_{1}}{a D_{j}}} & \text{otherwise,} \\
    \end{cases} & \text{for } 2 \leq i \leq d, \\
    a &= \sum_{i \not \in A} D_i \lambda_i + h^2 \sum_{i \in A} D_i.
\end{align*}
\end{proposition}

We now present the unique solution for $a$ for a given active set, which then allows us to immediately compute the unique candidate solution given by Proposition \ref{prop:FONC}.
\begin{proposition}\label{prop:solve_a}
Let $K = \{i \vert D_i \leq 0\}$ and $A_J = \{1, \dots, J\}$. Then the $a$ which satisfies the equations given in Proposition \ref{prop:FONC} is unique and is given by
\begin{itemize}
\item If $A_J$ is empty, then we have
    \begin{align*}
        a = \sqrt{2} \left(\frac{d \sigma^2}{D_{1}} \left(C \sqrt{8 D_{1} (d + 1) + C^2} + 2D_{1}(d + 1) + C^2 \right)\right)^{1/2}
    \end{align*}
    with $C = \sum_{i = 2}^d \sqrt{D_{i}}$;
\item Otherwise, 
    \begin{align*}
        a &= 
        \begin{cases}
            \frac{\sqrt[3]{\frac{2}{3}} c_2}{\sqrt[3]{\sqrt{3} \sqrt{27 c_1^2-4 c_2^3}+9 c_1}}+\frac{\sqrt[3]{\sqrt{3} \sqrt{27 c_1^2-4 c_2^3}+9 c_1}}{\sqrt[3]{18}} & \text{if } 27 c^2_1 - 4c_2^3 \geq 0, \\
            \cos(\theta/3) 2 \sqrt{\frac{c_2}{3}} & \text{otherwise,}
        \end{cases}
    \end{align*}
    where $c_1 = \sigma \sqrt{\frac{d h^2}{2}} \sum_{i = J + 1}^d \sqrt{D_{i}}$, $c_2 = \sum_{i=1}^J D_{i} h^2$ and $\theta = \text{arccos}(\frac{9 c_1}{\sqrt{12 c_2^3}})$.
\end{itemize}
\end{proposition}

\begin{algorithm}[h]
\caption{Get$\Sigma^*$}\label{alg:get_sigma_star}
\begin{algorithmic}[1]
\Procedure{Get$\Sigma^*$}{$D, \sigma, h$}
\State $d \gets \text{Dimension of } D$
\If{$\sum_{i=1}^d D_i = 0$}
    \State \Return $h$
\EndIf
\State $J \gets \vert \{i | D_i \leq 0, 1 \leq i \leq d\} \vert$ 
\State $c_1, c_2 \gets \sum_{i = J + 1}^{d} \sqrt{D_i}, \sum_{i = 1}^J D_i$
\While{True}
    \State $\lambda_{J + 1}, a \gets \text{Get}\lambda_{J + 1}(J, D, \sigma, h, c_1, c_2)$
    \State $c_1, c_2 \gets c_1 - \sqrt{D_{J + 1}}, c_2 + D_{J + 1}$ 
    \If{$\lambda_{J + 1} \leq h^2$ and $J < d$}
        \State break
    \EndIf
    \State $J \gets J + 1$
\EndWhile 
\State $\lambda_i \gets h^2$ \Comment{For $1 \leq i \leq J$.}
\State $\lambda_j \gets \sigma \sqrt{\frac{2 d h^2}{a D_j}}$ \Comment{For $J + 1 \leq j \leq d$.}
\State \Return $\sqrt{\lambda}$ 
\EndProcedure
\end{algorithmic}
\end{algorithm}

\begin{algorithm}[h]
\caption{Get $\lambda_{J+1}$}\label{alg:get_lambda_J}
\begin{algorithmic}[1]
\Procedure{Get$\lambda_{J + 1}$}{$J, D, \sigma, h, c_1, c_2$}
\State $d \gets \text{Dimension of } D$
\If{$J = 0$}
    \State $J \gets c_1 - \sqrt{D_1}$
    \State $a \gets \sqrt{2} (\frac{d \sigma^2}{D_{1}} (J \sqrt{8 D_{1} (d + 1) + J^2} + 2D_{1}(d + 1) + J^2 ))^{1/2}$
    \State $\lambda_1 \gets \frac{2d}{a D_{1}}\left(\frac{a^2}{4d} + \sigma^2 (d + 1) \right)$
\Else
    \State $c_1, c_2 \gets \sigma \sqrt{\frac{d h^2}{2}}c_1, h^2 c_2$
    \State $disc \gets 81 c_1^2 - 12 c_2^3$
    \If{$disc < 0$} 
        \State $\theta \gets \cos^{-1}(9 \frac{c_1}{\sqrt{12 c_2^3}})$
        \State $a \gets \frac{4}{3} \cos(\theta/3)^2 c_2 $ 
    \Else 
        \State $a \gets \frac{\sqrt[3]{\frac{2}{3}} c_2}{\sqrt[3]{\sqrt{disc} + 9 c_1}} + \frac{\sqrt[3]{\sqrt{disc} + 9 c_1}}{\sqrt[3]{18}}$
    \EndIf
    \State $\lambda_{J + 1} \gets \sigma \sqrt{\frac{2 d h^2}{a D_{J + 1}}}$
\EndIf
\State \Return $\lambda_{J + 1}, a$
\EndProcedure
\end{algorithmic}
\end{algorithm}

\subsubsection{Constructing $\lambda^*$ in Linear Runtime} 
While Propositions \ref{prop:FONC} and \ref{prop:solve_a} tell us how to construct solutions for an active set, it remains to select the global minimum of all possible solutions. The naive implementation of finding $\lambda^*$ has runtime $\mathcal{O}(d^2)$, since there are $d$ active sets to consider by Proposition \ref{prop:active_sets_solution_exists} and constructing $\lambda_A^*$ for each $A$ is also of linear runtime. We can improve the runtime to be linear. While this will not improve the theoretical runtime of the method, in practice we have seen it be helpful. 

To construct $\lambda^*$ in linear time, we observe the following proposition. The Algorithm \ref{alg:get_sigma_star} then uses the result as given by Algorithm \ref{alg:get_lambda_J} to compute $\lambda^*$ in linear time. Since by Proposition \ref{prop:FONC} and \ref{prop:solve_a} the candidate solution for an active set is unique, we let $\lambda^*_{A}$ represent the solution rather than the solution set. 

\begin{proposition}\label{prop:linear_runtime}
Let $K = \{i \vert D_i \leq 0\}$ and $J$ be the smallest integer in the set $\{\vert K \vert, \dots, d\}$ such that $\lambda_{A_J}^*$ satisfies all constraints, i.e. $h^2 \succeq \lambda_{A_J}^*$. Then it follows that $l(\lambda^*) = l(\lambda_{A_J}^*)$. Additionally, if $J \geq 1$ the comparison $h^2 \succeq \lambda_{A_J}^*$ can be done in constant time.
\end{proposition}
The proof is in Appendix \ref{app:linear_runtime}.

\section{extended Curvature Aligned Simplex Gradient (eCASG) }\label{sec:subspace}
Theorem \ref{thm:SVD_Opt} only applies for dimensions which are powers of two. Hence, we introduce a heuristic method which subdivides the full space such that we can use CASG via Theorem \ref{thm:SVD_Opt}. First we describe how to partition the subspace and how that may lead to a useful decomposition of the objective function. We then describe a method by which to choose the partition such that we can exploit the results for spaces with dimensions of powers of two and which aligns with the intuition developed in Section \ref{sec:intuition}.

\paragraph{Partition Set-up} Let $\{B_i\}_{1 \leq i \leq k}$ be a partition of the standard basis $\{e_i\}_{1 \leq i \leq d}$ into $k$ sets, with each $B_i$ being a cell of the partition represented as a matrix. We let $B$ be the concatenation of the cells, $B = [B_1 \quad B_2 \quad \dots \quad B_k]$. Since $B$ is simply a reordering of the standard basis, the matrix is orthogonal. Using this partition, we enforce that each difference vector is spanned by one of the cells. We can now write $S = B S_B$, where the $i$th columns of $S_B$ are the coefficients of the difference vector $s_i$ in the $B_i$ basis. It follows that $S_B$ is a block-diagonal matrix, with the $i$th block representing the coefficients of the differences vectors spanned by cell $B_i$. It immediately follows that, if it exists, the matrix $S^{-1}_B$ is also a block diagonal matrix, with the $i$th block of $S_B^{-1}$ being equal to $S^{-1}_{B_i}$. 

\begin{proposition}\label{prop:eCASG}
    Let Assumptions \ref{as:exact_D_sig} and \ref{as:D} hold. Given a fixed partition $\{B_i\}_{1 \leq i \leq k}$ and a block-diagonal matrix $S_B$ with $S_{B_i}$ being the $i$th block providing the coordinates of the difference vectors in space spanned by $B_i$, we have $l_{D, \sigma, h}(B S_B) = \sum_{i = 1}^k l_{D_{B_i}, \sigma, h}(S_{B_i})$ with $D_{B_i} = B_i D B_i^T$. 
\end{proposition}
\begin{proof}
If $S_B$ is not invertible then at least one block will not be invertible and hence the two sides of the equation are trivially equal as they equal $\infty$. When $S_B$ is invertible the result follows immediately from the block structure of $S_B$ and orthogonality of $B$. 
\end{proof}

The Proposition allows us to minimize with respect to the block-diagonal matrix $S_B$. Letting each cell in the partition have dimension power of two, we can use Theorem \ref{thm:SVD_Opt} to find the optimal solution by applying CASG to each cell. While the minimizer may have a larger function value than minimizing $l_{D, \sigma, h}(S)$ directly, we have empirically found that carefully choosing $B$ leads to good gradient estimates. We now describe our heuristic method of choosing such a partition. 
\begin{algorithm}[h]
\caption{Subdivide Space (eCASG)}\label{alg:H_Subdivision}
\begin{algorithmic}[1]
\Procedure{Subdivide Space}{$D$} \Comment{$D$ is sorted in increasing order}
\State $d \gets \text{Dimension of } D$
\State $d_{bin} \gets \text{binary coefficients of } d$ \Comment{Least significant digit last} 
\State $start, end \gets 0, d$
\State $i \gets d_{bin}.length - 1$ \Comment{Cell index starts with most significant digit}
\State $B \gets \text{List of cells}$
\While{$start \leq end$}
    \If{$d_{bin}[i] = 1$  and \text{cell $B[i]$ is not full}}
        \If{$i = 0$}
            \State $B[i].add(D[start])$
            \State $start \gets start + 1$
        \Else
            \State $B[i].add(D[start])$
            \State $B[i].add(D[end])$
            \State $start \gets start + 1$
            \State $end \gets end - 1$
        \EndIf
    \EndIf
    \State $i \gets (i - 1) \text{ mod } d_{bin}.length$
\EndWhile
\State \Return B
\EndProcedure
\end{algorithmic}
\end{algorithm}

\paragraph{Choosing the Partition}
To use Theorem \ref{thm:SVD_Opt}, we enforce that the dimension of the span of each cell be a power of two. We do so by using a binary expansion of the dimension $d = \sum_{i = 1}^{\lfloor \log(d) \rfloor} c_i 2^i$, and adding a cell of size $2^i$ for every $c_i = 1$. It now remains to describe the subspace each cell will span. 

As developed in Section \ref{sec:intuition}, we use the intuition that when negative and positive curvature directions exist, they are paired together such that CASG can exploit zero-set lines. When only positive or negative curvature regions remain, we pair high with low curvature directions to allow CASG to align the sample set along the low curvature direction.

We describe the full eCASG method in Algorithm \ref{alg:H_Subdivision}. Since the Hessian is assumed to be diagonal, the curvature directions, i.e. eigenvectors, are given by the standard basis vectors. We proceed greedily by iterating through the partition cells, $B_i$, and adding the direction of highest curvature and the direction of lowest curvature from the set of possible directions. We then remove these directions from the set of possible directions and move to the next cell and repeat.

\section{Global Model Framework and Numerical Results}\label{sec:Numerics}

\subsection{Global Model Framework}
We now discuss the global model framework for obtaining the Hessian needed for computing CASG. The idea behind the framework is simple. In many applications, the function will be evaluated multiple times in various regions of the domain. While the information we obtain may not be accurate enough for a gradient estimate, it may still be valuable to guide us in selecting the right sample set for a local gradient estimate. In our framework we formalize this via Algorithm \ref{alg:global}. For the global model framework several input variables are needed. One is the evaluation history set $\mathcal{E}$ as described in Section \ref{sec:setup}. Then we need a model space and a model fitting function. We let the fitting function be given by $\Phi(\mathcal{E}) \rightarrow \phi$ which maps evaluation history set to a global model $\phi : \mathbb{R}^d \rightarrow \mathbb{R}$. In many applications we would not like to use the complete evaluation history, as some points may be poorly spread out for the chosen model space and fitting function. Hence, we allow for a filter function on the evaluation history $\mathcal{F}(\mathcal{E}) \rightarrow \mathcal{E}_{\mathcal{F}}$ where  $\mathcal{E}_{\mathcal{F}} \subseteq \mathcal{E}$. Lastly, we denote by $\mathcal{G}$ the gradient estimator which maps a Hessian estimate and the current evaluation point to a gradient estimate and also returns the function evaluations which happened as a consequence of estimating the gradient. We return the gradient estimate and the new function evaluation history.

\begin{algorithm}
\caption{Global Model Framework}\label{alg:global}
\begin{algorithmic}[1]
\Procedure{Global Model Framework}{$\mathcal{E}, \Phi, \mathcal{F}, \mathcal{G}, x_0$} 
\State $\mathcal{E}_{\mathcal{F}} \gets \mathcal{F}(\mathcal{E})$ \Comment{Filter evaluation history}
\State $\phi \gets \Phi(\mathcal{E}_{\mathcal{F}})$ \Comment{Fit global model}
\State $H_{global} = \nabla^2 \phi(x_0)$
\State $g, \mathcal{E}_{curr} \gets \mathcal{G}(H_{global}, x_0)$ \Comment{Grad estimate}
\State $\mathcal{E} \gets \mathcal{E} \cup \mathcal{E}_{curr}$
\Return \State $g, \mathcal{E}$
\EndProcedure
\end{algorithmic}
\end{algorithm}

While the framework requires the choice of a model space, fitting function, and an evaluation history filter, such parameters often arise from the application domain as we will see in the following applications. 

A natural question is that if the global model is good enough as a Hessian estimator, should it not also be a good gradient estimator. We will empirically demonstrate this to be not true. The intuition is that a gradient estimator is allowed to sample locally, while the global model will generally not have these local evaluations present in the evaluation history and hence be less accurate close to the evaluation point.

\subsubsection{Global Model: Cubic Spline}
For our experiments we found the cubic spline radial basis function to be a good global model. We use the implementation by \Verb+scipy+ \cite{2020SciPy-NMeth} and extend it to allow for the analytic computation of the gradient and Hessian once the model has been fit. When some of the points in the evaluation history are close together, and the fit needs to be regularized, we also used the smoothing parameter that is part of the fitting procedure. We will specify when we use the smoothing parameter.

\subsection{Sensitivity Analysis}
Sensitivity analysis is used to understand how the response of a system depends on small changes in the input. Often this is done by computing the gradient locally \cite{qian2020sensitivity}, which provides an ideal application area under which to assess CASG and eCASG. To replicate a corrupted system, we add additive Gaussian noise of which we assume to know the noise level as is common for assessing gradient estimation methods \cite{shi2022adaptive, berahas2019derivative, brekelmans2005gradient}.

We compare our method against forward (FD) and central (CD) differences and the gradient estimate given by the global model, i.e. $\nabla \phi$. For comparison, our method and FD will be given both global model information and exact Hessian information. Forward differences will be used as given in \cite{berahas2022theoretical}.  We use CD where the difference vector lengths are set to the step-size parameter $h$ \cite{burden2015numerical}. This is a fair comparison as CD has no second order error and is used as a best case comparison. For our CASG, eCASG, FD and CD, the step-size hyperparameter $h$ is found by a hyperparameter search. The global model gradient estimate is attained by simply differentiating the global model.

To assess the goodness of the gradient estimate by our method, FD and CD, we use the exact version of the mean squared error as given in Proposition \ref{prop:mse_noise}. While Proposition \ref{prop:mse_noise} is stated only for the standard simplex gradient, it can be immediately extended for central differences. The gradient estimate from the global method is assessed by considering the squared norm of the difference to the true gradient. For the global model there is no noise component as the randomness is only due to the function evaluation history, which we assume to be given for any local gradient estimate. 

The main results are given in Figures \ref{fig:Ackley_sens} and \ref{fig:Cancer_sens}. For Figures \ref{fig:exact_Ackley}, \ref{fig:RBF_Ackley}, \ref{fig:exact_Cancer} and \ref{fig:RBF_Cancer} we chose 100 points uniformly at random from the domain and for each gradient estimation method computed the log ratio as in \cite{more2009benchmarking} between the respective method and the given version of CASG, i.e. $\log_2(\frac{MSE_{\text{method}}}{MSE_{\text{CASG}}})$. 

To make the cost of using a function evaluation history more explicit we also plot the error of each method against the number of points in $\mathcal{E}$. Specifically, for Figures \ref{fig:Nerrs_Ackley} and \ref{fig:Nerrs_Cancer} the global method was fitted with a variable number of points as given by the x-axis. The error-bars represent the 25th and 75th percentile of the MSE values at the 100 randomly chosen points, with the line being the median.

For the global model we use no-smoothing, as the function evaluation points will be sufficiently spread out to mitigate the effect of noise. 

\subsubsection{Ackley function}
The Ackley function is a common test function due to its high degree of non-linearity and is given by $x_i \in [-0.5, 0.5]$
\begin{align*}
    f(x_0 \cdots x_n) &= -20 \exp(-0.2 \sqrt{\frac{1}{n} \sum_{i=1}^n x_i^2}) - \exp(\frac{1}{n} \sum_{i=1}^n \cos(2\pi x_i)) + 20 + e.
\end{align*}

For the experiment we set the dimension to $8$ and the noise level to $\sigma = 10^{-5}$. The set of step-sizes we consider as hyperparameter is $\{0.1, 0.05, 0.01\}$. For the RBF we sample $10^4$ points for the global model for the box plots. Since the dimension is $8$ we can use CASG directly. 
\begin{figure}[h]
\begin{minipage}[t]{0.32\textwidth}
    \centering
    \resizebox{\textwidth}{!}{\input{figs/Sensitivity/Box_Ackley_1693402667.6911428_H_exact_True}}
    \subcaption{CASG with Exact Hessian}
    \label{fig:exact_Ackley}
\end{minipage}\hfill
\begin{minipage}[t]{0.32\textwidth}
    \centering
    \resizebox{\textwidth}{!}{\input{figs/Sensitivity/Box_Ackley_1693402667.6911428_H_exact_False}}
    \subcaption{CASG with RBF Hessian}
    \label{fig:RBF_Ackley}
\end{minipage}\hfill
\begin{minipage}[t]{0.32\textwidth}
    \centering
    \resizebox{\textwidth}{!}{\input{figs/Sensitivity/Nerrs_Ackley_1693402667.6911428}}
    \subcaption{RBF points vs. MSE}
    \label{fig:Nerrs_Ackley}
\end{minipage}\hfill

\caption{Comparison of gradient estimation methods for Ackley with $d=8$ and $\sigma=10^{-5}$.}
\label{fig:Ackley_sens}
\end{figure}
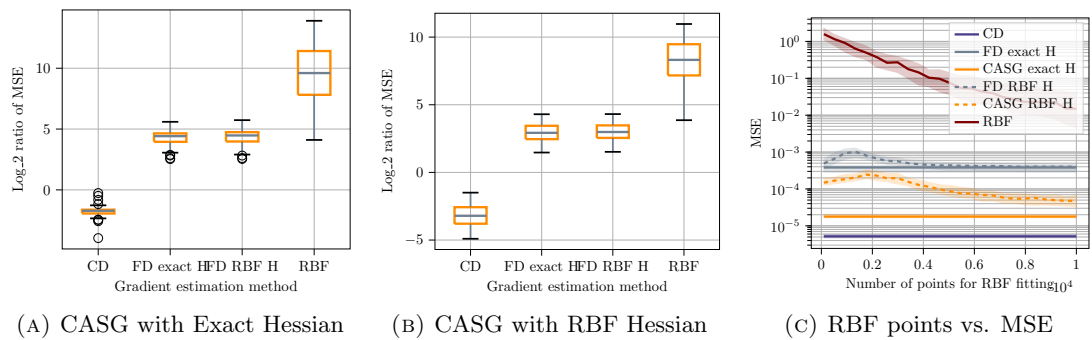

\subsubsection{ODE Model: Cell differentiation in the colon}
We consider a model for cell differentiation in the colon which has been used for sensitivity analysis \cite{qian2020sensitivity, johnston2007examples}. The process tracks three variables and is given by the following dynamics
\begin{align*}
    \frac{dN_0}{dt} &= (\alpha_3 - \alpha_1 - \alpha_2)N_0 - \frac{k_0 N_0^2}{1 + m_0 N_0} \\
    \frac{dN_1}{dt} &= (\beta_3 - \beta_1 - \beta_2) N_1 + \alpha_2 N_0 - \frac{c_1 N_1^2}{1 + m_1 N_1} + \frac{k_0 N_0^2}{1 + m_0 N_0} \\
    \frac{dN_2}{dt} &= -\gamma N_2 + \beta_2 N_1 + \frac{c_1 N_1^2}{1 + m_1 N_1}.
\end{align*}
In the dynamics, the time variable $t$ is in units of days, $N_0$ represents the stem cell population, $N_1$ the semi-differentiated cell population and $N_2$ the fully differentiated cell population. For the dynamics we fix the initial conditions to be $N_0=1$, $N_1=100$, $N_2=100$ and consider a time-span of $100$ days. The remaining coefficients are the independent variables with which we perform the sensitivity analysis. Their interpretation and recommended values are given in Table 11 of \cite{qian2020sensitivity}. $N_0$, the stem cell population, will be the dependent variable against which the sensitivity analysis is performed. 

The domain we use is $\pm 10\%$ of the suggested values listed in \cite{qian2020sensitivity}. To evaluate the differential equation we use the standard one-step explicit Euler scheme with a step size of $0.01$, amounting to $10^4$ steps per equation. To simulate noise in the observation of the final sample, we add noise with $\sigma=10^{-3}$ which we found to be significantly more than the numerical error in simulating the dynamics. To get an understanding of how dynamics can change with different input parameters, we sample choose different parameters from the specified domain and plot them in Figure \ref{fig:Cancer_vis}. 

The dynamics do not have an analytic gradient easily available. Instead, we use central differences with a step-size of $10^{-11}$ to construct a gradient from the process without noise. We found the error due to machine precision and the discretization of the dynamics to be insignificant compared to the artificial noise added. To obtain the true Hessian we apply the same procedure to the estimated gradient values. 
\begin{figure}
\begin{subfigure}[t]{0.32\textwidth}
    \centering
    \resizebox{\textwidth}{!}{\input{figs/Sensitivity/CancerVis0}}
    \label{fig:Cancer_1}
    \subcaption{Stem cell population}
\end{subfigure}\hfill
\begin{subfigure}[t]{0.32\textwidth}
    \centering
    \resizebox{\textwidth}{!}{\input{figs/Sensitivity/CancerVis1}}
    \label{fig:Cancer_2}
    \subcaption{Semi-Differentiated cell population}
\end{subfigure}\hfill
\begin{subfigure}[t]{0.32\textwidth}
    \centering
    \resizebox{\textwidth}{!}{\input{figs/Sensitivity/CancerVis2}}
    \label{fig:Cancer_3}
    \subcaption{Fully Differentiated cell population}
\end{subfigure}\hfill
\caption{Visualization of ODE path with three randomly chosen parameters around the recommended values.}
\label{fig:Cancer_vis}
\end{figure}
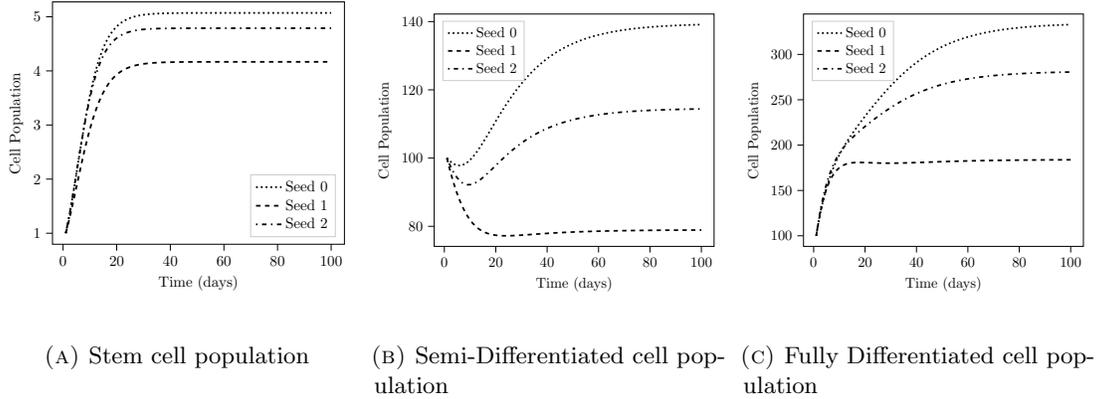
The range of step-sizes for the gradient estimation is $\{3.16 \times 10^{-4}, 10^{-3}, 3.16 \times 10^{-3}, 10^{-2}, 3.16 \times 10^{-2}, 10^{-1}\}$. Since the number of variables is eleven, we use eCASG as described in Section \ref{sec:subspace}. For Figure \ref{fig:RBF_Cancer} the global model is fitted with $1,000$ points. 
\begin{figure}[h!]
\begin{minipage}[t]{0.32\textwidth}
    \centering
    \resizebox{\textwidth}{!}{\input{figs/Sensitivity/Box_Cancer_1693393050.1899579_H_exact_True}}
    \subcaption{eCASG using True Hessian}
    \label{fig:exact_Cancer}
\end{minipage}\hfill
\begin{minipage}[t]{0.32\textwidth}
    \centering
    \resizebox{\textwidth}{!}{\input{figs/Sensitivity/Box_Cancer_1693393050.1899579_H_exact_False}}
    \subcaption{eCASG using RBF Hessian}
    \label{fig:RBF_Cancer}
\end{minipage}\hfill
\begin{minipage}[t]{0.32\textwidth}
    \centering
    \resizebox{\textwidth}{!}{\input{figs/Sensitivity/Nerrs_Cancer_1693393050.1899579}}
    \subcaption{Varying RBF Points}
    \label{fig:Nerrs_Cancer}
\end{minipage}\hfill
\caption{Comparison of gradient estimation methods for sensitivity analysis of cell differentiation in the colon with $\sigma=10^{-3}$.}
\label{fig:Cancer_sens}
\end{figure}
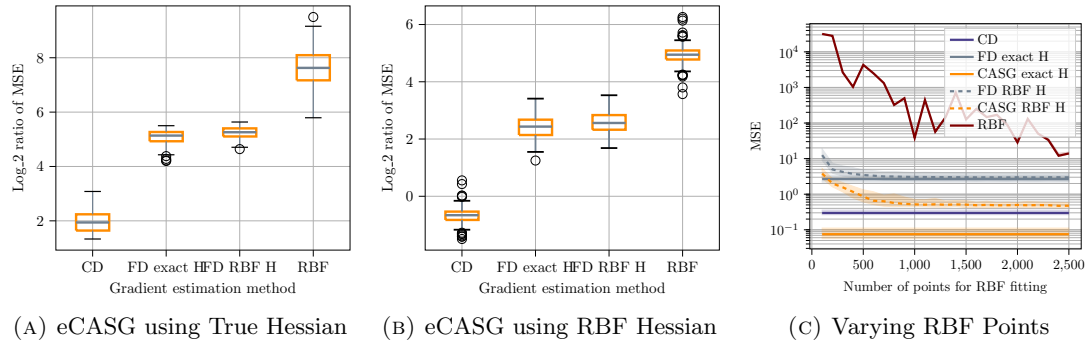

\subsubsection{Sensitivity Analysis: Discussion}

Our method outperforms the forward difference based methods and the gradient estimates by the global model by up to an order of magnitude. More interestingly, our method appears to perform close to central differences. For the ODE problem, eCASG with the exact Hessian even outperforms central differences. 

While outperforming central differences seems like a contradiction, as it supposed to be a best-case comparison, for some problems it may be the case that the layout of the difference vectors is not only good for a second-order approximation, but also a third order approximation. Hence, it may be that the specific sample set geometry also reduces the third-order error and hence improves upon central differences. Understanding how to purposefully distribute the sample set for a central-difference type gradient estimator is left for future work. Overall it is promising that the optimal sample set can achieve a comparable error as central difference with half the number of samples.

We now note the dependence on the number of sample points. While the box-plots are given a global model with $10,000$ and $1,000$ points, this corresponds to roughly $300$ and $90$ gradient estimates respectively. When a good understanding of the sensitivity of the system is needed, these are numbers which can be easily achieved. Furthermore, as seen in Figure \ref{fig:Nerrs_Ackley} far fewer points suffice to obtain a strong gradient estimate.


\subsection{Derivative Free Optimization}
Another application in which gradient estimation is crucial is derivative free optimization (DFO). To compare the gradient estimation methods, we use the bench-marking methodology developed in \cite{more2009benchmarking} and test on a range of CUTEst problems \cite{gould2015cutest}. For optimization we use a L-BFGS DFO framework from \cite{berahas2019derivative} which allows various gradient estimation methods to be plugged in while keeping everything else fixed. 

Besides the methods considered for Sensitivity Analysis, we also consider the Adaptive Finite Difference (AFD) method \cite{shi2022adaptive}. AFD is based on forward differences and computes the lengths of difference vectors adaptively instead of using a global model. We consider AFD to be the state-of-the-art method for using finite differences in DFO as it has been shown to be more accurate and robust than comparable methods \cite{shi2022adaptive}.

The sole hyperparameter $h$ is set to $0.1$, $0.01$, and  $0.001$. While AFD is designed to not need such a parameter, to ensure fairness, in addition to leaving AFD with the default parameters we also assess it when upper-bounding the difference vectors by $h$ as done for FD and CASG. To account for noise, for each hyperparameter we obtain ten random runs. We then set the optimal hyperparameter to be the one given by the random run which achieved the lowest function value. We then have ten random-runs with the chosen hyperparameter. 

To provide some initial global understanding of the function, at initialization we sample $100 \times d$ points randomly in a cube of side-length two, where $d$ is the dimension of the specific problem. During optimization, we save the function evaluations which happen as part of gradient estimation. While other points are evaluated as part of L-BFGS, adding them to the function evaluation history made the global model perform worse. When fitting the global model, we use the minimum number of $100 \times d$ and $1,000$ data-points from the function evaluations along the optimization path. When using the the global model for gradient estimation, to obtain a function evaluation history we artificially construct a FD sample set with difference vector lengths of $0.1$ after each gradient estimation by the global model. We then add these function evaluation to $\mathcal{E}$. Lastly, the smoothing parameter is set to $0.1$.

In Figure \ref{fig:Opt} we plot the data-profiles \cite{more2009benchmarking} of each method. To create the plots, for each CUTEst problem $P$ we let $f^P_L$ be the lowest average function value achieved by any gradient estimation method for problem $P$, with the average being taken over the corresponding random runs. We then let the convergence criterion for the CUTEst problem be $(f(x^P_0) - f^P_L) \geq (1 - \tau)(f(x^P_0) - f(x^P_k)$, with $x^P_0$ being the starting position for the problem, $x^P_k$ the first point which satisfies the condition and $\tau$ being the tolerance. Notice that each random run has a natural convergence criterion associated via the problem it was executed on. The data-profile then plots for each gradient estimation method the fraction of random runs which achieved their corresponding convergence criterion within a certain number of function evaluations. We use the unit of simplex gradients, which is simply the number of total function evaluations divided by the dimension of the problem.

\begin{figure}[h!]
\begin{minipage}[t]{0.32\textwidth}
    \centering
    \resizebox{\textwidth}{!}{\input{figs/Optimization/Opt_sig_0.1_tau_0.1}}
    \subcaption{$\sigma=10^{-1}$ and $\tau=10^{-1}$}
    \label{fig:opt_sig_0.1_tau_0.1}
\end{minipage}\hfill
\begin{minipage}[t]{0.32\textwidth}
    \centering
    \resizebox{\textwidth}{!}{\input{figs/Optimization/Opt_sig_0.001_tau_0.1}}
    \subcaption{$\sigma=10^{-3}$ and $\tau=10^{-1}$}
    \label{fig:opt_sig_0.001_tau_0.1}
\end{minipage}\hfill
\begin{minipage}[t]{0.32\textwidth}
    \centering
    \resizebox{\textwidth}{!}{\input{figs/Optimization/Opt_sig_1e-05_tau_0.1}}
    \subcaption{$\sigma=10^{-5}$ and $\tau=10^{-1}$}
    \label{fig:opt_sig_1e-05_tau_0.1}
\end{minipage}\hfill

\begin{minipage}[t]{0.32\textwidth}
    \centering
    \resizebox{\textwidth}{!}{\input{figs/Optimization/Opt_sig_0.1_tau_1e-05}}
    \subcaption{$\sigma=10^{-1}$ and $\tau=10^{-5}$}
    \label{fig:opt_sig_0.1_tau_1e-05}
\end{minipage}\hfill
\begin{minipage}[t]{0.32\textwidth}
    \centering
    \resizebox{\textwidth}{!}{\input{figs/Optimization/Opt_sig_0.001_tau_1e-05}}
    \subcaption{$\sigma=10^{-3}$ and $\tau=10^{-5}$}
    \label{fig:opt_sig_0.001_tau_1e-05}
\end{minipage}\hfill
\begin{minipage}[t]{0.32\textwidth}
    \centering
    \resizebox{\textwidth}{!}{\input{figs/Optimization/Opt_sig_1e-05_tau_1e-05}}
    \subcaption{$\sigma=10^{-5}$ and $\tau=10^{-5}$}
    \label{fig:opt_sig_1e-05_tau_1e-05}
\end{minipage}\hfill

\caption{Data profiles for L-BFGS using various gradient estimation methods with noise levels $10^{-1}, 10^{-3}$ and $ 10^{-5}$. The top row are the profiles with $\tau=10^{-1}$, associated with approximate convergence. The bottom row is $\tau=10^{-5}$, under which the convergence criterion is only satisfied if a random run is able to get very close to the lowest achievable function value.}
\label{fig:Opt}
\end{figure}
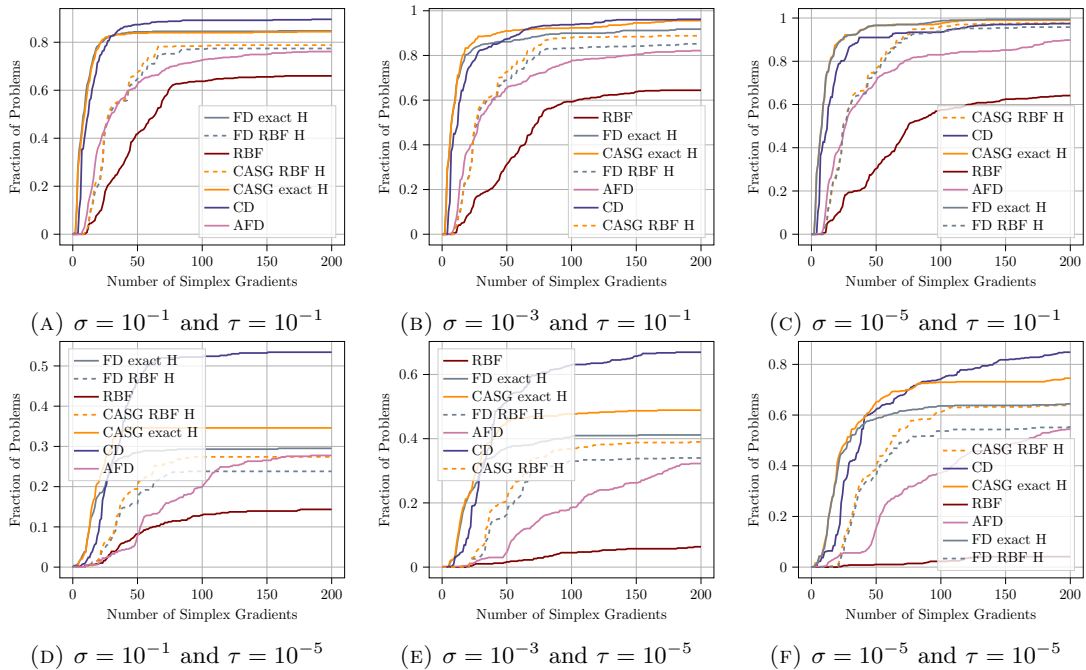


In Figure \ref{fig:Opt} the lines for methods which use the global methods are shifted by ten simplex gradient evaluations due to the initialization data. In some cases the initialization data may not be considered as part of the expense of optimization as the function may have been evaluated at various locations already. In which case the graph would be shifted over by ten simplex gradient evaluations, improving our method in terms of budget required. 

\subsubsection{DFO: Discussion.}
In Figure \ref{fig:Opt} we observe that when $\tau=10^{-1}$, many of the methods perform similarly, as the convergence criterion is relaxed. So, if an approximate solution is required then CASG and FD may suffice. Especially when the initialization data for the global model is not included, the methods using the global model may be favorable, as fewer evaluations than CD may be necessary. 

However, when $\tau=10^{-5}$ more accurate methods perform significantly better. When the true Hessian is used, at first both CASG and FD perform better, since per gradient estimate fewer evaluations are needed. However, there comes a point where the lines intersect with CD, as then the methods can not match the accuracy of CD.

We also observe that CASG with the global model outperforms all other forward difference based models. Specifically, when higher accuracy is needed the gap between the two methods is significant. Additionally, even standard forward differences using our framework tends to outperform AFD, implying that our framework which uses a global model may be more suited for DFO. 

We do note that AFD tends to continue improving after 200 steps. As the other methods have already long obtained convergence, we decided on a comprise to show the initial variation while showing the approximate convergence of AFD. However, AFD does not outperform CASG even when more evaluations are used and tends to eventually become similar to FD using the global model framework. 

Overall, CASG using the true Hessian far outperforms all other gradient estimation methods except for CD once enough function evaluations have been used. In practice only CASG with the global model framework can be used, which also outperforms all comparable methods. Importantly, CASG with the global model performs comparable to FD using the true Hessian. Hence, even if a better finite difference based method is developed, CASG with our framework should provide an upper-bound in performance.

\section{Conclusion}

We studied the optimal sample set for the simplex gradient and introduced a framework which makes our method useful in practice. We showed that the improvement in CASG comes from aligning the sample set along the directions with lowest curvature. We numerically found that CASG consistently outperforms optimal forward differences and hence will improve on any forward difference based methods in practice. Additionally, we found CASG to often be closer in performance to central differences than forward differences, even though central differences needs twice as many function evaluations. 

The implications are that in many applications in which function evaluations are costly, CASG may be used as an immediate improvement to forward difference type schemes, or drop-in for central difference methods, which would otherwise require twice the number of function evaluations. In settings where each function evaluation requires an expensive simulation or completion of a process, reducing the number of evaluations required by a factor of two will be substantial. Overall, CASG and eCASG, in conjunction with our global model framework, are highly competitive gradient estimators which may be used as a replacement for many numerical differentiation methods in real-world applications. 


\bibliographystyle{plain} 
\bibliography{ref}

\clearpage

\appendix
\section{Appendix: Proofs}\label{app:proofs}

\subsection{Taylor Bound}\label{app:taylor_bound}
\begin{proof}[Proof of Proposition \ref{lma:usg_error}]
By Taylor's Theorem for a fixed $x_0$ and bounded $x_i$ (which we have due to the assumption $\Vert S \Vert_2 \leq h$) we have
\begin{align*}
    f(x_i) - f(x_0) &= (\nabla f(x_0))^T s_i + \frac{1}{2} s_i^T \nabla^2 f(x_0) s_i + O(\Vert S \Vert_2^3) \\
    \rightarrow \delta \mathbf{\tilde{f}} &= S^T \nabla f(x_0) + \frac{1}{2} \text{diag}(S^T \nabla^2 f(x_0) S) + \mathbf{O}(\Vert S \Vert_2^3) \\
    \rightarrow \nabla_S f(x_0) = S^{-T} \delta \mathbf{\tilde{f}} &=  \nabla f(x_0) + \frac{1}{2} S^{-T} \text{diag}(S^T \nabla^2 f(x_0) S) + \mathbf{O}(\Vert S \Vert_2^2).
\end{align*}

\begin{align*}
    \Vert \nabla_S f(x_0)  - \nabla f(x_0) \Vert_2^2 &= \Vert S^{-T} \delta \mathbf{\tilde{f}} - \nabla f(x_0) \Vert_2^2 \\
    &= \Vert \frac{1}{2} S^{-T} \text{diag}(S^T \nabla^2 f(x_0) S) + \mathbf{O}(\Vert S \Vert_2^2) \Vert_2^2 \\
    &= \frac{1}{4} \Vert  S^{-T} \text{diag}(S^T \nabla^2 f(x_0) S) \Vert_2^2 + O(\Vert S \Vert_2^3). 
\end{align*}
\end{proof}

\subsection{Poofs for Theorem \ref{thm:SVD_Opt}}\label{app:prop_main_thm_proof}
\subsubsection{Proof of Proposition \ref{prop:obj_lower_bound}}
We first present the following results. 
\begin{lemma}\label{lma:lower_bound_opt}
Given an invertible matrix $S$ with singular value decomposition $S = U \Sigma V^T$, diagonal matrix $D$ and real number $\sigma$ we have 
\begin{align*}
    &\frac{1}{4} \Vert S^{-T} [s_i^T D s_i]^T_{1 \leq i \leq d} \Vert_2^2 
    + \sigma^2 \Vert S^{-1} \Vert_F^2 + \sigma^2 \Vert S^{-T} \mathbf{1} \Vert_2^2 \\
    &\geq \frac{1}{4} \frac{1}{\Sigma_{max}^2} \Vert [s_1^T D s_1, \dots, s_d^T D s_d]^T \Vert_2^2 + \sigma^2 \sum_{i = 1}^d \frac{1}{\Sigma_i^2} + \sigma^2 \frac{d}{\Sigma_{max}^2}. 
\end{align*}
The lower bound is achieved if both $[s_1^T D s_1, \dots, s_d^T D s_d]^T$ and $\mathbf{1}$ are right singular vectors, i.e. columns of $V$, associated with the singular value $\Sigma_{max}$.
\end{lemma}
\begin{proof}
    This is immediate by noticing that for any vector $v$ and matrix $A$ we have $\Vert A v \Vert_2 \geq \sigma_{min} \Vert v \Vert_2$, and that the Frobenius norm is the sum of the squares of singular values. 
    Clearly, $\Vert A v \Vert_2 = \sigma_{min} \Vert v \Vert_2$ when $v$ is the right singular vector associated with the smallest singular value of $A$. 
\end{proof}

\begin{lemma}\label{lma:min_norm_sum}
    Let $c \in \mathbb{R}_+$, then the solution $w^*$ to 
    \begin{align*}
        \min_{w \in \mathbb{R}^d} &\quad \Vert w \Vert_2^2 \\
         s.t. &\quad \sum_{i = 1}^d w_i = c
    \end{align*}
    is given by $w^*_i = \frac{c}{d}$ for $1 \leq i \leq d$.
\end{lemma}
\begin{proof}
    The proof is a simple application of Lagrange multipliers.
\end{proof}

\begin{lemma}\label{lma:VLevelSet}
Let $V$ be an orthogonal matrix and $D$ a diagonal matrix. Then for any orthogonal matrix $U$ and diagonal matrix $\Sigma$ we have
\begin{align*}
    \Vert [v_i^T \Sigma U^T D U \Sigma v_i]^T_{1 \leq i \leq d} \Vert_2^2 \geq \frac{1}{d} \text{tr}(\Sigma U^T D U \Sigma)^2.
\end{align*}
\end{lemma}
\begin{proof}
Due to the cyclic property of the trace we have
\begin{align*}
    \sum_{i = 1}^d v_i^T \Sigma U^T D U \Sigma v_i &= \text{tr}(V \Sigma U^T D U \Sigma V^T) \\
    &= \text{tr}(\Sigma U^T D U \Sigma).
\end{align*}
By Lemma \ref{lma:min_norm_sum} the vector $w^* \in \mathbb{R}^d$ which minimizes $\Vert w \Vert_2^2$ under the constraint $\sum_{i = 1}^d w_i =  \text{tr}(\Sigma U^T D U \Sigma)$ is given by $w^*_i = \frac{1}{d}\text{tr}(\Sigma U^T D U \Sigma)$. 
\end{proof}

\begin{proof}[Proof of Proposition \ref{prop:obj_lower_bound}]
Consider the case when $h I \succeq \Sigma \succ 0$. Otherwise, the functions both evaluate to $\infty$ and the bound trivially holds on the extended real number line. 
Letting $s_i^T D s_i = v_i^T \Sigma U^T D U \Sigma v_i$ and combining Lemma \ref{lma:lower_bound_opt} and \ref{lma:VLevelSet} we have that
\begin{align*}
    l(U, \Sigma, V) &\geq \frac{1}{4} \frac{1}{\Sigma_{max}^2} \Vert [s_1^T D s_1, \dots, s_d^T D s_d]^T \Vert_2^2 + \sigma^2 \sum_{i = 1}^d \frac{1}{\Sigma_i^2} + \sigma^2 \frac{1}{\Sigma_{max}^2} \Vert \mathbf{1} \Vert_2^2 \\
    &\geq \frac{1}{4 d \Sigma_{max}^2} \text{tr}(\Sigma U^T D U \Sigma)^2 + \sigma^2 \sum_{i = 1}^d \frac{1}{\Sigma_i^2} + \sigma^2 \frac{d}{\Sigma_{max}^2}
\end{align*}
which by taking the minimum over $V$ on both sides completes the proof.
\end{proof}

\subsubsection{Proof of Proposition \ref{prop:Uopt}}

We first consider the form of $U'$ for any given $\Sigma$. Since only $a(U, \Sigma)$ depends on $U$, it suffices to focus on this term.  
\begin{lemma}\label{lma:U_ext}
Fix a diagonal matrix $\Sigma$ and we have that
\begin{itemize}
    \item the extrema of $a(U, \Sigma)$ are given by the set of permutation matrices and,
    \item the extrema of $a(U, \Sigma)^2$ are given by the set of permutation matrices and all matrices $U_0$ such that $a(U_0, \Sigma) = 0$. 
\end{itemize}
\end{lemma}
\begin{proof}
For $a(U, \Sigma)$ and $a(U, \Sigma)^2$, the minimium with respect to $U$ under the orthogonality constraint $U^T U = I$ is achieved since the objective function is continuous and the set of orthogonal matrices is compact. 

The Lagrangian for minimzing $a(U, \Sigma)$ is given by $\mathcal{L}(U, \mu) = tr\Big(\Sigma U^T D U \Sigma \Big) - \sum_{i, j} \mu_{ij} (u_i^T u_j - \delta_{ij})$.
Taking the gradient with respect to $U$ and setting to zero we obtain
\begin{align*}
    2 U^T D V \Sigma^2 &= (\mu + \mu^T).
\end{align*}
Due to $\mu + \mu^T$ we have the symmetry relation
\begin{align*}
    U^T D U \Sigma^2 &= \Sigma^2 U^T D V
\end{align*}
which means that $U^T D U$ and $\Sigma^2$ commute, implying that they are simultaneously diagonalizable. Therefore, $\Sigma^2$ and $U^T D U$ share the same eigenvectors. 

Since $\Sigma^2$ is diagonal, it has the standard basis vectors as eigenvectors. The set of orthogonal matrices that represent the standard basis are permutation matrices $P$. Hence, we want that $P U^T D U P^T = D$ which can only happen when $P U^T = I$ and hence $U = P$. Hence, $U$ has to be a permutation matrix. 

For minimizing $a^2 (U, \Sigma)$ we have that by setting the gradient of the associated Lagrangian to zero provides
\begin{align*}
    4 tr\Big(\Sigma U^T D U \Sigma \Big) U^T D U \Sigma^2 &= (\mu + \mu^T).
\end{align*}
If it exists, the equation can be satisfied by an orthogonal matrix $U_0$ such that $tr\Big(\Sigma U^T D U \Sigma \Big) = a(U_0, \Sigma) = 0$. Otherwise, just as before we see that $U^T D U$ and $\Sigma^2$ commute, implying that $U$ has to be a permutation matrix.
\end{proof}
\begin{remark}
Letting $U_{max}(\Sigma)$ and $U_{min}(\Sigma)$ be the permutation matrix which respectively maximize and minimize $a(U, \Sigma)$ for a given $\Sigma$, if there exists $U_0$ such that $a(U_0, \Sigma) = 0$ it immediately follows that 
\begin{align*}
    a(U_{max}(\Sigma), \Sigma) \geq 0 \geq a(U_{min}(\Sigma), \Sigma).
\end{align*}
\end{remark}

By showing that we can discard $U_0$, which means that we only have to work with permutation matrices, we can significantly simplify the problem. To show this, we first present the following result which tells us that for the optimal $(U', \Sigma')$ the sign of $a(\Sigma', U_{ext})$ is the same for all extrema $U_{ext}$.
\begin{lemma}\label{lma:a_sign}
Without loss of generality assume that $\sum_{i = 1}^d D_i \geq 0$ and let $(U', \Sigma')$ be optimal. Then for all extrema $U_{ext}$
\begin{itemize}
\item if $\sum_{i = 1}^d D_{i} = 0$ then $a(\Sigma', U_{ext}) = 0$ and $\Sigma' = h I$ and $U' = I$;
\item if $\sum_{i = 1}^d D_{i} > 0$ then $a(\Sigma', U_{ext}) \geq 0$.
\end{itemize}
\end{lemma}

\begin{proof}
We proceed by cases. 

(i) Letting $\sum_{i = 1}^d D_i = 0$ we have that for any $\Sigma$ and all extrema $U_{ext}$
\begin{align*}
    \frac{1}{4 \Sigma_{max}^2} \frac{a(U_{ext}, \Sigma)}{d} + \sigma^2 \sum_{i = 1}^d \frac{1}{\Sigma_{i}^2} + \sigma^2 \frac{1}{\Sigma_{max}^2} d \geq \sigma^2 \sum_{i = 1}^d \frac{1}{\Sigma_{i}^2} + \sigma^2 \frac{1}{\Sigma_{max}^2} d.
\end{align*}
To minimize $\sigma^2 \sum_{i = 1}^d \frac{1}{\Sigma_{i}^2} + \sigma^2 \frac{1}{\Sigma_{max}^2} d$ we want to maximize each element along the diagonal $\Sigma$, which is achieved when $\Sigma = h I$. Hence, $\Sigma' = h I$ as this will minimize the loss for all extrema $U_{ext}$ chosen. Specifically, $I$ is a permutation matrix and hence $U' = I$. 

(ii) Assume that $\sum_{i = 1}^d D_{i} > 0$. If $D \succeq 0$ the result is immediate. Hence, assume that at least one entry of $D$ is negative. For contradiction, assume that for $\Sigma'$ we have
$a(\Sigma', U_{min}) < 0$, where $U_{min}$ is the permutation matrix which minimizes $a(\Sigma', U)$. For ease of notation and without loss of generality, we assume that $\Sigma'$ is sorted such that $U_{min}(\Sigma') = I$. We proceed by showing that we can find $\Sigma \succeq \Sigma'$ which reduces the objective function leading to a contradiction.

For the proof we introduce some further notation.
Let $I = \{i \vert D_i < 0 \}$. We assign $c = (\sum_{i \in I} \vert D_i \vert (\Sigma'_i)^2)/(\sum_{i \in I} \vert D_i \vert)$ where $\sqrt{c} \leq h$ since by assumption $\Sigma' \preceq hI$. Set $J = \{j | (\Sigma'_j)^2 \leq c, j \not \in I\}$. 
Lastly, define a function $\Sigma(\alpha) : [0, 1] \rightarrow \mathbb{R}^{d \times d}$ such that 
\begin{align*}
    (\Sigma(\alpha))_{i} =
    \begin{cases}
        \Sigma_{i} + \alpha (\sqrt{c} - \Sigma_{i}) & \text{if } i \in J \\
        \Sigma_{i} & \text{otherwise.}
    \end{cases}
\end{align*}

We show that given this construction $a(\Sigma(1), U_{min}) \geq 0 > a(\Sigma(0), U_{min})$. The lower bound is immediate. Letting $\alpha = 1$
\begin{align*}
    \sum_{i \not \in I} D_{i}(\Sigma'_{i}(1))^2  &= \sum_{i \in J, i \not \in I} c D_{i} + \sum_{i \not \in J, i \not \in I} D_{i} (\Sigma'_{i})^2 \\
    &\geq c \sum_{i \not \in I} D_{i} \\
    &\geq \sum_{i \in I} \vert D_i \vert (\Sigma'_{i})^2.
\end{align*}
The last step follows from $\sum_{i \in I} \vert D_{i} \vert < \sum_{i \not \in I}  \vert D_{i} \vert$ due to the assumption $\sum_{i = 1}^d D_i > 0$. By the above $\sum_{i \not \in I} D_{i}(\Sigma'_{i}(1))^2 - \sum_{i \in I} \vert D_i \vert (\Sigma'_{i})^2 \geq 0$ which by definition implies $a(\Sigma(1), U_{min}) \geq 0$. 

By the continuity of $a(\Sigma(\cdot), U_{min}))$ and the above, we have that by the intermediate value theorem there exists an $\alpha_0 \in (0, 1]$ such that $a(\Sigma(\alpha_0), U_{min}) = 0$. Combined with the fact that $\Sigma(\alpha_0) \succeq \Sigma'$ the objective function has been strictly decreased by using $\Sigma(\alpha_0)$ leading to a contraction.
\end{proof}

\begin{proof}[Proof of Proposition \ref{prop:Uopt}]
    If $\sum_{i} D_i = 0$ we are immediately done as by the first part of Lemma \ref{lma:a_sign} $U = I$. When $\sum_i D_i > 0$, we have that $a(\Sigma', U') \geq 0$ which implies that even if $U_0$ exists
    \begin{align*}
        a(U_{max}(\Sigma), \Sigma) \geq a(U_{0}(\Sigma), \Sigma) \geq a(U_{min}(\Sigma), \Sigma) \geq 0
    \end{align*}
    and hence $U_{min}$ is optimal. If $U_0$ does not exist, then trivially $U' = U_{min}$. If $U_{min}$ is not the identity matrix, we can permute $U_{min}$ and any $\Sigma$ without changing the objective function. Hence, $U' = I$ and $\Sigma' = \argmin_{\Sigma} \underline{l}(I, \Sigma)$.
\end{proof}

\subsubsection{Proof of Proposition \ref{prop:optV}}

\begin{proof}
Since $\Sigma'$ is assumed to be a solution, $\Sigma' \succ 0$. We proceed by showing that the terms in the respective objective functions $l(U', \Sigma', V')$ and $\underline{l}(U', \Sigma')$ can be matched up and shown to be equal. Let $S' = \Sigma' (V')^T$, recalling that $U' = I$.

To show that the first term of the objective functions are equal note that
\begin{align*}
    (s'_i)^T D s'_i &= (v'_i)^T \Sigma' D \Sigma' v'_i \\
    &= \sum_{j = 1}^d (V'_{i, j})^2 D_{j} (\Sigma'_{j})^2 \\
    &= \frac{1}{d} \sum_{j = 1}^d D_{j} (\Sigma'_{j})^2 & \text{Due to $V$ being the Hadamard matrix.}\\
    &= \frac{1}{d} \text{tr}((\Sigma')^T D \Sigma') = \frac{1}{d} a(U', \Sigma').
\end{align*}
It then follows that $[(s'_1)^T D s'_1, \dots, (s'_n)^T D s'_n]^T = \frac{1}{d} a(U', \Sigma') \mathbf{1}$. Since $V' = M_{k}$, $\mathbf{1}$ is the right singular vector of $S'$ associated with $\Sigma_{max}$. Hence, by  Lemma \ref{lma:lower_bound_opt} we have that $\frac{1}{4} \Vert S^{-T} [s_i^T D s_i]^T_{1 \leq i \leq d} \Vert_2^2 = \frac{1}{4d (\Sigma'_{max})^2} a(U', \Sigma')^2$.

The second terms are trivially equal since the Frobenius norm is always the sum of the squares of the singular values. The third terms are equivalent as by Lemma \ref{lma:lower_bound_opt}, as similarly to the first term, $\mathbf{1}$ is the right singular vector associated with the largest singular value as .
\end{proof}

\subsection{Proofs for Proposition \ref{prop:sig_alg}}\label{app:sig_alg}

For the following proofs we write $l(\lambda) := l(I, diag(\lambda), V^*) = \underline{l}(I, diag(\lambda))$. More specifically
\begin{align*}
    l(\lambda) := 
    \begin{cases}
        \frac{(\sum_{i = 1}^d D_i \lambda^2_i)^2}{4 d \lambda_{max}^2} + \sigma^2 \sum_{i =0}^d \frac{1}{\lambda^2_i} + \sigma^2 \frac{1}{\Sigma_{max}^2} d & \text{if } h I \succeq \lambda \succ 0 \\
        \infty & \text{otherwise.}
    \end{cases}
\end{align*}
Hence, our objective is to find $\lambda^* = \argmin_{\lambda} l(\lambda)$. We will also use $l_A$ which we define in Section \ref{sec:cand_sol}.

\subsubsection{Proof of Proposition \ref{prop:sig_dec}}\label{app:sig_dec}
\begin{proof}
Define $K = \{i \vert D_i \leq 0\}$ for which the complement is given by $K^c = \{i \vert D_i > 0 \}$. We first show that $\lambda_i = h^2$ for $D_i \leq 0$. Assume for contradiction that for some $1 \leq i \leq d$ with $D_i \leq 0$ we have $\lambda^*_i < h^2$.
Now, define the function $\lambda(\alpha) \in \mathbb{R}^d$ for $\alpha \in [0, 1]$ as 
\begin{align*}
    \lambda_i(\alpha) &= 
    \begin{cases}
        h^2 & \text{if $i \in K$ or $\lambda^*_i + \alpha h^2 \geq h^2$,} \\
        \lambda^*_i + \alpha h^2 & \text{otherwise.}
    \end{cases}
\end{align*}
Given this definition $\lambda(\alpha) \succ \lambda^*$ for any chosen $\alpha$ value. Hence, the noise error component of the objective function will have strictly decreased. To complete the proof, we show that there exists an $\alpha^*$ for which the approximation error does not increase. 

By Lemma \ref{lma:a_sign} and the assumption that $\sum_{i = 1}^d D_i \geq 0$ we know that for the optimal $\lambda^*$ we have $\sum_{i = 1}^d D_i \lambda^*_i \geq 0$. If instead of $\lambda^*$ we use $\lambda(0)$ and $\sum_{i = 1}^d D_i \lambda(0) \geq 0$ then we immediately know that we did not increase the approximation error. This is because now $\lambda_{max} = h^2$, placing a larger quantity in the denominator, and we made the $\lambda_i$ corresponding to $D_i \leq 0$ larger which results in adding a larger negative quantity to the numerator, i.e. 
\begin{align*}
    \sum_{i = 1}^d \lambda^*_i D_i &\geq \sum_{i \in K} D_i h^2 + \sum_{i \in K^c} \lambda^*_i D_i \\
    &= \sum_{i = 1}^d \lambda(0)_i D_i \geq 0.
\end{align*}

If instead $\sum_{i = 1}^d D_i \lambda(0) \leq 0$ we use the intermediate value theorem to find an $\alpha^*$ which makes the sum equal to zero. By the assumption $\sum_{i = 1}^d D_i \geq 0$ and that $\lambda(1) = h^2$ it follows that $\sum_{i = 1}^d D_i \lambda(1) \geq 0$. Hence, by the intermediate value theorem there exists an $\alpha^*$ such that $\sum_{i = 1}^d D_i \lambda(\alpha^*) = 0$ which implies that the approximation error is now zero. Hence, under this $\alpha^*$ we could not have increased the approximation error.

We know that $\lambda^*_i = h^2$ for all $i \in K$. Since $D$ is in increasing order, these terms will fill the first $\vert K \vert$ entries of $\lambda$. For the remainder we show that $\lambda$ has to be in decreasing order by way of contradiction. Assume that there exist $i, j \in K^c$ with $i < j$ and $\lambda_i < \lambda_j$. We then define
\begin{align*}
    s &= \frac{D_i \lambda_i + D_j \lambda_j}{D_i + D_j}.
\end{align*}
This is well-defined as $D_i, D_j > 0$ by assumption. 
Due to the assumption we have immediately that $D_j \lambda_j > D_i \lambda_i$ from which it then follows that $\frac{1}{\lambda_i} + \frac{1}{\lambda_j} > \frac{2}{s}$. Since $s (D_i + D_j) = \lambda_i D_i + \lambda_j D_j$, we did not change the approximation error. However, we reduced the noise error since $\frac{1}{\lambda_i} + \frac{1}{\lambda_j} > \frac{2}{s}$. So the objective function can be reduced by replacing $\lambda_i$ and $\lambda_j$ with $s$, leading to a contradiction.
\end{proof}

\subsubsection{Proofs for Finding all Candidate Solutions}\label{app:active_sets}
\begin{lemma}\label{lma:active_existance}
    There exists $A \in \mathcal{P}(\{1, \dots, d\})$ such that $\lambda^* \in \lambda^*_A$.
\end{lemma}
\begin{proof}
    Let $A = \{i \vert \lambda^*_i = h^2\}$, i.e. the set of all indices of the global optimum at which the constraints are strict. Since $\lambda^*$ is a minimizer of $l(\lambda)$, $h \succeq \lambda^* \succ 0$ and hence $\lambda^*$ is finite and in the domain of $l_{A}(\lambda)$. Additionally, Proposition \ref{prop:sig_dec} implies $\lambda_1 = \lambda_{max}$. Hence, $l_A(\lambda^*) = l(\lambda^*)$. Since by Fermat's theorem for all differentiable functions on open sets a point is a local minima if and only if it is a stationary point, it suffices to show that $\lambda^*$ is a local minimum of $l_A(\lambda)$. For this, note that for any in the domain of $l_A$ we have $l_A(\lambda) \geq l(\lambda^{(A)})$. Since $\lambda^*$ is a minimizer of $l(\lambda)$ and the two functions coincide there, it follows that $\lambda^*$ is a local minimum of $l_A$ as well and hence $\lambda^* \in \lambda_A^*$.
\end{proof}

\begin{proof}[Proof of Proposition \ref{prop:active_sets_solution_exists}]
    By Lemma \ref{lma:active_existance} we know that there exists an active set for which one of the solutions provides the minimizer $\lambda^*$.
    Now, assume for contradiction that there does not exist $J \in \{\vert K \vert, \dots, d\}$ such that $\lambda^* \in \lambda^*_{A_J}$.  Then, there has to either exist indices $i < j < k$ such that $\lambda^*_i = \lambda^*_k = h^2$ but $\lambda^*_j < h^2$ or an index $1 \leq l \leq \vert K \vert$ such that $\lambda_l < h^2$. However, by Proposition \ref{prop:sig_dec} we know that $\lambda^*$ is decreasing and that the first $\vert K \vert$ entries have to be set to $h^2$. 
\end{proof}

\begin{proof}[Proof of Proposition \ref{prop:FONC}]
Let $a = \sum_{i \not \in A} D_i \lambda_i + h^2 \sum_{i \in A} D_i$ and $\mathbbm{1}$ be the indicator function. Then the stationary points are simply all the points which satisfy
\begin{align}
    \frac{\partial l_A(\lambda)}{\partial \lambda_{1}} &= \mathbbm{1}(1 \not \in A) \left( \frac{D_{1}}{2 d \lambda_{1}} a - \frac{1}{4 d \lambda_{1}^2} a^2 - \sigma^2 \frac{d + 1}{\lambda_{1}^2} \right) = 0 \\
    \frac{\partial l_A(\lambda)}{\partial \lambda_{i}}  &= \mathbbm{1}(i \not \in A) \left( \frac{D_i}{2 d \lambda_{1}} a - \sigma^2 \frac{1}{\lambda_{i}^2} \right) = 0. \label{eq:FONC1}
\end{align}

To manipulate the above equations we note that if not all constraints are set, i.e. $\vert A \vert < d$, it holds that $a > 0$ for every stationary point. This is because $a$ is non-negative due to Lemma \ref{lma:a_sign} and $a \neq 0$ as otherwise the stationary point can not exist since Equation \ref{eq:FONC1} can not be satisfied. It also holds that $D_i > 0$ for all $i \not \in A$, as is immediate by the assumption that $K \subseteq A$. The result then follows by simple algebra.
\end{proof}

The following Lemmas prove Proposition \ref{prop:solve_a} by splitting it into two cases.
\begin{lemma}\label{lma:solve_pos_a}
When the active set $A_J$ is empty then the unique non-negative solution to $a$ as given by the equations in Proposition \ref{prop:FONC} is 
\begin{align*}
    a = \sqrt{2} \left(\frac{d \sigma^2}{D_{1}} \left(K \sqrt{8 D_{1} (d + 1) + K^2} + 2D_{1}(d + 1) + K^2 \right)\right)^{1/2}
\end{align*}
with $K = \sum_{i = 2}^d \sqrt{D_{i}}$.
\end{lemma}
\begin{proof}
We rewrite $a = \sum_{i = 1}^d D_i \lambda_i$ by plugging in $\lambda_i$ to obtain
\begin{align*}
    a = \sigma \sqrt{\frac{2 d \lambda_{1}}{a}} \sum_{i = 2}^d  \sqrt{D_i} + \lambda_{1} D_{1}.
\end{align*}
Letting $K = \sum_{i = 2}^d \sqrt{D_{i}}$, substituting $\lambda_{1}$ and simplifying we get
\begin{align*}
    a = \frac{1}{a} \left( \sigma \sqrt{\frac{2d}{D_{1}}\left(\frac{a^2}{2} + 2 \sigma^2 d (d + 1) \right)} K + \left(\frac{a^2}{2} + 2 \sigma^2 d (d + 1) \right)\right).
\end{align*}
The equation can be rewritten as solving for the roots of a fourth-order polynomial 
\begin{align*}
    \left(\frac{a^2}{2} - 2\sigma^2 d (d + 1) \right)^2 - K^2 \frac{2d \sigma^2}{D_{1}} \left(\frac{a^2}{2} + 2 \sigma^2 d(d + 1)\right) &= 0.
\end{align*}
However, during the algebraic manipulation, additional roots may be have been introduced. 
The potentially complex roots of the fourth order polynomial are given by two pairs
\begin{align*}
    a'_1 &= \pm \sqrt{2} \left(\frac{d \sigma^2}{D_{1}} \left(K \sqrt{8 D_{1} (d + 1) + K^2} + 2D_{1}(d + 1) + K^2 \right)\right)^{1/2} \\
    a'_2 &= \pm \sqrt{2} \left(\frac{d \sigma^2}{D_{1}} \left(-K \sqrt{8 D_{1} (d + 1) + K^2} + 2D_{1}(d + 1) + K^2 \right)\right)^{1/2}
\end{align*}
which can be easily verified by plugging into the polynomial. 

It is easy to verify that while $a'_1$ satisfies the original equation, $a'_2$ will not. 
Of the pair $a'_1$, by assumption only the non-negative can be valid. Since all terms under square root are positive, the solution is also real and hence provides the unique solution
\begin{align*}
    a = \sqrt{2} \left(\frac{d \sigma^2}{D_{1}} \left(K \sqrt{8 D_{1} (d + 1) + K^2} + 2D_{1}(d + 1) + K^2 \right)\right)^{1/2}.
\end{align*}

\end{proof}

\begin{lemma}\label{lma:solve_indef_a}
When the active set $A_J$ is non-empty the solution for $a$ given by the equations in Proposition \ref{prop:FONC} is given by 
\begin{align*}
    a &= 
    \begin{cases}
        \frac{\sqrt[3]{\frac{2}{3}} c_2}{\sqrt[3]{\sqrt{3} \sqrt{27 c_1^2-4 c_2^3}+9 c_1}}+\frac{\sqrt[3]{\sqrt{3} \sqrt{27 c_1^2-4 c_2^3}+9 c_1}}{\sqrt[3]{18}} & \text{if } 27 c^1 - 4c_2^3 \geq 0, \\
        \cos(\theta/3) 2 \sqrt{\frac{c_2}{3}} & \text{otherwise,}
    \end{cases}
\end{align*}
where $c_1 = \sigma \sqrt{\frac{2 d h^2}{a}} \sum_{i = I + 1}^d \sqrt{D_{i}}$ and $c_2 = \sum_{i=1}^I D_{i} h^2$.
\end{lemma}
For the proof we note the following Lemma 
\begin{lemma}\label{lma:cubic_root}
For $c_1 > 0$ and $c_2 \in \mathbb{R}$ there exists only one positive real root for the polynomial $f(x) = x^3 - c_2 x - c_1$. 
\end{lemma}
\begin{proof}
Since $c_1 > 0$ and for any polynomial the negative of the order zero coefficient is the product of the roots, the product of the roots is positive and no root is zero. Furthermore, the second order term presents the sum of the roots, which is zero. For the sum of the roots to be zero, at least one root has to have a negative real component. For the product to be positive there has to exist another root with negative real component and one positive real root. Hence, there exists a unique real positive root. 
\end{proof}

\begin{proof}[Poof of Lemma \ref{lma:solve_indef_a}]
Plugging the equations for $\lambda_{i}$ in $a = \sum_{i = 1}^d D_i \lambda_i$ and after some simplification we get
\begin{align*}
    a^{3/2} &= \sigma \sqrt{2 d h^2}\sum_{i \not \in I} \sqrt{D_{i}} + \sqrt{a} \sum_{i \in I} D_{i} h^2.
\end{align*}
Letting $c_1 = \sigma \sqrt{\frac{2 d h^2}{a}} \sum_{i \not \in I}\sqrt{D_{i}}$, $c_2 = \sum_{i \in I} D_{i} h^2$, and $x = \sqrt{a}$ the equation becomes 
\begin{align*}
    x^3 = c_1 + c_2 x
\end{align*}
when restricting $x \geq 0$. By Lemma \ref{lma:cubic_root} and $c_1 > 0$, there exists a unique positive $x$ satisfying the equation which provides the solution $\sqrt{a}$.

The root is given by
\begin{align*}
    x &= \frac{\sqrt[3]{\frac{2}{3}} c_2}{\sqrt[3]{\sqrt{3} \sqrt{27 c_1^2-4 c_2^3}+9 c_1}}+\frac{\sqrt[3]{\sqrt{3} \sqrt{27 c_1^2-4 c_2^3}+9 c_1}}{\sqrt[3]{18}}.
\end{align*}
We now show that the root is always well-defined, real, and non-negative. The root can be easily verified to be a root by plugging it into the polynomial. 

To show well defined, real, and non-negative we split in two cases: $27 c_1^2-4 c_2^3 \geq 0$ and $27 c_1^2-4 c_2^3 < 0$.

i) Assume that $27 c_1 - 4c_2^3 \geq 0$. It is well defined because then 
\\$\sqrt[3]{\sqrt{3} \sqrt{27 c_1^2-4 c_2^3}+9 c_1} > 0$ since $c_1 > 0$. Furthermore, this quantity is real since every individual term is real. We now need to show non-negativeness. If $c_2 \geq 0$ it is trivial, hence assume $c_2 < 0$ for which we have

\begin{align*}
    &\frac{\sqrt[3]{\frac{2}{3}} c_2}{\sqrt[3]{\sqrt{3} \sqrt{27 c_1^2-4 c_2^3}+9 c_1}}+\frac{\sqrt[3]{\sqrt{3} \sqrt{27 c_1^2-4 c_2^3}+9 c_1}}{\sqrt[3]{18}} \\
    &> - \frac{\sqrt[3]{\frac{2}{3}} \vert c_2 \vert}{\sqrt[3]{\sqrt{3} \sqrt{4 \vert c_2 \vert ^3}}} + \frac{\sqrt[3]{\sqrt{3} \sqrt{4 \vert c_2 \vert^3}}}{\sqrt[3]{18}} = 0.
\end{align*}

ii) $27 c_1 - 4c_2^3 < 0$ implies $c_2 > 0$ since $c_1 > 0$, which yields
\begin{align*}
    \sqrt{3} \sqrt{27 c_1^2-4 c_2^3}+9 c_1 &= i \sqrt{3} \sqrt{4 c_2^3 - 27 c_1^2}+9 c_1 \\
    &= r \left(\cos(\theta) + i \sin(\theta)\right) \\
    &= r e^{i \theta} = z
\end{align*}
where $\theta = \arccos(\frac{9 c_1}{r})$ and after simplifying $r = \sqrt{12 c_2^3}$. Since $c_2 > 0$ we have $r > 0$ and hence the root is well-defined. Substituting and simplifying yields that
\begin{align*}
    x = \cos(\theta/3) 2 \sqrt{\frac{c_2}{3}}. 
\end{align*}
which shows that the root is real. The root is non-negative because $\cos(\theta) = \frac{9c_1}{r} > 0$ and hence $\theta \in (-\frac{\pi}{2}, \frac{\pi}{2})$, which implies that $\theta/3 \in (-\frac{\pi}{2}, \frac{\pi}{2})$ and $\cos(\theta/3) > 0$. 
\end{proof}

\begin{proof}[Proof of Proposition \ref{prop:solve_a}]
    The result follows from combining Lemma \ref{lma:a_sign} and \ref{lma:solve_indef_a}.
\end{proof}

\subsubsection{Proof of Proposition \ref{prop:linear_runtime}}\label{app:linear_runtime}
We prove the construction of $\lambda^*$ in linear time.
\begin{proof}
    Assume for contradiction that there exists another index $J' > J$ for which $h^2 \succeq \lambda_{A_J'}$ and $l_{A_J}(\lambda_{A_J}^*) \geq l_{A_{J'}}(\lambda_{A_J'})$. Since $\lambda_{A_{J'}}$ also satisfies the active set $A_J$ we have that $l_{A_J}(\lambda_{A_J}^*) \geq l_{A_{J}}(\lambda_{A_J'})$ implying that $\lambda_{A_J}^*$ was not the global minimum. By Proposition \ref{prop:solve_a} this non-uniqueness is a contradiction. The second part follows from Proposition \ref{prop:FONC} that $\lambda_{A_J}^*$ is decreasing from the $J+1$th entry on-wards when $J \geq 1$. Hence it suffices to only compute and check that the $J+1$th entry is less than $h^2$, i.e. $h^2 \geq (\lambda_{A_J}^*)_{J + 1}$
\end{proof}
\end{document}

%% file: figs/Intro/sig_0.1_num_1000_2d_Error.tex
\begin{tikzpicture}

\definecolor{darkgray176}{RGB}{176,176,176}
\definecolor{darkorange}{RGB}{255,140,0}
\definecolor{lightgray204}{RGB}{204,204,204}
\definecolor{slategray}{RGB}{112,128,144}

\begin{axis}[
legend cell align={left},
legend style={
  fill opacity=0.8,
  draw opacity=1,
  text opacity=1,
  at={(0.03,0.97)},
  anchor=north west,
  draw=lightgray204
},
log basis x={10},
log basis y={10},
tick align=outside,
tick pos=left,
x grid style={darkgray176},
xlabel={k},
xmajorgrids,
xmin=3.98107170553497e-05, xmax=25118.8643150958,
xmode=log,
xtick style={color=black},
y grid style={darkgray176},
ylabel={MSE},
ymajorgrids,
ymin=0.000387780626451852, ymax=4551.78267445907,
ymode=log,
ytick style={color=black}
]
\path [draw=slategray, very thick]
(axis cs:0.0001,0.0250706296739161)
--(axis cs:0.0001,0.222661191955462);

\path [draw=slategray, very thick]
(axis cs:0.000215443469003188,0.0254840426282619)
--(axis cs:0.000215443469003188,0.222705378249929);

\path [draw=slategray, very thick]
(axis cs:0.000464158883361278,0.0256505401810127)
--(axis cs:0.000464158883361278,0.222881876986588);

\path [draw=slategray, very thick]
(axis cs:0.001,0.0257987878485915)
--(axis cs:0.001,0.222786117023253);

\path [draw=slategray, very thick]
(axis cs:0.00215443469003188,0.0260231678918288)
--(axis cs:0.00215443469003188,0.222892343783248);

\path [draw=slategray, very thick]
(axis cs:0.00464158883361277,0.0262278101672985)
--(axis cs:0.00464158883361277,0.223223227140672);

\path [draw=slategray, very thick]
(axis cs:0.01,0.0271539926019424)
--(axis cs:0.01,0.224444774558397);

\path [draw=slategray, very thick]
(axis cs:0.0215443469003188,0.0291227995343936)
--(axis cs:0.0215443469003188,0.227152374012732);

\path [draw=slategray, very thick]
(axis cs:0.0464158883361277,0.0322262665272487)
--(axis cs:0.0464158883361277,0.235252352260476);

\path [draw=slategray, very thick]
(axis cs:0.1,0.0363952250191387)
--(axis cs:0.1,0.24293676289293);

\path [draw=slategray, very thick]
(axis cs:0.215443469003188,0.0453427334583145)
--(axis cs:0.215443469003188,0.264207615072335);

\path [draw=slategray, very thick]
(axis cs:0.464158883361277,0.0625855084577217)
--(axis cs:0.464158883361277,0.317138958925778);

\path [draw=slategray, very thick]
(axis cs:1,0.0894358450539658)
--(axis cs:1,0.434249675988706);

\path [draw=slategray, very thick]
(axis cs:2.15443469003188,0.136456207532081)
--(axis cs:2.15443469003188,0.675934835399248);

\path [draw=slategray, very thick]
(axis cs:4.64158883361277,0.218039096903232)
--(axis cs:4.64158883361277,1.21141705053057);

\path [draw=slategray, very thick]
(axis cs:10,0.395413388152802)
--(axis cs:10,2.40459161480812);

\path [draw=slategray, very thick]
(axis cs:21.5443469003188,0.732604860202893)
--(axis cs:21.5443469003188,4.97889045845351);

\path [draw=slategray, very thick]
(axis cs:46.4158883361277,1.4686840472171)
--(axis cs:46.4158883361277,10.4650406983849);

\path [draw=slategray, very thick]
(axis cs:100,3.06269475530976)
--(axis cs:100,22.1242080282261);

\path [draw=slategray, very thick]
(axis cs:215.443469003188,6.61699143494287)
--(axis cs:215.443469003188,47.1480666202184);

\path [draw=slategray, very thick]
(axis cs:464.158883361277,14.0210487111525)
--(axis cs:464.158883361277,101.122727165642);

\path [draw=slategray, very thick]
(axis cs:1000,30.0158420049983)
--(axis cs:1000,217.42907518568);

\path [draw=slategray, very thick]
(axis cs:2154.43469003188,64.4039590904278)
--(axis cs:2154.43469003188,467.850980248647);

\path [draw=slategray, very thick]
(axis cs:4641.58883361277,141.023731049926)
--(axis cs:4641.58883361277,1008.82987297188);

\path [draw=slategray, very thick]
(axis cs:10000,302.723563489361)
--(axis cs:10000,2171.88223247336);

\path [draw=darkorange, very thick]
(axis cs:0.0001,0.000812702047368524)
--(axis cs:0.0001,0.00162779367448015);

\path [draw=darkorange, very thick]
(axis cs:0.000215443469003188,0.00119935664949563)
--(axis cs:0.000215443469003188,0.00238278809102601);

\path [draw=darkorange, very thick]
(axis cs:0.000464158883361278,0.00177986957289695)
--(axis cs:0.000464158883361278,0.00351119053685034);

\path [draw=darkorange, very thick]
(axis cs:0.001,0.00264554126526487)
--(axis cs:0.001,0.00517842109225748);

\path [draw=darkorange, very thick]
(axis cs:0.00215443469003188,0.0039445825479577)
--(axis cs:0.00215443469003188,0.00757918627213002);

\path [draw=darkorange, very thick]
(axis cs:0.00464158883361277,0.00574703726220615)
--(axis cs:0.00464158883361277,0.0111466101911726);

\path [draw=darkorange, very thick]
(axis cs:0.01,0.00839791304806945)
--(axis cs:0.01,0.0167758343799578);

\path [draw=darkorange, very thick]
(axis cs:0.0215443469003188,0.012380716174009)
--(axis cs:0.0215443469003188,0.0258583724155986);

\path [draw=darkorange, very thick]
(axis cs:0.0464158883361277,0.0177594452235408)
--(axis cs:0.0464158883361277,0.0404828307406087);

\path [draw=darkorange, very thick]
(axis cs:0.1,0.0252534485069084)
--(axis cs:0.1,0.0659963064738792);

\path [draw=darkorange, very thick]
(axis cs:0.215443469003188,0.0356088283936337)
--(axis cs:0.215443469003188,0.110216044997327);

\path [draw=darkorange, very thick]
(axis cs:0.464158883361277,0.0523483097160043)
--(axis cs:0.464158883361277,0.187330718673323);

\path [draw=darkorange, very thick]
(axis cs:1,0.0789150667500906)
--(axis cs:1,0.32631794513325);

\path [draw=darkorange, very thick]
(axis cs:2.15443469003188,0.112075135670052)
--(axis cs:2.15443469003188,0.40307571735911);

\path [draw=darkorange, very thick]
(axis cs:4.64158883361277,0.163600213738521)
--(axis cs:4.64158883361277,0.51103431208763);

\path [draw=darkorange, very thick]
(axis cs:10,0.250526755257073)
--(axis cs:10,0.657069213545414);

\path [draw=darkorange, very thick]
(axis cs:21.5443469003188,0.379159386566379)
--(axis cs:21.5443469003188,0.869599485018276);

\path [draw=darkorange, very thick]
(axis cs:46.4158883361277,0.569051351746026)
--(axis cs:46.4158883361277,1.19568383802553);

\path [draw=darkorange, very thick]
(axis cs:100,0.832018224653835)
--(axis cs:100,1.67142345551356);

\path [draw=darkorange, very thick]
(axis cs:215.443469003188,1.22200849314043)
--(axis cs:215.443469003188,2.39018040165007);

\path [draw=darkorange, very thick]
(axis cs:464.158883361277,1.79851542479933)
--(axis cs:464.158883361277,3.50813112748241);

\path [draw=darkorange, very thick]
(axis cs:1000,2.60106380363168)
--(axis cs:1000,5.1393682027671);

\path [draw=darkorange, very thick]
(axis cs:2154.43469003188,3.77175486145884)
--(axis cs:2154.43469003188,7.49932415250254);

\path [draw=darkorange, very thick]
(axis cs:4641.58883361277,5.48536902519206)
--(axis cs:4641.58883361277,10.9664434337961);

\path [draw=darkorange, very thick]
(axis cs:10000,8.01769371796854)
--(axis cs:10000,16.1346916976195);

\addplot [ultra thick, darkorange]
table {%
0.0001 0.00141442561492706
0.000215443469003188 0.00207645191055442
0.000464158883361278 0.00304894840070202
0.001 0.00447881921948498
0.00215443469003188 0.00658524283200809
0.00464158883361277 0.00970088398456405
0.01 0.0143469331942249
0.0215443469003188 0.0213826044904683
0.0464158883361277 0.0323190589074853
0.1 0.0499776830006803
0.215443469003188 0.0798395669232646
0.464158883361277 0.132805688935633
1 0.23094010767585
2.15443469003188 0.28612118327651
4.64158883361277 0.370582442311505
10 0.499776830006804
21.5443469003188 0.6962930165947
46.4158883361277 0.99249258236515
100 1.4346933194225
215.443469003188 2.08999209803196
464.158883361277 3.05659895956765
1000 4.47881921948497
2154.43469003188 6.56876020258966
4641.58883361277 9.63803600156377
10000 14.1442561492699
};
\addlegendentry{CASG}
\addplot [ultra thick, slategray, dashed]
table {%
0.0001 0.150015
0.000215443469003188 0.150032316520351
0.000464158883361278 0.150069623832504
0.001 0.15015
0.00215443469003188 0.150323165203505
0.00464158883361277 0.150696238325042
0.01 0.1515
0.0215443469003188 0.153231652035048
0.0464158883361277 0.156962383250419
0.1 0.165
0.215443469003188 0.182316520350478
0.464158883361277 0.219623832504192
1 0.3
2.15443469003188 0.473165203504782
4.64158883361277 0.846238325041915
10 1.64999999999999
21.5443469003188 3.38165203504782
46.4158883361277 7.11238325041935
100 15.1500000000001
215.443469003188 32.4665203504765
464.158883361277 69.7738325041942
1000 150.150000000045
2154.43469003188 323.315203504406
4641.58883361277 696.388325042176
10000 1500.15000001357
};
\addlegendentry{FD}

\end{axis}

\end{tikzpicture}

%% file: figs/Intro/sig_0.1_num_1_2d_S_scale.tex
\begin{tikzpicture}

\definecolor{darkgray176}{RGB}{176,176,176}
\definecolor{darkorange}{RGB}{255,140,0}
\definecolor{lightgray204}{RGB}{204,204,204}
\definecolor{slategray}{RGB}{112,128,144}

\begin{axis}[
legend cell align={left},
legend style={fill opacity=0.8, draw opacity=1, text opacity=1, draw=lightgray204},
log basis x={10},
log basis y={10},
tick align=outside,
tick pos=left,
x grid style={darkgray176},
xlabel={k},
xmajorgrids,
xmin=3.98107170553497e-05, xmax=25118.8643150958,
xmode=log,
xtick style={color=black},
y grid style={darkgray176},
ylabel={Radius},
ymajorgrids,
ymin=0.22818710416578, ymax=372.354998109738,
ymode=log,
ytick style={color=black}
]
\addplot [very thick, slategray, dashed]
table {%
0.0001 31.9390043677006
0.000215443469003188 21.8605746663357
0.000464158883361278 14.9942204422375
0.001 10.3162277660168
0.00215443469003188 7.12914845659645
0.00464158883361277 4.95781659962962
0.01 3.47850542618522
0.0215443469003188 2.47066245604872
0.0464158883361277 1.78402703363891
0.1 1.31622776601684
0.215443469003188 0.9975198350748
0.464158883361277 0.780386649378116
1 0.632455532033676
2.15443469003188 0.531671235020026
4.64158883361277 0.463007692779045
10 0.416227766016838
21.5443469003188 0.384356972922634
46.4158883361277 0.362643654352966
100 0.347850542618522
215.443469003188 0.337772112917157
464.158883361277 0.330905758693059
1000 0.326227766016838
2154.43469003188 0.323040686707418
4641.58883361277 0.320869354850451
10000 0.319390043677006
};
\addlegendentry{FD}
\addplot [very thick, darkorange]
table {%
0.0001 266.027731366106
0.000215443469003188 149.671596054966
0.000464158883361278 84.254999009103
0.001 47.4866915608733
0.00215443469003188 26.831462114779
0.00464158883361277 15.2396019863778
0.01 8.74451476247172
0.0215443469003188 5.11075167148166
0.0464158883361277 3.07440692384265
0.1 1.92035161531998
0.215443469003188 1.24909364479616
0.464158883361277 0.843612351527944
1 0.588566191276542
2.15443469003188 0.574746404455325
4.64158883361277 0.579777911382256
10 0.607268501279462
21.5443469003188 0.662360892800081
46.4158883361277 0.750155756039905
100 0.874451476247172
215.443469003188 1.03826199689191
464.158883361277 1.24540614941463
1000 1.50166103878256
2154.43469003188 1.81521892673814
4641.58883361277 2.19687859073305
10000 2.66027731366106
};
\addlegendentry{CASG}
\end{axis}

\end{tikzpicture}

%% file: figs/Intro/neg_a_1_sig_0.1_2d.tex
\begin{tikzpicture}

\definecolor{darkgray176}{RGB}{176,176,176}
\definecolor{darkorange}{RGB}{255,140,0}
\definecolor{lightgray204}{RGB}{204,204,204}
\definecolor{slategray}{RGB}{112,128,144}

\begin{axis}[
legend cell align={left},
legend style={
  fill opacity=0.8,
  draw opacity=1,
  text opacity=1,
  at={(0.03,0.97)},
  anchor=north west,
  draw=lightgray204
},
tick align=outside,
tick pos=left,
x grid style={darkgray176},
xlabel={$x_1$},
xmin=-0.919238815542512, xmax=0.919238815542512,
xtick style={color=black},
y grid style={darkgray176},
ylabel={$x_2$},
ymin=-0.919238815542512, ymax=0.919238815542512,
ytick style={color=black}
]
\addplot [draw opacity=0.25, draw=black, forget plot]
table{%
x  y
0.919238815542512 -0.308172706521648
0.918628283784815 -0.306412938514171
0.912564453214343 -0.2878424573921
0.906879612054526 -0.269271976270029
0.901573760305363 -0.250701495147958
0.900668334420441 -0.247288736043253
0.896563118040738 -0.232131014025887
0.89192049776022 -0.213560532903816
0.887664762503079 -0.194990051781745
0.883795912269314 -0.176419570659674
0.88209785329837 -0.167363256147975
0.880275991607022 -0.157849089537603
0.877308373666245 -0.140414334135535
};
\addplot [draw opacity=0.25, draw=black, forget plot]
table{%
x  y
0.880403341541997 0.158514139198025
0.88209785329837 0.167363256147975
0.883795912269314 0.176419570659674
0.887664762503078 0.194990051781745
0.89192049776022 0.213560532903816
0.896563118040738 0.232131014025887
0.900668334420441 0.247288736043253
0.901573760305363 0.250701495147958
0.906879612054526 0.269271976270029
0.912564453214343 0.2878424573921
0.918628283784816 0.306412938514171
0.919238815542512 0.308172706521648
};
\addplot [draw opacity=0.25, draw=black, forget plot]
table{%
x  y
-0.919238815542512 0.308172706521648
-0.918628283784815 0.306412938514171
-0.912564453214343 0.2878424573921
-0.906879612054526 0.269271976270029
-0.901573760305363 0.250701495147958
-0.900668334420441 0.247288736043253
-0.896563118040738 0.232131014025887
-0.89192049776022 0.213560532903816
-0.887664762503078 0.194990051781745
-0.883795912269314 0.176419570659674
-0.88209785329837 0.167363256147974
-0.880275991607022 0.157849089537603
-0.877308373666245 0.140414334135536
};
\addplot [draw opacity=0.25, draw=black, forget plot]
table{%
x  y
-0.880403341541997 -0.158514139198025
-0.88209785329837 -0.167363256147974
-0.883795912269314 -0.176419570659674
-0.887664762503078 -0.194990051781745
-0.89192049776022 -0.213560532903816
-0.896563118040737 -0.232131014025887
-0.900668334420441 -0.247288736043253
-0.901573760305362 -0.250701495147958
-0.906879612054526 -0.269271976270029
-0.912564453214343 -0.2878424573921
-0.918628283784815 -0.306412938514171
-0.919238815542512 -0.308172706521648
};
\addplot [draw opacity=0.25, draw=black, forget plot]
table{%
x  y
0.919238815542512 -0.587334389119744
0.917694826138751 -0.584970155345235
0.905946154408461 -0.566399674223164
0.900668334420441 -0.557779234909398
0.89444955829025 -0.547829193101093
0.883229892612332 -0.529258711979022
0.88209785329837 -0.527318073155087
0.872190713205864 -0.510688230856951
0.863527372176299 -0.495607600175857
0.861478981700074 -0.49211774973488
0.85098262280499 -0.473547268612809
0.844956891054228 -0.462459922191407
0.840799593998301 -0.454976787490738
0.830895337399863 -0.436406306368667
0.826386409932157 -0.427584491757938
0.821290516073956 -0.417835825246596
0.8209727503892 -0.417200293877085
};
\addplot [draw opacity=0.25, draw=black, forget plot]
table{%
x  y
0.722377746239017 -0.147738918585946
0.720642297998932 -0.139278608415532
0.717309134720612 -0.120708127293461
0.714963523199732 -0.105461652407738
0.714438680109782 -0.10213764617139
0.711995195751615 -0.0835671650493192
0.710040408265081 -0.0649966839272482
0.708574317650181 -0.0464262028051773
0.707596923906914 -0.0278557216831064
0.70710822703528 -0.00928524056103532
0.70710822703528 0.00928524056103558
0.707596923906914 0.0278557216831066
0.708574317650181 0.0464262028051775
0.710040408265081 0.0649966839272484
0.711995195751615 0.0835671650493194
0.714438680109782 0.10213764617139
0.714963523199732 0.105461652407738
0.717309134720612 0.120708127293461
0.720642297998932 0.139278608415532
0.72445162745987 0.157849089537603
0.728737123103425 0.176419570659674
0.733498784929597 0.194990051781745
0.733534004321803 0.195114920535929
0.738606547722972 0.213560532903816
0.744177692059593 0.232131014025887
0.750213098424266 0.250701495147958
0.752104485443873 0.256105458061122
0.756600369710329 0.269271976270029
0.763394448169624 0.2878424573921
0.770641465192871 0.306412938514171
0.770674966565944 0.306493735943348
0.778158886155925 0.324983419636242
0.78611766377967 0.343553900758313
0.789245447688015 0.350467949397813
0.79439596436636 0.362124381880384
0.803033397446393 0.380694863002455
0.807815928810086 0.390487665318588
0.81200527551292 0.399265344124525
0.821290516073956 0.417835825246596
0.826386409932157 0.427584491757938
0.830895337399863 0.436406306368667
0.840799593998301 0.454976787490738
0.844956891054228 0.462459922191407
0.85098262280499 0.473547268612809
0.861478981700074 0.49211774973488
0.863527372176299 0.495607600175857
0.872190713205864 0.510688230856951
0.88209785329837 0.527318073155086
0.883229892612333 0.529258711979022
0.89444955829025 0.547829193101093
0.900668334420441 0.557779234909398
0.905946154408461 0.566399674223164
0.917694826138751 0.584970155345235
0.919238815542512 0.587334389119744
};
\addplot [draw opacity=0.25, draw=black, forget plot]
table{%
x  y
-0.919238815542512 0.587334389119744
-0.917694826138751 0.584970155345235
-0.905946154408461 0.566399674223164
-0.900668334420441 0.557779234909398
-0.89444955829025 0.547829193101093
-0.883229892612333 0.529258711979022
-0.88209785329837 0.527318073155086
-0.872190713205864 0.510688230856951
-0.863527372176299 0.495607600175857
-0.861478981700074 0.49211774973488
-0.85098262280499 0.473547268612809
-0.844956891054228 0.462459922191407
-0.840799593998301 0.454976787490738
-0.830895337399863 0.436406306368667
-0.826386409932157 0.427584491757938
-0.821290516073956 0.417835825246596
-0.8209727503892 0.417200293877085
};
\addplot [draw opacity=0.25, draw=black, forget plot]
table{%
x  y
-0.722377746239017 0.147738918585947
-0.720642297998932 0.139278608415532
-0.717309134720612 0.120708127293461
-0.714963523199731 0.105461652407737
-0.714438680109782 0.10213764617139
-0.711995195751615 0.0835671650493194
-0.710040408265081 0.0649966839272484
-0.708574317650181 0.0464262028051775
-0.707596923906914 0.0278557216831066
-0.70710822703528 0.00928524056103558
-0.70710822703528 -0.00928524056103532
-0.707596923906914 -0.0278557216831064
-0.708574317650181 -0.0464262028051773
-0.710040408265081 -0.0649966839272482
-0.711995195751615 -0.0835671650493192
-0.714438680109782 -0.10213764617139
-0.714963523199731 -0.105461652407737
-0.717309134720612 -0.120708127293461
-0.720642297998932 -0.139278608415532
-0.72445162745987 -0.157849089537603
-0.728737123103425 -0.176419570659674
-0.733498784929597 -0.194990051781745
-0.733534004321802 -0.195114920535928
-0.738606547722972 -0.213560532903816
-0.744177692059593 -0.232131014025887
-0.750213098424266 -0.250701495147958
-0.752104485443873 -0.256105458061122
-0.756600369710329 -0.269271976270029
-0.763394448169624 -0.2878424573921
-0.770641465192871 -0.306412938514171
-0.770674966565944 -0.306493735943348
-0.778158886155925 -0.324983419636241
-0.78611766377967 -0.343553900758312
-0.789245447688015 -0.350467949397813
-0.79439596436636 -0.362124381880383
-0.803033397446393 -0.380694863002454
-0.807815928810086 -0.390487665318587
-0.81200527551292 -0.399265344124525
-0.821290516073956 -0.417835825246596
-0.826386409932157 -0.427584491757938
-0.830895337399863 -0.436406306368667
-0.840799593998301 -0.454976787490738
-0.844956891054228 -0.462459922191407
-0.85098262280499 -0.473547268612809
-0.861478981700074 -0.49211774973488
-0.863527372176299 -0.495607600175856
-0.872190713205864 -0.510688230856951
-0.88209785329837 -0.527318073155086
-0.883229892612332 -0.529258711979022
-0.89444955829025 -0.547829193101093
-0.900668334420441 -0.557779234909398
-0.905946154408461 -0.566399674223164
-0.917694826138751 -0.584970155345235
-0.919238815542512 -0.587334389119744
};
\addplot [draw opacity=0.25, draw=black, forget plot]
table{%
x  y
0.919238815542512 -0.771354549895256
0.918656315545959 -0.770674966565944
0.903117749709124 -0.752104485443873
0.900668334420441 -0.749103951715237
0.887693378259245 -0.733534004321802
0.88209785329837 -0.726647204369956
0.872402883816695 -0.714963523199731
0.863527372176299 -0.703985916697137
0.857254997490644 -0.69639304207766
0.844956891054228 -0.6811035043459
0.842259226488862 -0.677822560955589
0.827402841591205 -0.659252079833519
0.826386409932157 -0.657945239129028
0.8126539686909 -0.640681598711448
0.807815928810086 -0.634420605924512
0.798082845010212 -0.622111117589377
0.789245447688015 -0.610595721078636
0.78370216692197 -0.603540636467306
0.770674966565944 -0.586442436000022
0.769525869469525 -0.584970155345235
0.755484773986984 -0.566399674223164
0.752104485443873 -0.56177994654758
0.741641420359008 -0.547829193101093
0.733534004321803 -0.536646550291154
0.728040483525602 -0.529258711979022
0.714963523199732 -0.511044374382274
0.714701101654757 -0.510688230856951
0.701506286120653 -0.49211774973488
0.696393042077661 -0.484644546902813
0.68859495463063 -0.473547268612809
0.67782256095559 -0.457604125973748
0.675998020342388 -0.454976787490738
0.663617699594341 -0.436406306368667
0.659252079833519 -0.429573162395207
0.651538972564432 -0.417835825246596
0.640681598711448 -0.400562730480485
0.639842113422297 -0.399265344124525
0.628372110376312 -0.380694863002454
0.622111117589377 -0.370051175264664
0.617307000386782 -0.362124381880383
0.606614905195287 -0.343553900758312
0.603540636467306 -0.33791774142368
0.596265080461871 -0.324983419636241
0.586399512365771 -0.306412938514171
0.584970155345235 -0.303554224473098
0.576860856206655 -0.2878424573921
0.567875139534685 -0.269271976270029
0.566399674223164 -0.266004874508803
0.559258097188103 -0.250701495147958
0.551210888701872 -0.232131014025887
0.547829193101093 -0.223676775023938
0.543643161878973 -0.213560532903816
0.536599186280946 -0.194990051781745
0.530195572100922 -0.176419570659674
0.529258711979022 -0.173400799155774
0.524259948173181 -0.157849089537603
0.518954096424018 -0.139278608415532
0.5143114761435 -0.120708127293461
0.510688230856951 -0.103799649289564
0.51031889683069 -0.10213764617139
0.506879918845122 -0.0835671650493192
0.504128736456667 -0.0649966839272482
0.502065349665325 -0.0464262028051773
0.500689758471098 -0.0278557216831064
0.500001962873984 -0.00928524056103532
0.500001962873984 0.00928524056103558
0.500689758471098 0.0278557216831066
0.500701265273086 0.0280110635099401
};
\addplot [draw opacity=0.25, draw=black, forget plot]
table{%
x  y
0.585277896334127 0.304169706450882
0.586399512365771 0.306412938514171
0.596265080461872 0.324983419636242
0.603540636467306 0.33791774142368
0.606614905195287 0.343553900758313
0.617307000386783 0.362124381880384
0.622111117589377 0.370051175264664
0.628372110376312 0.380694863002455
0.639842113422297 0.399265344124525
0.640681598711448 0.400562730480485
0.651538972564432 0.417835825246596
0.659252079833519 0.429573162395207
0.663617699594341 0.436406306368667
0.675998020342388 0.454976787490738
0.67782256095559 0.457604125973748
0.688594954630631 0.473547268612809
0.696393042077661 0.484644546902813
0.701506286120654 0.49211774973488
0.714701101654757 0.510688230856951
0.714963523199732 0.511044374382274
0.728040483525602 0.529258711979022
0.733534004321803 0.536646550291154
0.741641420359008 0.547829193101093
0.752104485443873 0.56177994654758
0.755484773986984 0.566399674223164
0.769525869469525 0.584970155345235
0.770674966565944 0.586442436000022
0.78370216692197 0.603540636467306
0.789245447688015 0.610595721078636
0.798082845010212 0.622111117589377
0.807815928810086 0.634420605924512
0.8126539686909 0.640681598711448
0.826386409932157 0.657945239129028
0.827402841591205 0.659252079833519
0.842259226488862 0.67782256095559
0.844956891054228 0.6811035043459
0.857254997490645 0.696393042077661
0.863527372176299 0.703985916697137
0.872402883816695 0.714963523199732
0.88209785329837 0.726647204369956
0.887693378259246 0.733534004321803
0.900668334420441 0.749103951715237
0.903117749709124 0.752104485443873
0.918656315545959 0.770674966565944
0.919238815542512 0.771354549895256
};
\addplot [draw opacity=0.25, draw=black, forget plot]
table{%
x  y
-0.919238815542512 0.771354549895256
-0.918656315545959 0.770674966565944
-0.903117749709124 0.752104485443873
-0.900668334420441 0.749103951715237
-0.887693378259246 0.733534004321803
-0.88209785329837 0.726647204369956
-0.872402883816695 0.714963523199732
-0.863527372176299 0.703985916697137
-0.857254997490645 0.696393042077661
-0.844956891054228 0.681103504345899
-0.842259226488862 0.67782256095559
-0.827402841591205 0.659252079833519
-0.826386409932157 0.657945239129028
-0.8126539686909 0.640681598711448
-0.807815928810086 0.634420605924512
-0.798082845010212 0.622111117589377
-0.789245447688015 0.610595721078636
-0.78370216692197 0.603540636467306
-0.770674966565944 0.586442436000022
-0.769525869469525 0.584970155345235
-0.755484773986984 0.566399674223164
-0.752104485443873 0.561779946547579
-0.741641420359008 0.547829193101093
-0.733534004321802 0.536646550291154
-0.728040483525602 0.529258711979022
-0.714963523199731 0.511044374382274
-0.714701101654757 0.510688230856951
-0.701506286120654 0.49211774973488
-0.69639304207766 0.484644546902813
-0.688594954630631 0.473547268612809
-0.677822560955589 0.457604125973748
-0.675998020342388 0.454976787490738
-0.663617699594341 0.436406306368667
-0.659252079833519 0.429573162395206
-0.651538972564432 0.417835825246596
-0.640681598711448 0.400562730480485
-0.639842113422297 0.399265344124525
-0.628372110376312 0.380694863002455
-0.622111117589377 0.370051175264664
-0.617307000386783 0.362124381880384
-0.606614905195287 0.343553900758313
-0.603540636467306 0.33791774142368
-0.596265080461872 0.324983419636242
-0.586399512365771 0.306412938514171
-0.584970155345235 0.303554224473097
-0.576860856206655 0.2878424573921
-0.567875139534685 0.269271976270029
-0.566399674223164 0.266004874508802
-0.559258097188103 0.250701495147958
-0.551210888701872 0.232131014025887
-0.547829193101093 0.223676775023938
-0.543643161878973 0.213560532903816
-0.536599186280946 0.194990051781745
-0.530195572100922 0.176419570659674
-0.529258711979022 0.173400799155774
-0.524259948173181 0.157849089537603
-0.518954096424018 0.139278608415532
-0.5143114761435 0.120708127293461
-0.510688230856951 0.103799649289563
-0.51031889683069 0.10213764617139
-0.506879918845122 0.0835671650493194
-0.504128736456667 0.0649966839272484
-0.502065349665325 0.0464262028051775
-0.500689758471098 0.0278557216831066
-0.500001962873984 0.00928524056103558
-0.500001962873984 -0.00928524056103532
-0.500689758471098 -0.0278557216831064
-0.500701265273085 -0.0280110635099392
};
\addplot [draw opacity=0.25, draw=black, forget plot]
table{%
x  y
-0.585277896334127 -0.304169706450881
-0.586399512365771 -0.30641293851417
-0.596265080461872 -0.324983419636241
-0.603540636467306 -0.33791774142368
-0.606614905195287 -0.343553900758312
-0.617307000386782 -0.362124381880383
-0.622111117589377 -0.370051175264664
-0.628372110376312 -0.380694863002454
-0.639842113422297 -0.399265344124525
-0.640681598711448 -0.400562730480485
-0.651538972564432 -0.417835825246596
-0.659252079833519 -0.429573162395206
-0.663617699594341 -0.436406306368667
-0.675998020342388 -0.454976787490738
-0.677822560955589 -0.457604125973748
-0.68859495463063 -0.473547268612809
-0.69639304207766 -0.484644546902813
-0.701506286120654 -0.49211774973488
-0.714701101654757 -0.510688230856951
-0.714963523199731 -0.511044374382274
-0.728040483525602 -0.529258711979022
-0.733534004321802 -0.536646550291154
-0.741641420359008 -0.547829193101093
-0.752104485443873 -0.561779946547579
-0.755484773986984 -0.566399674223164
-0.769525869469525 -0.584970155345235
-0.770674966565944 -0.586442436000022
-0.78370216692197 -0.603540636467306
-0.789245447688015 -0.610595721078636
-0.798082845010212 -0.622111117589377
-0.807815928810086 -0.634420605924512
-0.8126539686909 -0.640681598711448
-0.826386409932157 -0.657945239129028
-0.827402841591205 -0.659252079833519
-0.842259226488862 -0.677822560955589
-0.844956891054228 -0.681103504345899
-0.857254997490644 -0.69639304207766
-0.863527372176299 -0.703985916697137
-0.872402883816694 -0.714963523199731
-0.88209785329837 -0.726647204369956
-0.887693378259245 -0.733534004321802
-0.900668334420441 -0.749103951715237
-0.903117749709124 -0.752104485443873
-0.918656315545959 -0.770674966565944
-0.919238815542512 -0.771354549895256
};
\addplot [draw opacity=0.25, draw=black, forget plot]
table{%
x  y
0.919238815542512 -0.919238815542512
0.900668334420441 -0.900668334420441
0.900668334420441 -0.900668334420441
0.88209785329837 -0.88209785329837
0.88209785329837 -0.88209785329837
0.863527372176299 -0.863527372176299
0.863527372176299 -0.863527372176299
0.844956891054228 -0.844956891054228
0.844956891054228 -0.844956891054228
0.826386409932157 -0.826386409932157
0.826386409932157 -0.826386409932157
0.807815928810086 -0.807815928810086
0.807815928810086 -0.807815928810086
0.789245447688015 -0.789245447688015
0.789245447688015 -0.789245447688015
0.770674966565944 -0.770674966565944
0.770674966565944 -0.770674966565944
0.752104485443873 -0.752104485443873
0.752104485443873 -0.752104485443873
0.733534004321803 -0.733534004321803
0.733534004321802 -0.733534004321802
0.714963523199732 -0.714963523199731
0.714963523199731 -0.714963523199731
0.696393042077661 -0.696393042077661
0.69639304207766 -0.69639304207766
0.67782256095559 -0.67782256095559
0.677822560955589 -0.677822560955589
0.659252079833519 -0.659252079833519
0.659252079833519 -0.659252079833519
0.640681598711448 -0.640681598711448
0.640681598711448 -0.640681598711448
0.622111117589377 -0.622111117589377
0.622111117589377 -0.622111117589377
0.603540636467306 -0.603540636467306
0.603540636467306 -0.603540636467306
0.584970155345235 -0.584970155345235
0.584970155345235 -0.584970155345235
0.566399674223164 -0.566399674223164
0.566399674223164 -0.566399674223164
0.547829193101093 -0.547829193101093
0.547829193101093 -0.547829193101093
0.529258711979022 -0.529258711979022
0.529258711979022 -0.529258711979022
0.510688230856951 -0.510688230856951
0.510688230856951 -0.510688230856951
0.49211774973488 -0.49211774973488
0.49211774973488 -0.49211774973488
0.487310423676866 -0.487310423676866
};
\addplot [draw opacity=0.25, draw=black, forget plot]
table{%
x  y
0.348361226816326 -0.348361226816326
0.343553900758313 -0.343553900758313
0.343553900758312 -0.343553900758312
0.324983419636242 -0.324983419636242
0.324983419636241 -0.324983419636241
0.306412938514171 -0.306412938514171
0.306412938514171 -0.306412938514171
0.2878424573921 -0.2878424573921
0.2878424573921 -0.2878424573921
0.269271976270029 -0.269271976270029
0.269271976270029 -0.269271976270029
0.250701495147958 -0.250701495147958
0.250701495147958 -0.250701495147958
0.232131014025887 -0.232131014025887
0.232131014025887 -0.232131014025887
0.213560532903816 -0.213560532903816
0.213560532903816 -0.213560532903816
0.194990051781745 -0.194990051781745
0.194990051781745 -0.194990051781745
0.176419570659674 -0.176419570659674
0.176419570659674 -0.176419570659674
0.157849089537603 -0.157849089537603
0.157849089537603 -0.157849089537603
0.139278608415532 -0.139278608415532
0.139278608415532 -0.139278608415532
0.120708127293461 -0.120708127293461
0.120708127293461 -0.120708127293461
0.10213764617139 -0.10213764617139
0.10213764617139 -0.10213764617139
0.0835671650493194 -0.0835671650493194
0.0835671650493191 -0.0835671650493192
0.0649966839272484 -0.0649966839272483
0.0649966839272481 -0.0649966839272482
0.0464262028051775 -0.0464262028051774
0.0464262028051772 -0.0464262028051773
0.0278557216831066 -0.0278557216831065
0.0278557216831063 -0.0278557216831064
0.00928524056103558 -0.00928524056103545
-0.00927328481795417 -0.00928524056103532
0.00928524056103558 0.00927639432820322
0.00928524056103558 0.00928524056103558
0.0278557216831066 0.0278557216831066
0.0464262028051775 0.0464262028051775
0.0649966839272484 0.0649966839272484
0.0835671650493194 0.0835671650493194
0.0835671650493194 0.0835671650493194
0.10213764617139 0.10213764617139
0.120708127293461 0.120708127293461
0.139278608415532 0.139278608415532
0.157849089537603 0.157849089537603
0.176419570659674 0.176419570659674
0.194990051781745 0.194990051781745
0.194990051781745 0.194990051781745
0.213560532903816 0.213560532903816
0.232131014025887 0.232131014025887
0.232131014025887 0.232131014025887
0.250701495147958 0.250701495147958
0.250701495147958 0.250701495147958
0.269271976270029 0.269271976270029
0.2878424573921 0.2878424573921
0.306412938514171 0.306412938514171
0.306412938514171 0.306412938514171
0.324983419636242 0.324983419636242
0.343553900758313 0.343553900758313
0.362124381880384 0.362124381880384
0.362124381880384 0.362124381880384
0.380694863002455 0.380694863002455
0.399265344124525 0.399265344124525
0.417835825246596 0.417835825246596
0.436406306368667 0.436406306368667
0.454976787490738 0.454976787490738
0.473547268612809 0.473547268612809
0.49211774973488 0.49211774973488
0.510688230856951 0.510688230856951
0.529258711979022 0.529258711979022
0.547829193101093 0.547829193101093
0.566399674223164 0.566399674223164
0.584970155345235 0.584970155345235
0.603540636467306 0.603540636467306
0.622111117589377 0.622111117589377
0.640681598711448 0.640681598711448
0.659252079833519 0.659252079833519
0.67782256095559 0.67782256095559
0.696393042077661 0.696393042077661
0.714963523199732 0.714963523199732
0.733534004321803 0.733534004321803
0.752104485443873 0.752104485443873
0.770674966565944 0.770674966565944
0.789245447688015 0.789245447688015
0.807815928810086 0.807815928810086
0.826386409932157 0.826386409932157
0.844956891054228 0.844956891054228
0.863527372176299 0.863527372176299
0.88209785329837 0.88209785329837
0.900668334420441 0.900668334420441
0.919238815542512 0.919238815542512
};
\addplot [draw opacity=0.25, draw=black, forget plot]
table{%
x  y
-0.919238815542512 0.919238815542512
-0.900668334420441 0.900668334420441
-0.900668334420441 0.900668334420441
-0.88209785329837 0.88209785329837
-0.88209785329837 0.88209785329837
-0.863527372176299 0.863527372176299
-0.863527372176299 0.863527372176299
-0.844956891054228 0.844956891054228
-0.844956891054228 0.844956891054228
-0.826386409932157 0.826386409932157
-0.826386409932157 0.826386409932157
-0.807815928810086 0.807815928810086
-0.807815928810086 0.807815928810086
-0.789245447688015 0.789245447688015
-0.789245447688015 0.789245447688015
-0.770674966565944 0.770674966565944
-0.770674966565944 0.770674966565944
-0.752104485443874 0.752104485443874
-0.752104485443873 0.752104485443873
-0.733534004321803 0.733534004321803
-0.733534004321802 0.733534004321802
-0.714963523199732 0.714963523199732
-0.714963523199731 0.714963523199731
-0.696393042077661 0.696393042077661
-0.696393042077661 0.69639304207766
-0.67782256095559 0.67782256095559
-0.677822560955589 0.677822560955589
-0.659252079833519 0.659252079833519
-0.659252079833519 0.659252079833519
-0.640681598711448 0.640681598711448
-0.640681598711448 0.640681598711448
-0.622111117589377 0.622111117589377
-0.622111117589377 0.622111117589377
-0.603540636467306 0.603540636467306
-0.603540636467306 0.603540636467306
-0.584970155345235 0.584970155345235
-0.584970155345235 0.584970155345235
-0.566399674223164 0.566399674223164
-0.566399674223164 0.566399674223164
-0.547829193101093 0.547829193101093
-0.547829193101093 0.547829193101093
-0.529258711979022 0.529258711979022
-0.529258711979022 0.529258711979022
-0.510688230856951 0.510688230856951
-0.510688230856951 0.510688230856951
-0.49211774973488 0.49211774973488
-0.49211774973488 0.49211774973488
-0.473547268612809 0.473547268612809
-0.473547268612809 0.473547268612809
-0.454976787490738 0.454976787490738
-0.454976787490738 0.454976787490738
-0.436406306368667 0.436406306368667
-0.436406306368667 0.436406306368667
-0.417835825246596 0.417835825246596
-0.417835825246596 0.417835825246596
-0.399265344124525 0.399265344124525
-0.399265344124525 0.399265344124525
-0.380694863002455 0.380694863002455
-0.380694863002454 0.380694863002454
-0.362124381880384 0.362124381880384
-0.362124381880383 0.362124381880383
-0.343553900758313 0.343553900758313
-0.343553900758312 0.343553900758312
-0.324983419636242 0.324983419636242
-0.324983419636241 0.324983419636241
-0.306412938514171 0.306412938514171
-0.306412938514171 0.306412938514171
-0.2878424573921 0.2878424573921
-0.2878424573921 0.2878424573921
-0.269271976270029 0.269271976270029
-0.269271976270029 0.269271976270029
-0.250701495147958 0.250701495147958
-0.250701495147958 0.250701495147958
-0.232131014025887 0.232131014025887
-0.232131014025887 0.232131014025887
-0.213560532903816 0.213560532903816
-0.213560532903816 0.213560532903816
-0.194990051781745 0.194990051781745
-0.194990051781745 0.194990051781745
-0.176419570659674 0.176419570659674
-0.176419570659674 0.176419570659674
-0.157849089537603 0.157849089537603
-0.157849089537603 0.157849089537603
-0.139278608415532 0.139278608415532
-0.139278608415532 0.139278608415532
-0.120708127293461 0.120708127293461
-0.120708127293461 0.120708127293461
-0.10213764617139 0.10213764617139
-0.10213764617139 0.10213764617139
-0.0835671650493194 0.0835671650493194
-0.0835671650493192 0.0835671650493191
-0.0649966839272483 0.0649966839272484
-0.0649966839272482 0.0649966839272481
-0.0464262028051775 0.0464262028051775
-0.0464262028051773 0.0464262028051772
-0.0278557216831065 0.0278557216831066
-0.0278557216831064 0.0278557216831063
-0.00928524056103545 0.00928524056103558
-0.00928524056103532 -0.00928524056103532
-0.0278557216831064 -0.0278557216831064
-0.0464262028051773 -0.0464262028051773
-0.0464262028051773 -0.0464262028051773
-0.0649966839272482 -0.0649966839272482
-0.0835671650493192 -0.0835671650493192
-0.10213764617139 -0.10213764617139
-0.120708127293461 -0.120708127293461
-0.139278608415532 -0.139278608415532
-0.157849089537603 -0.157849089537603
-0.176419570659674 -0.176419570659674
-0.194990051781745 -0.194990051781745
-0.213560532903816 -0.213560532903816
-0.232131014025887 -0.232131014025887
-0.250701495147958 -0.250701495147958
-0.250701495147958 -0.250701495147958
-0.269271976270029 -0.269271976270029
-0.2878424573921 -0.2878424573921
-0.306412938514171 -0.306412938514171
-0.324983419636241 -0.324983419636241
-0.343553900758312 -0.343553900758312
-0.362124381880383 -0.362124381880383
-0.380694863002454 -0.380694863002454
-0.380694863002454 -0.380694863002454
-0.399265344124525 -0.399265344124525
-0.404072670182539 -0.404072670182539
};
\addplot [draw opacity=0.25, draw=black, forget plot]
table{%
x  y
-0.543021867043079 -0.543021867043079
-0.547829193101093 -0.547829193101093
-0.566399674223164 -0.566399674223164
-0.584970155345235 -0.584970155345235
-0.603540636467306 -0.603540636467306
-0.603540636467306 -0.603540636467306
-0.622111117589377 -0.622111117589377
-0.640681598711448 -0.640681598711448
-0.659252079833518 -0.659252079833519
-0.659252079833519 -0.659252079833519
-0.677822560955589 -0.677822560955589
-0.69639304207766 -0.69639304207766
-0.714963523199731 -0.714963523199731
-0.733534004321802 -0.733534004321802
-0.733534004321802 -0.733534004321802
-0.752104485443873 -0.752104485443873
-0.770674966565944 -0.770674966565944
-0.789245447688015 -0.789245447688015
-0.807815928810086 -0.807815928810086
-0.826386409932157 -0.826386409932157
-0.844956891054228 -0.844956891054228
-0.844956891054228 -0.844956891054228
-0.863527372176299 -0.863527372176299
-0.88209785329837 -0.88209785329837
-0.900668334420441 -0.900668334420441
-0.919238815542512 -0.919238815542512
};
\addplot [draw opacity=0.25, draw=black, forget plot]
table{%
x  y
0.771354549895256 -0.919238815542512
0.770674966565944 -0.918656315545959
0.752104485443873 -0.903117749709124
0.749103951715237 -0.900668334420441
0.733534004321803 -0.887693378259246
0.726647204369956 -0.88209785329837
0.714963523199732 -0.872402883816695
0.703985916697137 -0.863527372176299
0.696393042077661 -0.857254997490645
0.681103504345899 -0.844956891054228
0.67782256095559 -0.842259226488862
0.661554302326234 -0.829244619585378
};
\addplot [draw opacity=0.25, draw=black, forget plot]
table{%
x  y
0.510341641093661 -0.714454840507156
0.49211774973488 -0.701506286120654
0.484644546902813 -0.69639304207766
0.473547268612809 -0.688594954630631
0.457604125973748 -0.677822560955589
0.454976787490738 -0.675998020342388
0.436406306368667 -0.663617699594341
0.429573162395206 -0.659252079833519
0.417835825246596 -0.651538972564432
0.400562730480485 -0.640681598711448
0.399265344124525 -0.639842113422297
0.380694863002455 -0.628372110376312
0.370051175264664 -0.622111117589377
0.362124381880384 -0.617307000386783
0.343553900758313 -0.606614905195287
0.33791774142368 -0.603540636467306
0.324983419636242 -0.596265080461872
0.306412938514171 -0.586399512365771
0.303554224473097 -0.584970155345235
0.2878424573921 -0.576860856206655
0.269271976270029 -0.567875139534685
0.266004874508802 -0.566399674223164
0.250701495147958 -0.559258097188103
0.232131014025887 -0.551210888701872
0.223676775023938 -0.547829193101093
0.213560532903816 -0.543643161878973
0.194990051781745 -0.536599186280946
0.176419570659674 -0.530195572100922
0.173400799155774 -0.529258711979022
0.157849089537603 -0.524259948173181
0.139278608415532 -0.518954096424018
0.120708127293461 -0.5143114761435
0.103799649289563 -0.510688230856951
0.10213764617139 -0.51031889683069
0.0835671650493194 -0.506879918845122
0.0649966839272484 -0.504128736456666
0.0464262028051775 -0.502065349665325
0.0278557216831066 -0.500689758471098
0.00928524056103558 -0.500001962873984
-0.00928524056103532 -0.500001962873984
-0.0278557216831064 -0.500689758471098
-0.0464262028051773 -0.502065349665325
-0.0649966839272482 -0.504128736456667
-0.0835671650493192 -0.506879918845122
-0.10213764617139 -0.51031889683069
-0.103799649289563 -0.510688230856951
-0.120708127293461 -0.5143114761435
-0.139278608415532 -0.518954096424018
-0.157849089537603 -0.524259948173181
-0.173400799155774 -0.529258711979022
-0.176419570659674 -0.530195572100922
-0.194990051781745 -0.536599186280946
-0.213560532903816 -0.543643161878973
-0.223676775023938 -0.547829193101093
-0.232131014025887 -0.551210888701872
-0.250701495147958 -0.559258097188103
-0.266004874508802 -0.566399674223164
-0.269271976270029 -0.567875139534685
-0.2878424573921 -0.576860856206655
-0.303554224473097 -0.584970155345235
-0.306412938514171 -0.586399512365771
-0.324983419636241 -0.596265080461872
-0.33791774142368 -0.603540636467306
-0.343553900758312 -0.606614905195287
-0.362124381880383 -0.617307000386782
-0.370051175264664 -0.622111117589377
-0.380694863002454 -0.628372110376312
-0.399265344124525 -0.639842113422297
-0.400562730480485 -0.640681598711448
-0.417835825246596 -0.651538972564432
-0.429573162395206 -0.659252079833519
-0.436406306368667 -0.663617699594341
-0.454976787490738 -0.675998020342388
-0.457604125973748 -0.677822560955589
-0.473547268612809 -0.68859495463063
-0.484644546902813 -0.69639304207766
-0.49211774973488 -0.701506286120654
-0.510688230856951 -0.714701101654757
-0.511044374382274 -0.714963523199731
-0.529258711979022 -0.728040483525602
-0.536646550291154 -0.733534004321802
-0.547829193101093 -0.741641420359008
-0.561779946547579 -0.752104485443873
-0.566399674223164 -0.755484773986984
-0.584970155345235 -0.769525869469525
-0.586442436000022 -0.770674966565944
-0.603540636467306 -0.78370216692197
-0.610595721078636 -0.789245447688015
-0.622111117589377 -0.798082845010212
-0.634420605924512 -0.807815928810086
-0.640681598711448 -0.8126539686909
-0.657945239129028 -0.826386409932157
-0.659252079833519 -0.827402841591205
-0.677822560955589 -0.842259226488862
-0.681103504345899 -0.844956891054228
-0.69639304207766 -0.857254997490644
-0.703985916697137 -0.863527372176299
-0.714963523199731 -0.872402883816694
-0.726647204369956 -0.88209785329837
-0.733534004321802 -0.887693378259245
-0.749103951715237 -0.900668334420441
-0.752104485443873 -0.903117749709124
-0.770674966565944 -0.918656315545959
-0.771354549895256 -0.919238815542512
};
\addplot [draw opacity=0.25, draw=black, forget plot]
table{%
x  y
-0.771354549895256 0.919238815542512
-0.770674966565944 0.918656315545958
-0.752104485443873 0.903117749709124
-0.749103951715237 0.900668334420441
-0.733534004321802 0.887693378259245
-0.726647204369956 0.88209785329837
-0.714963523199731 0.872402883816694
-0.703985916697137 0.863527372176299
-0.69639304207766 0.857254997490644
-0.6811035043459 0.844956891054228
-0.677822560955589 0.842259226488862
-0.661554302326234 0.829244619585378
};
\addplot [draw opacity=0.25, draw=black, forget plot]
table{%
x  y
-0.510341641093661 0.714454840507156
-0.49211774973488 0.701506286120653
-0.484644546902813 0.696393042077661
-0.473547268612809 0.68859495463063
-0.457604125973748 0.67782256095559
-0.454976787490738 0.675998020342388
-0.436406306368667 0.663617699594341
-0.429573162395207 0.659252079833519
-0.417835825246596 0.651538972564432
-0.400562730480485 0.640681598711448
-0.399265344124525 0.639842113422297
-0.380694863002454 0.628372110376312
-0.370051175264664 0.622111117589377
-0.362124381880383 0.617307000386782
-0.343553900758312 0.606614905195287
-0.33791774142368 0.603540636467306
-0.324983419636241 0.596265080461871
-0.306412938514171 0.586399512365771
-0.303554224473098 0.584970155345235
-0.2878424573921 0.576860856206655
-0.269271976270029 0.567875139534685
-0.266004874508803 0.566399674223164
-0.250701495147958 0.559258097188103
-0.232131014025887 0.551210888701872
-0.223676775023938 0.547829193101093
-0.213560532903816 0.543643161878973
-0.194990051781745 0.536599186280946
-0.176419570659674 0.530195572100922
-0.173400799155774 0.529258711979022
-0.157849089537603 0.524259948173181
-0.139278608415532 0.518954096424018
-0.120708127293461 0.5143114761435
-0.103799649289564 0.510688230856951
-0.10213764617139 0.51031889683069
-0.0835671650493192 0.506879918845122
-0.0649966839272482 0.504128736456667
-0.0464262028051773 0.502065349665325
-0.0278557216831064 0.500689758471098
-0.00928524056103532 0.500001962873984
0.00928524056103558 0.500001962873984
0.0278557216831066 0.500689758471098
0.0464262028051775 0.502065349665325
0.0649966839272484 0.504128736456667
0.0835671650493194 0.506879918845122
0.10213764617139 0.51031889683069
0.103799649289564 0.510688230856951
0.120708127293461 0.5143114761435
0.139278608415532 0.518954096424018
0.157849089537603 0.524259948173181
0.173400799155774 0.529258711979022
0.176419570659674 0.530195572100922
0.194990051781745 0.536599186280946
0.213560532903816 0.543643161878973
0.223676775023938 0.547829193101093
0.232131014025887 0.551210888701872
0.250701495147958 0.559258097188103
0.266004874508803 0.566399674223164
0.269271976270029 0.567875139534686
0.2878424573921 0.576860856206655
0.303554224473098 0.584970155345235
0.306412938514171 0.586399512365771
0.324983419636242 0.596265080461872
0.33791774142368 0.603540636467306
0.343553900758313 0.606614905195287
0.362124381880384 0.617307000386783
0.370051175264664 0.622111117589377
0.380694863002455 0.628372110376312
0.399265344124525 0.639842113422297
0.400562730480485 0.640681598711448
0.417835825246596 0.651538972564432
0.429573162395207 0.659252079833519
0.436406306368667 0.663617699594341
0.454976787490738 0.675998020342388
0.457604125973748 0.67782256095559
0.473547268612809 0.688594954630631
0.484644546902813 0.696393042077661
0.49211774973488 0.701506286120654
0.510688230856951 0.714701101654757
0.511044374382274 0.714963523199732
0.529258711979022 0.728040483525602
0.536646550291154 0.733534004321803
0.547829193101093 0.741641420359008
0.56177994654758 0.752104485443873
0.566399674223164 0.755484773986984
0.584970155345235 0.769525869469525
0.586442436000022 0.770674966565944
0.603540636467306 0.78370216692197
0.610595721078636 0.789245447688015
0.622111117589377 0.798082845010212
0.634420605924512 0.807815928810086
0.640681598711448 0.8126539686909
0.657945239129028 0.826386409932157
0.659252079833519 0.827402841591205
0.67782256095559 0.842259226488862
0.6811035043459 0.844956891054228
0.696393042077661 0.857254997490645
0.703985916697137 0.863527372176299
0.714963523199732 0.872402883816695
0.726647204369956 0.88209785329837
0.733534004321803 0.887693378259246
0.749103951715237 0.900668334420441
0.752104485443873 0.903117749709124
0.770674966565944 0.918656315545959
0.771354549895256 0.919238815542512
};
\addplot [draw opacity=0.25, draw=black, forget plot]
table{%
x  y
0.587334389119744 -0.919238815542512
0.584970155345235 -0.917694826138751
0.566399674223164 -0.905946154408461
0.557779234909398 -0.900668334420441
0.547829193101093 -0.89444955829025
0.529258711979022 -0.883229892612333
0.527318073155086 -0.88209785329837
0.510688230856951 -0.872190713205864
0.495607600175857 -0.863527372176299
0.49211774973488 -0.861478981700074
0.473547268612809 -0.85098262280499
0.462459922191407 -0.844956891054228
0.461630828671186 -0.844496283542994
};
\addplot [draw opacity=0.25, draw=black, forget plot]
table{%
x  y
0.298224158949158 -0.767445843899207
0.2878424573921 -0.763394448169624
0.269271976270029 -0.756600369710329
0.256105458061121 -0.752104485443873
0.250701495147958 -0.750213098424266
0.232131014025887 -0.744177692059593
0.213560532903816 -0.738606547722972
0.195114920535927 -0.733534004321802
0.194990051781745 -0.733498784929597
0.176419570659674 -0.728737123103425
0.157849089537603 -0.72445162745987
0.139278608415532 -0.720642297998932
0.120708127293461 -0.717309134720612
0.105461652407737 -0.714963523199731
0.10213764617139 -0.714438680109782
0.0835671650493194 -0.711995195751615
0.0649966839272484 -0.710040408265081
0.0464262028051775 -0.708574317650181
0.0278557216831066 -0.707596923906914
0.00928524056103558 -0.70710822703528
-0.00928524056103532 -0.70710822703528
-0.0278557216831064 -0.707596923906914
-0.0464262028051773 -0.708574317650181
-0.0649966839272482 -0.710040408265081
-0.0835671650493192 -0.711995195751615
-0.10213764617139 -0.714438680109782
-0.105461652407737 -0.714963523199731
-0.120708127293461 -0.717309134720612
-0.139278608415532 -0.720642297998932
-0.157849089537603 -0.72445162745987
-0.176419570659674 -0.728737123103425
-0.194990051781745 -0.733498784929597
-0.195114920535927 -0.733534004321802
-0.213560532903816 -0.738606547722972
-0.232131014025887 -0.744177692059593
-0.250701495147958 -0.750213098424266
-0.256105458061121 -0.752104485443873
-0.269271976270029 -0.756600369710329
-0.2878424573921 -0.763394448169624
-0.306412938514171 -0.770641465192871
-0.306493735943348 -0.770674966565944
-0.324983419636241 -0.778158886155925
-0.343553900758312 -0.78611766377967
-0.350467949397813 -0.789245447688015
-0.362124381880383 -0.79439596436636
-0.380694863002454 -0.803033397446393
-0.390487665318587 -0.807815928810086
-0.399265344124525 -0.81200527551292
-0.417835825246596 -0.821290516073956
-0.427584491757938 -0.826386409932157
-0.436406306368667 -0.830895337399863
-0.454976787490738 -0.840799593998301
-0.462459922191407 -0.844956891054228
-0.473547268612809 -0.85098262280499
-0.49211774973488 -0.861478981700074
-0.495607600175856 -0.863527372176299
-0.510688230856951 -0.872190713205864
-0.527318073155086 -0.88209785329837
-0.529258711979022 -0.883229892612332
-0.547829193101093 -0.89444955829025
-0.557779234909398 -0.900668334420441
-0.566399674223164 -0.905946154408461
-0.584970155345235 -0.917694826138751
-0.587334389119744 -0.919238815542512
};
\addplot [draw opacity=0.25, draw=black, forget plot]
table{%
x  y
-0.587334389119744 0.919238815542512
-0.584970155345235 0.917694826138751
-0.566399674223164 0.905946154408461
-0.557779234909398 0.900668334420441
-0.547829193101093 0.89444955829025
-0.529258711979022 0.883229892612332
-0.527318073155086 0.88209785329837
-0.510688230856951 0.872190713205864
-0.495607600175857 0.863527372176299
-0.49211774973488 0.861478981700074
-0.473547268612809 0.85098262280499
-0.462459922191407 0.844956891054228
-0.461630828671186 0.844496283542994
};
\addplot [draw opacity=0.25, draw=black, forget plot]
table{%
x  y
-0.298224158949158 0.767445843899207
-0.2878424573921 0.763394448169624
-0.269271976270029 0.756600369710329
-0.256105458061122 0.752104485443873
-0.250701495147958 0.750213098424266
-0.232131014025887 0.744177692059593
-0.213560532903816 0.738606547722972
-0.195114920535929 0.733534004321803
-0.194990051781745 0.733498784929597
-0.176419570659674 0.728737123103425
-0.157849089537603 0.72445162745987
-0.139278608415532 0.720642297998932
-0.120708127293461 0.717309134720612
-0.105461652407738 0.714963523199732
-0.10213764617139 0.714438680109782
-0.0835671650493192 0.711995195751615
-0.0649966839272482 0.710040408265081
-0.0464262028051773 0.70857431765018
-0.0278557216831064 0.707596923906914
-0.00928524056103532 0.70710822703528
0.00928524056103558 0.70710822703528
0.0278557216831066 0.707596923906914
0.0464262028051775 0.70857431765018
0.0649966839272484 0.710040408265081
0.0835671650493194 0.711995195751615
0.10213764617139 0.714438680109782
0.105461652407738 0.714963523199732
0.120708127293461 0.717309134720612
0.139278608415532 0.720642297998932
0.157849089537603 0.72445162745987
0.176419570659674 0.728737123103425
0.194990051781745 0.733498784929597
0.195114920535929 0.733534004321803
0.213560532903816 0.738606547722972
0.232131014025887 0.744177692059593
0.250701495147958 0.750213098424266
0.256105458061122 0.752104485443873
0.269271976270029 0.756600369710329
0.2878424573921 0.763394448169624
0.306412938514171 0.770641465192871
0.306493735943348 0.770674966565944
0.324983419636242 0.778158886155925
0.343553900758313 0.78611766377967
0.350467949397813 0.789245447688015
0.362124381880384 0.79439596436636
0.380694863002454 0.803033397446393
0.390487665318588 0.807815928810086
0.399265344124525 0.81200527551292
0.417835825246596 0.821290516073956
0.427584491757938 0.826386409932157
0.436406306368667 0.830895337399863
0.454976787490738 0.840799593998301
0.462459922191407 0.844956891054228
0.473547268612809 0.85098262280499
0.49211774973488 0.861478981700074
0.495607600175857 0.863527372176299
0.510688230856951 0.872190713205864
0.527318073155086 0.88209785329837
0.529258711979022 0.883229892612333
0.547829193101093 0.89444955829025
0.557779234909398 0.900668334420441
0.566399674223164 0.905946154408461
0.584970155345235 0.917694826138751
0.587334389119744 0.919238815542512
};
\addplot [draw opacity=0.25, draw=black, forget plot]
table{%
x  y
0.308172706521648 -0.919238815542512
0.306412938514171 -0.918628283784815
0.2878424573921 -0.912564453214343
0.269271976270029 -0.906879612054526
0.250701495147958 -0.901573760305363
0.247288736043253 -0.900668334420441
0.232131014025887 -0.896563118040737
0.213560532903816 -0.89192049776022
0.194990051781745 -0.887664762503078
0.176419570659674 -0.883795912269314
0.167363256147974 -0.88209785329837
0.157849089537603 -0.880275991607022
0.139278608415532 -0.877115058650074
0.120708127293461 -0.874349242312744
0.114316813149623 -0.87353332986885
};
\addplot [draw opacity=0.25, draw=black, forget plot]
table{%
x  y
-0.0587583005391791 -0.868024298334122
-0.0649966839272482 -0.868422493018466
-0.0835671650493192 -0.870002959496941
-0.10213764617139 -0.871978542595033
-0.120708127293461 -0.874349242312744
-0.139278608415532 -0.877115058650074
-0.157849089537603 -0.880275991607022
-0.167363256147974 -0.88209785329837
-0.176419570659674 -0.883795912269314
-0.194990051781745 -0.887664762503078
-0.213560532903816 -0.89192049776022
-0.232131014025887 -0.896563118040737
-0.247288736043253 -0.900668334420441
-0.250701495147958 -0.901573760305363
-0.269271976270029 -0.906879612054525
-0.2878424573921 -0.912564453214343
-0.306412938514171 -0.918628283784815
-0.308172706521648 -0.919238815542512
};
\addplot [draw opacity=0.25, draw=black, forget plot]
table{%
x  y
-0.308172706521648 0.919238815542512
-0.306412938514171 0.918628283784815
-0.2878424573921 0.912564453214343
-0.269271976270029 0.906879612054526
-0.250701495147958 0.901573760305362
-0.247288736043254 0.900668334420441
-0.232131014025887 0.896563118040738
-0.213560532903816 0.89192049776022
-0.194990051781745 0.887664762503078
-0.176419570659674 0.883795912269314
-0.167363256147975 0.88209785329837
-0.157849089537603 0.880275991607022
-0.139278608415532 0.877115058650074
-0.120708127293461 0.874349242312744
-0.114316813149623 0.87353332986885
};
\addplot [draw opacity=0.25, draw=black, forget plot]
table{%
x  y
0.0587583005391793 0.868024298334122
0.0649966839272484 0.868422493018466
0.0835671650493194 0.870002959496941
0.10213764617139 0.871978542595033
0.120708127293461 0.874349242312744
0.139278608415532 0.877115058650074
0.157849089537603 0.880275991607022
0.167363256147975 0.88209785329837
0.176419570659674 0.883795912269314
0.194990051781745 0.887664762503078
0.213560532903816 0.89192049776022
0.232131014025887 0.896563118040738
0.247288736043253 0.900668334420441
0.250701495147958 0.901573760305363
0.269271976270029 0.906879612054526
0.2878424573921 0.912564453214343
0.306412938514171 0.918628283784815
0.308172706521648 0.919238815542512
};
\addplot [very thick, darkorange, opacity=0.75, forget plot]
table {%
0 0
0.707106781186547 0.707106781186547
};
\addplot [very thick, darkorange, opacity=0.75]
table {%
0 0
0.707106781186547 -0.707106781186547
};
\addlegendentry{CASG}
\addplot [very thick, slategray, opacity=0.75, dashed, forget plot]
table {%
0 0
0.447213595499958 0
};
\addplot [very thick, slategray, opacity=0.75, dashed]
table {%
0 0
0 0.447213595499958
};
\addlegendentry{FD}
\addplot [draw=darkorange, fill=darkorange, mark=*, only marks]
table{%
x  y
0.707106781186547 0.707106781186547
0.707106781186547 -0.707106781186547
};
\addplot [draw=slategray, fill=slategray, mark=*, only marks]
table{%
x  y
0.447213595499958 0
0 0.447213595499958
};
\draw (axis cs:0.866051793300755,0.00928524056103575) node[
  scale=0.6,
  text=black,
  rotate=89.2
]{-0.75};
\draw (axis cs:-0.866051793300755,-0.0092852405610353) node[
  scale=0.6,
  text=black,
  rotate=89.2
]{-0.75};
\draw (axis cs:0.763394448169623,-0.2878424573921) node[
  scale=0.6,
  text=black,
  rotate=296.4
]{-0.50};
\draw (axis cs:-0.763394448169624,0.2878424573921) node[
  scale=0.6,
  text=black,
  rotate=296.4
]{-0.50};
\draw (axis cs:0.529258711979022,0.173400799155774) node[
  scale=0.6,
  text=black,
  rotate=67.3
]{-0.25};
\draw (axis cs:-0.529258711979022,-0.173400799155774) node[
  scale=0.6,
  text=black,
  rotate=67.3
]{-0.25};
\draw (axis cs:-0.473547268612809,-0.473547268612809) node[
  scale=0.6,
  text=black,
  rotate=36.7
]{0.00};
\draw (axis cs:0.586442436000022,-0.770674966565944) node[
  scale=0.6,
  text=black,
  rotate=330.5
]{0.25};
\draw (axis cs:-0.586442436000022,0.770674966565944) node[
  scale=0.6,
  text=black,
  rotate=330.5
]{0.25};
\draw (axis cs:0.380694863002455,-0.803033397446393) node[
  scale=0.6,
  text=black,
  rotate=340.6
]{0.50};
\draw (axis cs:-0.380694863002454,0.803033397446393) node[
  scale=0.6,
  text=black,
  rotate=340.6
]{0.50};
\draw (axis cs:0.0278557216831066,-0.866446909920374) node[
  scale=0.6,
  text=black,
  rotate=358.6
]{0.75};
\draw (axis cs:-0.0278557216831063,0.866446909920374) node[
  scale=0.6,
  text=black,
  rotate=358.6
]{0.75};
\end{axis}

\end{tikzpicture}

%% file: figs/Intro/sig_0.1_num_1000_2d_Error_negative_h_100.0.tex
\begin{tikzpicture}

\definecolor{darkgray176}{RGB}{176,176,176}
\definecolor{darkorange}{RGB}{255,140,0}
\definecolor{lightgray204}{RGB}{204,204,204}
\definecolor{slategray}{RGB}{112,128,144}

\begin{axis}[
legend cell align={left},
legend style={
  fill opacity=0.8,
  draw opacity=1,
  text opacity=1,
  at={(0.03,0.97)},
  anchor=north west,
  draw=lightgray204
},
log basis x={10},
log basis y={10},
tick align=outside,
tick pos=left,
x grid style={darkgray176},
xlabel={$\vert k \vert$},
xmajorgrids,
xmin=3.98107170553497e-05, xmax=25118.8643150958,
xmode=log,
xtick style={color=black},
y grid style={darkgray176},
ylabel={MSE},
ymajorgrids,
ymin=3.32463384532822e-07, ymax=6197.82747461685,
ymode=log,
ytick style={color=black}
]
\path [draw=slategray, very thick]
(axis cs:0.0001,0.0250588312740593)
--(axis cs:0.0001,0.222692850069456);

\path [draw=slategray, very thick]
(axis cs:0.000215443469003188,0.0254565327361613)
--(axis cs:0.000215443469003188,0.222742250205184);

\path [draw=slategray, very thick]
(axis cs:0.000464158883361278,0.0255215859928113)
--(axis cs:0.000464158883361278,0.222961262576761);

\path [draw=slategray, very thick]
(axis cs:0.001,0.0255913967510313)
--(axis cs:0.001,0.222957758262479);

\path [draw=slategray, very thick]
(axis cs:0.00215443469003188,0.0258260989069748)
--(axis cs:0.00215443469003188,0.223262488018683);

\path [draw=slategray, very thick]
(axis cs:0.00464158883361277,0.026136806939885)
--(axis cs:0.00464158883361277,0.223846923871053);

\path [draw=slategray, very thick]
(axis cs:0.01,0.0275745852475938)
--(axis cs:0.01,0.224632580009822);

\path [draw=slategray, very thick]
(axis cs:0.0215443469003188,0.0316118337944905)
--(axis cs:0.0215443469003188,0.226347468163999);

\path [draw=slategray, very thick]
(axis cs:0.0464158883361277,0.0356902591968598)
--(axis cs:0.0464158883361277,0.229388645468076);

\path [draw=slategray, very thick]
(axis cs:0.1,0.0502073498764723)
--(axis cs:0.1,0.233874858791721);

\path [draw=slategray, very thick]
(axis cs:0.215443469003188,0.0743757467520832)
--(axis cs:0.215443469003188,0.248798827229353);

\path [draw=slategray, very thick]
(axis cs:0.464158883361277,0.1162701505717)
--(axis cs:0.464158883361277,0.293391584174746);

\path [draw=slategray, very thick]
(axis cs:1,0.170872785428567)
--(axis cs:1,0.3982647531127);

\path [draw=slategray, very thick]
(axis cs:2.15443469003188,0.253934848585357)
--(axis cs:2.15443469003188,0.628327912383266);

\path [draw=slategray, very thick]
(axis cs:4.64158883361277,0.365467471012473)
--(axis cs:4.64158883361277,1.12279578887516);

\path [draw=slategray, very thick]
(axis cs:10,0.525489748042936)
--(axis cs:10,2.23229308432523);

\path [draw=slategray, very thick]
(axis cs:21.5443469003188,0.845939427945751)
--(axis cs:21.5443469003188,4.65776240292586);

\path [draw=slategray, very thick]
(axis cs:46.4158883361277,1.56661030022817)
--(axis cs:46.4158883361277,9.88644348497382);

\path [draw=slategray, very thick]
(axis cs:100,3.1735036030191)
--(axis cs:100,21.284492211169);

\path [draw=slategray, very thick]
(axis cs:215.443469003188,6.61066086343518)
--(axis cs:215.443469003188,45.7540121947405);

\path [draw=slategray, very thick]
(axis cs:464.158883361277,14.2816004777729)
--(axis cs:464.158883361277,98.350000959858);

\path [draw=slategray, very thick]
(axis cs:1000,30.7938977500185)
--(axis cs:1000,211.844601567272);

\path [draw=slategray, very thick]
(axis cs:2154.43469003188,66.4695615632991)
--(axis cs:2154.43469003188,456.534920364466);

\path [draw=slategray, very thick]
(axis cs:4641.58883361277,142.989480384543)
--(axis cs:4641.58883361277,983.157535467372);

\path [draw=slategray, very thick]
(axis cs:10000,309.73193311184)
--(axis cs:10000,2115.4305454567);

\path [draw=darkorange, very thick]
(axis cs:0.0001,0.000776291816667858)
--(axis cs:0.0001,0.00224698495626332);

\path [draw=darkorange, very thick]
(axis cs:0.000215443469003188,0.000646721485792401)
--(axis cs:0.000215443469003188,0.00204002384867061);

\path [draw=darkorange, very thick]
(axis cs:0.000464158883361278,0.00044035476795741)
--(axis cs:0.000464158883361278,0.00166792889667421);

\path [draw=darkorange, very thick]
(axis cs:0.001,0.000204729135267459)
--(axis cs:0.001,0.00111980103906782);

\path [draw=darkorange, very thick]
(axis cs:0.00215443469003188,7.58572864442662e-05)
--(axis cs:0.00215443469003188,0.000592562481601007);

\path [draw=darkorange, very thick]
(axis cs:0.00464158883361277,3.31677182716965e-05)
--(axis cs:0.00464158883361277,0.000281399910203329);

\path [draw=darkorange, very thick]
(axis cs:0.01,1.74011673268563e-05)
--(axis cs:0.01,0.000133729251234036);

\path [draw=darkorange, very thick]
(axis cs:0.0215443469003188,9.91440389144716e-06)
--(axis cs:0.0215443469003188,6.378173242798e-05);

\path [draw=darkorange, very thick]
(axis cs:0.0464158883361277,5.69485764127594e-06)
--(axis cs:0.0464158883361277,3.20367763456967e-05);

\path [draw=darkorange, very thick]
(axis cs:0.1,3.57655073618113e-06)
--(axis cs:0.1,1.72611189651656e-05);

\path [draw=darkorange, very thick]
(axis cs:0.215443469003188,2.22688735961844e-06)
--(axis cs:0.215443469003188,1.02766120294191e-05);

\path [draw=darkorange, very thick]
(axis cs:0.464158883361277,1.44420478914508e-06)
--(axis cs:0.464158883361277,7.06357684218474e-06);

\path [draw=darkorange, very thick]
(axis cs:1,9.74057363115545e-07)
--(axis cs:1,5.24727930512281e-06);

\path [draw=darkorange, very thick]
(axis cs:2.15443469003188,1.41905991492398e-06)
--(axis cs:2.15443469003188,7.05894842993064e-06);

\path [draw=darkorange, very thick]
(axis cs:4.64158883361277,2.17245907504294e-06)
--(axis cs:4.64158883361277,1.02964843823209e-05);

\path [draw=darkorange, very thick]
(axis cs:10,3.45286287683102e-06)
--(axis cs:10,1.71700781003913e-05);

\path [draw=darkorange, very thick]
(axis cs:21.5443469003188,5.44534436377368e-06)
--(axis cs:21.5443469003188,3.17999503100777e-05);

\path [draw=darkorange, very thick]
(axis cs:46.4158883361277,9.26064091463167e-06)
--(axis cs:46.4158883361277,6.35076658595913e-05);

\path [draw=darkorange, very thick]
(axis cs:100,1.56937127554393e-05)
--(axis cs:100,0.000133127917508866);

\path [draw=darkorange, very thick]
(axis cs:215.443469003188,3.06773997412931e-05)
--(axis cs:215.443469003188,0.000282179462630119);

\path [draw=darkorange, very thick]
(axis cs:464.158883361277,6.30064092221967e-05)
--(axis cs:464.158883361277,0.000601773172353279);

\path [draw=darkorange, very thick]
(axis cs:1000,0.000129576933900532)
--(axis cs:1000,0.00129552829300947);

\path [draw=darkorange, very thick]
(axis cs:2154.43469003188,0.000274514549793192)
--(axis cs:2154.43469003188,0.00278994592970017);

\path [draw=darkorange, very thick]
(axis cs:4641.58883361277,0.000588328305529667)
--(axis cs:4641.58883361277,0.00597072189970415);

\path [draw=darkorange, very thick]
(axis cs:10000,0.00125762035695606)
--(axis cs:10000,0.0128571654374788);

\addplot [very thick, slategray, dashed]
table {%
0.0001 0.150015
0.000215443469003188 0.150032316520351
0.000464158883361278 0.150069623832504
0.001 0.15015
0.00215443469003188 0.150323165203505
0.00464158883361277 0.150696238325042
0.01 0.1515
0.0215443469003188 0.153231652035048
0.0464158883361277 0.156962383250419
0.1 0.165
0.215443469003188 0.182316520350478
0.464158883361277 0.219623832504192
1 0.3
2.15443469003188 0.473165203504782
4.64158883361277 0.846238325041915
10 1.64999999999999
21.5443469003188 3.38165203504782
46.4158883361277 7.11238325041935
100 15.1500000000001
215.443469003188 32.4665203504765
464.158883361277 69.7738325041942
1000 150.150000000045
2154.43469003188 323.315203504406
4641.58883361277 696.388325042176
10000 1500.15000001357
};
\addlegendentry{FD}
\addplot [very thick, darkorange]
table {%
0.0001 0.00186281646480198
0.000215443469003188 0.0016793253644317
0.000464158883361278 0.00134952168126993
0.001 0.000881093044498962
0.00215443469003188 0.000458540475802403
0.00464158883361277 0.000218015988392368
0.01 0.000102980015980828
0.0215443469003188 4.94149600926176e-05
0.0464158883361277 2.45443038119698e-05
0.1 1.29999980000016e-05
0.215443469003188 7.64158874078102e-06
0.464158883361277 5.15443468572302e-06
1 4e-06
2.15443469003188 5.1544346891036e-06
4.64158883361277 7.64158882930391e-06
10 1.299999998e-05
21.5443469003188 2.45443468074871e-05
46.4158883361277 4.94158879052408e-05
100 0.000102999998
215.443469003188 0.000218443459720012
464.158883361277 0.0004671588402726
1000 0.00100299980000016
2154.43469003188 0.00215743376171572
4641.58883361277 0.0046445845247594
10000 0.01000298000016
};
\addlegendentry{CASG}
\end{axis}

\end{tikzpicture}

%% file: figs/Intro/sig_0.1_num_1_2d_S_scale_negative_h_100.0.tex
\begin{tikzpicture}

\definecolor{darkgray176}{RGB}{176,176,176}
\definecolor{darkorange}{RGB}{255,140,0}
\definecolor{lightgray204}{RGB}{204,204,204}
\definecolor{slategray}{RGB}{112,128,144}

\begin{axis}[
legend cell align={left},
legend style={
  fill opacity=0.8,
  draw opacity=1,
  text opacity=1,
  at={(0.03,0.03)},
  anchor=south west,
  draw=lightgray204
},
log basis x={10},
log basis y={10},
tick align=outside,
tick pos=left,
x grid style={darkgray176},
xlabel={k},
xmajorgrids,
xmin=3.98107170553497e-05, xmax=25118.8643150958,
xmode=log,
xtick style={color=black},
y grid style={darkgray176},
ylabel={Radius},
ymajorgrids,
ymin=0.239627933481057, ymax=133.285814820189,
ymode=log,
ytick style={color=black}
]
\addplot [very thick, slategray, dashed]
table {%
0.0001 31.9390043677006
0.000215443469003188 21.8605746663357
0.000464158883361278 14.9942204422375
0.001 10.3162277660168
0.00215443469003188 7.12914845659645
0.00464158883361277 4.95781659962962
0.01 3.47850542618522
0.0215443469003188 2.47066245604872
0.0464158883361277 1.78402703363891
0.1 1.31622776601684
0.215443469003188 0.9975198350748
0.464158883361277 0.780386649378116
1 0.632455532033676
2.15443469003188 0.531671235020026
4.64158883361277 0.463007692779045
10 0.416227766016838
21.5443469003188 0.384356972922634
46.4158883361277 0.362643654352966
100 0.347850542618522
215.443469003188 0.337772112917157
464.158883361277 0.330905758693059
1000 0.326227766016838
2154.43469003188 0.323040686707418
4641.58883361277 0.320869354850451
10000 0.319390043677006
};
\addlegendentry{FD}
\addplot [very thick, darkorange]
table {%
0.0001 70.737956325679
0.000215443469003188 70.7395246791251
0.000464158883361278 70.7434914945086
0.001 70.755009110111
0.00215443469003188 70.7896370897145
0.00464158883361277 70.8752428109617
0.01 71.063492624711
0.0215443469003188 71.4683569352641
0.0464158883361277 72.3331208367989
0.1 74.1619862193547
0.215443469003188 77.9565095941536
0.464158883361277 85.5616410903897
1 100
2.15443469003188 85.5616410478085
4.64158883361277 77.9565093306081
10 74.1619848844297
21.5443469003188 72.3331144337089
46.4158883361277 71.4683268063654
100 71.0633520318311
215.443469003188 70.8745930934919
464.158883361277 70.7868079196269
1000 70.746024637431
2154.43469003188 70.7270867236648
4641.58883361277 70.7182947994718
10000 70.7142135783181
};
\addlegendentry{CASG}
\end{axis}

\end{tikzpicture}

%% file: figs/Sensitivity/Box_Ackley_1693402667.6911428_H_exact_True.tex
\begin{tikzpicture}

\definecolor{darkgray176}{RGB}{176,176,176}
\definecolor{darkorange}{RGB}{255,140,0}
\definecolor{slategray}{RGB}{112,128,144}

\begin{axis}[
tick align=outside,
tick pos=left,
x grid style={darkgray176},
xlabel={Gradient estimation method},
xmajorgrids,
xmin=0.5, xmax=4.5,
xminorgrids,
xtick style={color=black},
xtick={1,2,3,4},
xticklabels={CD,FD exact H,FD RBF H,RBF},
y grid style={darkgray176},
ylabel={Log\_2 ratio of MSE},
ymajorgrids,
ymin=-4.85578888621393, ymax=14.791561490649,
yminorgrids,
ytick style={color=black}
]
\addplot [ultra thick, darkorange]
table {%
0.775 -1.93918972493189
1.225 -1.93918972493189
1.225 -1.63195718086311
0.775 -1.63195718086311
0.775 -1.93918972493189
};
\addplot [very thick, slategray]
table {%
1 -1.93918972493189
1 -2.34259235823921
};
\addplot [very thick, slategray]
table {%
1 -1.63195718086311
1 -1.26951823401726
};
\addplot [very thick, black]
table {%
0.8875 -2.34259235823921
1.1125 -2.34259235823921
};
\addplot [very thick, black]
table {%
0.8875 -1.26951823401726
1.1125 -1.26951823401726
};
\addplot [black, mark=o, mark size=3, mark options={solid,fill opacity=0}, only marks]
table {%
1 -2.57519436501143
1 -2.46743239969635
1 -2.44869714479688
1 -3.96272750544743
1 -1.17082168978118
1 -0.258766509528899
1 -0.492219243183851
1 -0.822574809315738
};
\addplot [ultra thick, darkorange]
table {%
1.775 3.95065095014494
2.225 3.95065095014494
2.225 4.64444288652892
1.775 4.64444288652892
1.775 3.95065095014494
};
\addplot [very thick, slategray]
table {%
2 3.95065095014494
2 3.0583751043768
};
\addplot [very thick, slategray]
table {%
2 4.64444288652892
2 5.59503122753346
};
\addplot [very thick, black]
table {%
1.8875 3.0583751043768
2.1125 3.0583751043768
};
\addplot [very thick, black]
table {%
1.8875 5.59503122753346
2.1125 5.59503122753346
};
\addplot [black, mark=o, mark size=3, mark options={solid,fill opacity=0}, only marks]
table {%
2 2.88088212820547
2 2.59673350281449
2 2.53362303870701
2 2.82741763922701
};
\addplot [ultra thick, darkorange]
table {%
2.775 3.97830044534018
3.225 3.97830044534018
3.225 4.74469591401045
2.775 4.74469591401045
2.775 3.97830044534018
};
\addplot [very thick, slategray]
table {%
3 3.97830044534018
3 2.90491786716968
};
\addplot [very thick, slategray]
table {%
3 4.74469591401045
3 5.73292352526508
};
\addplot [very thick, black]
table {%
2.8875 2.90491786716968
3.1125 2.90491786716968
};
\addplot [very thick, black]
table {%
2.8875 5.73292352526508
3.1125 5.73292352526508
};
\addplot [black, mark=o, mark size=3, mark options={solid,fill opacity=0}, only marks]
table {%
3 2.61204713767528
3 2.54165582293965
3 2.82246118648005
};
\addplot [ultra thick, darkorange]
table {%
3.775 7.814579746207
4.225 7.814579746207
4.225 11.4145553102629
3.775 11.4145553102629
3.775 7.814579746207
};
\addplot [very thick, slategray]
table {%
4 7.814579746207
4 4.10631205888409
};
\addplot [very thick, slategray]
table {%
4 11.4145553102629
4 13.8985001098825
};
\addplot [very thick, black]
table {%
3.8875 4.10631205888409
4.1125 4.10631205888409
};
\addplot [very thick, black]
table {%
3.8875 13.8985001098825
4.1125 13.8985001098825
};
\addplot [ultra thick, slategray]
table {%
0.775 -1.71905319062284
1.225 -1.71905319062284
};
\addplot [ultra thick, slategray]
table {%
1.775 4.41847484462118
2.225 4.41847484462118
};
\addplot [ultra thick, slategray]
table {%
2.775 4.47547627884176
3.225 4.47547627884176
};
\addplot [ultra thick, slategray]
table {%
3.775 9.60231199307902
4.225 9.60231199307902
};
\end{axis}

\end{tikzpicture}

%% file: figs/Sensitivity/Box_Ackley_1693402667.6911428_H_exact_False.tex
\begin{tikzpicture}

\definecolor{darkgray176}{RGB}{176,176,176}
\definecolor{darkorange}{RGB}{255,140,0}
\definecolor{slategray}{RGB}{112,128,144}

\begin{axis}[
tick align=outside,
tick pos=left,
x grid style={darkgray176},
xlabel={Gradient estimation method},
xmajorgrids,
xmin=0.5, xmax=4.5,
xminorgrids,
xtick style={color=black},
xtick={1,2,3,4},
xticklabels={CD,FD exact H,FD RBF H,RBF},
y grid style={darkgray176},
ylabel={Log\_2 ratio of MSE},
ymajorgrids,
ymin=-5.69834479674235, ymax=11.7619039870834,
yminorgrids,
ytick style={color=black}
]
\addplot [ultra thick, darkorange]
table {%
0.775 -3.79194590044491
1.225 -3.79194590044491
1.225 -2.57013163202006
0.775 -2.57013163202006
0.775 -3.79194590044491
};
\addplot [very thick, slategray]
table {%
1 -3.79194590044491
1 -4.90469712475027
};
\addplot [very thick, slategray]
table {%
1 -2.57013163202006
1 -1.4957887452366
};
\addplot [very thick, black]
table {%
0.8875 -4.90469712475027
1.1125 -4.90469712475027
};
\addplot [very thick, black]
table {%
0.8875 -1.4957887452366
1.1125 -1.4957887452366
};
\addplot [ultra thick, darkorange]
table {%
1.775 2.4578453477132
2.225 2.4578453477132
2.225 3.4457718131937
1.775 3.4457718131937
1.775 2.4578453477132
};
\addplot [very thick, slategray]
table {%
2 2.4578453477132
2 1.46860287882206
};
\addplot [very thick, slategray]
table {%
2 3.4457718131937
2 4.29843185126188
};
\addplot [very thick, black]
table {%
1.8875 1.46860287882206
2.1125 1.46860287882206
};
\addplot [very thick, black]
table {%
1.8875 4.29843185126188
2.1125 4.29843185126188
};
\addplot [ultra thick, darkorange]
table {%
2.775 2.54978428588405
3.225 2.54978428588405
3.225 3.47553861928768
2.775 3.47553861928768
2.775 2.54978428588405
};
\addplot [very thick, slategray]
table {%
3 2.54978428588405
3 1.51877681373981
};
\addplot [very thick, slategray]
table {%
3 3.47553861928768
3 4.31067237709107
};
\addplot [very thick, black]
table {%
2.8875 1.51877681373981
3.1125 1.51877681373981
};
\addplot [very thick, black]
table {%
2.8875 4.31067237709107
3.1125 4.31067237709107
};
\addplot [ultra thick, darkorange]
table {%
3.775 7.16780704297296
4.225 7.16780704297296
4.225 9.47396177334047
3.775 9.47396177334047
3.775 7.16780704297296
};
\addplot [very thick, slategray]
table {%
4 7.16780704297296
4 3.85929583503016
};
\addplot [very thick, slategray]
table {%
4 9.47396177334047
4 10.9682563150913
};
\addplot [very thick, black]
table {%
3.8875 3.85929583503016
4.1125 3.85929583503016
};
\addplot [very thick, black]
table {%
3.8875 10.9682563150913
4.1125 10.9682563150913
};
\addplot [ultra thick, slategray]
table {%
0.775 -3.20580883789029
1.225 -3.20580883789029
};
\addplot [ultra thick, slategray]
table {%
1.775 2.92465964370428
2.225 2.92465964370428
};
\addplot [ultra thick, slategray]
table {%
2.775 2.98597474426704
3.225 2.98597474426704
};
\addplot [ultra thick, slategray]
table {%
3.775 8.31705174159183
4.225 8.31705174159183
};
\end{axis}

\end{tikzpicture}

%% file: figs/Sensitivity/Nerrs_Ackley_1693402667.6911428.tex
\begin{tikzpicture}

\definecolor{darkgray176}{RGB}{176,176,176}
\definecolor{darkorange}{RGB}{255,140,0}
\definecolor{darkslateblue}{RGB}{72,61,139}
\definecolor{lightgray204}{RGB}{204,204,204}
\definecolor{maroon}{RGB}{128,0,0}
\definecolor{slategray}{RGB}{112,128,144}

\begin{axis}[
legend cell align={left},
legend style={fill opacity=0.8, draw opacity=1, text opacity=1, draw=lightgray204},
log basis y={10},
tick align=outside,
tick pos=left,
x grid style={darkgray176},
xlabel={Number of points for RBF fitting},
xmajorgrids,
xmin=-395, xmax=10495,
xminorgrids,
xtick style={color=black},
y grid style={darkgray176},
ylabel={MSE},
ymajorgrids,
ymin=2.4460594772134e-06, ymax=4.34695182526077,
yminorgrids,
ymode=log,
ytick style={color=black}
]
\path [draw=darkslateblue, fill=darkslateblue, opacity=0.2]
(axis cs:100,5.90809456628385e-06)
--(axis cs:100,4.70484256382362e-06)
--(axis cs:512.5,4.70484256382362e-06)
--(axis cs:925,4.70484256382362e-06)
--(axis cs:1337.5,4.70484256382362e-06)
--(axis cs:1750,4.70484256382362e-06)
--(axis cs:2162.5,4.70484256382362e-06)
--(axis cs:2575,4.70484256382362e-06)
--(axis cs:2987.5,4.70484256382362e-06)
--(axis cs:3400,4.70484256382362e-06)
--(axis cs:3812.5,4.70484256382362e-06)
--(axis cs:4225,4.70484256382362e-06)
--(axis cs:4637.5,4.70484256382362e-06)
--(axis cs:5050,4.70484256382362e-06)
--(axis cs:5462.5,4.70484256382362e-06)
--(axis cs:5875,4.70484256382362e-06)
--(axis cs:6287.5,4.70484256382362e-06)
--(axis cs:6700,4.70484256382362e-06)
--(axis cs:7112.5,4.70484256382362e-06)
--(axis cs:7525,4.70484256382362e-06)
--(axis cs:7937.5,4.70484256382362e-06)
--(axis cs:8350,4.70484256382362e-06)
--(axis cs:8762.5,4.70484256382362e-06)
--(axis cs:9175,4.70484256382362e-06)
--(axis cs:9587.5,4.70484256382362e-06)
--(axis cs:10000,4.70484256382362e-06)
--(axis cs:10000,5.90809456628385e-06)
--(axis cs:10000,5.90809456628385e-06)
--(axis cs:9587.5,5.90809456628385e-06)
--(axis cs:9175,5.90809456628385e-06)
--(axis cs:8762.5,5.90809456628385e-06)
--(axis cs:8350,5.90809456628385e-06)
--(axis cs:7937.5,5.90809456628385e-06)
--(axis cs:7525,5.90809456628385e-06)
--(axis cs:7112.5,5.90809456628385e-06)
--(axis cs:6700,5.90809456628385e-06)
--(axis cs:6287.5,5.90809456628385e-06)
--(axis cs:5875,5.90809456628385e-06)
--(axis cs:5462.5,5.90809456628385e-06)
--(axis cs:5050,5.90809456628385e-06)
--(axis cs:4637.5,5.90809456628385e-06)
--(axis cs:4225,5.90809456628385e-06)
--(axis cs:3812.5,5.90809456628385e-06)
--(axis cs:3400,5.90809456628385e-06)
--(axis cs:2987.5,5.90809456628385e-06)
--(axis cs:2575,5.90809456628385e-06)
--(axis cs:2162.5,5.90809456628385e-06)
--(axis cs:1750,5.90809456628385e-06)
--(axis cs:1337.5,5.90809456628385e-06)
--(axis cs:925,5.90809456628385e-06)
--(axis cs:512.5,5.90809456628385e-06)
--(axis cs:100,5.90809456628385e-06)
--cycle;

\path [draw=slategray, fill=slategray, opacity=0.2]
(axis cs:100,0.000437389042942241)
--(axis cs:100,0.000282618269419477)
--(axis cs:512.5,0.000282618269419477)
--(axis cs:925,0.000282618269419477)
--(axis cs:1337.5,0.000282618269419477)
--(axis cs:1750,0.000282618269419477)
--(axis cs:2162.5,0.000282618269419477)
--(axis cs:2575,0.000282618269419477)
--(axis cs:2987.5,0.000282618269419477)
--(axis cs:3400,0.000282618269419477)
--(axis cs:3812.5,0.000282618269419477)
--(axis cs:4225,0.000282618269419477)
--(axis cs:4637.5,0.000282618269419477)
--(axis cs:5050,0.000282618269419477)
--(axis cs:5462.5,0.000282618269419477)
--(axis cs:5875,0.000282618269419477)
--(axis cs:6287.5,0.000282618269419477)
--(axis cs:6700,0.000282618269419477)
--(axis cs:7112.5,0.000282618269419477)
--(axis cs:7525,0.000282618269419477)
--(axis cs:7937.5,0.000282618269419477)
--(axis cs:8350,0.000282618269419477)
--(axis cs:8762.5,0.000282618269419477)
--(axis cs:9175,0.000282618269419477)
--(axis cs:9587.5,0.000282618269419477)
--(axis cs:10000,0.000282618269419477)
--(axis cs:10000,0.000437389042942241)
--(axis cs:10000,0.000437389042942241)
--(axis cs:9587.5,0.000437389042942241)
--(axis cs:9175,0.000437389042942241)
--(axis cs:8762.5,0.000437389042942241)
--(axis cs:8350,0.000437389042942241)
--(axis cs:7937.5,0.000437389042942241)
--(axis cs:7525,0.000437389042942241)
--(axis cs:7112.5,0.000437389042942241)
--(axis cs:6700,0.000437389042942241)
--(axis cs:6287.5,0.000437389042942241)
--(axis cs:5875,0.000437389042942241)
--(axis cs:5462.5,0.000437389042942241)
--(axis cs:5050,0.000437389042942241)
--(axis cs:4637.5,0.000437389042942241)
--(axis cs:4225,0.000437389042942241)
--(axis cs:3812.5,0.000437389042942241)
--(axis cs:3400,0.000437389042942241)
--(axis cs:2987.5,0.000437389042942241)
--(axis cs:2575,0.000437389042942241)
--(axis cs:2162.5,0.000437389042942241)
--(axis cs:1750,0.000437389042942241)
--(axis cs:1337.5,0.000437389042942241)
--(axis cs:925,0.000437389042942241)
--(axis cs:512.5,0.000437389042942241)
--(axis cs:100,0.000437389042942241)
--cycle;

\path [draw=darkorange, fill=darkorange, opacity=0.2]
(axis cs:100,2.05991852093213e-05)
--(axis cs:100,1.6505433508623e-05)
--(axis cs:512.5,1.6505433508623e-05)
--(axis cs:925,1.6505433508623e-05)
--(axis cs:1337.5,1.6505433508623e-05)
--(axis cs:1750,1.6505433508623e-05)
--(axis cs:2162.5,1.6505433508623e-05)
--(axis cs:2575,1.6505433508623e-05)
--(axis cs:2987.5,1.6505433508623e-05)
--(axis cs:3400,1.6505433508623e-05)
--(axis cs:3812.5,1.6505433508623e-05)
--(axis cs:4225,1.6505433508623e-05)
--(axis cs:4637.5,1.6505433508623e-05)
--(axis cs:5050,1.6505433508623e-05)
--(axis cs:5462.5,1.6505433508623e-05)
--(axis cs:5875,1.6505433508623e-05)
--(axis cs:6287.5,1.6505433508623e-05)
--(axis cs:6700,1.6505433508623e-05)
--(axis cs:7112.5,1.6505433508623e-05)
--(axis cs:7525,1.6505433508623e-05)
--(axis cs:7937.5,1.6505433508623e-05)
--(axis cs:8350,1.6505433508623e-05)
--(axis cs:8762.5,1.6505433508623e-05)
--(axis cs:9175,1.6505433508623e-05)
--(axis cs:9587.5,1.6505433508623e-05)
--(axis cs:10000,1.6505433508623e-05)
--(axis cs:10000,2.05991852093213e-05)
--(axis cs:10000,2.05991852093213e-05)
--(axis cs:9587.5,2.05991852093213e-05)
--(axis cs:9175,2.05991852093213e-05)
--(axis cs:8762.5,2.05991852093213e-05)
--(axis cs:8350,2.05991852093213e-05)
--(axis cs:7937.5,2.05991852093213e-05)
--(axis cs:7525,2.05991852093213e-05)
--(axis cs:7112.5,2.05991852093213e-05)
--(axis cs:6700,2.05991852093213e-05)
--(axis cs:6287.5,2.05991852093213e-05)
--(axis cs:5875,2.05991852093213e-05)
--(axis cs:5462.5,2.05991852093213e-05)
--(axis cs:5050,2.05991852093213e-05)
--(axis cs:4637.5,2.05991852093213e-05)
--(axis cs:4225,2.05991852093213e-05)
--(axis cs:3812.5,2.05991852093213e-05)
--(axis cs:3400,2.05991852093213e-05)
--(axis cs:2987.5,2.05991852093213e-05)
--(axis cs:2575,2.05991852093213e-05)
--(axis cs:2162.5,2.05991852093213e-05)
--(axis cs:1750,2.05991852093213e-05)
--(axis cs:1337.5,2.05991852093213e-05)
--(axis cs:925,2.05991852093213e-05)
--(axis cs:512.5,2.05991852093213e-05)
--(axis cs:100,2.05991852093213e-05)
--cycle;

\path [draw=slategray, fill=slategray, opacity=0.2]
(axis cs:100,0.000633180082859254)
--(axis cs:100,0.000354153996532433)
--(axis cs:512.5,0.000483516009485687)
--(axis cs:925,0.00063780695786942)
--(axis cs:1337.5,0.000711924245904253)
--(axis cs:1750,0.00064955293073352)
--(axis cs:2162.5,0.000534280317442517)
--(axis cs:2575,0.000466249528373596)
--(axis cs:2987.5,0.000443114342018799)
--(axis cs:3400,0.000390867890395871)
--(axis cs:3812.5,0.000365584682269395)
--(axis cs:4225,0.000359726202704075)
--(axis cs:4637.5,0.000334625756048314)
--(axis cs:5050,0.000323623667455741)
--(axis cs:5462.5,0.000328758442421921)
--(axis cs:5875,0.000326534321055359)
--(axis cs:6287.5,0.000314733753921985)
--(axis cs:6700,0.000319764692334186)
--(axis cs:7112.5,0.000316552407395349)
--(axis cs:7525,0.000314129173979031)
--(axis cs:7937.5,0.000307659628075519)
--(axis cs:8350,0.000311429740540499)
--(axis cs:8762.5,0.00030511728830517)
--(axis cs:9175,0.000308186901706709)
--(axis cs:9587.5,0.000303021258346185)
--(axis cs:10000,0.00030429458476038)
--(axis cs:10000,0.00047071092885941)
--(axis cs:10000,0.00047071092885941)
--(axis cs:9587.5,0.000467516081219606)
--(axis cs:9175,0.000471359666872374)
--(axis cs:8762.5,0.000477568109358844)
--(axis cs:8350,0.000482602903146885)
--(axis cs:7937.5,0.000485236272362372)
--(axis cs:7525,0.000481827044621674)
--(axis cs:7112.5,0.00049293871228473)
--(axis cs:6700,0.000493996698963085)
--(axis cs:6287.5,0.000499594435725431)
--(axis cs:5875,0.000501990974771653)
--(axis cs:5462.5,0.000513192500281139)
--(axis cs:5050,0.00051928295642683)
--(axis cs:4637.5,0.000530352784509472)
--(axis cs:4225,0.000546561122020233)
--(axis cs:3812.5,0.000597030888441475)
--(axis cs:3400,0.000612648196366267)
--(axis cs:2987.5,0.000694522503693274)
--(axis cs:2575,0.000723016009116455)
--(axis cs:2162.5,0.000835454524211463)
--(axis cs:1750,0.000997050772856599)
--(axis cs:1337.5,0.00127796022785321)
--(axis cs:925,0.0012126886715873)
--(axis cs:512.5,0.00096405669579333)
--(axis cs:100,0.000633180082859254)
--cycle;

\path [draw=darkorange, fill=darkorange, opacity=0.2]
(axis cs:100,0.000188618185932992)
--(axis cs:100,0.000127566862969082)
--(axis cs:512.5,0.000137640642036047)
--(axis cs:925,0.000152399537818178)
--(axis cs:1337.5,0.000167303569619888)
--(axis cs:1750,0.000190628243095838)
--(axis cs:2162.5,0.000170974075235567)
--(axis cs:2575,0.000156932567100273)
--(axis cs:2987.5,0.000136620644706095)
--(axis cs:3400,0.000122110618843871)
--(axis cs:3812.5,9.32387210832173e-05)
--(axis cs:4225,8.51382871515015e-05)
--(axis cs:4637.5,7.25375797796653e-05)
--(axis cs:5050,6.4846274908308e-05)
--(axis cs:5462.5,5.84282565298256e-05)
--(axis cs:5875,5.61684108763923e-05)
--(axis cs:6287.5,4.65676731942285e-05)
--(axis cs:6700,4.75610080516845e-05)
--(axis cs:7112.5,4.1620149811421e-05)
--(axis cs:7525,4.15039176380978e-05)
--(axis cs:7937.5,3.77937022886658e-05)
--(axis cs:8350,3.58625021779062e-05)
--(axis cs:8762.5,3.67886286864205e-05)
--(axis cs:9175,3.55168205466853e-05)
--(axis cs:9587.5,3.48904897478106e-05)
--(axis cs:10000,3.28647888650943e-05)
--(axis cs:10000,6.64390870613659e-05)
--(axis cs:10000,6.64390870613659e-05)
--(axis cs:9587.5,6.15456332710578e-05)
--(axis cs:9175,6.95381277477876e-05)
--(axis cs:8762.5,6.61621520680716e-05)
--(axis cs:8350,7.08899988766109e-05)
--(axis cs:7937.5,7.41541267354247e-05)
--(axis cs:7525,7.39675615629689e-05)
--(axis cs:7112.5,7.71706547824369e-05)
--(axis cs:6700,8.74148403093207e-05)
--(axis cs:6287.5,8.94868272028976e-05)
--(axis cs:5875,0.000104150067720097)
--(axis cs:5462.5,0.000111997371367121)
--(axis cs:5050,0.000123624308264502)
--(axis cs:4637.5,0.000130056779277723)
--(axis cs:4225,0.000150579458427727)
--(axis cs:3812.5,0.000176868806497063)
--(axis cs:3400,0.000194672914097644)
--(axis cs:2987.5,0.00026800425811038)
--(axis cs:2575,0.000262121218477851)
--(axis cs:2162.5,0.000309763756178436)
--(axis cs:1750,0.00032590307960349)
--(axis cs:1337.5,0.000256474204357545)
--(axis cs:925,0.000245195174280356)
--(axis cs:512.5,0.000209211064487022)
--(axis cs:100,0.000188618185932992)
--cycle;

\path [draw=maroon, fill=maroon, opacity=0.2]
(axis cs:100,2.25999118247388)
--(axis cs:100,1.10479202731573)
--(axis cs:512.5,0.781980320378375)
--(axis cs:925,0.518366917605136)
--(axis cs:1337.5,0.384681932480582)
--(axis cs:1750,0.321808161696875)
--(axis cs:2162.5,0.202227636299021)
--(axis cs:2575,0.168768458150832)
--(axis cs:2987.5,0.144119072790464)
--(axis cs:3400,0.0984020386211799)
--(axis cs:3812.5,0.0660308952136366)
--(axis cs:4225,0.0585689753260923)
--(axis cs:4637.5,0.043201811889405)
--(axis cs:5050,0.0354798102662031)
--(axis cs:5462.5,0.0266014262036502)
--(axis cs:5875,0.0228573758224614)
--(axis cs:6287.5,0.0210653817484393)
--(axis cs:6700,0.0192075465855674)
--(axis cs:7112.5,0.0147977960594343)
--(axis cs:7525,0.0116316992321052)
--(axis cs:7937.5,0.0103070963647574)
--(axis cs:8350,0.0104149798285388)
--(axis cs:8762.5,0.00908561670984256)
--(axis cs:9175,0.00781394508888777)
--(axis cs:9587.5,0.00578842856255866)
--(axis cs:10000,0.004571858042632)
--(axis cs:10000,0.0405645468235243)
--(axis cs:10000,0.0405645468235243)
--(axis cs:9587.5,0.0447186351533149)
--(axis cs:9175,0.04780473414792)
--(axis cs:8762.5,0.0517837349989761)
--(axis cs:8350,0.0594933803682741)
--(axis cs:7937.5,0.0550393660557102)
--(axis cs:7525,0.0710495440032916)
--(axis cs:7112.5,0.0703289156307632)
--(axis cs:6700,0.0885682076059698)
--(axis cs:6287.5,0.0939611060241903)
--(axis cs:5875,0.118075022633795)
--(axis cs:5462.5,0.118099632122802)
--(axis cs:5050,0.139714082124134)
--(axis cs:4637.5,0.170613290892081)
--(axis cs:4225,0.168381474354531)
--(axis cs:3812.5,0.247475978210732)
--(axis cs:3400,0.312950620339656)
--(axis cs:2987.5,0.389248192934471)
--(axis cs:2575,0.438972430556215)
--(axis cs:2162.5,0.52776219773176)
--(axis cs:1750,0.760282881470869)
--(axis cs:1337.5,0.927187318006926)
--(axis cs:925,1.19496921209945)
--(axis cs:512.5,1.53796287960706)
--(axis cs:100,2.25999118247388)
--cycle;

\addplot [ultra thick, darkslateblue]
table {%
100 5.19252608331239e-06
512.5 5.19252608331239e-06
925 5.19252608331239e-06
1337.5 5.19252608331239e-06
1750 5.19252608331239e-06
2162.5 5.19252608331239e-06
2575 5.19252608331239e-06
2987.5 5.19252608331239e-06
3400 5.19252608331239e-06
3812.5 5.19252608331239e-06
4225 5.19252608331239e-06
4637.5 5.19252608331239e-06
5050 5.19252608331239e-06
5462.5 5.19252608331239e-06
5875 5.19252608331239e-06
6287.5 5.19252608331239e-06
6700 5.19252608331239e-06
7112.5 5.19252608331239e-06
7525 5.19252608331239e-06
7937.5 5.19252608331239e-06
8350 5.19252608331239e-06
8762.5 5.19252608331239e-06
9175 5.19252608331239e-06
9587.5 5.19252608331239e-06
10000 5.19252608331239e-06
};
\addlegendentry{CD}
\addplot [ultra thick, slategray]
table {%
100 0.000382037862055063
512.5 0.000382037862055063
925 0.000382037862055063
1337.5 0.000382037862055063
1750 0.000382037862055063
2162.5 0.000382037862055063
2575 0.000382037862055063
2987.5 0.000382037862055063
3400 0.000382037862055063
3812.5 0.000382037862055063
4225 0.000382037862055063
4637.5 0.000382037862055063
5050 0.000382037862055063
5462.5 0.000382037862055063
5875 0.000382037862055063
6287.5 0.000382037862055063
6700 0.000382037862055063
7112.5 0.000382037862055063
7525 0.000382037862055063
7937.5 0.000382037862055063
8350 0.000382037862055063
8762.5 0.000382037862055063
9175 0.000382037862055063
9587.5 0.000382037862055063
10000 0.000382037862055063
};
\addlegendentry{FD exact H}
\addplot [ultra thick, darkorange]
table {%
100 1.78208346377775e-05
512.5 1.78208346377775e-05
925 1.78208346377775e-05
1337.5 1.78208346377775e-05
1750 1.78208346377775e-05
2162.5 1.78208346377775e-05
2575 1.78208346377775e-05
2987.5 1.78208346377775e-05
3400 1.78208346377775e-05
3812.5 1.78208346377775e-05
4225 1.78208346377775e-05
4637.5 1.78208346377775e-05
5050 1.78208346377775e-05
5462.5 1.78208346377775e-05
5875 1.78208346377775e-05
6287.5 1.78208346377775e-05
6700 1.78208346377775e-05
7112.5 1.78208346377775e-05
7525 1.78208346377775e-05
7937.5 1.78208346377775e-05
8350 1.78208346377775e-05
8762.5 1.78208346377775e-05
9175 1.78208346377775e-05
9587.5 1.78208346377775e-05
10000 1.78208346377775e-05
};
\addlegendentry{CASG exact H}
\addplot [ultra thick, slategray, dashed]
table {%
100 0.000490862276715626
512.5 0.000662084448371793
925 0.000927907908997179
1337.5 0.00102398296255653
1750 0.000808859388953321
2162.5 0.000671295052225644
2575 0.000579531252578307
2987.5 0.000565760803788903
3400 0.000501173744192592
3812.5 0.000467898079633429
4225 0.000451988997336596
4637.5 0.000434005761784605
5050 0.000436836656458528
5462.5 0.000425630211589714
5875 0.000423232643197742
6287.5 0.000415968561478121
6700 0.000420040693532933
7112.5 0.000415829708719085
7525 0.000403439477404465
7937.5 0.00039953281833234
8350 0.000403382272824953
8762.5 0.000400577422087667
9175 0.000402167413616971
9587.5 0.000400083934887841
10000 0.000399533511480414
};
\addlegendentry{FD RBF H}
\addplot [ultra thick, darkorange, dashed]
table {%
100 0.000148283604040011
512.5 0.000168351350225915
925 0.000184338287093844
1337.5 0.000204946900483186
1750 0.000245176931591172
2162.5 0.000230416211332138
2575 0.000195013823538125
2987.5 0.000197556702694365
3400 0.000155885495609126
3812.5 0.0001291104623844
4225 0.000114593206453084
4637.5 9.72655145110895e-05
5050 8.63839024636096e-05
5462.5 7.65634136425767e-05
5875 7.47586341271343e-05
6287.5 6.90619789797448e-05
6700 6.74446018855408e-05
7112.5 5.88433240936449e-05
7525 5.44090293930297e-05
7937.5 5.4540814587375e-05
8350 5.75912443328119e-05
8762.5 5.24298305694992e-05
9175 4.95726964334566e-05
9587.5 4.79731222872702e-05
10000 4.73055116670909e-05
};
\addlegendentry{CASG RBF H}
\addplot [ultra thick, maroon]
table {%
100 1.57889599748629
512.5 1.15429972427352
925 0.910079574588856
1337.5 0.632897364609617
1750 0.497569411389307
2162.5 0.377631057047273
2575 0.264580704483289
2987.5 0.271910901855671
3400 0.178008911133494
3812.5 0.144552670723146
4225 0.103719552865179
4637.5 0.0971438401565984
5050 0.0741496705974314
5462.5 0.0549157810301262
5875 0.0535861118102238
6287.5 0.0436296616649223
6700 0.0398552970987745
7112.5 0.0312317789797038
7525 0.0271993263231673
7937.5 0.0231841185890696
8350 0.0228847548685636
8762.5 0.0226386603839835
9175 0.0221698672184434
9587.5 0.0151225854271444
10000 0.0149685793209682
};
\addlegendentry{RBF}
\end{axis}

\end{tikzpicture}

%% file: figs/Sensitivity/CancerVis0.tex
\begin{tikzpicture}

\definecolor{darkgray176}{RGB}{176,176,176}
\definecolor{lightgray204}{RGB}{204,204,204}

\begin{axis}[
legend cell align={left},
legend style={
  fill opacity=0.8,
  draw opacity=1,
  text opacity=1,
  at={(0.97,0.03)},
  anchor=south east,
  draw=lightgray204
},
tick align=outside,
tick pos=left,
x grid style={darkgray176},
xlabel={Time (days)},
xmin=-3.95, xmax=104.95,
xtick style={color=black},
y grid style={darkgray176},
ylabel={Cell Population},
ymin=0.79643105734517, ymax=5.27494779575143,
ytick style={color=black}
]
\addplot [very thick, black, dotted]
table {%
1 1
1.99 1.23741264622317
2.98 1.50352081980304
3.97 1.7922145090399
4.96 2.09530142564936
5.95 2.40352763124358
6.94 2.70773858492524
7.93 2.99988342689019
8.92 3.27366859957791
9.91 3.52481295248261
10.9 3.7509712407462
11.89 3.95144538778689
12.88 4.12680179663179
13.87 4.27848333590537
14.86 4.40846856702883
15.85 4.51900077702917
16.84 4.61238938441009
17.83 4.6908754255229
18.82 4.7565485068696
19.81 4.81130221459177
20.8 4.85681653692049
21.79 4.89455812403793
22.78 4.92579149148433
23.77 4.95159625036447
24.76 4.97288702047702
25.75 4.99043386293835
26.74 5.00488191584986
27.73 5.01676950109743
28.72 5.02654435881436
29.71 5.03457791359236
30.7 5.04117762611377
31.69 5.046597567387
32.68 5.05104739307543
33.67 5.05469990848517
34.66 5.05769741148163
35.65 5.06015698815311
36.64 5.0621749190903
37.63 5.06383033562167
38.62 5.06518824697954
39.61 5.06630204215073
40.6 5.06721555457196
41.59 5.06796476404572
42.58 5.06857919827149
43.57 5.0690830861057
44.56 5.06949630592322
45.55 5.06983516507716
46.54 5.07011304026461
47.53 5.07034090343592
48.52 5.0705277535815
49.51 5.07068097115847
50.5 5.07080660896151
51.49 5.07090963079724
52.48 5.07099410730366
53.47 5.07106337659236
54.46 5.07112017602158
55.45 5.07116675028041
56.44 5.07120494003768
57.43 5.07123625464679
58.42 5.07126193177181
59.41 5.07128298628605
60.4 5.07130025037194
61.39 5.07131440640495
62.38 5.07132601391933
63.37 5.07133553172068
64.36 5.07134333601843
65.35 5.07134973529458
66.34 5.07135498249594
67.33 5.07135928503173
68.32 5.07136281297147
69.31 5.07136570576705
70.3 5.07136807776479
71.29 5.07137002272509
72.28 5.07137161752855
73.27 5.07137292521481
74.26 5.07137399747436
75.25 5.07137487669174
76.24 5.07137559762088
77.23 5.0713761887589
78.22 5.07137667347247
79.21 5.07137707092151
80.2 5.07137739681653
81.19 5.07137766403963
82.18 5.07137788315369
83.17 5.07137806281994
84.16 5.07137821014029
85.15 5.07137833093808
86.14 5.07137842998826
87.13 5.07137851120611
88.12 5.07137857780205
89.11 5.0713786324085
90.1 5.07137867718398
91.09 5.07137871389838
92.08 5.07137874400296
93.07 5.07137876868772
94.06 5.0713787889284
95.05 5.07137880552509
96.04 5.07137881913382
97.03 5.07137883029252
98.02 5.07137883944229
99.01 5.0713788469448
100 5.0713788530966
};
\addlegendentry{Seed 0}
\addplot [very thick, black, dashed]
table {%
1 1
1.99 1.19212757262914
2.98 1.40069065138351
3.97 1.62129333864956
4.96 1.84866668352845
5.95 2.07720095891616
6.94 2.30149784496627
7.93 2.51683335823812
8.92 2.71946252219749
9.91 2.9067480541328
10.9 3.07713702585627
11.89 3.23003161232327
12.88 3.36560358734869
13.87 3.4845938108284
14.86 3.58812482852389
15.85 3.67754198425151
16.84 3.7542886156952
17.83 3.81981451626728
18.82 3.87551340505471
19.81 3.92268381273481
20.8 3.96250775542179
21.79 3.99604220995614
22.78 4.02421931011258
23.77 4.04785211138705
24.76 4.06764360042715
25.75 4.08419730712942
26.74 4.09802840853027
27.73 4.10957460999403
28.72 4.11920637455762
29.71 4.12723626991918
30.7 4.13392733605837
31.69 4.13950046260635
32.68 4.14414081771473
33.67 4.14800339974849
34.66 4.15121779728576
35.65 4.15389224710047
36.64 4.15611707783652
37.63 4.15796762155238
38.62 4.15950666795443
39.61 4.16078652806479
40.6 4.16185076598629
41.59 4.16273564974332
42.58 4.16347136511651
43.57 4.16408303004966
44.56 4.16459154160914
45.55 4.16501428259541
46.54 4.16536571069014
47.53 4.16565784940843
48.52 4.16590069704482
49.51 4.16610256718843
50.5 4.16627037217344
51.49 4.16640985896955
52.48 4.16652580545203
53.47 4.16662218367837
54.46 4.16670229569835
55.45 4.16676888650484
56.44 4.16682423796363
57.43 4.16687024691926
58.42 4.16690849013806
59.41 4.16694027830392
60.4 4.16696670090969
61.39 4.16698866357823
62.38 4.16700691908852
63.37 4.16702209316802
64.36 4.16703470593359
65.35 4.16704518971487
66.34 4.16705390387021
67.33 4.16706114710242
68.32 4.16706716769618
69.31 4.16707217202774
70.3 4.16707633163852
71.29 4.16707978911474
72.28 4.16708266297478
73.27 4.16708505173165
74.26 4.16708703726973
75.25 4.16708868765158
76.24 4.16709005945099
77.23 4.16709119969223
78.22 4.167092147462
79.21 4.16709293524921
80.2 4.16709359005874
81.19 4.16709413433705
82.18 4.1670945867417
83.17 4.16709496278085
84.16 4.16709527534494
85.15 4.16709553514848
86.14 4.16709575109741
87.13 4.16709593059434
88.12 4.16709607979235
89.11 4.16709620380586
90.1 4.16709630688599
91.09 4.16709639256628
92.08 4.16709646378381
93.07 4.16709652297988
94.06 4.1670965721837
95.05 4.16709661308195
96.04 4.16709664707661
97.03 4.167096675333
98.02 4.16709669881974
99.01 4.16709671834193
100 4.16709673456879
};
\addlegendentry{Seed 1}
\addplot [very thick, black, dash pattern=on 1pt off 3pt on 3pt off 3pt]
table {%
1 1
1.99 1.24129578691552
2.98 1.51089590655208
3.97 1.80168252471494
4.96 2.10442851213679
5.95 2.40904854077654
6.94 2.70593797458122
7.93 2.98704580228614
8.92 3.24648368638932
9.91 3.48066081426979
10.9 3.6880622382325
11.89 3.86883059690271
12.88 4.02429056114955
13.87 4.15650846214729
14.86 4.26793314502798
15.85 4.36113024148447
16.84 4.43860239398381
17.83 4.50267933784449
18.82 4.55546003651667
19.81 4.59879090532422
20.8 4.63426734176114
21.79 4.66324908139581
22.78 4.68688275790957
23.77 4.70612729030963
24.76 4.72177936972405
25.75 4.73449747023097
26.74 4.74482357786741
27.73 4.75320232375846
28.72 4.75999750560104
29.71 4.76550615014204
30.7 4.76997035353319
31.69 4.77358716822322
32.68 4.77651680567031
33.67 4.77888940754971
34.66 4.78081061311757
35.65 4.7823661223746
36.64 4.78362542680272
37.63 4.78464485343401
38.62 4.7854700446634
39.61 4.78613797580336
40.6 4.78667859485523
41.59 4.78711615412948
42.58 4.78747029089993
43.57 4.78775690391738
44.56 4.78798886403617
45.55 4.78817659014664
46.54 4.78832851581101
47.53 4.78845146725767
48.52 4.78855096951528
49.51 4.78863149431078
50.5 4.78869666078522
51.49 4.78874939799174
52.48 4.78879207644206
53.47 4.78882661459023
54.46 4.78885456502412
55.45 4.78887718422881
56.44 4.78889548905113
57.43 4.78891030239918
58.42 4.78892229022848
59.41 4.78893199147543
60.4 4.78893984228273
61.39 4.78894619560477
62.38 4.78895133707418
63.37 4.78895549784227
64.36 4.78895886497048
65.35 4.78896158983999
66.34 4.78896379495728
67.33 4.78896557946168
68.32 4.78896702358241
69.31 4.78896819224555
70.3 4.78896913799297
71.29 4.78896990334449
72.28 4.78897052270957
73.27 4.78897102393427
74.26 4.78897142955319
75.25 4.78897175780259
76.24 4.78897202344025
77.23 4.78897223840902
78.22 4.7889724123737
79.21 4.78897255315558
80.2 4.78897266708411
81.19 4.7889727592814
82.18 4.78897283389258
83.17 4.78897289427208
84.16 4.78897294313454
85.15 4.78897298267675
86.14 4.7889730146765
87.13 4.78897304057248
88.12 4.78897306152894
89.11 4.78897307848807
90.1 4.78897309221235
91.09 4.78897310331879
92.08 4.78897311230675
93.07 4.7889731195803
94.06 4.78897312546646
95.05 4.78897313022987
96.04 4.78897313408468
97.03 4.78897313720421
98.02 4.78897313972871
99.01 4.78897314177168
100 4.78897314342495
};
\addlegendentry{Seed 2}
\end{axis}

\end{tikzpicture}

%% file: figs/Sensitivity/CancerVis1.tex
\begin{tikzpicture}

\definecolor{darkgray176}{RGB}{176,176,176}
\definecolor{lightgray204}{RGB}{204,204,204}

\begin{axis}[
legend cell align={left},
legend style={
  fill opacity=0.8,
  draw opacity=1,
  text opacity=1,
  at={(0.03,0.97)},
  anchor=north west,
  draw=lightgray204
},
tick align=outside,
tick pos=left,
x grid style={darkgray176},
xlabel={Time (days)},
xmin=-3.95, xmax=104.95,
xtick style={color=black},
y grid style={darkgray176},
ylabel={Cell Population},
ymin=74.124470306005, ymax=142.219170864975,
ytick style={color=black}
]
\addplot [very thick, black, dotted]
table {%
1 100
1.99 99.2352543399553
2.98 98.6223037866236
3.97 98.1769380481944
4.96 97.9121946223307
5.95 97.8364616359544
6.94 97.9521720851516
7.93 98.2553579148661
8.92 98.7360712844189
9.91 99.3794637360108
10.9 100.167215065788
11.89 101.079019220024
12.88 102.093916673462
13.87 103.191360415188
14.86 104.351983642825
15.85 105.55809011099
16.84 106.79391472467
17.83 108.045709757472
18.82 109.301708727417
19.81 110.552011425067
20.8 111.788423662178
21.79 113.004276048678
22.78 114.194238391958
23.77 115.354140341097
24.76 116.480804523169
25.75 117.571895359367
26.74 118.625784704355
27.73 119.641434154279
28.72 120.618293098344
29.71 121.556211178262
30.7 122.455363646023
31.69 123.316188085959
32.68 124.139331031626
33.67 124.925603121211
34.66 125.675941570687
35.65 126.391378885162
36.64 127.073016866069
37.63 127.722005099388
38.62 128.339523225487
39.61 128.926766393416
40.6 129.48493339196
41.59 130.015217027096
42.58 130.518796381973
43.57 130.99683065224
44.56 131.450454297748
45.55 131.88077329251
46.54 132.288862289315
47.53 132.675762544505
48.52 133.042480473
49.51 133.389986724275
50.5 133.719215687443
51.49 134.031065348189
52.48 134.326397432663
53.47 134.606037783779
54.46 134.87077692413
55.45 135.121370767077
56.44 135.35854144377
57.43 135.582978219104
58.42 135.795338474025
59.41 135.996248735331
60.4 136.186305737258
61.39 136.366077501824
62.38 136.536104427121
63.37 136.6969003747
64.36 136.848953748748
65.35 136.992728561167
66.34 137.12866547779
67.33 137.25718284195
68.32 137.378677672439
69.31 137.493526633596
70.3 137.60208697583
71.29 137.70469744539
72.28 137.801679162605
73.27 137.89333646814
74.26 137.979957737144
75.25 138.061816161345
76.24 138.139170499401
77.23 138.21226579593
78.22 138.281334069798
79.21 138.34659497234
80.2 138.408256416266
81.19 138.466515176073
82.18 138.521557460825
83.17 138.573559460205
84.16 138.622687864754
85.15 138.669100361231
86.14 138.712946104035
87.13 138.754366163615
88.12 138.793493952807
89.11 138.830455632002
90.1 138.865370494048
91.09 138.898351329764
92.08 138.929504774922
93.07 138.958931639537
94.06 138.986727220267
95.05 139.012981596709
96.04 139.037779912349
97.03 139.061202640896
98.02 139.083325838701
99.01 139.10422138394
100 139.123957203203
};
\addlegendentry{Seed 0}
\addplot [very thick, black, dashed]
table {%
1 100
1.99 97.0184912466875
2.98 94.3007065148146
3.97 91.8402400164883
4.96 89.6301909573598
5.95 87.6622185868419
6.94 85.9259307182913
7.93 84.4086725737087
8.92 83.0956939043618
9.91 81.9706019862596
10.9 81.0159752992079
11.89 80.2140169602072
12.88 79.5471552518099
13.87 78.9985352417305
14.86 78.5523789344207
15.85 78.1942159325483
16.84 77.9110012556726
17.83 77.6911434416728
18.82 77.5244668967533
19.81 77.4021300503873
20.8 77.3165170115826
21.79 77.2611162730799
22.78 77.2303962117473
23.77 77.2196839677764
24.76 77.2250518145198
25.75 77.2432133009783
26.74 77.2714301566548
27.73 77.3074300795134
28.72 77.3493349767279
29.71 77.3955989063331
30.7 77.4449548059297
31.69 77.496369039397
32.68 77.5490028055509
33.67 77.6021795065233
34.66 77.6553572495448
35.65 77.7081057412658
36.64 77.7600869207326
37.63 77.8110387607503
38.62 77.8607617448618
39.61 77.9091075972211
40.6 77.9559699048251
41.59 78.0012763260213
42.58 78.0449821264274
43.57 78.0870648239878
44.56 78.1275197595963
45.55 78.1663564392058
46.54 78.2035955183336
47.53 78.2392663209588
48.52 78.2734048025585
49.51 78.3060518819469
50.5 78.3372520790823
51.49 78.3670524064827
52.48 78.3955014706414
53.47 78.4226487471509
54.46 78.4485439993515
55.45 78.4732368154056
56.44 78.4967762429532
57.43 78.5192105040306
58.42 78.5405867758864
59.41 78.5609510257742
60.4 78.580347889844
61.39 78.5988205879481
62.38 78.616410867589
63.37 78.6331589714085
64.36 78.6491036235889
65.35 78.6642820313503
66.34 78.6787298983962
67.33 78.6924814477169
68.32 78.7055694516268
69.31 78.7180252672889
70.3 78.7298788763038
71.29 78.7411589271995
72.28 78.7518927798784
73.27 78.762106551257
74.26 78.7718251614842
75.25 78.7810723802447
76.24 78.7898708727585
77.23 78.7982422451678
78.22 78.806207089074
79.21 78.8137850250407
80.2 78.8209947449289
81.19 78.8278540529635
82.18 78.8343799054644
83.17 78.8405884491983
84.16 78.8464950583288
85.15 78.8521143699561
86.14 78.8574603182537
87.13 78.8625461672167
88.12 78.8673845420455
89.11 78.8719874591943
90.1 78.8763663551195
91.09 78.8805321137646
92.08 78.8844950928214
93.07 78.8882651488104
94.06 78.8918516610211
95.05 78.8952635543557
96.04 78.8985093211186
97.03 78.9015970417956
98.02 78.9045344048614
99.01 78.9073287256591
100 78.9099869643904
};
\addlegendentry{Seed 1}
\addplot [very thick, black, dash pattern=on 1pt off 3pt on 3pt off 3pt]
table {%
1 100
1.99 98.2268254849788
2.98 96.6704813206481
3.97 95.342119350211
4.96 94.2497471549478
5.95 93.3961469096721
6.94 92.777615739911
7.93 92.3837803314598
8.92 92.1984019079889
9.91 92.2008452181875
10.9 92.3678104361886
11.89 92.6749905646609
12.88 93.0984440191986
13.87 93.6155975184334
14.86 94.2058850854987
15.85 94.8510791374548
16.84 95.5353877994121
17.83 96.2453908164976
18.82 96.969874964931
19.81 97.6996154628471
20.8 98.4271362613636
21.79 99.1464708430513
22.78 99.8529366193726
23.77 100.542929947815
24.76 101.213744707931
25.75 101.863414785536
26.74 102.490579283643
27.73 103.094368466712
28.72 103.674308099517
29.71 104.230239785166
30.7 104.762255017303
31.69 105.270640859544
32.68 105.75583540025
33.67 106.218391372098
34.66 106.658946556231
35.65 107.078199800886
36.64 107.476891670676
37.63 107.855788904533
38.62 108.215671998805
39.61 108.557325349371
40.6 108.881529485152
41.59 109.189055007749
42.58 109.480657920267
43.57 109.757076085061
44.56 110.01902659687
45.55 110.26720389632
46.54 110.502278480493
47.53 110.724896093233
48.52 110.935677299302
49.51 111.135217363942
50.5 111.324086373842
51.49 111.502829547219
52.48 111.671967690442
53.47 111.831997766465
54.46 111.983393546895
55.45 112.12660632476
56.44 112.262065669448
57.43 112.390180208845
58.42 112.511338426606
59.41 112.625909464913
60.4 112.734243925046
61.39 112.836674659669
62.38 112.933517552101
63.37 113.025072278868
64.36 113.111623052777
65.35 113.193439344414
66.34 113.270776580602
67.33 113.343876818809
68.32 113.4129693969
69.31 113.478271557927
70.3 113.539989049918
71.29 113.598316700814
72.28 113.653438968889
73.27 113.705530469071
74.26 113.754756475739
75.25 113.801273402581
76.24 113.845229260215
77.23 113.886764092263
78.22 113.926010390631
79.21 113.963093490749
80.2 113.998131947527
81.19 114.031237892797
82.18 114.062517375004
83.17 114.092070681877
84.16 114.119992646838
85.15 114.146372939854
86.14 114.171296343431
87.13 114.194843014442
88.12 114.217088732439
89.11 114.238105135091
90.1 114.257959941359
91.09 114.276717163009
92.08 114.294437305014
93.07 114.311177555416
94.06 114.326991965153
95.05 114.341931618351
96.04 114.356044793585
97.03 114.369377116534
98.02 114.381971704497
99.01 114.393869303165
100 114.405108416068
};
\addlegendentry{Seed 2}
\end{axis}

\end{tikzpicture}

%% file: figs/Sensitivity/CancerVis2.tex
\begin{tikzpicture}

\definecolor{darkgray176}{RGB}{176,176,176}
\definecolor{lightgray204}{RGB}{204,204,204}

\begin{axis}[
legend cell align={left},
legend style={
  fill opacity=0.8,
  draw opacity=1,
  text opacity=1,
  at={(0.03,0.97)},
  anchor=north west,
  draw=lightgray204
},
tick align=outside,
tick pos=left,
x grid style={darkgray176},
xlabel={Time (days)},
xmin=-3.95, xmax=104.95,
xtick style={color=black},
y grid style={darkgray176},
ylabel={Cell Population},
ymin=88.3557862329637, ymax=344.528489107763,
ytick style={color=black}
]
\addplot [very thick, black, dotted]
table {%
1 100
1.99 116.105986253389
2.98 130.021307914979
3.97 142.064520521777
4.96 152.519436981005
5.95 161.63827009212
6.94 169.643898798909
7.93 176.731590886104
8.92 183.070513970892
9.91 188.805291982733
10.9 194.057751802689
11.89 198.928895309598
12.88 203.50105248476
13.87 207.840129252066
14.86 211.997853737566
15.85 216.013935525794
16.84 219.918073487205
17.83 223.731770984366
18.82 227.469938137549
19.81 231.142277324323
20.8 234.754459836312
21.79 238.309109126523
22.78 241.806610206708
23.77 245.245766390142
24.76 248.624324508381
25.75 251.939388588242
26.74 255.187740224453
27.73 258.366081852334
28.72 261.471217032031
29.71 264.500179837762
30.7 267.450323580653
31.69 270.319377420404
32.68 273.105477951764
33.67 275.807181582847
34.66 278.423462440594
35.65 280.953699626353
36.64 283.397656881996
37.63 285.755457094475
38.62 288.027553545822
39.61 290.214699389317
40.6 292.317916485884
41.59 294.338464454641
42.58 296.277810566531
43.57 298.137600930585
44.56 299.919633280352
45.55 301.625831556609
46.54 303.258222395983
47.53 304.818913568651
48.52 306.310074358212
49.51 307.733917839623
50.5 309.092684984347
51.49 310.388630503294
52.48 311.624010325995
53.47 312.801070607397
54.46 313.922038150372
55.45 314.989112131804
56.44 316.004457021977
57.43 316.970196590486
58.42 317.888408896522
59.41 318.761122166772
60.4 319.590311469976
61.39 320.377896103315
62.38 321.125737611933
63.37 321.835638368981
64.36 322.509340649535
65.35 323.14852613743
66.34 323.754815809525
67.33 324.329770147033
68.32 324.874889628384
69.31 325.391615462614
70.3 325.881330526416
71.29 326.345360471862
72.28 326.78497497536
73.27 327.201389101656
74.26 327.595764759641
75.25 327.969212229475
76.24 328.322791742943
77.23 328.657515101216
78.22 328.974347316194
79.21 329.274208263395
80.2 329.557974336016
81.19 329.826480091227
82.18 330.080519881058
83.17 330.320849461437
84.16 330.548187573922
85.15 330.763217495644
86.14 330.966588553753
87.13 331.158917601398
88.12 331.340790452904
89.11 331.512763276349
90.1 331.675363942263
91.09 331.829093327558
92.08 331.974426574194
93.07 332.111814302392
94.06 332.241683778468
95.05 332.364440037627
96.04 332.480466962214
97.03 332.590128316115
98.02 332.693768736122
99.01 332.791714681209
100 332.884275340727
};
\addlegendentry{Seed 0}
\addplot [very thick, black, dashed]
table {%
1 100
1.99 117.008060187344
2.98 130.915913328156
3.97 142.207774724362
4.96 151.302933281975
5.95 158.563843534557
6.94 164.302929663757
7.93 168.788387933188
8.92 172.249241029632
9.91 174.879844143533
10.9 176.843980645241
11.89 178.27862809416
12.88 179.297431811992
13.87 179.993896180612
14.86 180.444291378282
15.85 180.710271325929
16.84 180.84120274263
17.83 180.876211796461
18.82 180.845961472161
19.81 180.774178151696
20.8 180.67894952119
21.79 180.573817786865
22.78 180.468692572032
23.77 180.370607138289
24.76 180.2843400849
25.75 180.212922737349
26.74 180.158050278738
27.73 180.120412479471
28.72 180.099957758273
29.71 180.096102333042
30.7 180.107894431975
31.69 180.134141949066
32.68 180.173510543053
33.67 180.224597985015
34.66 180.285989541409
35.65 180.356298318174
36.64 180.434193768635
37.63 180.518420965019
38.62 180.607812733309
39.61 180.701296338547
40.6 180.797896068711
41.59 180.896732788147
42.58 180.997021305813
43.57 181.09806622059
44.56 181.1992567581
45.55 181.300060994715
46.54 181.400019769348
47.53 181.498740507976
48.52 181.595891125946
49.51 181.691194125979
50.5 181.784420973055
51.49 181.875386798808
52.48 181.963945466338
53.47 182.049985009743
54.46 182.133423450475
55.45 182.21420498359
56.44 182.292296520674
57.43 182.367684571879
58.42 182.440372446776
59.41 182.510377752168
60.4 182.577730164404
61.39 182.642469453717
62.38 182.704643738653
63.37 182.764307949508
64.36 182.821522480756
65.35 182.876352013698
66.34 182.928864491856
67.33 182.979130232949
68.32 183.027221162634
69.31 183.073210156463
70.3 183.117170477736
71.29 183.159175300111
72.28 183.199297304899
73.27 183.237608344015
74.26 183.274179160473
75.25 183.30907915918
76.24 183.342376221577
77.23 183.374136558373
78.22 183.404424595291
79.21 183.433302887292
80.2 183.460832057323
81.19 183.487070756035
82.18 183.512075639424
83.17 183.535901361639
84.16 183.55860058061
85.15 183.580223974401
86.14 183.600820266479
87.13 183.620436258317
88.12 183.639116867968
89.11 183.656905173408
90.1 183.673842459637
91.09 183.689968268643
92.08 183.705320451483
93.07 183.719935221813
94.06 183.733847210336
95.05 183.747089519681
96.04 183.759693779329
97.03 183.771690200252
98.02 183.783107628984
99.01 183.7939736009
100 183.804314392521
};
\addlegendentry{Seed 1}
\addplot [very thick, black, dash pattern=on 1pt off 3pt on 3pt off 3pt]
table {%
1 100
1.99 118.193548445568
2.98 133.466119080743
3.97 146.272660128733
4.96 157.011997658386
5.95 166.032729644166
6.94 173.637803928626
7.93 180.088274259407
8.92 185.606686656693
9.91 190.380414907686
10.9 194.565098031981
11.89 198.288190198369
12.88 201.652543930075
13.87 204.739911965237
14.86 207.614257426504
15.85 210.324787896595
16.84 212.908661903151
17.83 215.393347039604
18.82 217.79863320926
19.81 220.138321259678
20.8 222.421617313445
21.79 224.654267839867
22.78 226.839471478355
23.77 228.978602144214
24.76 231.071775049473
25.75 233.118283691536
26.74 235.116932092721
27.73 237.0662829199
28.72 238.964838747826
29.71 240.811170734771
30.7 242.604006379178
31.69 244.342285810848
32.68 246.025194209983
33.67 247.652176403559
34.66 249.222938419567
35.65 250.737439745344
36.64 252.195879199343
37.63 253.598676653134
38.62 254.946452303402
39.61 256.240004767588
40.6 257.480288940978
41.59 258.668394290064
42.58 259.805524052776
43.57 260.892975658804
44.56 261.932122563296
45.55 262.924397596834
46.54 263.871277867521
47.53 264.774271202252
48.52 265.634904079623
49.51 266.45471098334
50.5 267.235225089766
51.49 267.977970194418
52.48 268.684453778121
53.47 269.356161112948
54.46 269.994550310017
55.45 270.601048214916
56.44 271.177047061382
57.43 271.723901799484
58.42 272.242928020542
59.41 272.735400407174
60.4 273.202551642954
61.39 273.645571722129
62.38 274.065607605497
63.37 274.463763173975
64.36 274.841099436365
65.35 275.198634952516
66.34 275.537346437391
67.33 275.858169515448
68.32 276.161999598349
69.31 276.449692862272
70.3 276.722067304008
71.29 276.979903857687
72.28 277.223947556335
73.27 277.454908724591
74.26 277.673464190783
75.25 277.880258508272
76.24 278.075905177426
77.23 278.260987860916
78.22 278.436061586205
79.21 278.6016539301
80.2 278.758266181132
81.19 278.906374476344
82.18 279.046430909709
83.17 279.178864610039
84.16 279.304082786732
85.15 279.422471742165
86.14 279.534397849929
87.13 279.640208498413
88.12 279.74023299956
89.11 279.834783462813
90.1 279.924155634513
91.09 280.00862970315
92.08 280.088471071016
93.07 280.163931092934
94.06 280.235247782812
95.05 280.302646488864
96.04 280.366340538366
97.03 280.426531852909
98.02 280.48341153508
99.01 280.53716042757
100 280.5879496457
};
\addlegendentry{Seed 2}
\end{axis}

\end{tikzpicture}

%% file: figs/Sensitivity/Box_Cancer_1693393050.1899579_H_exact_True.tex
\begin{tikzpicture}

\definecolor{darkgray176}{RGB}{176,176,176}
\definecolor{darkorange}{RGB}{255,140,0}
\definecolor{slategray}{RGB}{112,128,144}

\begin{axis}[
tick align=outside,
tick pos=left,
x grid style={darkgray176},
xlabel={Gradient estimation method},
xmajorgrids,
xmin=0.5, xmax=4.5,
xminorgrids,
xtick style={color=black},
xtick={1,2,3,4},
xticklabels={CD,FD exact H,FD RBF H,RBF},
y grid style={darkgray176},
ylabel={Log\_2 ratio of MSE},
ymajorgrids,
ymin=0.923250635456771, ymax=9.90739292570078,
yminorgrids,
ytick style={color=black}
]
\addplot [ultra thick, darkorange]
table {%
0.775 1.64244287266919
1.225 1.64244287266919
1.225 2.24052653298308
0.775 2.24052653298308
0.775 1.64244287266919
};
\addplot [thick, slategray]
table {%
1 1.64244287266919
1 1.33162073955877
};
\addplot [thick, slategray]
table {%
1 2.24052653298308
1 3.08051293148887
};
\addplot [thick, black]
table {%
0.8875 1.33162073955877
1.1125 1.33162073955877
};
\addplot [thick, black]
table {%
0.8875 3.08051293148887
1.1125 3.08051293148887
};
\addplot [ultra thick, darkorange]
table {%
1.775 4.92780211568513
2.225 4.92780211568513
2.225 5.26852429618165
1.775 5.26852429618165
1.775 4.92780211568513
};
\addplot [thick, slategray]
table {%
2 4.92780211568513
2 4.43051778096231
};
\addplot [thick, slategray]
table {%
2 5.26852429618165
2 5.50030024724763
};
\addplot [thick, black]
table {%
1.8875 4.43051778096231
2.1125 4.43051778096231
};
\addplot [thick, black]
table {%
1.8875 5.50030024724763
2.1125 5.50030024724763
};
\addplot [black, mark=o, mark size=3, mark options={solid,fill opacity=0}, only marks]
table {%
2 4.27904247770956
2 4.38068895334355
2 4.22375058489392
2 4.2016645031843
};
\addplot [ultra thick, darkorange]
table {%
2.775 5.1036828350546
3.225 5.1036828350546
3.225 5.40205168182309
2.775 5.40205168182309
2.775 5.1036828350546
};
\addplot [thick, slategray]
table {%
3 5.1036828350546
3 4.7064421178144
};
\addplot [thick, slategray]
table {%
3 5.40205168182309
3 5.63547043985872
};
\addplot [thick, black]
table {%
2.8875 4.7064421178144
3.1125 4.7064421178144
};
\addplot [thick, black]
table {%
2.8875 5.63547043985872
3.1125 5.63547043985872
};
\addplot [black, mark=o, mark size=3, mark options={solid,fill opacity=0}, only marks]
table {%
3 4.63864920082838
};
\addplot [ultra thick, darkorange]
table {%
3.775 7.16970801789136
4.225 7.16970801789136
4.225 8.09465905795993
3.775 8.09465905795993
3.775 7.16970801789136
};
\addplot [thick, slategray]
table {%
4 7.16970801789136
4 5.7925973816951
};
\addplot [thick, slategray]
table {%
4 8.09465905795993
4 9.15831942534157
};
\addplot [thick, black]
table {%
3.8875 5.7925973816951
4.1125 5.7925973816951
};
\addplot [thick, black]
table {%
3.8875 9.15831942534157
4.1125 9.15831942534157
};
\addplot [black, mark=o, mark size=3, mark options={solid,fill opacity=0}, only marks]
table {%
4 9.49902282159878
};
\addplot [ultra thick, slategray]
table {%
0.775 1.94063571508539
1.225 1.94063571508539
};
\addplot [ultra thick, slategray]
table {%
1.775 5.1367883749905
2.225 5.1367883749905
};
\addplot [ultra thick, slategray]
table {%
2.775 5.26003941467109
3.225 5.26003941467109
};
\addplot [ultra thick, slategray]
table {%
3.775 7.62660478025917
4.225 7.62660478025917
};
\end{axis}

\end{tikzpicture}

%% file: figs/Sensitivity/Box_Cancer_1693393050.1899579_H_exact_False.tex
\begin{tikzpicture}

\definecolor{darkgray176}{RGB}{176,176,176}
\definecolor{darkorange}{RGB}{255,140,0}
\definecolor{slategray}{RGB}{112,128,144}

\begin{axis}[
tick align=outside,
tick pos=left,
x grid style={darkgray176},
xlabel={Gradient estimation method},
xmajorgrids,
xmin=0.5, xmax=4.5,
xminorgrids,
xtick style={color=black},
xtick={1,2,3,4},
xticklabels={CD,FD exact H,FD RBF H,RBF},
y grid style={darkgray176},
ylabel={Log\_2 ratio of MSE},
ymajorgrids,
ymin=-1.88048780788092, ymax=6.65016357987935,
yminorgrids,
ytick style={color=black}
]
\addplot [ultra thick, darkorange]
table {%
0.775 -0.826367505959196
1.225 -0.826367505959196
1.225 -0.535661847192979
0.775 -0.535661847192979
0.775 -0.826367505959196
};
\addplot [very thick, slategray]
table {%
1 -0.826367505959196
1 -1.16829031217203
};
\addplot [very thick, slategray]
table {%
1 -0.535661847192979
1 -0.157648942013808
};
\addplot [very thick, black]
table {%
0.8875 -1.16829031217203
1.1125 -1.16829031217203
};
\addplot [very thick, black]
table {%
0.8875 -0.157648942013808
1.1125 -0.157648942013808
};
\addplot [black, mark=o, mark size=3, mark options={solid,fill opacity=0}, only marks]
table {%
1 -1.27375600869195
1 -1.30969347334351
1 -1.36368053777837
1 -1.40508872331307
1 -1.31724564739098
1 -1.49273092661909
1 -1.41214985792445
1 0.0249018303144849
1 0.555294249996948
1 0.42653831421147
1 0.018327272550543
1 -0.0210151926192046
};
\addplot [ultra thick, darkorange]
table {%
1.775 2.14005632913103
2.225 2.14005632913103
2.225 2.6759248481046
1.775 2.6759248481046
1.775 2.14005632913103
};
\addplot [very thick, slategray]
table {%
2 2.14005632913103
2 1.54959556379556
};
\addplot [very thick, slategray]
table {%
2 2.6759248481046
2 3.41114001785647
};
\addplot [very thick, black]
table {%
1.8875 1.54959556379556
2.1125 1.54959556379556
};
\addplot [very thick, black]
table {%
1.8875 3.41114001785647
2.1125 3.41114001785647
};
\addplot [black, mark=o, mark size=3, mark options={solid,fill opacity=0}, only marks]
table {%
2 1.24855081298085
};
\addplot [ultra thick, darkorange]
table {%
2.775 2.3248490801544
3.225 2.3248490801544
3.225 2.8361611757745
2.775 2.8361611757745
2.775 2.3248490801544
};
\addplot [very thick, slategray]
table {%
3 2.3248490801544
3 1.68476575640665
};
\addplot [very thick, slategray]
table {%
3 2.8361611757745
3 3.52938058589831
};
\addplot [very thick, black]
table {%
2.8875 1.68476575640665
3.1125 1.68476575640665
};
\addplot [very thick, black]
table {%
2.8875 3.52938058589831
3.1125 3.52938058589831
};
\addplot [ultra thick, darkorange]
table {%
3.775 4.77787236621759
4.225 4.77787236621759
4.225 5.09335228464449
3.775 5.09335228464449
3.775 4.77787236621759
};
\addplot [very thick, slategray]
table {%
4 4.77787236621759
4 4.36217334964591
};
\addplot [very thick, slategray]
table {%
4 5.09335228464449
4 5.45067952829083
};
\addplot [very thick, black]
table {%
3.8875 4.36217334964591
4.1125 4.36217334964591
};
\addplot [very thick, black]
table {%
3.8875 5.45067952829083
4.1125 5.45067952829083
};
\addplot [black, mark=o, mark size=3, mark options={solid,fill opacity=0}, only marks]
table {%
4 4.23510447232935
4 3.56969454420835
4 3.80253828821119
4 4.2269932013303
4 4.25022280580585
4 4.18413447779293
4 5.71847508868803
4 5.60608469231318
4 5.6283362043218
4 6.26240669861752
4 5.5741063951626
4 6.13683865550711
4 6.20398441897635
};
\addplot [ultra thick, slategray]
table {%
0.775 -0.660645746364098
1.225 -0.660645746364098
};
\addplot [ultra thick, slategray]
table {%
1.775 2.43384405019805
2.225 2.43384405019805
};
\addplot [ultra thick, slategray]
table {%
2.775 2.56036485033628
3.225 2.56036485033628
};
\addplot [ultra thick, slategray]
table {%
3.775 4.94587501110114
4.225 4.94587501110114
};
\end{axis}

\end{tikzpicture}

%% file: figs/Sensitivity/Nerrs_Cancer_1693393050.1899579.tex
\begin{tikzpicture}

\definecolor{darkgray176}{RGB}{176,176,176}
\definecolor{darkorange}{RGB}{255,140,0}
\definecolor{darkslateblue}{RGB}{72,61,139}
\definecolor{lightgray204}{RGB}{204,204,204}
\definecolor{maroon}{RGB}{128,0,0}
\definecolor{slategray}{RGB}{112,128,144}

\begin{axis}[
legend cell align={left},
legend style={fill opacity=0.8, draw opacity=1, text opacity=1, draw=lightgray204},
log basis y={10},
tick align=outside,
tick pos=left,
x grid style={darkgray176},
xlabel={Number of points for RBF fitting},
xmajorgrids,
xmin=-20, xmax=2620,
xminorgrids,
xtick style={color=black},
y grid style={darkgray176},
ylabel={MSE},
ymajorgrids,
ymin=0.0289350045288796, ymax=62054.0249327189,
yminorgrids,
ymode=log,
ytick style={color=black}
]
\path [draw=darkslateblue, fill=darkslateblue, opacity=0.2]
(axis cs:100,0.35726330854282)
--(axis cs:100,0.264820548889559)
--(axis cs:200,0.264820548889559)
--(axis cs:300,0.264820548889559)
--(axis cs:400,0.264820548889559)
--(axis cs:500,0.264820548889559)
--(axis cs:600,0.264820548889559)
--(axis cs:700,0.264820548889559)
--(axis cs:800,0.264820548889559)
--(axis cs:900,0.264820548889559)
--(axis cs:1000,0.264820548889559)
--(axis cs:1100,0.264820548889559)
--(axis cs:1200,0.264820548889559)
--(axis cs:1300,0.264820548889559)
--(axis cs:1400,0.264820548889559)
--(axis cs:1500,0.264820548889559)
--(axis cs:1600,0.264820548889559)
--(axis cs:1700,0.264820548889559)
--(axis cs:1800,0.264820548889559)
--(axis cs:1900,0.264820548889559)
--(axis cs:2000,0.264820548889559)
--(axis cs:2100,0.264820548889559)
--(axis cs:2200,0.264820548889559)
--(axis cs:2300,0.264820548889559)
--(axis cs:2400,0.264820548889559)
--(axis cs:2500,0.264820548889559)
--(axis cs:2500,0.35726330854282)
--(axis cs:2500,0.35726330854282)
--(axis cs:2400,0.35726330854282)
--(axis cs:2300,0.35726330854282)
--(axis cs:2200,0.35726330854282)
--(axis cs:2100,0.35726330854282)
--(axis cs:2000,0.35726330854282)
--(axis cs:1900,0.35726330854282)
--(axis cs:1800,0.35726330854282)
--(axis cs:1700,0.35726330854282)
--(axis cs:1600,0.35726330854282)
--(axis cs:1500,0.35726330854282)
--(axis cs:1400,0.35726330854282)
--(axis cs:1300,0.35726330854282)
--(axis cs:1200,0.35726330854282)
--(axis cs:1100,0.35726330854282)
--(axis cs:1000,0.35726330854282)
--(axis cs:900,0.35726330854282)
--(axis cs:800,0.35726330854282)
--(axis cs:700,0.35726330854282)
--(axis cs:600,0.35726330854282)
--(axis cs:500,0.35726330854282)
--(axis cs:400,0.35726330854282)
--(axis cs:300,0.35726330854282)
--(axis cs:200,0.35726330854282)
--(axis cs:100,0.35726330854282)
--cycle;

\path [draw=slategray, fill=slategray, opacity=0.2]
(axis cs:100,3.43582123041205)
--(axis cs:100,2.17791333771533)
--(axis cs:200,2.17791333771533)
--(axis cs:300,2.17791333771533)
--(axis cs:400,2.17791333771533)
--(axis cs:500,2.17791333771533)
--(axis cs:600,2.17791333771533)
--(axis cs:700,2.17791333771533)
--(axis cs:800,2.17791333771533)
--(axis cs:900,2.17791333771533)
--(axis cs:1000,2.17791333771533)
--(axis cs:1100,2.17791333771533)
--(axis cs:1200,2.17791333771533)
--(axis cs:1300,2.17791333771533)
--(axis cs:1400,2.17791333771533)
--(axis cs:1500,2.17791333771533)
--(axis cs:1600,2.17791333771533)
--(axis cs:1700,2.17791333771533)
--(axis cs:1800,2.17791333771533)
--(axis cs:1900,2.17791333771533)
--(axis cs:2000,2.17791333771533)
--(axis cs:2100,2.17791333771533)
--(axis cs:2200,2.17791333771533)
--(axis cs:2300,2.17791333771533)
--(axis cs:2400,2.17791333771533)
--(axis cs:2500,2.17791333771533)
--(axis cs:2500,3.43582123041205)
--(axis cs:2500,3.43582123041205)
--(axis cs:2400,3.43582123041205)
--(axis cs:2300,3.43582123041205)
--(axis cs:2200,3.43582123041205)
--(axis cs:2100,3.43582123041205)
--(axis cs:2000,3.43582123041205)
--(axis cs:1900,3.43582123041205)
--(axis cs:1800,3.43582123041205)
--(axis cs:1700,3.43582123041205)
--(axis cs:1600,3.43582123041205)
--(axis cs:1500,3.43582123041205)
--(axis cs:1400,3.43582123041205)
--(axis cs:1300,3.43582123041205)
--(axis cs:1200,3.43582123041205)
--(axis cs:1100,3.43582123041205)
--(axis cs:1000,3.43582123041205)
--(axis cs:900,3.43582123041205)
--(axis cs:800,3.43582123041205)
--(axis cs:700,3.43582123041205)
--(axis cs:600,3.43582123041205)
--(axis cs:500,3.43582123041205)
--(axis cs:400,3.43582123041205)
--(axis cs:300,3.43582123041205)
--(axis cs:200,3.43582123041205)
--(axis cs:100,3.43582123041205)
--cycle;

\path [draw=darkorange, fill=darkorange, opacity=0.2]
(axis cs:100,0.112050203983193)
--(axis cs:100,0.0561321952844768)
--(axis cs:200,0.0561321952844768)
--(axis cs:300,0.0561321952844768)
--(axis cs:400,0.0561321952844768)
--(axis cs:500,0.0561321952844768)
--(axis cs:600,0.0561321952844768)
--(axis cs:700,0.0561321952844768)
--(axis cs:800,0.0561321952844768)
--(axis cs:900,0.0561321952844768)
--(axis cs:1000,0.0561321952844768)
--(axis cs:1100,0.0561321952844768)
--(axis cs:1200,0.0561321952844768)
--(axis cs:1300,0.0561321952844768)
--(axis cs:1400,0.0561321952844768)
--(axis cs:1500,0.0561321952844768)
--(axis cs:1600,0.0561321952844768)
--(axis cs:1700,0.0561321952844768)
--(axis cs:1800,0.0561321952844768)
--(axis cs:1900,0.0561321952844768)
--(axis cs:2000,0.0561321952844768)
--(axis cs:2100,0.0561321952844768)
--(axis cs:2200,0.0561321952844768)
--(axis cs:2300,0.0561321952844768)
--(axis cs:2400,0.0561321952844768)
--(axis cs:2500,0.0561321952844768)
--(axis cs:2500,0.112050203983193)
--(axis cs:2500,0.112050203983193)
--(axis cs:2400,0.112050203983193)
--(axis cs:2300,0.112050203983193)
--(axis cs:2200,0.112050203983193)
--(axis cs:2100,0.112050203983193)
--(axis cs:2000,0.112050203983193)
--(axis cs:1900,0.112050203983193)
--(axis cs:1800,0.112050203983193)
--(axis cs:1700,0.112050203983193)
--(axis cs:1600,0.112050203983193)
--(axis cs:1500,0.112050203983193)
--(axis cs:1400,0.112050203983193)
--(axis cs:1300,0.112050203983193)
--(axis cs:1200,0.112050203983193)
--(axis cs:1100,0.112050203983193)
--(axis cs:1000,0.112050203983193)
--(axis cs:900,0.112050203983193)
--(axis cs:800,0.112050203983193)
--(axis cs:700,0.112050203983193)
--(axis cs:600,0.112050203983193)
--(axis cs:500,0.112050203983193)
--(axis cs:400,0.112050203983193)
--(axis cs:300,0.112050203983193)
--(axis cs:200,0.112050203983193)
--(axis cs:100,0.112050203983193)
--cycle;

\path [draw=slategray, fill=slategray, opacity=0.2]
(axis cs:100,18.7617418109445)
--(axis cs:100,8.90383817922333)
--(axis cs:200,3.88863865525857)
--(axis cs:300,3.23410890102602)
--(axis cs:400,2.97404904426265)
--(axis cs:500,2.74594688826152)
--(axis cs:600,2.60350002028961)
--(axis cs:700,2.48110663672605)
--(axis cs:800,2.46748487721461)
--(axis cs:900,2.43369777464162)
--(axis cs:1000,2.40149049412742)
--(axis cs:1100,2.40396404776602)
--(axis cs:1200,2.38838908113543)
--(axis cs:1300,2.38138427738912)
--(axis cs:1400,2.39485390970242)
--(axis cs:1500,2.38628891633823)
--(axis cs:1600,2.35130590317633)
--(axis cs:1700,2.39484615675804)
--(axis cs:1800,2.37260516503721)
--(axis cs:1900,2.34663891848755)
--(axis cs:2000,2.39742676078128)
--(axis cs:2100,2.36395213458521)
--(axis cs:2200,2.35693112414496)
--(axis cs:2300,2.39102413558775)
--(axis cs:2400,2.36872644746802)
--(axis cs:2500,2.35239875819297)
--(axis cs:2500,3.87913504498897)
--(axis cs:2500,3.87913504498897)
--(axis cs:2400,3.88943136315593)
--(axis cs:2300,3.92134907095725)
--(axis cs:2200,3.91785849347127)
--(axis cs:2100,3.92974079530254)
--(axis cs:2000,3.95277214397612)
--(axis cs:1900,3.92982042600284)
--(axis cs:1800,3.96933036883572)
--(axis cs:1700,3.98861125618217)
--(axis cs:1600,3.98319263694699)
--(axis cs:1500,4.01098388314239)
--(axis cs:1400,4.0485884137452)
--(axis cs:1300,4.05933914202845)
--(axis cs:1200,4.09119434163446)
--(axis cs:1100,4.14183042456664)
--(axis cs:1000,4.09784747175052)
--(axis cs:900,4.18322709090151)
--(axis cs:800,4.35340903582276)
--(axis cs:700,4.31201576126785)
--(axis cs:600,4.59583500883387)
--(axis cs:500,4.81336458044352)
--(axis cs:400,5.59129164117161)
--(axis cs:300,6.24374434804604)
--(axis cs:200,7.41899067756437)
--(axis cs:100,18.7617418109445)
--cycle;

\path [draw=darkorange, fill=darkorange, opacity=0.2]
(axis cs:100,5.25768745720571)
--(axis cs:100,2.77529344718441)
--(axis cs:200,1.60927860731059)
--(axis cs:300,1.16012320039537)
--(axis cs:400,0.784834463020484)
--(axis cs:500,0.605838950498003)
--(axis cs:600,0.509735785452871)
--(axis cs:700,0.460144372461318)
--(axis cs:800,0.447524767950326)
--(axis cs:900,0.456565822970907)
--(axis cs:1000,0.452385400716444)
--(axis cs:1100,0.444962128418757)
--(axis cs:1200,0.436743030187055)
--(axis cs:1300,0.432872343754123)
--(axis cs:1400,0.455633464027063)
--(axis cs:1500,0.419954926133936)
--(axis cs:1600,0.422845637946745)
--(axis cs:1700,0.413541282219932)
--(axis cs:1800,0.426219455615581)
--(axis cs:1900,0.418887620627351)
--(axis cs:2000,0.421717002620867)
--(axis cs:2100,0.417588594201332)
--(axis cs:2200,0.425815621273312)
--(axis cs:2300,0.424252140672446)
--(axis cs:2400,0.411924563857243)
--(axis cs:2500,0.415933193420202)
--(axis cs:2500,0.578530260707154)
--(axis cs:2500,0.578530260707154)
--(axis cs:2400,0.591431673411942)
--(axis cs:2300,0.595227192896625)
--(axis cs:2200,0.607111858088316)
--(axis cs:2100,0.629565873809635)
--(axis cs:2000,0.581216009395377)
--(axis cs:1900,0.598448450302596)
--(axis cs:1800,0.587467450173335)
--(axis cs:1700,0.638418435238991)
--(axis cs:1600,0.59405913617077)
--(axis cs:1500,0.669355140617384)
--(axis cs:1400,0.666773043413206)
--(axis cs:1300,0.642709721093685)
--(axis cs:1200,0.762787428295657)
--(axis cs:1100,0.700655487913946)
--(axis cs:1000,0.710874876736456)
--(axis cs:900,0.835429255661995)
--(axis cs:800,1.02298531520122)
--(axis cs:700,0.938836409456494)
--(axis cs:600,1.17185018971258)
--(axis cs:500,1.33283384985894)
--(axis cs:400,1.67387343423706)
--(axis cs:300,2.12421403600977)
--(axis cs:200,2.75687083575431)
--(axis cs:100,5.25768745720571)
--cycle;

\path [draw=maroon, fill=maroon, opacity=0.2]
(axis cs:100,31987.5872191298)
--(axis cs:100,31861.1632122783)
--(axis cs:200,28146.1068582512)
--(axis cs:300,2641.00900442342)
--(axis cs:400,1031.09508970499)
--(axis cs:500,4303.73403940091)
--(axis cs:600,2475.92868421192)
--(axis cs:700,1323.91226381387)
--(axis cs:800,320.048510994401)
--(axis cs:900,493.969353441053)
--(axis cs:1000,37.0168889833837)
--(axis cs:1100,438.535541693953)
--(axis cs:1200,54.8871216780781)
--(axis cs:1300,141.904948211248)
--(axis cs:1400,706.604004281149)
--(axis cs:1500,125.577437161288)
--(axis cs:1600,244.702937709844)
--(axis cs:1700,145.747991409242)
--(axis cs:1800,164.70309638174)
--(axis cs:1900,88.8014545586472)
--(axis cs:2000,27.0727587045372)
--(axis cs:2100,128.596669736744)
--(axis cs:2200,49.3980985841036)
--(axis cs:2300,31.7496093427173)
--(axis cs:2400,11.3865528171652)
--(axis cs:2500,13.3575670605859)
--(axis cs:2500,15.2432319377997)
--(axis cs:2500,15.2432319377997)
--(axis cs:2400,14.1734015340202)
--(axis cs:2300,36.8627823983173)
--(axis cs:2200,53.5348078791383)
--(axis cs:2100,135.130502638299)
--(axis cs:2000,32.2545821212999)
--(axis cs:1900,91.9216588855907)
--(axis cs:1800,177.574960131831)
--(axis cs:1700,149.448658724163)
--(axis cs:1600,260.528858577296)
--(axis cs:1500,131.495803042652)
--(axis cs:1400,729.847566842032)
--(axis cs:1300,149.736441150985)
--(axis cs:1200,61.8574411735335)
--(axis cs:1100,452.785857988839)
--(axis cs:1000,44.9553347611214)
--(axis cs:900,505.920772908246)
--(axis cs:800,330.910228328104)
--(axis cs:700,1349.34415828169)
--(axis cs:600,2503.83172150564)
--(axis cs:500,4329.0617224783)
--(axis cs:400,1060.86460734338)
--(axis cs:300,2688.81318975994)
--(axis cs:200,28247.335451971)
--(axis cs:100,31987.5872191298)
--cycle;

\addplot [ultra thick, darkslateblue]
table {%
100 0.296835658515268
200 0.296835658515268
300 0.296835658515268
400 0.296835658515268
500 0.296835658515268
600 0.296835658515268
700 0.296835658515268
800 0.296835658515268
900 0.296835658515268
1000 0.296835658515268
1100 0.296835658515268
1200 0.296835658515268
1300 0.296835658515268
1400 0.296835658515268
1500 0.296835658515268
1600 0.296835658515268
1700 0.296835658515268
1800 0.296835658515268
1900 0.296835658515268
2000 0.296835658515268
2100 0.296835658515268
2200 0.296835658515268
2300 0.296835658515268
2400 0.296835658515268
2500 0.296835658515268
};
\addlegendentry{CD}
\addplot [ultra thick, slategray]
table {%
100 2.65533206679491
200 2.65533206679491
300 2.65533206679491
400 2.65533206679491
500 2.65533206679491
600 2.65533206679491
700 2.65533206679491
800 2.65533206679491
900 2.65533206679491
1000 2.65533206679491
1100 2.65533206679491
1200 2.65533206679491
1300 2.65533206679491
1400 2.65533206679491
1500 2.65533206679491
1600 2.65533206679491
1700 2.65533206679491
1800 2.65533206679491
1900 2.65533206679491
2000 2.65533206679491
2100 2.65533206679491
2200 2.65533206679491
2300 2.65533206679491
2400 2.65533206679491
2500 2.65533206679491
};
\addlegendentry{FD exact H}
\addplot [ultra thick, darkorange]
table {%
100 0.0751816256525313
200 0.0751816256525313
300 0.0751816256525313
400 0.0751816256525313
500 0.0751816256525313
600 0.0751816256525313
700 0.0751816256525313
800 0.0751816256525313
900 0.0751816256525313
1000 0.0751816256525313
1100 0.0751816256525313
1200 0.0751816256525313
1300 0.0751816256525313
1400 0.0751816256525313
1500 0.0751816256525313
1600 0.0751816256525313
1700 0.0751816256525313
1800 0.0751816256525313
1900 0.0751816256525313
2000 0.0751816256525313
2100 0.0751816256525313
2200 0.0751816256525313
2300 0.0751816256525313
2400 0.0751816256525313
2500 0.0751816256525313
};
\addlegendentry{CASG exact H}
\addplot [ultra thick, slategray, dashed]
table {%
100 12.5430927029067
200 4.96770883479904
300 4.24685842601816
400 3.73126119992744
500 3.45633599598204
600 3.30954461780296
700 3.15225595275167
800 3.13771486031936
900 3.07842946005677
1000 3.04137115036649
1100 3.03726057878324
1200 3.01613558103335
1300 3.0225844510001
1400 2.99576614305945
1500 2.9964082012326
1600 2.96422422913409
1700 2.98783181768542
1800 2.97966677729741
1900 2.95871454730248
2000 2.9900423660116
2100 2.97125653489903
2200 2.96171093538006
2300 2.99040024484066
2400 2.99167240925137
2500 2.97208209530406
};
\addlegendentry{FD RBF H}
\addplot [ultra thick, darkorange, dashed]
table {%
100 3.80237209269471
200 2.00127367047948
300 1.5616225605326
400 1.16648249419255
500 0.864433383370198
600 0.65343572599297
700 0.633506914568015
800 0.565900758631338
900 0.535627782675064
1000 0.518411552837606
1100 0.507892048031461
1200 0.520084407276167
1300 0.503835043985541
1400 0.516935063743683
1500 0.508664147316414
1600 0.490398447728599
1700 0.501286134037362
1800 0.487747116053852
1900 0.490342406752279
2000 0.49830792069964
2100 0.488892102833227
2200 0.494874309132923
2300 0.491153260297961
2400 0.47176073349961
2500 0.475221224718081
};
\addlegendentry{CASG RBF H}
\addplot [ultra thick, maroon]
table {%
100 31900.7048177506
200 28167.4317839143
300 2656.42271838259
400 1038.41148750347
500 4310.60040927647
600 2483.07997972704
700 1328.10817761052
800 322.113196035068
900 497.860199768424
1000 38.3226066903953
1100 444.336331248118
1200 56.9464836872183
1300 144.105014281372
1400 714.747852797695
1500 126.479405314847
1600 250.192842396546
1700 147.120133928384
1800 169.489650547774
1900 89.6264694080064
2000 28.5241333886645
2100 130.788152665439
2200 50.4065512325168
2300 33.4679943425858
2400 12.0476061654803
2500 13.9769050735891
};
\addlegendentry{RBF}
\end{axis}

\end{tikzpicture}

%% file: figs/Optimization/Opt_sig_0.1_tau_0.1.tex
\begin{tikzpicture}

\definecolor{darkgray176}{RGB}{176,176,176}
\definecolor{darkorange}{RGB}{255,140,0}
\definecolor{darkslateblue}{RGB}{72,61,139}
\definecolor{lightgray204}{RGB}{204,204,204}
\definecolor{maroon}{RGB}{128,0,0}
\definecolor{palevioletred204121167}{RGB}{204,121,167}
\definecolor{slategray}{RGB}{112,128,144}

\begin{axis}[
legend cell align={left},
legend style={
  fill opacity=0.8,
  draw opacity=1,
  text opacity=1,
  at={(0.97,0.03)},
  anchor=south east,
  draw=lightgray204
},
tick align=outside,
tick pos=left,
x grid style={darkgray176},
xlabel={Number of Simplex Gradients},
xmajorgrids,
xmin=-10, xmax=210,
xminorgrids,
xtick style={color=black},
y grid style={darkgray176},
ylabel={Fraction of Problems},
ymajorgrids,
ymin=-0.0447791164658635, ymax=0.940361445783133,
yminorgrids,
ytick style={color=black}
]
\addplot [very thick, slategray]
table {%
0 0
1 0.00401606425702811
2 0.00401606425702811
3 0.110441767068273
4 0.281124497991968
5 0.355421686746988
6 0.389558232931727
7 0.437751004016064
8 0.493975903614458
9 0.516064257028112
10 0.580321285140562
11 0.622489959839357
12 0.63855421686747
13 0.660642570281124
14 0.692771084337349
15 0.714859437751004
16 0.740963855421687
17 0.753012048192771
18 0.775100401606426
19 0.783132530120482
20 0.78714859437751
21 0.795180722891566
22 0.799196787148594
23 0.801204819277108
24 0.807228915662651
25 0.813253012048193
26 0.817269076305221
27 0.819277108433735
28 0.825301204819277
29 0.825301204819277
30 0.825301204819277
31 0.825301204819277
32 0.829317269076305
33 0.831325301204819
34 0.833333333333333
35 0.833333333333333
36 0.833333333333333
37 0.835341365461847
38 0.837349397590361
39 0.837349397590361
40 0.837349397590361
41 0.839357429718876
42 0.839357429718876
43 0.84136546184739
44 0.84136546184739
45 0.84136546184739
46 0.84136546184739
47 0.843373493975904
48 0.843373493975904
49 0.843373493975904
50 0.843373493975904
51 0.843373493975904
52 0.843373493975904
53 0.843373493975904
54 0.843373493975904
55 0.843373493975904
56 0.843373493975904
57 0.843373493975904
58 0.843373493975904
59 0.843373493975904
60 0.843373493975904
61 0.843373493975904
62 0.845381526104418
63 0.845381526104418
64 0.845381526104418
65 0.845381526104418
66 0.845381526104418
67 0.845381526104418
68 0.845381526104418
69 0.845381526104418
70 0.845381526104418
71 0.845381526104418
72 0.845381526104418
73 0.845381526104418
74 0.845381526104418
75 0.845381526104418
76 0.845381526104418
77 0.845381526104418
78 0.845381526104418
79 0.845381526104418
80 0.845381526104418
81 0.845381526104418
82 0.845381526104418
83 0.845381526104418
84 0.845381526104418
85 0.845381526104418
86 0.845381526104418
87 0.845381526104418
88 0.845381526104418
89 0.845381526104418
90 0.845381526104418
91 0.845381526104418
92 0.845381526104418
93 0.845381526104418
94 0.845381526104418
95 0.845381526104418
96 0.845381526104418
97 0.845381526104418
98 0.845381526104418
99 0.845381526104418
100 0.845381526104418
101 0.845381526104418
102 0.845381526104418
103 0.845381526104418
104 0.845381526104418
105 0.845381526104418
106 0.845381526104418
107 0.845381526104418
108 0.845381526104418
109 0.845381526104418
110 0.845381526104418
111 0.845381526104418
112 0.845381526104418
113 0.845381526104418
114 0.845381526104418
115 0.845381526104418
116 0.845381526104418
117 0.845381526104418
118 0.845381526104418
119 0.845381526104418
120 0.845381526104418
121 0.845381526104418
122 0.845381526104418
123 0.845381526104418
124 0.845381526104418
125 0.845381526104418
126 0.845381526104418
127 0.845381526104418
128 0.845381526104418
129 0.845381526104418
130 0.845381526104418
131 0.845381526104418
132 0.845381526104418
133 0.845381526104418
134 0.845381526104418
135 0.845381526104418
136 0.845381526104418
137 0.845381526104418
138 0.845381526104418
139 0.845381526104418
140 0.845381526104418
141 0.845381526104418
142 0.845381526104418
143 0.845381526104418
144 0.845381526104418
145 0.845381526104418
146 0.845381526104418
147 0.845381526104418
148 0.845381526104418
149 0.845381526104418
150 0.845381526104418
151 0.845381526104418
152 0.845381526104418
153 0.845381526104418
154 0.845381526104418
155 0.845381526104418
156 0.845381526104418
157 0.845381526104418
158 0.845381526104418
159 0.847389558232932
160 0.847389558232932
161 0.847389558232932
162 0.847389558232932
163 0.847389558232932
164 0.847389558232932
165 0.847389558232932
166 0.847389558232932
167 0.847389558232932
168 0.847389558232932
169 0.847389558232932
170 0.847389558232932
171 0.847389558232932
172 0.847389558232932
173 0.847389558232932
174 0.847389558232932
175 0.847389558232932
176 0.847389558232932
177 0.847389558232932
178 0.847389558232932
179 0.847389558232932
180 0.847389558232932
181 0.847389558232932
182 0.847389558232932
183 0.847389558232932
184 0.847389558232932
185 0.847389558232932
186 0.847389558232932
187 0.847389558232932
188 0.847389558232932
189 0.847389558232932
190 0.847389558232932
191 0.847389558232932
192 0.847389558232932
193 0.847389558232932
194 0.847389558232932
195 0.847389558232932
196 0.847389558232932
197 0.847389558232932
198 0.847389558232932
199 0.847389558232932
200 0.847389558232932
};
\addlegendentry{FD exact H}
\addplot [very thick, slategray, dashed]
table {%
0 0
1 0
2 0
3 0
4 0
5 0
6 0
7 0
8 0
9 0.002
10 0.002
11 0.006
12 0.024
13 0.074
14 0.108
15 0.116
16 0.142
17 0.174
18 0.184
19 0.2
20 0.208
21 0.214
22 0.262
23 0.286
24 0.376
25 0.44
26 0.452
27 0.47
28 0.494
29 0.504
30 0.516
31 0.526
32 0.528
33 0.542
34 0.546
35 0.546
36 0.548
37 0.548
38 0.55
39 0.556
40 0.556
41 0.556
42 0.558
43 0.598
44 0.618
45 0.624
46 0.624
47 0.628
48 0.636
49 0.642
50 0.648
51 0.654
52 0.654
53 0.658
54 0.678
55 0.68
56 0.682
57 0.684
58 0.706
59 0.706
60 0.708
61 0.71
62 0.71
63 0.71
64 0.73
65 0.732
66 0.75
67 0.75
68 0.75
69 0.75
70 0.75
71 0.752
72 0.752
73 0.752
74 0.752
75 0.752
76 0.752
77 0.752
78 0.752
79 0.752
80 0.752
81 0.754
82 0.756
83 0.764
84 0.768
85 0.768
86 0.768
87 0.768
88 0.77
89 0.77
90 0.77
91 0.77
92 0.77
93 0.772
94 0.772
95 0.772
96 0.772
97 0.772
98 0.772
99 0.772
100 0.772
101 0.772
102 0.772
103 0.772
104 0.774
105 0.774
106 0.774
107 0.774
108 0.774
109 0.774
110 0.774
111 0.774
112 0.774
113 0.774
114 0.774
115 0.774
116 0.774
117 0.774
118 0.774
119 0.774
120 0.774
121 0.774
122 0.774
123 0.774
124 0.774
125 0.774
126 0.774
127 0.774
128 0.774
129 0.774
130 0.774
131 0.774
132 0.774
133 0.774
134 0.774
135 0.774
136 0.774
137 0.774
138 0.774
139 0.774
140 0.774
141 0.774
142 0.774
143 0.774
144 0.774
145 0.774
146 0.774
147 0.774
148 0.774
149 0.774
150 0.774
151 0.774
152 0.774
153 0.774
154 0.774
155 0.774
156 0.774
157 0.774
158 0.774
159 0.774
160 0.774
161 0.774
162 0.774
163 0.774
164 0.774
165 0.774
166 0.774
167 0.774
168 0.774
169 0.774
170 0.774
171 0.774
172 0.774
173 0.774
174 0.774
175 0.774
176 0.774
177 0.774
178 0.774
179 0.774
180 0.774
181 0.774
182 0.774
183 0.774
184 0.774
185 0.774
186 0.774
187 0.774
188 0.774
189 0.774
190 0.774
191 0.774
192 0.774
193 0.774
194 0.774
195 0.774
196 0.774
197 0.774
198 0.774
199 0.774
200 0.774
};
\addlegendentry{FD RBF H}
\addplot [very thick, maroon]
table {%
0 0
1 0
2 0
3 0
4 0
5 0
6 0
7 0
8 0
9 0.00204918032786885
10 0.00409836065573771
11 0.00819672131147541
12 0.0368852459016393
13 0.0450819672131148
14 0.0450819672131148
15 0.0491803278688525
16 0.0532786885245902
17 0.0573770491803279
18 0.0758196721311475
19 0.0819672131147541
20 0.0881147540983607
21 0.0963114754098361
22 0.100409836065574
23 0.120901639344262
24 0.14344262295082
25 0.17827868852459
26 0.190573770491803
27 0.202868852459016
28 0.209016393442623
29 0.215163934426229
30 0.223360655737705
31 0.23155737704918
32 0.235655737704918
33 0.245901639344262
34 0.247950819672131
35 0.258196721311475
36 0.270491803278689
37 0.276639344262295
38 0.284836065573771
39 0.290983606557377
40 0.30327868852459
41 0.315573770491803
42 0.323770491803279
43 0.336065573770492
44 0.348360655737705
45 0.383196721311475
46 0.387295081967213
47 0.395491803278689
48 0.403688524590164
49 0.415983606557377
50 0.418032786885246
51 0.424180327868852
52 0.42827868852459
53 0.432377049180328
54 0.434426229508197
55 0.438524590163934
56 0.44672131147541
57 0.459016393442623
58 0.463114754098361
59 0.477459016393443
60 0.48155737704918
61 0.483606557377049
62 0.487704918032787
63 0.493852459016393
64 0.502049180327869
65 0.514344262295082
66 0.526639344262295
67 0.526639344262295
68 0.530737704918033
69 0.555327868852459
70 0.563524590163934
71 0.573770491803279
72 0.577868852459016
73 0.584016393442623
74 0.600409836065574
75 0.608606557377049
76 0.614754098360656
77 0.620901639344262
78 0.622950819672131
79 0.625
80 0.627049180327869
81 0.627049180327869
82 0.627049180327869
83 0.627049180327869
84 0.629098360655738
85 0.629098360655738
86 0.631147540983607
87 0.631147540983607
88 0.631147540983607
89 0.633196721311475
90 0.635245901639344
91 0.635245901639344
92 0.635245901639344
93 0.635245901639344
94 0.635245901639344
95 0.637295081967213
96 0.637295081967213
97 0.637295081967213
98 0.637295081967213
99 0.637295081967213
100 0.637295081967213
101 0.637295081967213
102 0.637295081967213
103 0.639344262295082
104 0.639344262295082
105 0.641393442622951
106 0.641393442622951
107 0.641393442622951
108 0.64344262295082
109 0.645491803278688
110 0.645491803278688
111 0.647540983606557
112 0.647540983606557
113 0.647540983606557
114 0.647540983606557
115 0.647540983606557
116 0.647540983606557
117 0.649590163934426
118 0.649590163934426
119 0.649590163934426
120 0.649590163934426
121 0.649590163934426
122 0.649590163934426
123 0.649590163934426
124 0.649590163934426
125 0.651639344262295
126 0.651639344262295
127 0.651639344262295
128 0.653688524590164
129 0.653688524590164
130 0.653688524590164
131 0.653688524590164
132 0.653688524590164
133 0.653688524590164
134 0.653688524590164
135 0.653688524590164
136 0.653688524590164
137 0.653688524590164
138 0.653688524590164
139 0.653688524590164
140 0.653688524590164
141 0.653688524590164
142 0.655737704918033
143 0.655737704918033
144 0.655737704918033
145 0.655737704918033
146 0.655737704918033
147 0.655737704918033
148 0.655737704918033
149 0.655737704918033
150 0.655737704918033
151 0.655737704918033
152 0.655737704918033
153 0.655737704918033
154 0.655737704918033
155 0.655737704918033
156 0.657786885245902
157 0.657786885245902
158 0.657786885245902
159 0.657786885245902
160 0.657786885245902
161 0.657786885245902
162 0.657786885245902
163 0.657786885245902
164 0.659836065573771
165 0.659836065573771
166 0.659836065573771
167 0.659836065573771
168 0.659836065573771
169 0.659836065573771
170 0.659836065573771
171 0.659836065573771
172 0.659836065573771
173 0.659836065573771
174 0.659836065573771
175 0.659836065573771
176 0.659836065573771
177 0.659836065573771
178 0.659836065573771
179 0.659836065573771
180 0.659836065573771
181 0.659836065573771
182 0.659836065573771
183 0.659836065573771
184 0.659836065573771
185 0.659836065573771
186 0.659836065573771
187 0.659836065573771
188 0.659836065573771
189 0.659836065573771
190 0.659836065573771
191 0.659836065573771
192 0.659836065573771
193 0.659836065573771
194 0.659836065573771
195 0.659836065573771
196 0.659836065573771
197 0.659836065573771
198 0.659836065573771
199 0.659836065573771
200 0.659836065573771
};
\addlegendentry{RBF}
\addplot [very thick, darkorange, dashed]
table {%
0 0
1 0
2 0
3 0
4 0
5 0
6 0
7 0
8 0
9 0.004
10 0.018
11 0.018
12 0.04
13 0.078
14 0.1
15 0.116
16 0.166
17 0.202
18 0.218
19 0.222
20 0.228
21 0.232
22 0.28
23 0.3
24 0.414
25 0.446
26 0.462
27 0.49
28 0.51
29 0.522
30 0.53
31 0.536
32 0.542
33 0.546
34 0.55
35 0.552
36 0.554
37 0.554
38 0.554
39 0.558
40 0.56
41 0.57
42 0.57
43 0.618
44 0.652
45 0.654
46 0.658
47 0.658
48 0.662
49 0.668
50 0.674
51 0.68
52 0.682
53 0.684
54 0.71
55 0.714
56 0.714
57 0.714
58 0.734
59 0.736
60 0.74
61 0.74
62 0.74
63 0.74
64 0.762
65 0.768
66 0.782
67 0.782
68 0.782
69 0.782
70 0.782
71 0.782
72 0.782
73 0.782
74 0.782
75 0.782
76 0.782
77 0.782
78 0.782
79 0.784
80 0.784
81 0.784
82 0.784
83 0.784
84 0.784
85 0.784
86 0.784
87 0.784
88 0.784
89 0.784
90 0.784
91 0.784
92 0.784
93 0.784
94 0.784
95 0.784
96 0.784
97 0.784
98 0.784
99 0.784
100 0.784
101 0.784
102 0.784
103 0.784
104 0.786
105 0.786
106 0.786
107 0.786
108 0.786
109 0.788
110 0.788
111 0.788
112 0.788
113 0.788
114 0.788
115 0.788
116 0.788
117 0.788
118 0.788
119 0.788
120 0.788
121 0.788
122 0.788
123 0.788
124 0.788
125 0.788
126 0.788
127 0.788
128 0.788
129 0.788
130 0.788
131 0.788
132 0.788
133 0.788
134 0.788
135 0.788
136 0.788
137 0.788
138 0.788
139 0.788
140 0.788
141 0.788
142 0.788
143 0.788
144 0.788
145 0.788
146 0.788
147 0.788
148 0.788
149 0.788
150 0.788
151 0.788
152 0.788
153 0.788
154 0.788
155 0.788
156 0.788
157 0.788
158 0.788
159 0.788
160 0.788
161 0.788
162 0.788
163 0.788
164 0.788
165 0.788
166 0.788
167 0.788
168 0.788
169 0.788
170 0.788
171 0.788
172 0.788
173 0.788
174 0.788
175 0.788
176 0.788
177 0.788
178 0.788
179 0.788
180 0.788
181 0.788
182 0.788
183 0.788
184 0.788
185 0.788
186 0.788
187 0.788
188 0.788
189 0.788
190 0.788
191 0.788
192 0.788
193 0.788
194 0.788
195 0.788
196 0.788
197 0.788
198 0.788
199 0.788
200 0.788
};
\addlegendentry{CASG RBF H}
\addplot [very thick, darkorange]
table {%
0 0
1 0
2 0
3 0.102
4 0.276
5 0.32
6 0.37
7 0.416
8 0.466
9 0.496
10 0.542
11 0.584
12 0.624
13 0.638
14 0.682
15 0.704
16 0.722
17 0.736
18 0.766
19 0.77
20 0.778
21 0.796
22 0.804
23 0.812
24 0.812
25 0.82
26 0.822
27 0.822
28 0.824
29 0.826
30 0.826
31 0.826
32 0.828
33 0.83
34 0.832
35 0.832
36 0.832
37 0.832
38 0.832
39 0.832
40 0.832
41 0.834
42 0.836
43 0.836
44 0.836
45 0.836
46 0.838
47 0.838
48 0.838
49 0.838
50 0.838
51 0.838
52 0.838
53 0.838
54 0.838
55 0.838
56 0.838
57 0.838
58 0.838
59 0.84
60 0.84
61 0.84
62 0.84
63 0.84
64 0.84
65 0.84
66 0.84
67 0.84
68 0.84
69 0.84
70 0.84
71 0.84
72 0.84
73 0.84
74 0.84
75 0.84
76 0.84
77 0.84
78 0.84
79 0.84
80 0.84
81 0.84
82 0.84
83 0.84
84 0.84
85 0.84
86 0.84
87 0.84
88 0.84
89 0.84
90 0.84
91 0.84
92 0.84
93 0.84
94 0.84
95 0.84
96 0.84
97 0.84
98 0.84
99 0.84
100 0.84
101 0.84
102 0.84
103 0.84
104 0.84
105 0.84
106 0.84
107 0.84
108 0.84
109 0.84
110 0.84
111 0.84
112 0.84
113 0.84
114 0.84
115 0.84
116 0.84
117 0.84
118 0.84
119 0.84
120 0.84
121 0.84
122 0.84
123 0.84
124 0.84
125 0.84
126 0.84
127 0.84
128 0.84
129 0.84
130 0.84
131 0.84
132 0.84
133 0.84
134 0.84
135 0.84
136 0.84
137 0.84
138 0.84
139 0.84
140 0.842
141 0.842
142 0.842
143 0.842
144 0.842
145 0.842
146 0.842
147 0.842
148 0.842
149 0.842
150 0.842
151 0.842
152 0.842
153 0.842
154 0.842
155 0.842
156 0.842
157 0.842
158 0.842
159 0.842
160 0.842
161 0.842
162 0.842
163 0.842
164 0.842
165 0.842
166 0.842
167 0.842
168 0.842
169 0.842
170 0.842
171 0.842
172 0.842
173 0.842
174 0.844
175 0.844
176 0.844
177 0.844
178 0.844
179 0.844
180 0.844
181 0.844
182 0.844
183 0.844
184 0.844
185 0.844
186 0.844
187 0.844
188 0.844
189 0.844
190 0.844
191 0.844
192 0.844
193 0.844
194 0.844
195 0.844
196 0.844
197 0.844
198 0.844
199 0.844
200 0.844
};
\addlegendentry{CASG exact H}
\addplot [very thick, darkslateblue]
table {%
0 0
1 0.00200803212851406
2 0.00200803212851406
3 0.00200803212851406
4 0.00200803212851406
5 0.104417670682731
6 0.14859437751004
7 0.353413654618474
8 0.355421686746988
9 0.399598393574297
10 0.437751004016064
11 0.477911646586345
12 0.502008032128514
13 0.558232931726908
14 0.566265060240964
15 0.606425702811245
16 0.626506024096386
17 0.654618473895582
18 0.670682730923695
19 0.714859437751004
20 0.720883534136546
21 0.744979919678715
22 0.748995983935743
23 0.763052208835341
24 0.765060240963855
25 0.783132530120482
26 0.78714859437751
27 0.805220883534137
28 0.807228915662651
29 0.823293172690763
30 0.825301204819277
31 0.829317269076305
32 0.829317269076305
33 0.837349397590361
34 0.843373493975904
35 0.849397590361446
36 0.85140562248996
37 0.857429718875502
38 0.857429718875502
39 0.86144578313253
40 0.865461847389558
41 0.865461847389558
42 0.865461847389558
43 0.867469879518072
44 0.867469879518072
45 0.867469879518072
46 0.8714859437751
47 0.8714859437751
48 0.873493975903614
49 0.873493975903614
50 0.873493975903614
51 0.873493975903614
52 0.873493975903614
53 0.875502008032129
54 0.875502008032129
55 0.877510040160643
56 0.877510040160643
57 0.877510040160643
58 0.877510040160643
59 0.877510040160643
60 0.879518072289157
61 0.883534136546185
62 0.883534136546185
63 0.883534136546185
64 0.883534136546185
65 0.885542168674699
66 0.885542168674699
67 0.885542168674699
68 0.885542168674699
69 0.885542168674699
70 0.885542168674699
71 0.885542168674699
72 0.885542168674699
73 0.885542168674699
74 0.887550200803213
75 0.887550200803213
76 0.887550200803213
77 0.887550200803213
78 0.887550200803213
79 0.887550200803213
80 0.887550200803213
81 0.887550200803213
82 0.887550200803213
83 0.887550200803213
84 0.889558232931727
85 0.891566265060241
86 0.891566265060241
87 0.891566265060241
88 0.891566265060241
89 0.891566265060241
90 0.891566265060241
91 0.891566265060241
92 0.891566265060241
93 0.891566265060241
94 0.891566265060241
95 0.891566265060241
96 0.891566265060241
97 0.891566265060241
98 0.891566265060241
99 0.891566265060241
100 0.891566265060241
101 0.891566265060241
102 0.891566265060241
103 0.891566265060241
104 0.891566265060241
105 0.891566265060241
106 0.891566265060241
107 0.891566265060241
108 0.891566265060241
109 0.891566265060241
110 0.891566265060241
111 0.891566265060241
112 0.891566265060241
113 0.891566265060241
114 0.891566265060241
115 0.891566265060241
116 0.891566265060241
117 0.891566265060241
118 0.891566265060241
119 0.891566265060241
120 0.891566265060241
121 0.891566265060241
122 0.891566265060241
123 0.891566265060241
124 0.891566265060241
125 0.891566265060241
126 0.891566265060241
127 0.891566265060241
128 0.891566265060241
129 0.891566265060241
130 0.891566265060241
131 0.891566265060241
132 0.891566265060241
133 0.891566265060241
134 0.891566265060241
135 0.891566265060241
136 0.891566265060241
137 0.891566265060241
138 0.891566265060241
139 0.891566265060241
140 0.891566265060241
141 0.891566265060241
142 0.891566265060241
143 0.891566265060241
144 0.891566265060241
145 0.891566265060241
146 0.891566265060241
147 0.891566265060241
148 0.891566265060241
149 0.891566265060241
150 0.891566265060241
151 0.891566265060241
152 0.891566265060241
153 0.891566265060241
154 0.891566265060241
155 0.893574297188755
156 0.893574297188755
157 0.893574297188755
158 0.893574297188755
159 0.893574297188755
160 0.893574297188755
161 0.893574297188755
162 0.893574297188755
163 0.893574297188755
164 0.893574297188755
165 0.893574297188755
166 0.893574297188755
167 0.893574297188755
168 0.893574297188755
169 0.893574297188755
170 0.893574297188755
171 0.893574297188755
172 0.893574297188755
173 0.893574297188755
174 0.893574297188755
175 0.893574297188755
176 0.893574297188755
177 0.893574297188755
178 0.893574297188755
179 0.893574297188755
180 0.893574297188755
181 0.893574297188755
182 0.893574297188755
183 0.893574297188755
184 0.893574297188755
185 0.895582329317269
186 0.895582329317269
187 0.895582329317269
188 0.895582329317269
189 0.895582329317269
190 0.895582329317269
191 0.895582329317269
192 0.895582329317269
193 0.895582329317269
194 0.895582329317269
195 0.895582329317269
196 0.895582329317269
197 0.895582329317269
198 0.895582329317269
199 0.895582329317269
200 0.895582329317269
};
\addlegendentry{CD}
\addplot [very thick, palevioletred204121167]
table {%
0 0
1 0
2 0
3 0
4 0
5 0
6 0
7 0.00612244897959184
8 0.0183673469387755
9 0.0285714285714286
10 0.0591836734693878
11 0.0857142857142857
12 0.130612244897959
13 0.153061224489796
14 0.173469387755102
15 0.2
16 0.23265306122449
17 0.283673469387755
18 0.324489795918367
19 0.351020408163265
20 0.36734693877551
21 0.377551020408163
22 0.397959183673469
23 0.418367346938776
24 0.426530612244898
25 0.444897959183673
26 0.459183673469388
27 0.46530612244898
28 0.477551020408163
29 0.489795918367347
30 0.504081632653061
31 0.506122448979592
32 0.510204081632653
33 0.522448979591837
34 0.53469387755102
35 0.544897959183673
36 0.553061224489796
37 0.561224489795918
38 0.56530612244898
39 0.56734693877551
40 0.569387755102041
41 0.581632653061224
42 0.593877551020408
43 0.6
44 0.6
45 0.604081632653061
46 0.604081632653061
47 0.608163265306122
48 0.608163265306122
49 0.616326530612245
50 0.624489795918367
51 0.63469387755102
52 0.63469387755102
53 0.640816326530612
54 0.644897959183673
55 0.648979591836735
56 0.653061224489796
57 0.653061224489796
58 0.653061224489796
59 0.657142857142857
60 0.657142857142857
61 0.661224489795918
62 0.661224489795918
63 0.661224489795918
64 0.663265306122449
65 0.663265306122449
66 0.663265306122449
67 0.669387755102041
68 0.673469387755102
69 0.679591836734694
70 0.681632653061224
71 0.683673469387755
72 0.689795918367347
73 0.691836734693878
74 0.691836734693878
75 0.697959183673469
76 0.697959183673469
77 0.7
78 0.7
79 0.702040816326531
80 0.702040816326531
81 0.704081632653061
82 0.708163265306123
83 0.708163265306123
84 0.710204081632653
85 0.710204081632653
86 0.712244897959184
87 0.712244897959184
88 0.712244897959184
89 0.714285714285714
90 0.714285714285714
91 0.714285714285714
92 0.716326530612245
93 0.718367346938775
94 0.718367346938775
95 0.720408163265306
96 0.720408163265306
97 0.722448979591837
98 0.722448979591837
99 0.724489795918367
100 0.726530612244898
101 0.726530612244898
102 0.728571428571429
103 0.728571428571429
104 0.728571428571429
105 0.730612244897959
106 0.730612244897959
107 0.730612244897959
108 0.730612244897959
109 0.730612244897959
110 0.73265306122449
111 0.73265306122449
112 0.73265306122449
113 0.73469387755102
114 0.73469387755102
115 0.73469387755102
116 0.736734693877551
117 0.736734693877551
118 0.736734693877551
119 0.736734693877551
120 0.736734693877551
121 0.736734693877551
122 0.738775510204082
123 0.738775510204082
124 0.738775510204082
125 0.738775510204082
126 0.738775510204082
127 0.738775510204082
128 0.740816326530612
129 0.740816326530612
130 0.740816326530612
131 0.742857142857143
132 0.744897959183674
133 0.746938775510204
134 0.746938775510204
135 0.746938775510204
136 0.746938775510204
137 0.748979591836735
138 0.748979591836735
139 0.748979591836735
140 0.748979591836735
141 0.748979591836735
142 0.748979591836735
143 0.748979591836735
144 0.748979591836735
145 0.748979591836735
146 0.748979591836735
147 0.748979591836735
148 0.748979591836735
149 0.748979591836735
150 0.748979591836735
151 0.751020408163265
152 0.751020408163265
153 0.751020408163265
154 0.755102040816326
155 0.755102040816326
156 0.755102040816326
157 0.755102040816326
158 0.755102040816326
159 0.755102040816326
160 0.755102040816326
161 0.755102040816326
162 0.755102040816326
163 0.755102040816326
164 0.755102040816326
165 0.757142857142857
166 0.757142857142857
167 0.757142857142857
168 0.757142857142857
169 0.757142857142857
170 0.759183673469388
171 0.759183673469388
172 0.759183673469388
173 0.759183673469388
174 0.759183673469388
175 0.759183673469388
176 0.759183673469388
177 0.759183673469388
178 0.759183673469388
179 0.759183673469388
180 0.759183673469388
181 0.759183673469388
182 0.761224489795918
183 0.761224489795918
184 0.761224489795918
185 0.761224489795918
186 0.761224489795918
187 0.761224489795918
188 0.761224489795918
189 0.761224489795918
190 0.761224489795918
191 0.761224489795918
192 0.761224489795918
193 0.761224489795918
194 0.761224489795918
195 0.761224489795918
196 0.761224489795918
197 0.761224489795918
198 0.761224489795918
199 0.761224489795918
200 0.761224489795918
};
\addlegendentry{AFD}
\end{axis}

\end{tikzpicture}

%% file: figs/Optimization/Opt_sig_0.001_tau_0.1.tex
\begin{tikzpicture}

\definecolor{darkgray176}{RGB}{176,176,176}
\definecolor{darkorange}{RGB}{255,140,0}
\definecolor{darkslateblue}{RGB}{72,61,139}
\definecolor{lightgray204}{RGB}{204,204,204}
\definecolor{maroon}{RGB}{128,0,0}
\definecolor{palevioletred204121167}{RGB}{204,121,167}
\definecolor{slategray}{RGB}{112,128,144}

\begin{axis}[
legend cell align={left},
legend style={
  fill opacity=0.8,
  draw opacity=1,
  text opacity=1,
  at={(0.97,0.03)},
  anchor=south east,
  draw=lightgray204
},
tick align=outside,
tick pos=left,
x grid style={darkgray176},
xlabel={Number of Simplex Gradients},
xmajorgrids,
xmin=-10, xmax=210,
xminorgrids,
xtick style={color=black},
y grid style={darkgray176},
ylabel={Fraction of Problems},
ymajorgrids,
ymin=-0.0480846774193548, ymax=1.00977822580645,
yminorgrids,
ytick style={color=black}
]
\addplot [very thick, maroon]
table {%
0 0
1 0
2 0
3 0
4 0
5 0
6 0
7 0
8 0
9 0.00613496932515337
10 0.00613496932515337
11 0.0081799591002045
12 0.0347648261758691
13 0.0408997955010225
14 0.0408997955010225
15 0.049079754601227
16 0.0511247443762781
17 0.0531697341513292
18 0.0674846625766871
19 0.0797546012269939
20 0.0858895705521472
21 0.0920245398773006
22 0.0961145194274029
23 0.114519427402863
24 0.124744376278119
25 0.153374233128834
26 0.16359918200409
27 0.169734151329243
28 0.171779141104294
29 0.173824130879346
30 0.179959100204499
31 0.18200408997955
32 0.186094069529652
33 0.188139059304703
34 0.192229038854806
35 0.192229038854806
36 0.196319018404908
37 0.20040899795501
38 0.204498977505112
39 0.206543967280164
40 0.214723926380368
41 0.220858895705521
42 0.229038854805726
43 0.23721881390593
44 0.243353783231084
45 0.265848670756646
46 0.278118609406953
47 0.282208588957055
48 0.292433537832311
49 0.304703476482618
50 0.312883435582822
51 0.319018404907975
52 0.32719836400818
53 0.337423312883436
54 0.343558282208589
55 0.34560327198364
56 0.347648261758691
57 0.357873210633947
58 0.361963190184049
59 0.372188139059305
60 0.388548057259714
61 0.394683026584867
62 0.396728016359918
63 0.398773006134969
64 0.402862985685072
65 0.415132924335378
66 0.441717791411043
67 0.443762781186094
68 0.445807770961145
69 0.451942740286299
70 0.474437627811861
71 0.480572597137014
72 0.49079754601227
73 0.494887525562372
74 0.501022494887526
75 0.523517382413088
76 0.525562372188139
77 0.529652351738241
78 0.54601226993865
79 0.554192229038855
80 0.556237218813906
81 0.560327198364008
82 0.562372188139059
83 0.562372188139059
84 0.562372188139059
85 0.566462167689162
86 0.566462167689162
87 0.570552147239264
88 0.572597137014315
89 0.576687116564417
90 0.578732106339468
91 0.580777096114519
92 0.582822085889571
93 0.586912065439673
94 0.591002044989775
95 0.591002044989775
96 0.593047034764826
97 0.593047034764826
98 0.593047034764826
99 0.593047034764826
100 0.593047034764826
101 0.593047034764826
102 0.595092024539877
103 0.597137014314928
104 0.59918200408998
105 0.601226993865031
106 0.603271983640082
107 0.603271983640082
108 0.603271983640082
109 0.605316973415133
110 0.607361963190184
111 0.607361963190184
112 0.607361963190184
113 0.607361963190184
114 0.607361963190184
115 0.611451942740286
116 0.611451942740286
117 0.611451942740286
118 0.611451942740286
119 0.611451942740286
120 0.613496932515337
121 0.613496932515337
122 0.613496932515337
123 0.61758691206544
124 0.619631901840491
125 0.621676891615542
126 0.621676891615542
127 0.621676891615542
128 0.621676891615542
129 0.623721881390593
130 0.623721881390593
131 0.623721881390593
132 0.623721881390593
133 0.623721881390593
134 0.623721881390593
135 0.623721881390593
136 0.625766871165644
137 0.625766871165644
138 0.627811860940695
139 0.627811860940695
140 0.627811860940695
141 0.629856850715746
142 0.629856850715746
143 0.629856850715746
144 0.629856850715746
145 0.629856850715746
146 0.631901840490798
147 0.631901840490798
148 0.631901840490798
149 0.6359918200409
150 0.638036809815951
151 0.638036809815951
152 0.640081799591002
153 0.640081799591002
154 0.640081799591002
155 0.640081799591002
156 0.640081799591002
157 0.642126789366053
158 0.642126789366053
159 0.642126789366053
160 0.642126789366053
161 0.642126789366053
162 0.642126789366053
163 0.642126789366053
164 0.642126789366053
165 0.642126789366053
166 0.642126789366053
167 0.642126789366053
168 0.642126789366053
169 0.644171779141104
170 0.644171779141104
171 0.644171779141104
172 0.644171779141104
173 0.644171779141104
174 0.644171779141104
175 0.644171779141104
176 0.644171779141104
177 0.644171779141104
178 0.644171779141104
179 0.644171779141104
180 0.644171779141104
181 0.644171779141104
182 0.644171779141104
183 0.644171779141104
184 0.644171779141104
185 0.644171779141104
186 0.644171779141104
187 0.644171779141104
188 0.644171779141104
189 0.644171779141104
190 0.644171779141104
191 0.644171779141104
192 0.644171779141104
193 0.644171779141104
194 0.644171779141104
195 0.644171779141104
196 0.644171779141104
197 0.644171779141104
198 0.644171779141104
199 0.644171779141104
200 0.644171779141104
};
\addlegendentry{RBF}
\addplot [very thick, slategray]
table {%
0 0
1 0
2 0
3 0.118473895582329
4 0.28714859437751
5 0.327309236947791
6 0.399598393574297
7 0.485943775100402
8 0.534136546184739
9 0.566265060240964
10 0.602409638554217
11 0.666666666666667
12 0.694779116465863
13 0.704819277108434
14 0.728915662650602
15 0.753012048192771
16 0.767068273092369
17 0.77710843373494
18 0.803212851405622
19 0.803212851405622
20 0.805220883534137
21 0.811244979919679
22 0.815261044176707
23 0.823293172690763
24 0.829317269076305
25 0.833333333333333
26 0.837349397590361
27 0.839357429718876
28 0.84136546184739
29 0.843373493975904
30 0.847389558232932
31 0.849397590361446
32 0.849397590361446
33 0.85140562248996
34 0.85140562248996
35 0.85140562248996
36 0.853413654618474
37 0.853413654618474
38 0.855421686746988
39 0.855421686746988
40 0.855421686746988
41 0.859437751004016
42 0.859437751004016
43 0.86144578313253
44 0.86144578313253
45 0.86144578313253
46 0.86144578313253
47 0.86144578313253
48 0.86144578313253
49 0.86144578313253
50 0.86144578313253
51 0.86144578313253
52 0.86144578313253
53 0.86144578313253
54 0.863453815261044
55 0.867469879518072
56 0.869477911646586
57 0.869477911646586
58 0.869477911646586
59 0.869477911646586
60 0.869477911646586
61 0.8714859437751
62 0.8714859437751
63 0.8714859437751
64 0.875502008032129
65 0.875502008032129
66 0.875502008032129
67 0.877510040160643
68 0.877510040160643
69 0.881526104417671
70 0.883534136546185
71 0.883534136546185
72 0.883534136546185
73 0.883534136546185
74 0.885542168674699
75 0.885542168674699
76 0.887550200803213
77 0.889558232931727
78 0.889558232931727
79 0.891566265060241
80 0.893574297188755
81 0.893574297188755
82 0.893574297188755
83 0.895582329317269
84 0.895582329317269
85 0.895582329317269
86 0.895582329317269
87 0.895582329317269
88 0.895582329317269
89 0.895582329317269
90 0.895582329317269
91 0.897590361445783
92 0.899598393574297
93 0.899598393574297
94 0.899598393574297
95 0.899598393574297
96 0.899598393574297
97 0.899598393574297
98 0.899598393574297
99 0.899598393574297
100 0.899598393574297
101 0.899598393574297
102 0.899598393574297
103 0.899598393574297
104 0.899598393574297
105 0.899598393574297
106 0.899598393574297
107 0.899598393574297
108 0.899598393574297
109 0.899598393574297
110 0.899598393574297
111 0.899598393574297
112 0.899598393574297
113 0.899598393574297
114 0.899598393574297
115 0.899598393574297
116 0.899598393574297
117 0.899598393574297
118 0.899598393574297
119 0.899598393574297
120 0.899598393574297
121 0.899598393574297
122 0.899598393574297
123 0.899598393574297
124 0.899598393574297
125 0.899598393574297
126 0.899598393574297
127 0.901606425702811
128 0.901606425702811
129 0.901606425702811
130 0.901606425702811
131 0.901606425702811
132 0.901606425702811
133 0.901606425702811
134 0.903614457831325
135 0.905622489959839
136 0.905622489959839
137 0.905622489959839
138 0.909638554216867
139 0.909638554216867
140 0.911646586345382
141 0.911646586345382
142 0.911646586345382
143 0.911646586345382
144 0.911646586345382
145 0.911646586345382
146 0.911646586345382
147 0.911646586345382
148 0.911646586345382
149 0.911646586345382
150 0.911646586345382
151 0.911646586345382
152 0.911646586345382
153 0.911646586345382
154 0.911646586345382
155 0.911646586345382
156 0.911646586345382
157 0.911646586345382
158 0.911646586345382
159 0.911646586345382
160 0.911646586345382
161 0.911646586345382
162 0.911646586345382
163 0.911646586345382
164 0.911646586345382
165 0.911646586345382
166 0.911646586345382
167 0.911646586345382
168 0.911646586345382
169 0.911646586345382
170 0.911646586345382
171 0.911646586345382
172 0.911646586345382
173 0.911646586345382
174 0.911646586345382
175 0.911646586345382
176 0.913654618473896
177 0.913654618473896
178 0.913654618473896
179 0.913654618473896
180 0.913654618473896
181 0.917670682730924
182 0.917670682730924
183 0.917670682730924
184 0.917670682730924
185 0.917670682730924
186 0.917670682730924
187 0.917670682730924
188 0.917670682730924
189 0.917670682730924
190 0.917670682730924
191 0.917670682730924
192 0.917670682730924
193 0.917670682730924
194 0.917670682730924
195 0.917670682730924
196 0.917670682730924
197 0.917670682730924
198 0.917670682730924
199 0.917670682730924
200 0.917670682730924
};
\addlegendentry{FD exact H}
\addplot [very thick, darkorange]
table {%
0 0
1 0
2 0
3 0.100603621730382
4 0.295774647887324
5 0.336016096579477
6 0.402414486921529
7 0.48692152917505
8 0.533199195171026
9 0.577464788732394
10 0.585513078470825
11 0.653923541247485
12 0.712273641851107
13 0.714285714285714
14 0.75251509054326
15 0.782696177062374
16 0.802816901408451
17 0.806841046277666
18 0.832997987927565
19 0.832997987927565
20 0.832997987927565
21 0.843058350100604
22 0.847082494969819
23 0.853118712273642
24 0.857142857142857
25 0.861167002012072
26 0.873239436619718
27 0.875251509054326
28 0.885311871227364
29 0.885311871227364
30 0.885311871227364
31 0.885311871227364
32 0.885311871227364
33 0.887323943661972
34 0.887323943661972
35 0.887323943661972
36 0.887323943661972
37 0.887323943661972
38 0.88933601609658
39 0.893360160965795
40 0.895372233400402
41 0.895372233400402
42 0.899396378269618
43 0.901408450704225
44 0.901408450704225
45 0.903420523138833
46 0.905432595573441
47 0.905432595573441
48 0.907444668008048
49 0.909456740442656
50 0.909456740442656
51 0.909456740442656
52 0.909456740442656
53 0.909456740442656
54 0.911468812877264
55 0.913480885311871
56 0.913480885311871
57 0.913480885311871
58 0.913480885311871
59 0.913480885311871
60 0.913480885311871
61 0.915492957746479
62 0.915492957746479
63 0.915492957746479
64 0.915492957746479
65 0.915492957746479
66 0.915492957746479
67 0.915492957746479
68 0.917505030181087
69 0.917505030181087
70 0.917505030181087
71 0.919517102615694
72 0.919517102615694
73 0.919517102615694
74 0.919517102615694
75 0.919517102615694
76 0.919517102615694
77 0.919517102615694
78 0.919517102615694
79 0.919517102615694
80 0.921529175050302
81 0.921529175050302
82 0.921529175050302
83 0.921529175050302
84 0.921529175050302
85 0.921529175050302
86 0.921529175050302
87 0.921529175050302
88 0.921529175050302
89 0.921529175050302
90 0.921529175050302
91 0.921529175050302
92 0.921529175050302
93 0.921529175050302
94 0.921529175050302
95 0.921529175050302
96 0.921529175050302
97 0.921529175050302
98 0.921529175050302
99 0.92354124748491
100 0.92354124748491
101 0.92354124748491
102 0.92354124748491
103 0.92354124748491
104 0.92354124748491
105 0.92354124748491
106 0.92354124748491
107 0.92354124748491
108 0.92354124748491
109 0.92354124748491
110 0.92354124748491
111 0.92354124748491
112 0.92354124748491
113 0.92354124748491
114 0.92354124748491
115 0.925553319919517
116 0.925553319919517
117 0.927565392354125
118 0.927565392354125
119 0.927565392354125
120 0.927565392354125
121 0.927565392354125
122 0.927565392354125
123 0.927565392354125
124 0.929577464788732
125 0.929577464788732
126 0.929577464788732
127 0.929577464788732
128 0.93158953722334
129 0.935613682092555
130 0.935613682092555
131 0.935613682092555
132 0.935613682092555
133 0.935613682092555
134 0.935613682092555
135 0.935613682092555
136 0.935613682092555
137 0.935613682092555
138 0.935613682092555
139 0.937625754527163
140 0.937625754527163
141 0.937625754527163
142 0.937625754527163
143 0.939637826961771
144 0.939637826961771
145 0.941649899396378
146 0.941649899396378
147 0.941649899396378
148 0.945674044265594
149 0.945674044265594
150 0.945674044265594
151 0.945674044265594
152 0.945674044265594
153 0.945674044265594
154 0.945674044265594
155 0.945674044265594
156 0.945674044265594
157 0.945674044265594
158 0.945674044265594
159 0.945674044265594
160 0.945674044265594
161 0.945674044265594
162 0.947686116700201
163 0.947686116700201
164 0.947686116700201
165 0.949698189134809
166 0.949698189134809
167 0.949698189134809
168 0.949698189134809
169 0.949698189134809
170 0.949698189134809
171 0.949698189134809
172 0.951710261569417
173 0.953722334004024
174 0.953722334004024
175 0.953722334004024
176 0.953722334004024
177 0.953722334004024
178 0.955734406438632
179 0.955734406438632
180 0.955734406438632
181 0.955734406438632
182 0.955734406438632
183 0.955734406438632
184 0.955734406438632
185 0.955734406438632
186 0.955734406438632
187 0.955734406438632
188 0.955734406438632
189 0.955734406438632
190 0.955734406438632
191 0.955734406438632
192 0.955734406438632
193 0.955734406438632
194 0.955734406438632
195 0.955734406438632
196 0.955734406438632
197 0.955734406438632
198 0.955734406438632
199 0.955734406438632
200 0.955734406438632
};
\addlegendentry{CASG exact H}
\addplot [very thick, slategray, dashed]
table {%
0 0
1 0
2 0
3 0
4 0
5 0
6 0
7 0
8 0
9 0
10 0.012
11 0.032
12 0.06
13 0.1
14 0.1
15 0.142
16 0.182
17 0.218
18 0.228
19 0.23
20 0.252
21 0.264
22 0.306
23 0.328
24 0.428
25 0.47
26 0.48
27 0.492
28 0.534
29 0.554
30 0.564
31 0.574
32 0.578
33 0.582
34 0.594
35 0.596
36 0.596
37 0.6
38 0.602
39 0.602
40 0.604
41 0.604
42 0.604
43 0.638
44 0.666
45 0.666
46 0.668
47 0.668
48 0.688
49 0.688
50 0.69
51 0.69
52 0.692
53 0.692
54 0.712
55 0.712
56 0.714
57 0.714
58 0.736
59 0.738
60 0.738
61 0.738
62 0.74
63 0.74
64 0.76
65 0.76
66 0.782
67 0.782
68 0.782
69 0.782
70 0.782
71 0.79
72 0.796
73 0.798
74 0.802
75 0.804
76 0.82
77 0.824
78 0.824
79 0.824
80 0.828
81 0.83
82 0.83
83 0.83
84 0.83
85 0.83
86 0.83
87 0.83
88 0.83
89 0.83
90 0.832
91 0.832
92 0.832
93 0.832
94 0.832
95 0.832
96 0.832
97 0.832
98 0.832
99 0.832
100 0.832
101 0.832
102 0.832
103 0.834
104 0.834
105 0.834
106 0.834
107 0.834
108 0.834
109 0.834
110 0.836
111 0.836
112 0.836
113 0.836
114 0.838
115 0.838
116 0.838
117 0.838
118 0.838
119 0.838
120 0.838
121 0.838
122 0.838
123 0.838
124 0.838
125 0.838
126 0.838
127 0.838
128 0.84
129 0.84
130 0.84
131 0.84
132 0.84
133 0.84
134 0.84
135 0.84
136 0.84
137 0.84
138 0.84
139 0.84
140 0.84
141 0.84
142 0.84
143 0.84
144 0.84
145 0.84
146 0.84
147 0.84
148 0.842
149 0.842
150 0.842
151 0.842
152 0.842
153 0.842
154 0.842
155 0.842
156 0.842
157 0.842
158 0.842
159 0.842
160 0.844
161 0.844
162 0.844
163 0.844
164 0.844
165 0.844
166 0.844
167 0.846
168 0.846
169 0.846
170 0.846
171 0.846
172 0.846
173 0.846
174 0.846
175 0.846
176 0.846
177 0.846
178 0.846
179 0.846
180 0.846
181 0.848
182 0.848
183 0.848
184 0.848
185 0.848
186 0.848
187 0.85
188 0.852
189 0.852
190 0.852
191 0.852
192 0.852
193 0.852
194 0.852
195 0.852
196 0.852
197 0.852
198 0.852
199 0.852
200 0.852
};
\addlegendentry{FD RBF H}
\addplot [very thick, palevioletred204121167]
table {%
0 0
1 0.00200400801603206
2 0.00200400801603206
3 0.00200400801603206
4 0.00200400801603206
5 0.00200400801603206
6 0.00200400801603206
7 0.00400801603206413
8 0.0100200400801603
9 0.0220440881763527
10 0.0480961923847695
11 0.0721442885771543
12 0.110220440881764
13 0.176352705410822
14 0.24248496993988
15 0.246492985971944
16 0.266533066132265
17 0.318637274549098
18 0.354709418837675
19 0.366733466933868
20 0.374749498997996
21 0.382765531062124
22 0.398797595190381
23 0.402805611222445
24 0.414829659318637
25 0.432865731462926
26 0.442885771543086
27 0.456913827655311
28 0.498997995991984
29 0.509018036072144
30 0.529058116232465
31 0.535070140280561
32 0.537074148296593
33 0.547094188376753
34 0.55310621242485
35 0.56312625250501
36 0.575150300601202
37 0.579158316633267
38 0.587174348697395
39 0.603206412825651
40 0.609218436873748
41 0.613226452905812
42 0.615230460921844
43 0.617234468937876
44 0.617234468937876
45 0.617234468937876
46 0.629258517034068
47 0.635270541082164
48 0.637274549098196
49 0.651302605210421
50 0.653306613226453
51 0.661322645290581
52 0.661322645290581
53 0.663326653306613
54 0.665330661322645
55 0.665330661322645
56 0.665330661322645
57 0.667334669338677
58 0.669338677354709
59 0.671342685370741
60 0.671342685370741
61 0.673346693386774
62 0.67935871743487
63 0.683366733466934
64 0.683366733466934
65 0.685370741482966
66 0.691382765531062
67 0.69939879759519
68 0.707414829659319
69 0.709418837675351
70 0.709418837675351
71 0.713426853707415
72 0.713426853707415
73 0.717434869739479
74 0.717434869739479
75 0.721442885771543
76 0.721442885771543
77 0.721442885771543
78 0.723446893787575
79 0.723446893787575
80 0.725450901803607
81 0.731462925851703
82 0.7374749498998
83 0.739478957915832
84 0.741482965931864
85 0.745490981963928
86 0.74749498997996
87 0.74749498997996
88 0.751503006012024
89 0.751503006012024
90 0.753507014028056
91 0.75751503006012
92 0.759519038076152
93 0.759519038076152
94 0.763527054108216
95 0.765531062124248
96 0.767535070140281
97 0.767535070140281
98 0.771543086172345
99 0.771543086172345
100 0.777555110220441
101 0.777555110220441
102 0.777555110220441
103 0.777555110220441
104 0.779559118236473
105 0.779559118236473
106 0.779559118236473
107 0.779559118236473
108 0.779559118236473
109 0.781563126252505
110 0.783567134268537
111 0.783567134268537
112 0.783567134268537
113 0.783567134268537
114 0.783567134268537
115 0.785571142284569
116 0.785571142284569
117 0.785571142284569
118 0.785571142284569
119 0.785571142284569
120 0.785571142284569
121 0.785571142284569
122 0.787575150300601
123 0.787575150300601
124 0.787575150300601
125 0.787575150300601
126 0.789579158316633
127 0.791583166332665
128 0.791583166332665
129 0.791583166332665
130 0.791583166332665
131 0.791583166332665
132 0.791583166332665
133 0.791583166332665
134 0.791583166332665
135 0.795591182364729
136 0.795591182364729
137 0.795591182364729
138 0.797595190380762
139 0.797595190380762
140 0.799599198396794
141 0.799599198396794
142 0.799599198396794
143 0.799599198396794
144 0.799599198396794
145 0.801603206412826
146 0.801603206412826
147 0.801603206412826
148 0.803607214428858
149 0.803607214428858
150 0.803607214428858
151 0.80561122244489
152 0.80561122244489
153 0.80561122244489
154 0.80561122244489
155 0.80561122244489
156 0.809619238476954
157 0.811623246492986
158 0.811623246492986
159 0.811623246492986
160 0.811623246492986
161 0.811623246492986
162 0.813627254509018
163 0.813627254509018
164 0.813627254509018
165 0.81563126252505
166 0.81563126252505
167 0.81563126252505
168 0.81563126252505
169 0.81563126252505
170 0.81563126252505
171 0.81563126252505
172 0.81563126252505
173 0.81563126252505
174 0.81563126252505
175 0.81563126252505
176 0.81563126252505
177 0.81563126252505
178 0.81563126252505
179 0.81563126252505
180 0.817635270541082
181 0.819639278557114
182 0.819639278557114
183 0.819639278557114
184 0.819639278557114
185 0.819639278557114
186 0.819639278557114
187 0.819639278557114
188 0.819639278557114
189 0.819639278557114
190 0.819639278557114
191 0.819639278557114
192 0.819639278557114
193 0.821643286573146
194 0.821643286573146
195 0.821643286573146
196 0.821643286573146
197 0.821643286573146
198 0.821643286573146
199 0.821643286573146
200 0.821643286573146
};
\addlegendentry{AFD}
\addplot [very thick, darkslateblue]
table {%
0 0
1 0
2 0
3 0
4 0
5 0.102822580645161
6 0.143145161290323
7 0.368951612903226
8 0.389112903225806
9 0.449596774193548
10 0.449596774193548
11 0.469758064516129
12 0.48991935483871
13 0.584677419354839
14 0.600806451612903
15 0.643145161290323
16 0.643145161290323
17 0.669354838709677
18 0.685483870967742
19 0.721774193548387
20 0.73991935483871
21 0.747983870967742
22 0.754032258064516
23 0.766129032258065
24 0.774193548387097
25 0.778225806451613
26 0.788306451612903
27 0.804435483870968
28 0.816532258064516
29 0.82258064516129
30 0.824596774193548
31 0.824596774193548
32 0.826612903225806
33 0.828629032258065
34 0.830645161290323
35 0.832661290322581
36 0.838709677419355
37 0.838709677419355
38 0.838709677419355
39 0.848790322580645
40 0.850806451612903
41 0.850806451612903
42 0.850806451612903
43 0.852822580645161
44 0.854838709677419
45 0.858870967741935
46 0.862903225806452
47 0.86491935483871
48 0.870967741935484
49 0.872983870967742
50 0.872983870967742
51 0.875
52 0.875
53 0.881048387096774
54 0.883064516129032
55 0.88508064516129
56 0.887096774193548
57 0.889112903225806
58 0.893145161290323
59 0.893145161290323
60 0.897177419354839
61 0.903225806451613
62 0.907258064516129
63 0.913306451612903
64 0.915322580645161
65 0.917338709677419
66 0.917338709677419
67 0.921370967741935
68 0.921370967741935
69 0.925403225806452
70 0.925403225806452
71 0.925403225806452
72 0.92741935483871
73 0.92741935483871
74 0.929435483870968
75 0.929435483870968
76 0.931451612903226
77 0.931451612903226
78 0.933467741935484
79 0.933467741935484
80 0.933467741935484
81 0.933467741935484
82 0.933467741935484
83 0.933467741935484
84 0.933467741935484
85 0.935483870967742
86 0.935483870967742
87 0.935483870967742
88 0.935483870967742
89 0.935483870967742
90 0.935483870967742
91 0.9375
92 0.9375
93 0.9375
94 0.9375
95 0.9375
96 0.9375
97 0.9375
98 0.9375
99 0.9375
100 0.9375
101 0.9375
102 0.9375
103 0.939516129032258
104 0.939516129032258
105 0.939516129032258
106 0.939516129032258
107 0.941532258064516
108 0.941532258064516
109 0.941532258064516
110 0.941532258064516
111 0.941532258064516
112 0.941532258064516
113 0.941532258064516
114 0.941532258064516
115 0.941532258064516
116 0.941532258064516
117 0.943548387096774
118 0.94758064516129
119 0.94758064516129
120 0.94758064516129
121 0.94758064516129
122 0.94758064516129
123 0.94758064516129
124 0.949596774193548
125 0.949596774193548
126 0.955645161290323
127 0.955645161290323
128 0.955645161290323
129 0.955645161290323
130 0.955645161290323
131 0.955645161290323
132 0.957661290322581
133 0.957661290322581
134 0.957661290322581
135 0.959677419354839
136 0.959677419354839
137 0.959677419354839
138 0.959677419354839
139 0.959677419354839
140 0.959677419354839
141 0.959677419354839
142 0.959677419354839
143 0.959677419354839
144 0.959677419354839
145 0.959677419354839
146 0.959677419354839
147 0.959677419354839
148 0.959677419354839
149 0.959677419354839
150 0.959677419354839
151 0.959677419354839
152 0.959677419354839
153 0.959677419354839
154 0.959677419354839
155 0.959677419354839
156 0.959677419354839
157 0.959677419354839
158 0.959677419354839
159 0.959677419354839
160 0.959677419354839
161 0.959677419354839
162 0.959677419354839
163 0.959677419354839
164 0.959677419354839
165 0.959677419354839
166 0.959677419354839
167 0.959677419354839
168 0.959677419354839
169 0.959677419354839
170 0.959677419354839
171 0.959677419354839
172 0.959677419354839
173 0.959677419354839
174 0.959677419354839
175 0.959677419354839
176 0.959677419354839
177 0.959677419354839
178 0.959677419354839
179 0.959677419354839
180 0.959677419354839
181 0.959677419354839
182 0.961693548387097
183 0.961693548387097
184 0.961693548387097
185 0.961693548387097
186 0.961693548387097
187 0.961693548387097
188 0.961693548387097
189 0.961693548387097
190 0.961693548387097
191 0.961693548387097
192 0.961693548387097
193 0.961693548387097
194 0.961693548387097
195 0.961693548387097
196 0.961693548387097
197 0.961693548387097
198 0.961693548387097
199 0.961693548387097
200 0.961693548387097
};
\addlegendentry{CD}
\addplot [very thick, darkorange, dashed]
table {%
0 0
1 0
2 0
3 0
4 0
5 0
6 0
7 0
8 0
9 0
10 0.02
11 0.04
12 0.06
13 0.102
14 0.104
15 0.132
16 0.168
17 0.226
18 0.234
19 0.238
20 0.246
21 0.26
22 0.298
23 0.33
24 0.436
25 0.478
26 0.506
27 0.52
28 0.558
29 0.57
30 0.584
31 0.584
32 0.588
33 0.604
34 0.606
35 0.61
36 0.612
37 0.614
38 0.618
39 0.618
40 0.62
41 0.62
42 0.62
43 0.678
44 0.7
45 0.702
46 0.702
47 0.702
48 0.716
49 0.724
50 0.728
51 0.728
52 0.728
53 0.728
54 0.75
55 0.75
56 0.75
57 0.75
58 0.774
59 0.774
60 0.774
61 0.774
62 0.776
63 0.778
64 0.8
65 0.8
66 0.82
67 0.82
68 0.82
69 0.82
70 0.82
71 0.832
72 0.834
73 0.838
74 0.846
75 0.846
76 0.846
77 0.852
78 0.858
79 0.868
80 0.868
81 0.87
82 0.87
83 0.87
84 0.874
85 0.876
86 0.876
87 0.876
88 0.876
89 0.876
90 0.876
91 0.878
92 0.878
93 0.878
94 0.878
95 0.878
96 0.878
97 0.878
98 0.88
99 0.88
100 0.88
101 0.88
102 0.882
103 0.882
104 0.882
105 0.882
106 0.882
107 0.882
108 0.882
109 0.882
110 0.882
111 0.882
112 0.882
113 0.882
114 0.882
115 0.882
116 0.882
117 0.882
118 0.882
119 0.882
120 0.882
121 0.882
122 0.882
123 0.882
124 0.882
125 0.882
126 0.882
127 0.882
128 0.882
129 0.884
130 0.884
131 0.884
132 0.884
133 0.884
134 0.884
135 0.884
136 0.884
137 0.884
138 0.884
139 0.884
140 0.884
141 0.884
142 0.884
143 0.884
144 0.884
145 0.884
146 0.884
147 0.884
148 0.884
149 0.884
150 0.884
151 0.884
152 0.884
153 0.884
154 0.884
155 0.884
156 0.884
157 0.884
158 0.884
159 0.884
160 0.884
161 0.884
162 0.884
163 0.886
164 0.886
165 0.888
166 0.888
167 0.888
168 0.888
169 0.888
170 0.888
171 0.888
172 0.888
173 0.888
174 0.888
175 0.888
176 0.888
177 0.888
178 0.888
179 0.888
180 0.888
181 0.888
182 0.888
183 0.888
184 0.888
185 0.888
186 0.888
187 0.888
188 0.888
189 0.888
190 0.888
191 0.888
192 0.888
193 0.888
194 0.888
195 0.888
196 0.888
197 0.888
198 0.888
199 0.888
200 0.888
};
\addlegendentry{CASG RBF H}
\end{axis}

\end{tikzpicture}

%% file: figs/Optimization/Opt_sig_1e-05_tau_0.1.tex
\begin{tikzpicture}

\definecolor{darkgray176}{RGB}{176,176,176}
\definecolor{darkorange}{RGB}{255,140,0}
\definecolor{darkslateblue}{RGB}{72,61,139}
\definecolor{lightgray204}{RGB}{204,204,204}
\definecolor{maroon}{RGB}{128,0,0}
\definecolor{palevioletred204121167}{RGB}{204,121,167}
\definecolor{slategray}{RGB}{112,128,144}

\begin{axis}[
legend cell align={left},
legend style={
  fill opacity=0.8,
  draw opacity=1,
  text opacity=1,
  at={(0.97,0.03)},
  anchor=south east,
  draw=lightgray204
},
tick align=outside,
tick pos=left,
x grid style={darkgray176},
xlabel={Number of Simplex Gradients},
xmajorgrids,
xmin=-10, xmax=210,
xminorgrids,
xtick style={color=black},
y grid style={darkgray176},
ylabel={Fraction of Problems},
ymajorgrids,
ymin=-0.0497, ymax=1.0437,
yminorgrids,
ytick style={color=black}
]
\addplot [very thick, darkorange, dashed]
table {%
0 0
1 0
2 0
3 0
4 0
5 0
6 0
7 0
8 0
9 0
10 0.02
11 0.04
12 0.06
13 0.12
14 0.12
15 0.168
16 0.21
17 0.254
18 0.26
19 0.27
20 0.304
21 0.304
22 0.35
23 0.372
24 0.474
25 0.52
26 0.522
27 0.528
28 0.574
29 0.584
30 0.596
31 0.616
32 0.642
33 0.648
34 0.654
35 0.656
36 0.658
37 0.66
38 0.662
39 0.664
40 0.664
41 0.664
42 0.664
43 0.704
44 0.746
45 0.746
46 0.746
47 0.746
48 0.746
49 0.766
50 0.766
51 0.772
52 0.776
53 0.784
54 0.814
55 0.814
56 0.816
57 0.824
58 0.848
59 0.854
60 0.856
61 0.858
62 0.86
63 0.86
64 0.88
65 0.88
66 0.9
67 0.9
68 0.9
69 0.9
70 0.902
71 0.914
72 0.922
73 0.924
74 0.924
75 0.924
76 0.926
77 0.948
78 0.948
79 0.948
80 0.948
81 0.948
82 0.948
83 0.948
84 0.948
85 0.948
86 0.948
87 0.95
88 0.95
89 0.95
90 0.95
91 0.952
92 0.952
93 0.952
94 0.952
95 0.954
96 0.956
97 0.956
98 0.956
99 0.958
100 0.962
101 0.962
102 0.962
103 0.962
104 0.962
105 0.962
106 0.962
107 0.962
108 0.964
109 0.964
110 0.964
111 0.968
112 0.968
113 0.97
114 0.97
115 0.97
116 0.97
117 0.97
118 0.97
119 0.97
120 0.97
121 0.97
122 0.97
123 0.972
124 0.972
125 0.972
126 0.972
127 0.972
128 0.972
129 0.972
130 0.972
131 0.972
132 0.972
133 0.972
134 0.972
135 0.972
136 0.972
137 0.972
138 0.972
139 0.972
140 0.972
141 0.974
142 0.974
143 0.974
144 0.974
145 0.974
146 0.974
147 0.974
148 0.974
149 0.974
150 0.974
151 0.974
152 0.974
153 0.974
154 0.974
155 0.974
156 0.976
157 0.976
158 0.976
159 0.976
160 0.976
161 0.976
162 0.976
163 0.976
164 0.976
165 0.976
166 0.976
167 0.976
168 0.976
169 0.976
170 0.976
171 0.976
172 0.976
173 0.976
174 0.976
175 0.976
176 0.976
177 0.976
178 0.976
179 0.976
180 0.976
181 0.976
182 0.976
183 0.976
184 0.976
185 0.976
186 0.976
187 0.976
188 0.976
189 0.976
190 0.976
191 0.976
192 0.976
193 0.976
194 0.976
195 0.976
196 0.976
197 0.976
198 0.976
199 0.976
200 0.976
};
\addlegendentry{CASG RBF H}
\addplot [very thick, darkslateblue]
table {%
0 0
1 0
2 0
3 0
4 0
5 0.102
6 0.142
7 0.364
8 0.384
9 0.426
10 0.426
11 0.466
12 0.486
13 0.586
14 0.606
15 0.634
16 0.634
17 0.684
18 0.708
19 0.748
20 0.768
21 0.788
22 0.79
23 0.802
24 0.802
25 0.802
26 0.81
27 0.826
28 0.844
29 0.864
30 0.87
31 0.872
32 0.878
33 0.878
34 0.884
35 0.884
36 0.89
37 0.91
38 0.91
39 0.91
40 0.91
41 0.91
42 0.91
43 0.91
44 0.91
45 0.91
46 0.91
47 0.91
48 0.91
49 0.91
50 0.91
51 0.91
52 0.91
53 0.91
54 0.91
55 0.91
56 0.91
57 0.91
58 0.91
59 0.91
60 0.91
61 0.916
62 0.916
63 0.928
64 0.928
65 0.93
66 0.93
67 0.93
68 0.93
69 0.93
70 0.93
71 0.93
72 0.93
73 0.93
74 0.932
75 0.934
76 0.934
77 0.934
78 0.934
79 0.934
80 0.934
81 0.934
82 0.934
83 0.934
84 0.934
85 0.934
86 0.934
87 0.934
88 0.934
89 0.934
90 0.934
91 0.934
92 0.934
93 0.934
94 0.934
95 0.934
96 0.934
97 0.934
98 0.934
99 0.934
100 0.934
101 0.936
102 0.936
103 0.938
104 0.942
105 0.944
106 0.946
107 0.946
108 0.946
109 0.948
110 0.948
111 0.948
112 0.95
113 0.95
114 0.95
115 0.952
116 0.954
117 0.956
118 0.956
119 0.956
120 0.956
121 0.96
122 0.96
123 0.962
124 0.962
125 0.962
126 0.964
127 0.964
128 0.964
129 0.966
130 0.966
131 0.966
132 0.966
133 0.966
134 0.966
135 0.966
136 0.968
137 0.968
138 0.968
139 0.968
140 0.968
141 0.97
142 0.97
143 0.97
144 0.97
145 0.97
146 0.97
147 0.97
148 0.97
149 0.97
150 0.97
151 0.97
152 0.97
153 0.97
154 0.97
155 0.97
156 0.97
157 0.97
158 0.97
159 0.97
160 0.97
161 0.97
162 0.97
163 0.972
164 0.972
165 0.972
166 0.972
167 0.972
168 0.972
169 0.972
170 0.972
171 0.972
172 0.972
173 0.972
174 0.972
175 0.972
176 0.972
177 0.972
178 0.972
179 0.972
180 0.972
181 0.972
182 0.972
183 0.972
184 0.974
185 0.974
186 0.974
187 0.974
188 0.974
189 0.974
190 0.974
191 0.974
192 0.974
193 0.974
194 0.974
195 0.974
196 0.974
197 0.974
198 0.974
199 0.974
200 0.974
};
\addlegendentry{CD}
\addplot [very thick, darkorange]
table {%
0 0
1 0
2 0
3 0.102204408817635
4 0.282565130260521
5 0.322645290581162
6 0.392785571142285
7 0.462925851703407
8 0.503006012024048
9 0.605210420841683
10 0.623246492985972
11 0.705410821643287
12 0.74749498997996
13 0.751503006012024
14 0.791583166332665
15 0.81563126252505
16 0.821643286573146
17 0.829659318637275
18 0.871743486973948
19 0.877755511022044
20 0.88376753507014
21 0.889779559118236
22 0.895791583166333
23 0.899799599198397
24 0.899799599198397
25 0.901803607214429
26 0.92184368737475
27 0.92184368737475
28 0.92184368737475
29 0.92184368737475
30 0.92184368737475
31 0.92184368737475
32 0.92184368737475
33 0.92184368737475
34 0.92184368737475
35 0.925851703406814
36 0.933867735470942
37 0.94188376753507
38 0.945891783567134
39 0.945891783567134
40 0.945891783567134
41 0.947895791583166
42 0.95190380761523
43 0.953907815631262
44 0.957915831663327
45 0.959919839679359
46 0.961923847695391
47 0.961923847695391
48 0.963927855711423
49 0.963927855711423
50 0.965931863727455
51 0.965931863727455
52 0.965931863727455
53 0.965931863727455
54 0.965931863727455
55 0.965931863727455
56 0.965931863727455
57 0.965931863727455
58 0.965931863727455
59 0.967935871743487
60 0.967935871743487
61 0.967935871743487
62 0.967935871743487
63 0.969939879759519
64 0.969939879759519
65 0.969939879759519
66 0.969939879759519
67 0.969939879759519
68 0.969939879759519
69 0.969939879759519
70 0.969939879759519
71 0.969939879759519
72 0.969939879759519
73 0.969939879759519
74 0.969939879759519
75 0.969939879759519
76 0.969939879759519
77 0.969939879759519
78 0.969939879759519
79 0.969939879759519
80 0.969939879759519
81 0.969939879759519
82 0.969939879759519
83 0.969939879759519
84 0.969939879759519
85 0.969939879759519
86 0.969939879759519
87 0.969939879759519
88 0.969939879759519
89 0.969939879759519
90 0.969939879759519
91 0.969939879759519
92 0.969939879759519
93 0.969939879759519
94 0.969939879759519
95 0.969939879759519
96 0.971943887775551
97 0.971943887775551
98 0.971943887775551
99 0.973947895791583
100 0.973947895791583
101 0.975951903807615
102 0.977955911823647
103 0.977955911823647
104 0.979959919839679
105 0.979959919839679
106 0.979959919839679
107 0.981963927855711
108 0.983967935871743
109 0.983967935871743
110 0.983967935871743
111 0.983967935871743
112 0.983967935871743
113 0.983967935871743
114 0.987975951903808
115 0.987975951903808
116 0.987975951903808
117 0.987975951903808
118 0.987975951903808
119 0.98997995991984
120 0.98997995991984
121 0.98997995991984
122 0.98997995991984
123 0.98997995991984
124 0.98997995991984
125 0.98997995991984
126 0.98997995991984
127 0.98997995991984
128 0.98997995991984
129 0.98997995991984
130 0.98997995991984
131 0.98997995991984
132 0.98997995991984
133 0.98997995991984
134 0.98997995991984
135 0.98997995991984
136 0.98997995991984
137 0.98997995991984
138 0.98997995991984
139 0.98997995991984
140 0.98997995991984
141 0.98997995991984
142 0.98997995991984
143 0.98997995991984
144 0.98997995991984
145 0.98997995991984
146 0.98997995991984
147 0.98997995991984
148 0.98997995991984
149 0.98997995991984
150 0.98997995991984
151 0.98997995991984
152 0.98997995991984
153 0.98997995991984
154 0.98997995991984
155 0.98997995991984
156 0.98997995991984
157 0.98997995991984
158 0.98997995991984
159 0.98997995991984
160 0.98997995991984
161 0.98997995991984
162 0.98997995991984
163 0.98997995991984
164 0.98997995991984
165 0.98997995991984
166 0.98997995991984
167 0.98997995991984
168 0.98997995991984
169 0.98997995991984
170 0.98997995991984
171 0.98997995991984
172 0.98997995991984
173 0.98997995991984
174 0.98997995991984
175 0.98997995991984
176 0.98997995991984
177 0.98997995991984
178 0.98997995991984
179 0.98997995991984
180 0.98997995991984
181 0.98997995991984
182 0.98997995991984
183 0.98997995991984
184 0.98997995991984
185 0.98997995991984
186 0.98997995991984
187 0.98997995991984
188 0.98997995991984
189 0.98997995991984
190 0.98997995991984
191 0.98997995991984
192 0.98997995991984
193 0.98997995991984
194 0.98997995991984
195 0.98997995991984
196 0.98997995991984
197 0.98997995991984
198 0.98997995991984
199 0.98997995991984
200 0.98997995991984
};
\addlegendentry{CASG exact H}
\addplot [very thick, maroon]
table {%
0 0
1 0
2 0
3 0
4 0
5 0
6 0
7 0
8 0
9 0.00618556701030928
10 0.00618556701030928
11 0.00824742268041237
12 0.0474226804123711
13 0.0556701030927835
14 0.0597938144329897
15 0.0639175257731959
16 0.0639175257731959
17 0.0701030927835052
18 0.0824742268041237
19 0.088659793814433
20 0.0948453608247423
21 0.101030927835052
22 0.103092783505155
23 0.117525773195876
24 0.127835051546392
25 0.162886597938144
26 0.181443298969072
27 0.183505154639175
28 0.189690721649485
29 0.193814432989691
30 0.193814432989691
31 0.193814432989691
32 0.195876288659794
33 0.2
34 0.2
35 0.2
36 0.202061855670103
37 0.202061855670103
38 0.204123711340206
39 0.210309278350515
40 0.212371134020619
41 0.214432989690722
42 0.222680412371134
43 0.243298969072165
44 0.251546391752577
45 0.274226804123711
46 0.274226804123711
47 0.280412371134021
48 0.288659793814433
49 0.298969072164948
50 0.303092783505155
51 0.307216494845361
52 0.321649484536082
53 0.325773195876289
54 0.331958762886598
55 0.344329896907216
56 0.34639175257732
57 0.356701030927835
58 0.358762886597938
59 0.360824742268041
60 0.364948453608247
61 0.369072164948454
62 0.379381443298969
63 0.383505154639175
64 0.383505154639175
65 0.4
66 0.414432989690722
67 0.418556701030928
68 0.42680412371134
69 0.437113402061856
70 0.447422680412371
71 0.457731958762887
72 0.468041237113402
73 0.482474226804124
74 0.488659793814433
75 0.505154639175258
76 0.51340206185567
77 0.515463917525773
78 0.517525773195876
79 0.519587628865979
80 0.523711340206186
81 0.523711340206186
82 0.525773195876289
83 0.525773195876289
84 0.529896907216495
85 0.531958762886598
86 0.534020618556701
87 0.538144329896907
88 0.538144329896907
89 0.542268041237113
90 0.544329896907216
91 0.550515463917526
92 0.550515463917526
93 0.556701030927835
94 0.560824742268041
95 0.560824742268041
96 0.569072164948454
97 0.569072164948454
98 0.571134020618557
99 0.571134020618557
100 0.575257731958763
101 0.575257731958763
102 0.575257731958763
103 0.575257731958763
104 0.577319587628866
105 0.577319587628866
106 0.577319587628866
107 0.581443298969072
108 0.581443298969072
109 0.581443298969072
110 0.583505154639175
111 0.585567010309278
112 0.585567010309278
113 0.585567010309278
114 0.585567010309278
115 0.585567010309278
116 0.587628865979381
117 0.589690721649485
118 0.591752577319588
119 0.595876288659794
120 0.597938144329897
121 0.6
122 0.6
123 0.6
124 0.602061855670103
125 0.604123711340206
126 0.604123711340206
127 0.606185567010309
128 0.608247422680412
129 0.608247422680412
130 0.608247422680412
131 0.608247422680412
132 0.608247422680412
133 0.608247422680412
134 0.608247422680412
135 0.608247422680412
136 0.610309278350516
137 0.610309278350516
138 0.610309278350516
139 0.610309278350516
140 0.610309278350516
141 0.612371134020619
142 0.612371134020619
143 0.614432989690722
144 0.614432989690722
145 0.614432989690722
146 0.618556701030928
147 0.620618556701031
148 0.622680412371134
149 0.622680412371134
150 0.624742268041237
151 0.624742268041237
152 0.624742268041237
153 0.624742268041237
154 0.624742268041237
155 0.624742268041237
156 0.624742268041237
157 0.624742268041237
158 0.62680412371134
159 0.62680412371134
160 0.62680412371134
161 0.62680412371134
162 0.62680412371134
163 0.62680412371134
164 0.62680412371134
165 0.62680412371134
166 0.62680412371134
167 0.62680412371134
168 0.628865979381443
169 0.628865979381443
170 0.628865979381443
171 0.630927835051546
172 0.630927835051546
173 0.630927835051546
174 0.630927835051546
175 0.632989690721649
176 0.632989690721649
177 0.635051546391753
178 0.635051546391753
179 0.635051546391753
180 0.635051546391753
181 0.635051546391753
182 0.635051546391753
183 0.635051546391753
184 0.637113402061856
185 0.637113402061856
186 0.637113402061856
187 0.637113402061856
188 0.639175257731959
189 0.639175257731959
190 0.639175257731959
191 0.639175257731959
192 0.639175257731959
193 0.639175257731959
194 0.641237113402062
195 0.641237113402062
196 0.641237113402062
197 0.641237113402062
198 0.641237113402062
199 0.641237113402062
200 0.641237113402062
};
\addlegendentry{RBF}
\addplot [very thick, palevioletred204121167]
table {%
0 0
1 0.002
2 0.002
3 0.002
4 0.002
5 0.002
6 0.002
7 0.002
8 0.006
9 0.02
10 0.058
11 0.112
12 0.176
13 0.228
14 0.254
15 0.26
16 0.284
17 0.314
18 0.358
19 0.374
20 0.38
21 0.388
22 0.406
23 0.418
24 0.44
25 0.476
26 0.496
27 0.516
28 0.542
29 0.556
30 0.568
31 0.576
32 0.586
33 0.588
34 0.594
35 0.604
36 0.612
37 0.62
38 0.63
39 0.638
40 0.648
41 0.656
42 0.658
43 0.664
44 0.672
45 0.674
46 0.676
47 0.682
48 0.692
49 0.7
50 0.708
51 0.724
52 0.726
53 0.732
54 0.734
55 0.738
56 0.744
57 0.746
58 0.746
59 0.748
60 0.75
61 0.752
62 0.756
63 0.758
64 0.758
65 0.764
66 0.766
67 0.776
68 0.784
69 0.79
70 0.79
71 0.794
72 0.802
73 0.804
74 0.806
75 0.808
76 0.808
77 0.814
78 0.816
79 0.818
80 0.818
81 0.818
82 0.82
83 0.82
84 0.82
85 0.82
86 0.822
87 0.822
88 0.824
89 0.824
90 0.83
91 0.83
92 0.83
93 0.83
94 0.83
95 0.83
96 0.83
97 0.83
98 0.83
99 0.83
100 0.83
101 0.83
102 0.83
103 0.832
104 0.834
105 0.834
106 0.834
107 0.836
108 0.838
109 0.838
110 0.838
111 0.838
112 0.838
113 0.84
114 0.842
115 0.844
116 0.844
117 0.844
118 0.844
119 0.844
120 0.844
121 0.844
122 0.844
123 0.844
124 0.844
125 0.844
126 0.844
127 0.844
128 0.844
129 0.846
130 0.846
131 0.846
132 0.846
133 0.846
134 0.846
135 0.846
136 0.846
137 0.846
138 0.846
139 0.846
140 0.848
141 0.848
142 0.848
143 0.848
144 0.848
145 0.848
146 0.85
147 0.852
148 0.852
149 0.852
150 0.852
151 0.852
152 0.852
153 0.852
154 0.852
155 0.852
156 0.854
157 0.854
158 0.854
159 0.856
160 0.858
161 0.86
162 0.862
163 0.864
164 0.864
165 0.864
166 0.864
167 0.864
168 0.866
169 0.866
170 0.866
171 0.868
172 0.874
173 0.876
174 0.876
175 0.876
176 0.876
177 0.878
178 0.88
179 0.88
180 0.888
181 0.888
182 0.888
183 0.89
184 0.89
185 0.89
186 0.892
187 0.894
188 0.894
189 0.894
190 0.894
191 0.894
192 0.894
193 0.896
194 0.896
195 0.896
196 0.898
197 0.898
198 0.898
199 0.898
200 0.898
};
\addlegendentry{AFD}
\addplot [very thick, slategray]
table {%
0 0
1 0
2 0
3 0.1
4 0.282
5 0.322
6 0.392
7 0.458
8 0.5
9 0.592
10 0.634
11 0.722
12 0.75
13 0.752
14 0.788
15 0.806
16 0.816
17 0.824
18 0.856
19 0.864
20 0.876
21 0.886
22 0.89
23 0.898
24 0.9
25 0.912
26 0.914
27 0.914
28 0.92
29 0.92
30 0.92
31 0.922
32 0.924
33 0.924
34 0.924
35 0.926
36 0.93
37 0.936
38 0.942
39 0.946
40 0.946
41 0.95
42 0.952
43 0.958
44 0.962
45 0.962
46 0.962
47 0.964
48 0.966
49 0.966
50 0.966
51 0.966
52 0.966
53 0.966
54 0.966
55 0.966
56 0.966
57 0.966
58 0.966
59 0.966
60 0.966
61 0.966
62 0.966
63 0.966
64 0.966
65 0.968
66 0.97
67 0.97
68 0.97
69 0.97
70 0.97
71 0.97
72 0.97
73 0.97
74 0.97
75 0.97
76 0.97
77 0.97
78 0.97
79 0.97
80 0.97
81 0.97
82 0.97
83 0.972
84 0.974
85 0.974
86 0.974
87 0.974
88 0.976
89 0.976
90 0.978
91 0.978
92 0.978
93 0.98
94 0.982
95 0.982
96 0.984
97 0.984
98 0.984
99 0.986
100 0.986
101 0.988
102 0.988
103 0.988
104 0.988
105 0.988
106 0.988
107 0.99
108 0.99
109 0.99
110 0.99
111 0.99
112 0.99
113 0.99
114 0.99
115 0.99
116 0.99
117 0.99
118 0.99
119 0.99
120 0.99
121 0.99
122 0.99
123 0.99
124 0.99
125 0.99
126 0.99
127 0.99
128 0.992
129 0.992
130 0.992
131 0.992
132 0.992
133 0.992
134 0.992
135 0.994
136 0.994
137 0.994
138 0.994
139 0.994
140 0.994
141 0.994
142 0.994
143 0.994
144 0.994
145 0.994
146 0.994
147 0.994
148 0.994
149 0.994
150 0.994
151 0.994
152 0.994
153 0.994
154 0.994
155 0.994
156 0.994
157 0.994
158 0.994
159 0.994
160 0.994
161 0.994
162 0.994
163 0.994
164 0.994
165 0.994
166 0.994
167 0.994
168 0.994
169 0.994
170 0.994
171 0.994
172 0.994
173 0.994
174 0.994
175 0.994
176 0.994
177 0.994
178 0.994
179 0.994
180 0.994
181 0.994
182 0.994
183 0.994
184 0.994
185 0.994
186 0.994
187 0.994
188 0.994
189 0.994
190 0.994
191 0.994
192 0.994
193 0.994
194 0.994
195 0.994
196 0.994
197 0.994
198 0.994
199 0.994
200 0.994
};
\addlegendentry{FD exact H}
\addplot [very thick, slategray, dashed]
table {%
0 0
1 0
2 0
3 0
4 0
5 0
6 0
7 0
8 0
9 0
10 0.0200400801603206
11 0.0400801603206413
12 0.0601202404809619
13 0.120240480961924
14 0.128256513026052
15 0.18436873747495
16 0.218436873747495
17 0.254509018036072
18 0.258517034068136
19 0.258517034068136
20 0.296593186372746
21 0.298597194388778
22 0.342685370741483
23 0.364729458917836
24 0.466933867735471
25 0.511022044088176
26 0.513026052104208
27 0.527054108216433
28 0.567134268537074
29 0.585170340681363
30 0.595190380761523
31 0.61122244488978
32 0.627254509018036
33 0.635270541082164
34 0.637274549098196
35 0.64128256513026
36 0.64128256513026
37 0.643286573146293
38 0.645290581162325
39 0.649298597194389
40 0.649298597194389
41 0.649298597194389
42 0.651302605210421
43 0.675350701402806
44 0.719438877755511
45 0.721442885771543
46 0.721442885771543
47 0.725450901803607
48 0.725450901803607
49 0.745490981963928
50 0.74749498997996
51 0.755511022044088
52 0.75751503006012
53 0.767535070140281
54 0.793587174348697
55 0.797595190380762
56 0.801603206412826
57 0.80561122244489
58 0.833667334669339
59 0.835671342685371
60 0.837675350701403
61 0.843687374749499
62 0.843687374749499
63 0.843687374749499
64 0.865731462925852
65 0.865731462925852
66 0.885771543086172
67 0.885771543086172
68 0.885771543086172
69 0.885771543086172
70 0.885771543086172
71 0.889779559118236
72 0.905811623246493
73 0.905811623246493
74 0.905811623246493
75 0.905811623246493
76 0.907815631262525
77 0.919839679358717
78 0.925851703406814
79 0.925851703406814
80 0.925851703406814
81 0.927855711422846
82 0.927855711422846
83 0.927855711422846
84 0.927855711422846
85 0.927855711422846
86 0.927855711422846
87 0.927855711422846
88 0.927855711422846
89 0.929859719438878
90 0.929859719438878
91 0.929859719438878
92 0.929859719438878
93 0.929859719438878
94 0.929859719438878
95 0.93186372745491
96 0.935871743486974
97 0.937875751503006
98 0.937875751503006
99 0.937875751503006
100 0.937875751503006
101 0.937875751503006
102 0.937875751503006
103 0.937875751503006
104 0.939879759519038
105 0.939879759519038
106 0.939879759519038
107 0.947895791583166
108 0.947895791583166
109 0.947895791583166
110 0.947895791583166
111 0.95190380761523
112 0.95190380761523
113 0.95190380761523
114 0.95190380761523
115 0.95190380761523
116 0.95190380761523
117 0.95190380761523
118 0.95190380761523
119 0.95190380761523
120 0.95190380761523
121 0.95190380761523
122 0.95190380761523
123 0.95190380761523
124 0.95190380761523
125 0.95190380761523
126 0.95190380761523
127 0.95190380761523
128 0.95190380761523
129 0.95190380761523
130 0.95190380761523
131 0.95190380761523
132 0.95190380761523
133 0.95190380761523
134 0.95190380761523
135 0.95190380761523
136 0.95190380761523
137 0.95190380761523
138 0.95190380761523
139 0.95190380761523
140 0.95190380761523
141 0.95190380761523
142 0.95190380761523
143 0.953907815631262
144 0.953907815631262
145 0.953907815631262
146 0.953907815631262
147 0.953907815631262
148 0.953907815631262
149 0.953907815631262
150 0.953907815631262
151 0.953907815631262
152 0.953907815631262
153 0.953907815631262
154 0.953907815631262
155 0.953907815631262
156 0.953907815631262
157 0.953907815631262
158 0.953907815631262
159 0.955911823647295
160 0.955911823647295
161 0.955911823647295
162 0.957915831663327
163 0.957915831663327
164 0.957915831663327
165 0.957915831663327
166 0.957915831663327
167 0.957915831663327
168 0.957915831663327
169 0.957915831663327
170 0.957915831663327
171 0.957915831663327
172 0.957915831663327
173 0.957915831663327
174 0.957915831663327
175 0.957915831663327
176 0.957915831663327
177 0.957915831663327
178 0.957915831663327
179 0.957915831663327
180 0.957915831663327
181 0.957915831663327
182 0.957915831663327
183 0.957915831663327
184 0.957915831663327
185 0.957915831663327
186 0.957915831663327
187 0.957915831663327
188 0.957915831663327
189 0.957915831663327
190 0.957915831663327
191 0.957915831663327
192 0.957915831663327
193 0.957915831663327
194 0.957915831663327
195 0.957915831663327
196 0.957915831663327
197 0.957915831663327
198 0.957915831663327
199 0.957915831663327
200 0.957915831663327
};
\addlegendentry{FD RBF H}
\end{axis}

\end{tikzpicture}

%% file: figs/Optimization/Opt_sig_0.1_tau_1e-05.tex
\begin{tikzpicture}

\definecolor{darkgray176}{RGB}{176,176,176}
\definecolor{darkorange}{RGB}{255,140,0}
\definecolor{darkslateblue}{RGB}{72,61,139}
\definecolor{lightgray204}{RGB}{204,204,204}
\definecolor{maroon}{RGB}{128,0,0}
\definecolor{palevioletred204121167}{RGB}{204,121,167}
\definecolor{slategray}{RGB}{112,128,144}

\begin{axis}[
legend cell align={left},
legend style={
  fill opacity=0.8,
  draw opacity=1,
  text opacity=1,
  at={(0.03,0.97)},
  anchor=north west,
  draw=lightgray204
},
tick align=outside,
tick pos=left,
x grid style={darkgray176},
xlabel={Number of Simplex Gradients},
xmajorgrids,
xmin=-10, xmax=210,
xminorgrids,
xtick style={color=black},
y grid style={darkgray176},
ylabel={Fraction of Problems},
ymajorgrids,
ymin=-0.0267068273092369, ymax=0.560843373493976,
yminorgrids,
ytick style={color=black}
]
\addplot [very thick, slategray]
table {%
0 0
1 0.00401606425702811
2 0.00401606425702811
3 0.00602409638554217
4 0.00803212851405622
5 0.0140562248995984
6 0.0180722891566265
7 0.0261044176706827
8 0.0321285140562249
9 0.0381526104417671
10 0.0381526104417671
11 0.0642570281124498
12 0.0803212851405622
13 0.100401606425703
14 0.122489959839357
15 0.1285140562249
16 0.136546184738956
17 0.14859437751004
18 0.168674698795181
19 0.174698795180723
20 0.178714859437751
21 0.184738955823293
22 0.184738955823293
23 0.186746987951807
24 0.194779116465863
25 0.200803212851406
26 0.216867469879518
27 0.220883534136546
28 0.230923694779116
29 0.234939759036145
30 0.238955823293173
31 0.259036144578313
32 0.261044176706827
33 0.263052208835341
34 0.263052208835341
35 0.269076305220884
36 0.269076305220884
37 0.271084337349398
38 0.271084337349398
39 0.273092369477912
40 0.275100401606426
41 0.275100401606426
42 0.275100401606426
43 0.275100401606426
44 0.27710843373494
45 0.279116465863454
46 0.283132530120482
47 0.283132530120482
48 0.283132530120482
49 0.285140562248996
50 0.285140562248996
51 0.285140562248996
52 0.285140562248996
53 0.28714859437751
54 0.28714859437751
55 0.289156626506024
56 0.289156626506024
57 0.289156626506024
58 0.289156626506024
59 0.289156626506024
60 0.289156626506024
61 0.289156626506024
62 0.289156626506024
63 0.289156626506024
64 0.289156626506024
65 0.289156626506024
66 0.289156626506024
67 0.289156626506024
68 0.289156626506024
69 0.289156626506024
70 0.289156626506024
71 0.289156626506024
72 0.291164658634538
73 0.291164658634538
74 0.291164658634538
75 0.291164658634538
76 0.291164658634538
77 0.291164658634538
78 0.291164658634538
79 0.291164658634538
80 0.291164658634538
81 0.293172690763052
82 0.293172690763052
83 0.293172690763052
84 0.293172690763052
85 0.293172690763052
86 0.293172690763052
87 0.293172690763052
88 0.293172690763052
89 0.293172690763052
90 0.293172690763052
91 0.293172690763052
92 0.293172690763052
93 0.293172690763052
94 0.293172690763052
95 0.293172690763052
96 0.293172690763052
97 0.293172690763052
98 0.293172690763052
99 0.293172690763052
100 0.293172690763052
101 0.293172690763052
102 0.293172690763052
103 0.293172690763052
104 0.293172690763052
105 0.293172690763052
106 0.293172690763052
107 0.293172690763052
108 0.293172690763052
109 0.293172690763052
110 0.293172690763052
111 0.293172690763052
112 0.293172690763052
113 0.293172690763052
114 0.293172690763052
115 0.293172690763052
116 0.293172690763052
117 0.293172690763052
118 0.293172690763052
119 0.293172690763052
120 0.293172690763052
121 0.293172690763052
122 0.293172690763052
123 0.293172690763052
124 0.293172690763052
125 0.293172690763052
126 0.293172690763052
127 0.293172690763052
128 0.293172690763052
129 0.293172690763052
130 0.293172690763052
131 0.293172690763052
132 0.293172690763052
133 0.293172690763052
134 0.293172690763052
135 0.293172690763052
136 0.293172690763052
137 0.293172690763052
138 0.293172690763052
139 0.293172690763052
140 0.293172690763052
141 0.293172690763052
142 0.293172690763052
143 0.293172690763052
144 0.293172690763052
145 0.293172690763052
146 0.293172690763052
147 0.293172690763052
148 0.293172690763052
149 0.293172690763052
150 0.293172690763052
151 0.293172690763052
152 0.293172690763052
153 0.293172690763052
154 0.293172690763052
155 0.293172690763052
156 0.293172690763052
157 0.293172690763052
158 0.293172690763052
159 0.295180722891566
160 0.295180722891566
161 0.295180722891566
162 0.295180722891566
163 0.295180722891566
164 0.295180722891566
165 0.295180722891566
166 0.295180722891566
167 0.295180722891566
168 0.295180722891566
169 0.295180722891566
170 0.295180722891566
171 0.295180722891566
172 0.295180722891566
173 0.295180722891566
174 0.295180722891566
175 0.295180722891566
176 0.295180722891566
177 0.295180722891566
178 0.295180722891566
179 0.295180722891566
180 0.295180722891566
181 0.295180722891566
182 0.295180722891566
183 0.295180722891566
184 0.295180722891566
185 0.295180722891566
186 0.295180722891566
187 0.295180722891566
188 0.295180722891566
189 0.295180722891566
190 0.295180722891566
191 0.295180722891566
192 0.295180722891566
193 0.295180722891566
194 0.295180722891566
195 0.295180722891566
196 0.295180722891566
197 0.295180722891566
198 0.295180722891566
199 0.295180722891566
200 0.295180722891566
};
\addlegendentry{FD exact H}
\addplot [very thick, slategray, dashed]
table {%
0 0
1 0
2 0
3 0
4 0
5 0
6 0
7 0
8 0
9 0
10 0
11 0.002
12 0.002
13 0.004
14 0.006
15 0.012
16 0.012
17 0.014
18 0.014
19 0.016
20 0.016
21 0.018
22 0.036
23 0.042
24 0.044
25 0.048
26 0.054
27 0.056
28 0.062
29 0.064
30 0.072
31 0.074
32 0.086
33 0.09
34 0.114
35 0.118
36 0.126
37 0.144
38 0.146
39 0.148
40 0.148
41 0.148
42 0.15
43 0.152
44 0.152
45 0.156
46 0.158
47 0.158
48 0.158
49 0.158
50 0.158
51 0.18
52 0.184
53 0.188
54 0.19
55 0.19
56 0.192
57 0.192
58 0.192
59 0.192
60 0.194
61 0.194
62 0.202
63 0.214
64 0.214
65 0.216
66 0.216
67 0.218
68 0.218
69 0.22
70 0.222
71 0.226
72 0.228
73 0.23
74 0.232
75 0.234
76 0.234
77 0.236
78 0.236
79 0.236
80 0.236
81 0.236
82 0.236
83 0.236
84 0.236
85 0.236
86 0.236
87 0.236
88 0.238
89 0.238
90 0.238
91 0.238
92 0.238
93 0.238
94 0.238
95 0.238
96 0.238
97 0.238
98 0.238
99 0.238
100 0.238
101 0.238
102 0.238
103 0.238
104 0.238
105 0.238
106 0.238
107 0.238
108 0.238
109 0.238
110 0.238
111 0.238
112 0.238
113 0.238
114 0.238
115 0.238
116 0.238
117 0.238
118 0.238
119 0.238
120 0.238
121 0.238
122 0.238
123 0.238
124 0.238
125 0.238
126 0.238
127 0.238
128 0.238
129 0.238
130 0.238
131 0.238
132 0.238
133 0.238
134 0.238
135 0.238
136 0.238
137 0.238
138 0.238
139 0.238
140 0.238
141 0.238
142 0.238
143 0.238
144 0.238
145 0.238
146 0.238
147 0.238
148 0.238
149 0.238
150 0.238
151 0.238
152 0.238
153 0.238
154 0.238
155 0.238
156 0.238
157 0.238
158 0.238
159 0.238
160 0.238
161 0.238
162 0.238
163 0.238
164 0.238
165 0.238
166 0.238
167 0.238
168 0.238
169 0.238
170 0.238
171 0.238
172 0.238
173 0.238
174 0.238
175 0.238
176 0.238
177 0.238
178 0.238
179 0.238
180 0.238
181 0.238
182 0.238
183 0.238
184 0.238
185 0.238
186 0.238
187 0.238
188 0.238
189 0.238
190 0.238
191 0.238
192 0.238
193 0.238
194 0.238
195 0.238
196 0.238
197 0.238
198 0.238
199 0.238
200 0.238
};
\addlegendentry{FD RBF H}
\addplot [very thick, maroon]
table {%
0 0
1 0
2 0
3 0
4 0
5 0
6 0
7 0
8 0
9 0.00204918032786885
10 0.00409836065573771
11 0.00409836065573771
12 0.00409836065573771
13 0.00614754098360656
14 0.00614754098360656
15 0.00614754098360656
16 0.00614754098360656
17 0.00614754098360656
18 0.00819672131147541
19 0.00819672131147541
20 0.00819672131147541
21 0.0102459016393443
22 0.0102459016393443
23 0.0102459016393443
24 0.0163934426229508
25 0.0163934426229508
26 0.0204918032786885
27 0.0245901639344262
28 0.0307377049180328
29 0.0327868852459016
30 0.0389344262295082
31 0.0389344262295082
32 0.0389344262295082
33 0.0389344262295082
34 0.0430327868852459
35 0.0471311475409836
36 0.0532786885245902
37 0.0573770491803279
38 0.0594262295081967
39 0.0594262295081967
40 0.0614754098360656
41 0.0635245901639344
42 0.0635245901639344
43 0.0655737704918033
44 0.0655737704918033
45 0.069672131147541
46 0.0717213114754098
47 0.0737704918032787
48 0.0778688524590164
49 0.0819672131147541
50 0.0819672131147541
51 0.0819672131147541
52 0.0840163934426229
53 0.0860655737704918
54 0.0881147540983607
55 0.0922131147540984
56 0.0922131147540984
57 0.0942622950819672
58 0.0942622950819672
59 0.0963114754098361
60 0.0983606557377049
61 0.0983606557377049
62 0.100409836065574
63 0.100409836065574
64 0.100409836065574
65 0.102459016393443
66 0.102459016393443
67 0.102459016393443
68 0.10655737704918
69 0.10655737704918
70 0.10655737704918
71 0.110655737704918
72 0.110655737704918
73 0.110655737704918
74 0.112704918032787
75 0.112704918032787
76 0.112704918032787
77 0.114754098360656
78 0.114754098360656
79 0.114754098360656
80 0.114754098360656
81 0.114754098360656
82 0.114754098360656
83 0.114754098360656
84 0.116803278688525
85 0.116803278688525
86 0.116803278688525
87 0.116803278688525
88 0.116803278688525
89 0.120901639344262
90 0.120901639344262
91 0.122950819672131
92 0.125
93 0.127049180327869
94 0.127049180327869
95 0.127049180327869
96 0.127049180327869
97 0.127049180327869
98 0.127049180327869
99 0.129098360655738
100 0.129098360655738
101 0.131147540983607
102 0.131147540983607
103 0.131147540983607
104 0.131147540983607
105 0.131147540983607
106 0.131147540983607
107 0.131147540983607
108 0.131147540983607
109 0.131147540983607
110 0.131147540983607
111 0.133196721311475
112 0.133196721311475
113 0.133196721311475
114 0.133196721311475
115 0.135245901639344
116 0.135245901639344
117 0.135245901639344
118 0.135245901639344
119 0.135245901639344
120 0.135245901639344
121 0.137295081967213
122 0.137295081967213
123 0.137295081967213
124 0.137295081967213
125 0.137295081967213
126 0.137295081967213
127 0.137295081967213
128 0.137295081967213
129 0.137295081967213
130 0.137295081967213
131 0.137295081967213
132 0.137295081967213
133 0.137295081967213
134 0.137295081967213
135 0.139344262295082
136 0.139344262295082
137 0.139344262295082
138 0.139344262295082
139 0.139344262295082
140 0.139344262295082
141 0.139344262295082
142 0.139344262295082
143 0.139344262295082
144 0.139344262295082
145 0.139344262295082
146 0.139344262295082
147 0.139344262295082
148 0.139344262295082
149 0.139344262295082
150 0.139344262295082
151 0.139344262295082
152 0.139344262295082
153 0.139344262295082
154 0.139344262295082
155 0.139344262295082
156 0.139344262295082
157 0.139344262295082
158 0.139344262295082
159 0.139344262295082
160 0.139344262295082
161 0.139344262295082
162 0.139344262295082
163 0.139344262295082
164 0.139344262295082
165 0.139344262295082
166 0.139344262295082
167 0.139344262295082
168 0.139344262295082
169 0.139344262295082
170 0.139344262295082
171 0.139344262295082
172 0.139344262295082
173 0.139344262295082
174 0.141393442622951
175 0.141393442622951
176 0.14344262295082
177 0.14344262295082
178 0.14344262295082
179 0.14344262295082
180 0.14344262295082
181 0.14344262295082
182 0.14344262295082
183 0.14344262295082
184 0.14344262295082
185 0.14344262295082
186 0.14344262295082
187 0.14344262295082
188 0.14344262295082
189 0.14344262295082
190 0.14344262295082
191 0.14344262295082
192 0.14344262295082
193 0.14344262295082
194 0.14344262295082
195 0.14344262295082
196 0.14344262295082
197 0.14344262295082
198 0.14344262295082
199 0.14344262295082
200 0.14344262295082
};
\addlegendentry{RBF}
\addplot [very thick, darkorange, dashed]
table {%
0 0
1 0
2 0
3 0
4 0
5 0
6 0
7 0
8 0
9 0
10 0.002
11 0.002
12 0.002
13 0.004
14 0.004
15 0.004
16 0.012
17 0.014
18 0.014
19 0.016
20 0.02
21 0.024
22 0.046
23 0.05
24 0.058
25 0.064
26 0.064
27 0.066
28 0.068
29 0.074
30 0.076
31 0.076
32 0.094
33 0.1
34 0.126
35 0.13
36 0.14
37 0.162
38 0.164
39 0.168
40 0.174
41 0.176
42 0.178
43 0.184
44 0.186
45 0.188
46 0.192
47 0.194
48 0.194
49 0.194
50 0.194
51 0.214
52 0.214
53 0.218
54 0.22
55 0.22
56 0.224
57 0.224
58 0.226
59 0.228
60 0.228
61 0.23
62 0.238
63 0.252
64 0.252
65 0.252
66 0.252
67 0.252
68 0.252
69 0.252
70 0.256
71 0.256
72 0.256
73 0.256
74 0.258
75 0.26
76 0.26
77 0.26
78 0.262
79 0.266
80 0.268
81 0.268
82 0.268
83 0.268
84 0.268
85 0.268
86 0.27
87 0.27
88 0.272
89 0.272
90 0.272
91 0.274
92 0.274
93 0.274
94 0.274
95 0.274
96 0.274
97 0.274
98 0.274
99 0.274
100 0.274
101 0.274
102 0.274
103 0.274
104 0.274
105 0.274
106 0.274
107 0.274
108 0.274
109 0.274
110 0.274
111 0.274
112 0.274
113 0.274
114 0.274
115 0.274
116 0.274
117 0.274
118 0.274
119 0.274
120 0.274
121 0.274
122 0.274
123 0.274
124 0.274
125 0.274
126 0.274
127 0.274
128 0.274
129 0.274
130 0.274
131 0.274
132 0.274
133 0.274
134 0.274
135 0.274
136 0.274
137 0.274
138 0.274
139 0.274
140 0.274
141 0.274
142 0.274
143 0.274
144 0.274
145 0.274
146 0.274
147 0.274
148 0.274
149 0.274
150 0.274
151 0.274
152 0.274
153 0.274
154 0.274
155 0.274
156 0.274
157 0.274
158 0.274
159 0.274
160 0.274
161 0.274
162 0.274
163 0.274
164 0.274
165 0.274
166 0.274
167 0.274
168 0.274
169 0.274
170 0.274
171 0.274
172 0.274
173 0.274
174 0.274
175 0.274
176 0.274
177 0.274
178 0.274
179 0.274
180 0.274
181 0.274
182 0.274
183 0.274
184 0.274
185 0.274
186 0.274
187 0.274
188 0.274
189 0.274
190 0.274
191 0.274
192 0.274
193 0.274
194 0.274
195 0.274
196 0.274
197 0.274
198 0.274
199 0.274
200 0.274
};
\addlegendentry{CASG RBF H}
\addplot [very thick, darkorange]
table {%
0 0
1 0
2 0
3 0.004
4 0.004
5 0.014
6 0.02
7 0.02
8 0.026
9 0.034
10 0.038
11 0.066
12 0.084
13 0.1
14 0.136
15 0.154
16 0.166
17 0.172
18 0.202
19 0.208
20 0.21
21 0.224
22 0.234
23 0.242
24 0.25
25 0.264
26 0.28
27 0.28
28 0.286
29 0.292
30 0.292
31 0.318
32 0.32
33 0.322
34 0.326
35 0.328
36 0.332
37 0.334
38 0.336
39 0.338
40 0.338
41 0.342
42 0.342
43 0.342
44 0.342
45 0.346
46 0.346
47 0.346
48 0.346
49 0.346
50 0.346
51 0.346
52 0.346
53 0.346
54 0.346
55 0.346
56 0.346
57 0.346
58 0.346
59 0.346
60 0.346
61 0.346
62 0.346
63 0.346
64 0.346
65 0.346
66 0.346
67 0.346
68 0.346
69 0.346
70 0.346
71 0.346
72 0.346
73 0.346
74 0.346
75 0.346
76 0.346
77 0.346
78 0.346
79 0.346
80 0.346
81 0.346
82 0.346
83 0.346
84 0.346
85 0.346
86 0.346
87 0.346
88 0.346
89 0.346
90 0.346
91 0.346
92 0.346
93 0.346
94 0.346
95 0.346
96 0.346
97 0.346
98 0.346
99 0.346
100 0.346
101 0.346
102 0.346
103 0.346
104 0.346
105 0.346
106 0.346
107 0.346
108 0.346
109 0.346
110 0.346
111 0.346
112 0.346
113 0.346
114 0.346
115 0.346
116 0.346
117 0.346
118 0.346
119 0.346
120 0.346
121 0.346
122 0.346
123 0.346
124 0.346
125 0.346
126 0.346
127 0.346
128 0.346
129 0.346
130 0.346
131 0.346
132 0.346
133 0.346
134 0.346
135 0.346
136 0.346
137 0.346
138 0.346
139 0.346
140 0.346
141 0.346
142 0.346
143 0.346
144 0.346
145 0.346
146 0.346
147 0.346
148 0.346
149 0.346
150 0.346
151 0.346
152 0.346
153 0.346
154 0.346
155 0.346
156 0.346
157 0.346
158 0.346
159 0.346
160 0.346
161 0.346
162 0.346
163 0.346
164 0.346
165 0.346
166 0.346
167 0.346
168 0.346
169 0.346
170 0.346
171 0.346
172 0.346
173 0.346
174 0.346
175 0.346
176 0.346
177 0.346
178 0.346
179 0.346
180 0.346
181 0.346
182 0.346
183 0.346
184 0.346
185 0.346
186 0.346
187 0.346
188 0.346
189 0.346
190 0.346
191 0.346
192 0.346
193 0.346
194 0.346
195 0.346
196 0.346
197 0.346
198 0.346
199 0.346
200 0.346
};
\addlegendentry{CASG exact H}
\addplot [very thick, darkslateblue]
table {%
0 0
1 0.00200803212851406
2 0.00200803212851406
3 0.00200803212851406
4 0.00200803212851406
5 0.00200803212851406
6 0.00200803212851406
7 0.00602409638554217
8 0.00803212851405622
9 0.00803212851405622
10 0.0120481927710843
11 0.0180722891566265
12 0.0240963855421687
13 0.0240963855421687
14 0.034136546184739
15 0.0381526104417671
16 0.0481927710843374
17 0.0562248995983936
18 0.0582329317269076
19 0.0783132530120482
20 0.0883534136546185
21 0.116465863453815
22 0.120481927710843
23 0.166666666666667
24 0.178714859437751
25 0.196787148594378
26 0.204819277108434
27 0.226907630522088
28 0.230923694779116
29 0.251004016064257
30 0.257028112449799
31 0.267068273092369
32 0.273092369477912
33 0.281124497991968
34 0.285140562248996
35 0.301204819277108
36 0.307228915662651
37 0.321285140562249
38 0.325301204819277
39 0.363453815261044
40 0.363453815261044
41 0.399598393574297
42 0.401606425702811
43 0.409638554216867
44 0.41566265060241
45 0.42570281124498
46 0.433734939759036
47 0.443775100401606
48 0.44578313253012
49 0.459839357429719
50 0.463855421686747
51 0.465863453815261
52 0.467871485943775
53 0.473895582329317
54 0.479919678714859
55 0.487951807228916
56 0.48995983935743
57 0.497991967871486
58 0.5
59 0.502008032128514
60 0.504016064257028
61 0.508032128514056
62 0.508032128514056
63 0.508032128514056
64 0.51004016064257
65 0.512048192771084
66 0.512048192771084
67 0.520080321285141
68 0.520080321285141
69 0.520080321285141
70 0.520080321285141
71 0.520080321285141
72 0.520080321285141
73 0.520080321285141
74 0.520080321285141
75 0.520080321285141
76 0.520080321285141
77 0.520080321285141
78 0.520080321285141
79 0.520080321285141
80 0.520080321285141
81 0.520080321285141
82 0.520080321285141
83 0.522088353413655
84 0.522088353413655
85 0.522088353413655
86 0.522088353413655
87 0.522088353413655
88 0.522088353413655
89 0.522088353413655
90 0.522088353413655
91 0.522088353413655
92 0.522088353413655
93 0.522088353413655
94 0.522088353413655
95 0.522088353413655
96 0.522088353413655
97 0.522088353413655
98 0.522088353413655
99 0.522088353413655
100 0.522088353413655
101 0.524096385542169
102 0.524096385542169
103 0.524096385542169
104 0.524096385542169
105 0.524096385542169
106 0.524096385542169
107 0.524096385542169
108 0.524096385542169
109 0.524096385542169
110 0.524096385542169
111 0.524096385542169
112 0.524096385542169
113 0.524096385542169
114 0.526104417670683
115 0.526104417670683
116 0.526104417670683
117 0.526104417670683
118 0.526104417670683
119 0.526104417670683
120 0.528112449799197
121 0.528112449799197
122 0.528112449799197
123 0.530120481927711
124 0.530120481927711
125 0.530120481927711
126 0.530120481927711
127 0.530120481927711
128 0.530120481927711
129 0.532128514056225
130 0.532128514056225
131 0.532128514056225
132 0.532128514056225
133 0.532128514056225
134 0.532128514056225
135 0.532128514056225
136 0.532128514056225
137 0.532128514056225
138 0.532128514056225
139 0.532128514056225
140 0.532128514056225
141 0.532128514056225
142 0.532128514056225
143 0.532128514056225
144 0.532128514056225
145 0.532128514056225
146 0.532128514056225
147 0.532128514056225
148 0.532128514056225
149 0.532128514056225
150 0.532128514056225
151 0.534136546184739
152 0.534136546184739
153 0.534136546184739
154 0.534136546184739
155 0.534136546184739
156 0.534136546184739
157 0.534136546184739
158 0.534136546184739
159 0.534136546184739
160 0.534136546184739
161 0.534136546184739
162 0.534136546184739
163 0.534136546184739
164 0.534136546184739
165 0.534136546184739
166 0.534136546184739
167 0.534136546184739
168 0.534136546184739
169 0.534136546184739
170 0.534136546184739
171 0.534136546184739
172 0.534136546184739
173 0.534136546184739
174 0.534136546184739
175 0.534136546184739
176 0.534136546184739
177 0.534136546184739
178 0.534136546184739
179 0.534136546184739
180 0.534136546184739
181 0.534136546184739
182 0.534136546184739
183 0.534136546184739
184 0.534136546184739
185 0.534136546184739
186 0.534136546184739
187 0.534136546184739
188 0.534136546184739
189 0.534136546184739
190 0.534136546184739
191 0.534136546184739
192 0.534136546184739
193 0.534136546184739
194 0.534136546184739
195 0.534136546184739
196 0.534136546184739
197 0.534136546184739
198 0.534136546184739
199 0.534136546184739
200 0.534136546184739
};
\addlegendentry{CD}
\addplot [very thick, palevioletred204121167]
table {%
0 0
1 0
2 0
3 0
4 0
5 0
6 0
7 0
8 0
9 0
10 0.00408163265306122
11 0.00408163265306122
12 0.00612244897959184
13 0.00612244897959184
14 0.00612244897959184
15 0.00612244897959184
16 0.00612244897959184
17 0.00612244897959184
18 0.00612244897959184
19 0.00612244897959184
20 0.0102040816326531
21 0.0122448979591837
22 0.0163265306122449
23 0.0224489795918367
24 0.0224489795918367
25 0.0224489795918367
26 0.0244897959183673
27 0.0244897959183673
28 0.026530612244898
29 0.026530612244898
30 0.0326530612244898
31 0.0326530612244898
32 0.0326530612244898
33 0.036734693877551
34 0.036734693877551
35 0.0387755102040816
36 0.0408163265306122
37 0.0408163265306122
38 0.0428571428571429
39 0.0428571428571429
40 0.0428571428571429
41 0.0448979591836735
42 0.0448979591836735
43 0.0448979591836735
44 0.0469387755102041
45 0.0469387755102041
46 0.0510204081632653
47 0.0510204081632653
48 0.0551020408163265
49 0.0653061224489796
50 0.0775510204081633
51 0.0836734693877551
52 0.0979591836734694
53 0.106122448979592
54 0.114285714285714
55 0.124489795918367
56 0.126530612244898
57 0.128571428571429
58 0.128571428571429
59 0.128571428571429
60 0.128571428571429
61 0.130612244897959
62 0.13265306122449
63 0.13265306122449
64 0.13265306122449
65 0.13265306122449
66 0.13469387755102
67 0.136734693877551
68 0.142857142857143
69 0.144897959183673
70 0.151020408163265
71 0.155102040816327
72 0.159183673469388
73 0.159183673469388
74 0.159183673469388
75 0.161224489795918
76 0.163265306122449
77 0.163265306122449
78 0.16530612244898
79 0.16530612244898
80 0.16734693877551
81 0.169387755102041
82 0.171428571428571
83 0.173469387755102
84 0.173469387755102
85 0.173469387755102
86 0.179591836734694
87 0.179591836734694
88 0.179591836734694
89 0.181632653061225
90 0.181632653061225
91 0.181632653061225
92 0.183673469387755
93 0.185714285714286
94 0.193877551020408
95 0.195918367346939
96 0.195918367346939
97 0.197959183673469
98 0.197959183673469
99 0.197959183673469
100 0.2
101 0.204081632653061
102 0.208163265306122
103 0.212244897959184
104 0.216326530612245
105 0.218367346938776
106 0.224489795918367
107 0.228571428571429
108 0.23469387755102
109 0.23469387755102
110 0.238775510204082
111 0.240816326530612
112 0.242857142857143
113 0.244897959183673
114 0.248979591836735
115 0.248979591836735
116 0.248979591836735
117 0.248979591836735
118 0.248979591836735
119 0.248979591836735
120 0.248979591836735
121 0.248979591836735
122 0.251020408163265
123 0.251020408163265
124 0.253061224489796
125 0.253061224489796
126 0.253061224489796
127 0.253061224489796
128 0.253061224489796
129 0.255102040816327
130 0.255102040816327
131 0.257142857142857
132 0.257142857142857
133 0.261224489795918
134 0.261224489795918
135 0.263265306122449
136 0.263265306122449
137 0.263265306122449
138 0.263265306122449
139 0.263265306122449
140 0.263265306122449
141 0.263265306122449
142 0.263265306122449
143 0.263265306122449
144 0.263265306122449
145 0.263265306122449
146 0.263265306122449
147 0.26530612244898
148 0.26530612244898
149 0.26530612244898
150 0.26530612244898
151 0.26530612244898
152 0.26530612244898
153 0.26530612244898
154 0.269387755102041
155 0.269387755102041
156 0.269387755102041
157 0.269387755102041
158 0.269387755102041
159 0.271428571428571
160 0.271428571428571
161 0.273469387755102
162 0.273469387755102
163 0.273469387755102
164 0.273469387755102
165 0.275510204081633
166 0.275510204081633
167 0.275510204081633
168 0.275510204081633
169 0.275510204081633
170 0.275510204081633
171 0.275510204081633
172 0.275510204081633
173 0.275510204081633
174 0.275510204081633
175 0.275510204081633
176 0.275510204081633
177 0.275510204081633
178 0.275510204081633
179 0.275510204081633
180 0.275510204081633
181 0.275510204081633
182 0.275510204081633
183 0.275510204081633
184 0.277551020408163
185 0.277551020408163
186 0.277551020408163
187 0.277551020408163
188 0.277551020408163
189 0.277551020408163
190 0.277551020408163
191 0.277551020408163
192 0.277551020408163
193 0.277551020408163
194 0.277551020408163
195 0.277551020408163
196 0.277551020408163
197 0.277551020408163
198 0.277551020408163
199 0.277551020408163
200 0.277551020408163
};
\addlegendentry{AFD}
\end{axis}

\end{tikzpicture}

%% file: figs/Optimization/Opt_sig_0.001_tau_1e-05.tex
\begin{tikzpicture}

\definecolor{darkgray176}{RGB}{176,176,176}
\definecolor{darkorange}{RGB}{255,140,0}
\definecolor{darkslateblue}{RGB}{72,61,139}
\definecolor{lightgray204}{RGB}{204,204,204}
\definecolor{maroon}{RGB}{128,0,0}
\definecolor{palevioletred204121167}{RGB}{204,121,167}
\definecolor{slategray}{RGB}{112,128,144}

\begin{axis}[
legend cell align={left},
legend style={
  fill opacity=0.8,
  draw opacity=1,
  text opacity=1,
  at={(0.03,0.97)},
  anchor=north west,
  draw=lightgray204
},
tick align=outside,
tick pos=left,
x grid style={darkgray176},
xlabel={Number of Simplex Gradients},
xmajorgrids,
xmin=-10, xmax=210,
xminorgrids,
xtick style={color=black},
y grid style={darkgray176},
ylabel={Fraction of Problems},
ymajorgrids,
ymin=-0.0334677419354839, ymax=0.702822580645161,
yminorgrids,
ytick style={color=black}
]
\addplot [very thick, maroon]
table {%
0 0
1 0
2 0
3 0
4 0
5 0
6 0
7 0
8 0
9 0
10 0
11 0
12 0
13 0
14 0
15 0
16 0
17 0
18 0
19 0.00204498977505112
20 0.00204498977505112
21 0.00204498977505112
22 0.00408997955010225
23 0.00408997955010225
24 0.00613496932515337
25 0.0081799591002045
26 0.0102249488752556
27 0.0102249488752556
28 0.0102249488752556
29 0.0102249488752556
30 0.0102249488752556
31 0.0102249488752556
32 0.0102249488752556
33 0.0102249488752556
34 0.0102249488752556
35 0.0102249488752556
36 0.0102249488752556
37 0.0102249488752556
38 0.0102249488752556
39 0.0102249488752556
40 0.0122699386503067
41 0.0122699386503067
42 0.0122699386503067
43 0.0122699386503067
44 0.0122699386503067
45 0.0122699386503067
46 0.0122699386503067
47 0.0122699386503067
48 0.0143149284253579
49 0.0143149284253579
50 0.016359918200409
51 0.016359918200409
52 0.016359918200409
53 0.0184049079754601
54 0.0184049079754601
55 0.0184049079754601
56 0.0184049079754601
57 0.0184049079754601
58 0.0184049079754601
59 0.0204498977505112
60 0.0204498977505112
61 0.0204498977505112
62 0.0204498977505112
63 0.0224948875255624
64 0.0224948875255624
65 0.0224948875255624
66 0.0224948875255624
67 0.0224948875255624
68 0.0224948875255624
69 0.0224948875255624
70 0.0224948875255624
71 0.0265848670756646
72 0.0265848670756646
73 0.0265848670756646
74 0.0286298568507157
75 0.0286298568507157
76 0.0286298568507157
77 0.0286298568507157
78 0.0286298568507157
79 0.032719836400818
80 0.032719836400818
81 0.032719836400818
82 0.032719836400818
83 0.0347648261758691
84 0.0347648261758691
85 0.0347648261758691
86 0.0368098159509202
87 0.0388548057259714
88 0.0388548057259714
89 0.0408997955010225
90 0.0408997955010225
91 0.0429447852760736
92 0.0449897750511247
93 0.0449897750511247
94 0.0449897750511247
95 0.0449897750511247
96 0.0449897750511247
97 0.0449897750511247
98 0.0449897750511247
99 0.0449897750511247
100 0.0449897750511247
101 0.0449897750511247
102 0.0449897750511247
103 0.0449897750511247
104 0.0449897750511247
105 0.0470347648261759
106 0.0470347648261759
107 0.0470347648261759
108 0.0470347648261759
109 0.0470347648261759
110 0.0470347648261759
111 0.0470347648261759
112 0.0470347648261759
113 0.0470347648261759
114 0.0470347648261759
115 0.0470347648261759
116 0.049079754601227
117 0.049079754601227
118 0.049079754601227
119 0.049079754601227
120 0.049079754601227
121 0.049079754601227
122 0.0511247443762781
123 0.0511247443762781
124 0.0531697341513292
125 0.0531697341513292
126 0.0531697341513292
127 0.0531697341513292
128 0.0531697341513292
129 0.0531697341513292
130 0.0531697341513292
131 0.0531697341513292
132 0.0531697341513292
133 0.0531697341513292
134 0.0531697341513292
135 0.0552147239263804
136 0.0552147239263804
137 0.0552147239263804
138 0.0552147239263804
139 0.0552147239263804
140 0.0552147239263804
141 0.0552147239263804
142 0.0552147239263804
143 0.0572597137014315
144 0.0572597137014315
145 0.0572597137014315
146 0.0572597137014315
147 0.0572597137014315
148 0.0572597137014315
149 0.0572597137014315
150 0.0572597137014315
151 0.0572597137014315
152 0.0572597137014315
153 0.0572597137014315
154 0.0572597137014315
155 0.0572597137014315
156 0.0572597137014315
157 0.0572597137014315
158 0.0572597137014315
159 0.0572597137014315
160 0.0572597137014315
161 0.0572597137014315
162 0.0572597137014315
163 0.0572597137014315
164 0.0572597137014315
165 0.0572597137014315
166 0.0572597137014315
167 0.0572597137014315
168 0.0572597137014315
169 0.0572597137014315
170 0.0572597137014315
171 0.0572597137014315
172 0.0572597137014315
173 0.0572597137014315
174 0.0572597137014315
175 0.0572597137014315
176 0.0572597137014315
177 0.0572597137014315
178 0.0572597137014315
179 0.0572597137014315
180 0.0572597137014315
181 0.0572597137014315
182 0.0593047034764826
183 0.0593047034764826
184 0.0593047034764826
185 0.0593047034764826
186 0.0593047034764826
187 0.0593047034764826
188 0.0593047034764826
189 0.0593047034764826
190 0.0613496932515337
191 0.0613496932515337
192 0.0613496932515337
193 0.0613496932515337
194 0.0613496932515337
195 0.0633946830265849
196 0.0633946830265849
197 0.0633946830265849
198 0.0633946830265849
199 0.0633946830265849
200 0.0633946830265849
};
\addlegendentry{RBF}
\addplot [very thick, slategray]
table {%
0 0
1 0
2 0
3 0
4 0
5 0
6 0.0120481927710843
7 0.0200803212851406
8 0.0220883534136546
9 0.0261044176706827
10 0.0281124497991968
11 0.0542168674698795
12 0.0823293172690763
13 0.104417670682731
14 0.126506024096386
15 0.146586345381526
16 0.170682730923695
17 0.182730923694779
18 0.198795180722892
19 0.206827309236948
20 0.216867469879518
21 0.228915662650602
22 0.236947791164659
23 0.238955823293173
24 0.244979919678715
25 0.248995983935743
26 0.261044176706827
27 0.263052208835341
28 0.283132530120482
29 0.289156626506024
30 0.29718875502008
31 0.323293172690763
32 0.325301204819277
33 0.335341365461847
34 0.337349397590361
35 0.337349397590361
36 0.337349397590361
37 0.339357429718876
38 0.339357429718876
39 0.343373493975904
40 0.343373493975904
41 0.347389558232932
42 0.35140562248996
43 0.355421686746988
44 0.357429718875502
45 0.36144578313253
46 0.365461847389558
47 0.365461847389558
48 0.367469879518072
49 0.369477911646586
50 0.3714859437751
51 0.373493975903614
52 0.373493975903614
53 0.373493975903614
54 0.377510040160643
55 0.377510040160643
56 0.377510040160643
57 0.377510040160643
58 0.377510040160643
59 0.377510040160643
60 0.377510040160643
61 0.377510040160643
62 0.379518072289157
63 0.379518072289157
64 0.379518072289157
65 0.379518072289157
66 0.379518072289157
67 0.379518072289157
68 0.381526104417671
69 0.381526104417671
70 0.381526104417671
71 0.383534136546185
72 0.385542168674699
73 0.385542168674699
74 0.387550200803213
75 0.389558232931727
76 0.389558232931727
77 0.391566265060241
78 0.391566265060241
79 0.391566265060241
80 0.391566265060241
81 0.393574297188755
82 0.393574297188755
83 0.393574297188755
84 0.395582329317269
85 0.397590361445783
86 0.397590361445783
87 0.399598393574297
88 0.399598393574297
89 0.399598393574297
90 0.399598393574297
91 0.399598393574297
92 0.399598393574297
93 0.399598393574297
94 0.401606425702811
95 0.403614457831325
96 0.403614457831325
97 0.403614457831325
98 0.403614457831325
99 0.405622489959839
100 0.407630522088353
101 0.407630522088353
102 0.409638554216867
103 0.409638554216867
104 0.409638554216867
105 0.409638554216867
106 0.409638554216867
107 0.409638554216867
108 0.409638554216867
109 0.409638554216867
110 0.409638554216867
111 0.409638554216867
112 0.409638554216867
113 0.409638554216867
114 0.409638554216867
115 0.409638554216867
116 0.409638554216867
117 0.409638554216867
118 0.409638554216867
119 0.409638554216867
120 0.409638554216867
121 0.409638554216867
122 0.409638554216867
123 0.409638554216867
124 0.409638554216867
125 0.409638554216867
126 0.409638554216867
127 0.409638554216867
128 0.409638554216867
129 0.409638554216867
130 0.409638554216867
131 0.409638554216867
132 0.409638554216867
133 0.409638554216867
134 0.409638554216867
135 0.409638554216867
136 0.409638554216867
137 0.409638554216867
138 0.409638554216867
139 0.409638554216867
140 0.409638554216867
141 0.409638554216867
142 0.409638554216867
143 0.409638554216867
144 0.409638554216867
145 0.409638554216867
146 0.409638554216867
147 0.409638554216867
148 0.409638554216867
149 0.409638554216867
150 0.409638554216867
151 0.409638554216867
152 0.411646586345382
153 0.411646586345382
154 0.411646586345382
155 0.411646586345382
156 0.411646586345382
157 0.411646586345382
158 0.411646586345382
159 0.411646586345382
160 0.411646586345382
161 0.411646586345382
162 0.411646586345382
163 0.411646586345382
164 0.411646586345382
165 0.411646586345382
166 0.411646586345382
167 0.411646586345382
168 0.411646586345382
169 0.411646586345382
170 0.411646586345382
171 0.411646586345382
172 0.411646586345382
173 0.411646586345382
174 0.411646586345382
175 0.411646586345382
176 0.411646586345382
177 0.411646586345382
178 0.411646586345382
179 0.411646586345382
180 0.411646586345382
181 0.411646586345382
182 0.411646586345382
183 0.411646586345382
184 0.411646586345382
185 0.411646586345382
186 0.411646586345382
187 0.411646586345382
188 0.411646586345382
189 0.411646586345382
190 0.411646586345382
191 0.411646586345382
192 0.411646586345382
193 0.411646586345382
194 0.411646586345382
195 0.411646586345382
196 0.411646586345382
197 0.411646586345382
198 0.411646586345382
199 0.411646586345382
200 0.411646586345382
};
\addlegendentry{FD exact H}
\addplot [very thick, darkorange]
table {%
0 0
1 0
2 0
3 0
4 0
5 0
6 0.0181086519114688
7 0.0221327967806841
8 0.0221327967806841
9 0.0261569416498994
10 0.03420523138833
11 0.06841046277666
12 0.0945674044265594
13 0.116700201207243
14 0.144869215291751
15 0.164989939637827
16 0.185110663983903
17 0.191146881287726
18 0.211267605633803
19 0.21327967806841
20 0.217303822937626
21 0.227364185110664
22 0.231388329979879
23 0.243460764587525
24 0.269617706237425
25 0.283702213279678
26 0.299798792756539
27 0.301810865191147
28 0.317907444668008
29 0.325955734406439
30 0.331991951710262
31 0.358148893360161
32 0.366197183098592
33 0.370221327967807
34 0.382293762575453
35 0.38430583501006
36 0.388329979879276
37 0.388329979879276
38 0.392354124748491
39 0.396378269617706
40 0.400402414486922
41 0.408450704225352
42 0.412474849094567
43 0.420523138832998
44 0.428571428571429
45 0.434607645875252
46 0.438631790744467
47 0.440643863179074
48 0.44466800804829
49 0.44466800804829
50 0.450704225352113
51 0.456740442655936
52 0.456740442655936
53 0.456740442655936
54 0.460764587525151
55 0.460764587525151
56 0.460764587525151
57 0.462776659959759
58 0.462776659959759
59 0.464788732394366
60 0.464788732394366
61 0.464788732394366
62 0.464788732394366
63 0.464788732394366
64 0.464788732394366
65 0.464788732394366
66 0.464788732394366
67 0.466800804828974
68 0.468812877263582
69 0.468812877263582
70 0.468812877263582
71 0.468812877263582
72 0.468812877263582
73 0.468812877263582
74 0.468812877263582
75 0.468812877263582
76 0.468812877263582
77 0.468812877263582
78 0.468812877263582
79 0.470824949698189
80 0.470824949698189
81 0.470824949698189
82 0.470824949698189
83 0.470824949698189
84 0.470824949698189
85 0.470824949698189
86 0.470824949698189
87 0.470824949698189
88 0.470824949698189
89 0.470824949698189
90 0.474849094567404
91 0.474849094567404
92 0.474849094567404
93 0.474849094567404
94 0.476861167002012
95 0.476861167002012
96 0.476861167002012
97 0.476861167002012
98 0.476861167002012
99 0.476861167002012
100 0.476861167002012
101 0.476861167002012
102 0.476861167002012
103 0.476861167002012
104 0.476861167002012
105 0.47887323943662
106 0.47887323943662
107 0.47887323943662
108 0.47887323943662
109 0.47887323943662
110 0.47887323943662
111 0.47887323943662
112 0.47887323943662
113 0.480885311871227
114 0.480885311871227
115 0.480885311871227
116 0.480885311871227
117 0.480885311871227
118 0.480885311871227
119 0.480885311871227
120 0.480885311871227
121 0.480885311871227
122 0.482897384305835
123 0.482897384305835
124 0.482897384305835
125 0.482897384305835
126 0.482897384305835
127 0.482897384305835
128 0.482897384305835
129 0.482897384305835
130 0.482897384305835
131 0.482897384305835
132 0.482897384305835
133 0.482897384305835
134 0.484909456740443
135 0.484909456740443
136 0.484909456740443
137 0.484909456740443
138 0.484909456740443
139 0.484909456740443
140 0.484909456740443
141 0.484909456740443
142 0.484909456740443
143 0.484909456740443
144 0.48692152917505
145 0.48692152917505
146 0.48692152917505
147 0.48692152917505
148 0.48692152917505
149 0.48692152917505
150 0.48692152917505
151 0.48692152917505
152 0.48692152917505
153 0.48692152917505
154 0.48692152917505
155 0.48692152917505
156 0.48692152917505
157 0.48692152917505
158 0.48692152917505
159 0.48692152917505
160 0.48692152917505
161 0.48692152917505
162 0.48692152917505
163 0.48692152917505
164 0.48692152917505
165 0.48692152917505
166 0.48692152917505
167 0.488933601609658
168 0.488933601609658
169 0.488933601609658
170 0.488933601609658
171 0.488933601609658
172 0.488933601609658
173 0.488933601609658
174 0.488933601609658
175 0.488933601609658
176 0.488933601609658
177 0.488933601609658
178 0.488933601609658
179 0.488933601609658
180 0.488933601609658
181 0.488933601609658
182 0.488933601609658
183 0.488933601609658
184 0.488933601609658
185 0.488933601609658
186 0.488933601609658
187 0.488933601609658
188 0.488933601609658
189 0.488933601609658
190 0.488933601609658
191 0.488933601609658
192 0.488933601609658
193 0.488933601609658
194 0.488933601609658
195 0.488933601609658
196 0.488933601609658
197 0.488933601609658
198 0.488933601609658
199 0.488933601609658
200 0.488933601609658
};
\addlegendentry{CASG exact H}
\addplot [very thick, slategray, dashed]
table {%
0 0
1 0
2 0
3 0
4 0
5 0
6 0
7 0
8 0
9 0
10 0
11 0
12 0
13 0
14 0
15 0
16 0
17 0
18 0.002
19 0.002
20 0.002
21 0.004
22 0.024
23 0.026
24 0.026
25 0.026
26 0.026
27 0.03
28 0.03
29 0.032
30 0.034
31 0.04
32 0.046
33 0.048
34 0.082
35 0.088
36 0.098
37 0.128
38 0.138
39 0.144
40 0.144
41 0.144
42 0.148
43 0.15
44 0.15
45 0.152
46 0.152
47 0.154
48 0.156
49 0.16
50 0.164
51 0.188
52 0.19
53 0.19
54 0.19
55 0.198
56 0.206
57 0.206
58 0.21
59 0.216
60 0.22
61 0.224
62 0.23
63 0.262
64 0.264
65 0.266
66 0.266
67 0.27
68 0.288
69 0.298
70 0.304
71 0.304
72 0.306
73 0.306
74 0.308
75 0.308
76 0.308
77 0.31
78 0.31
79 0.31
80 0.31
81 0.31
82 0.314
83 0.314
84 0.314
85 0.314
86 0.316
87 0.318
88 0.318
89 0.32
90 0.322
91 0.322
92 0.324
93 0.324
94 0.324
95 0.324
96 0.326
97 0.326
98 0.33
99 0.33
100 0.33
101 0.33
102 0.33
103 0.33
104 0.33
105 0.33
106 0.33
107 0.332
108 0.334
109 0.334
110 0.334
111 0.334
112 0.334
113 0.334
114 0.334
115 0.334
116 0.334
117 0.334
118 0.334
119 0.334
120 0.334
121 0.334
122 0.334
123 0.334
124 0.334
125 0.334
126 0.334
127 0.334
128 0.334
129 0.334
130 0.334
131 0.334
132 0.334
133 0.334
134 0.336
135 0.336
136 0.336
137 0.336
138 0.336
139 0.336
140 0.336
141 0.336
142 0.336
143 0.336
144 0.336
145 0.336
146 0.338
147 0.338
148 0.338
149 0.338
150 0.338
151 0.338
152 0.338
153 0.338
154 0.338
155 0.338
156 0.338
157 0.338
158 0.338
159 0.338
160 0.338
161 0.338
162 0.338
163 0.338
164 0.338
165 0.338
166 0.338
167 0.338
168 0.338
169 0.338
170 0.338
171 0.338
172 0.338
173 0.338
174 0.338
175 0.338
176 0.338
177 0.34
178 0.34
179 0.34
180 0.34
181 0.34
182 0.34
183 0.34
184 0.34
185 0.34
186 0.34
187 0.34
188 0.34
189 0.34
190 0.34
191 0.34
192 0.34
193 0.34
194 0.34
195 0.34
196 0.34
197 0.34
198 0.34
199 0.34
200 0.34
};
\addlegendentry{FD RBF H}
\addplot [very thick, palevioletred204121167]
table {%
0 0
1 0.00200400801603206
2 0.00200400801603206
3 0.00200400801603206
4 0.00200400801603206
5 0.00200400801603206
6 0.00200400801603206
7 0.00200400801603206
8 0.00200400801603206
9 0.00200400801603206
10 0.00200400801603206
11 0.00200400801603206
12 0.00200400801603206
13 0.00200400801603206
14 0.00200400801603206
15 0.00200400801603206
16 0.00400801603206413
17 0.00400801603206413
18 0.00400801603206413
19 0.00400801603206413
20 0.00601202404809619
21 0.00801603206412826
22 0.0100200400801603
23 0.0120240480961924
24 0.0120240480961924
25 0.0140280561122244
26 0.0160320641282565
27 0.0200400801603206
28 0.0200400801603206
29 0.0220440881763527
30 0.0240480961923848
31 0.0260521042084168
32 0.0260521042084168
33 0.0260521042084168
34 0.0280561122244489
35 0.0280561122244489
36 0.030060120240481
37 0.030060120240481
38 0.030060120240481
39 0.030060120240481
40 0.030060120240481
41 0.030060120240481
42 0.030060120240481
43 0.030060120240481
44 0.030060120240481
45 0.030060120240481
46 0.030060120240481
47 0.030060120240481
48 0.0360721442885772
49 0.0420841683366733
50 0.0460921843687375
51 0.0561122244488978
52 0.0681362725450902
53 0.0781563126252505
54 0.0841683366733467
55 0.0861723446893788
56 0.0921843687374749
57 0.0981963927855711
58 0.102204408817635
59 0.104208416833667
60 0.108216432865731
61 0.112224448897796
62 0.114228456913828
63 0.118236472945892
64 0.118236472945892
65 0.120240480961924
66 0.120240480961924
67 0.120240480961924
68 0.120240480961924
69 0.12625250501002
70 0.128256513026052
71 0.132264529058116
72 0.13627254509018
73 0.140280561122244
74 0.142284569138277
75 0.146292585170341
76 0.152304609218437
77 0.154308617234469
78 0.158316633266533
79 0.160320641282565
80 0.162324649298597
81 0.162324649298597
82 0.164328657314629
83 0.164328657314629
84 0.166332665330661
85 0.168336673346693
86 0.168336673346693
87 0.170340681362725
88 0.170340681362725
89 0.170340681362725
90 0.170340681362725
91 0.176352705410822
92 0.176352705410822
93 0.176352705410822
94 0.176352705410822
95 0.176352705410822
96 0.176352705410822
97 0.178356713426854
98 0.178356713426854
99 0.18436873747495
100 0.186372745490982
101 0.186372745490982
102 0.186372745490982
103 0.186372745490982
104 0.188376753507014
105 0.19438877755511
106 0.200400801603206
107 0.200400801603206
108 0.208416833667335
109 0.210420841683367
110 0.218436873747495
111 0.220440881763527
112 0.222444889779559
113 0.226452905811623
114 0.226452905811623
115 0.226452905811623
116 0.228456913827655
117 0.230460921843687
118 0.234468937875752
119 0.234468937875752
120 0.238476953907816
121 0.238476953907816
122 0.240480961923848
123 0.240480961923848
124 0.240480961923848
125 0.240480961923848
126 0.240480961923848
127 0.240480961923848
128 0.240480961923848
129 0.240480961923848
130 0.24248496993988
131 0.24248496993988
132 0.24248496993988
133 0.246492985971944
134 0.246492985971944
135 0.246492985971944
136 0.248496993987976
137 0.248496993987976
138 0.250501002004008
139 0.250501002004008
140 0.25250501002004
141 0.258517034068136
142 0.260521042084168
143 0.260521042084168
144 0.260521042084168
145 0.2625250501002
146 0.2625250501002
147 0.2625250501002
148 0.2625250501002
149 0.2625250501002
150 0.2625250501002
151 0.2625250501002
152 0.264529058116232
153 0.264529058116232
154 0.270541082164329
155 0.270541082164329
156 0.270541082164329
157 0.274549098196393
158 0.274549098196393
159 0.276553106212425
160 0.278557114228457
161 0.278557114228457
162 0.282565130260521
163 0.282565130260521
164 0.284569138276553
165 0.286573146292585
166 0.288577154308617
167 0.290581162324649
168 0.290581162324649
169 0.290581162324649
170 0.290581162324649
171 0.296593186372746
172 0.298597194388778
173 0.30060120240481
174 0.30060120240481
175 0.304609218436874
176 0.306613226452906
177 0.306613226452906
178 0.308617234468938
179 0.31062124248497
180 0.31062124248497
181 0.314629258517034
182 0.314629258517034
183 0.314629258517034
184 0.318637274549098
185 0.32064128256513
186 0.32064128256513
187 0.32064128256513
188 0.32064128256513
189 0.32064128256513
190 0.32064128256513
191 0.32064128256513
192 0.322645290581162
193 0.322645290581162
194 0.322645290581162
195 0.322645290581162
196 0.322645290581162
197 0.322645290581162
198 0.322645290581162
199 0.322645290581162
200 0.322645290581162
};
\addlegendentry{AFD}
\addplot [very thick, darkslateblue]
table {%
0 0
1 0
2 0
3 0
4 0
5 0.00201612903225806
6 0.00201612903225806
7 0.00201612903225806
8 0.00403225806451613
9 0.0100806451612903
10 0.032258064516129
11 0.032258064516129
12 0.0362903225806452
13 0.0383064516129032
14 0.0443548387096774
15 0.0463709677419355
16 0.0483870967741935
17 0.0685483870967742
18 0.0685483870967742
19 0.0887096774193548
20 0.0907258064516129
21 0.116935483870968
22 0.125
23 0.189516129032258
24 0.191532258064516
25 0.191532258064516
26 0.199596774193548
27 0.227822580645161
28 0.245967741935484
29 0.292338709677419
30 0.298387096774194
31 0.308467741935484
32 0.310483870967742
33 0.330645161290323
34 0.338709677419355
35 0.36491935483871
36 0.370967741935484
37 0.370967741935484
38 0.375
39 0.443548387096774
40 0.445564516129032
41 0.485887096774194
42 0.485887096774194
43 0.491935483870968
44 0.495967741935484
45 0.495967741935484
46 0.5
47 0.506048387096774
48 0.51008064516129
49 0.512096774193548
50 0.518145161290323
51 0.528225806451613
52 0.534274193548387
53 0.538306451612903
54 0.538306451612903
55 0.542338709677419
56 0.542338709677419
57 0.544354838709677
58 0.544354838709677
59 0.544354838709677
60 0.544354838709677
61 0.546370967741935
62 0.548387096774194
63 0.55241935483871
64 0.560483870967742
65 0.560483870967742
66 0.570564516129032
67 0.570564516129032
68 0.590725806451613
69 0.594758064516129
70 0.596774193548387
71 0.596774193548387
72 0.596774193548387
73 0.598790322580645
74 0.600806451612903
75 0.600806451612903
76 0.602822580645161
77 0.606854838709677
78 0.606854838709677
79 0.606854838709677
80 0.606854838709677
81 0.606854838709677
82 0.606854838709677
83 0.610887096774194
84 0.61491935483871
85 0.61491935483871
86 0.61491935483871
87 0.616935483870968
88 0.616935483870968
89 0.616935483870968
90 0.616935483870968
91 0.618951612903226
92 0.618951612903226
93 0.620967741935484
94 0.622983870967742
95 0.625
96 0.625
97 0.627016129032258
98 0.627016129032258
99 0.629032258064516
100 0.629032258064516
101 0.631048387096774
102 0.631048387096774
103 0.631048387096774
104 0.631048387096774
105 0.631048387096774
106 0.631048387096774
107 0.631048387096774
108 0.631048387096774
109 0.631048387096774
110 0.631048387096774
111 0.631048387096774
112 0.633064516129032
113 0.633064516129032
114 0.633064516129032
115 0.63508064516129
116 0.63508064516129
117 0.63508064516129
118 0.63508064516129
119 0.63508064516129
120 0.63508064516129
121 0.63508064516129
122 0.63508064516129
123 0.63508064516129
124 0.63508064516129
125 0.637096774193548
126 0.637096774193548
127 0.637096774193548
128 0.639112903225806
129 0.639112903225806
130 0.641129032258065
131 0.645161290322581
132 0.645161290322581
133 0.645161290322581
134 0.645161290322581
135 0.645161290322581
136 0.645161290322581
137 0.645161290322581
138 0.645161290322581
139 0.647177419354839
140 0.647177419354839
141 0.647177419354839
142 0.647177419354839
143 0.649193548387097
144 0.649193548387097
145 0.649193548387097
146 0.651209677419355
147 0.651209677419355
148 0.655241935483871
149 0.657258064516129
150 0.657258064516129
151 0.657258064516129
152 0.659274193548387
153 0.659274193548387
154 0.661290322580645
155 0.665322580645161
156 0.665322580645161
157 0.665322580645161
158 0.665322580645161
159 0.665322580645161
160 0.665322580645161
161 0.665322580645161
162 0.665322580645161
163 0.665322580645161
164 0.665322580645161
165 0.665322580645161
166 0.667338709677419
167 0.667338709677419
168 0.667338709677419
169 0.667338709677419
170 0.667338709677419
171 0.667338709677419
172 0.667338709677419
173 0.667338709677419
174 0.667338709677419
175 0.667338709677419
176 0.667338709677419
177 0.667338709677419
178 0.667338709677419
179 0.667338709677419
180 0.667338709677419
181 0.669354838709677
182 0.669354838709677
183 0.669354838709677
184 0.669354838709677
185 0.669354838709677
186 0.669354838709677
187 0.669354838709677
188 0.669354838709677
189 0.669354838709677
190 0.669354838709677
191 0.669354838709677
192 0.669354838709677
193 0.669354838709677
194 0.669354838709677
195 0.669354838709677
196 0.669354838709677
197 0.669354838709677
198 0.669354838709677
199 0.669354838709677
200 0.669354838709677
};
\addlegendentry{CD}
\addplot [very thick, darkorange, dashed]
table {%
0 0
1 0
2 0
3 0
4 0
5 0
6 0
7 0
8 0
9 0
10 0
11 0
12 0
13 0
14 0.002
15 0.002
16 0.002
17 0.002
18 0.002
19 0.002
20 0.002
21 0.004
22 0.02
23 0.028
24 0.03
25 0.032
26 0.046
27 0.052
28 0.056
29 0.062
30 0.062
31 0.066
32 0.068
33 0.076
34 0.12
35 0.126
36 0.148
37 0.172
38 0.176
39 0.182
40 0.188
41 0.19
42 0.19
43 0.194
44 0.198
45 0.198
46 0.2
47 0.202
48 0.202
49 0.204
50 0.206
51 0.228
52 0.234
53 0.244
54 0.246
55 0.256
56 0.264
57 0.266
58 0.266
59 0.268
60 0.27
61 0.274
62 0.276
63 0.296
64 0.306
65 0.314
66 0.314
67 0.316
68 0.328
69 0.338
70 0.338
71 0.34
72 0.34
73 0.34
74 0.34
75 0.34
76 0.34
77 0.342
78 0.344
79 0.344
80 0.344
81 0.35
82 0.354
83 0.356
84 0.358
85 0.358
86 0.364
87 0.364
88 0.364
89 0.364
90 0.364
91 0.366
92 0.366
93 0.366
94 0.366
95 0.368
96 0.368
97 0.368
98 0.368
99 0.368
100 0.37
101 0.37
102 0.37
103 0.37
104 0.37
105 0.372
106 0.374
107 0.374
108 0.374
109 0.374
110 0.374
111 0.374
112 0.374
113 0.376
114 0.376
115 0.376
116 0.376
117 0.38
118 0.38
119 0.38
120 0.382
121 0.382
122 0.382
123 0.382
124 0.382
125 0.382
126 0.382
127 0.382
128 0.382
129 0.382
130 0.382
131 0.382
132 0.382
133 0.382
134 0.382
135 0.382
136 0.384
137 0.384
138 0.384
139 0.384
140 0.386
141 0.386
142 0.386
143 0.386
144 0.386
145 0.386
146 0.386
147 0.386
148 0.386
149 0.388
150 0.388
151 0.388
152 0.388
153 0.388
154 0.388
155 0.388
156 0.388
157 0.388
158 0.388
159 0.388
160 0.388
161 0.388
162 0.388
163 0.388
164 0.388
165 0.388
166 0.388
167 0.388
168 0.388
169 0.388
170 0.388
171 0.388
172 0.388
173 0.388
174 0.388
175 0.388
176 0.388
177 0.388
178 0.388
179 0.388
180 0.388
181 0.388
182 0.388
183 0.388
184 0.388
185 0.388
186 0.388
187 0.388
188 0.388
189 0.388
190 0.388
191 0.388
192 0.39
193 0.39
194 0.39
195 0.39
196 0.39
197 0.39
198 0.39
199 0.39
200 0.39
};
\addlegendentry{CASG RBF H}
\end{axis}

\end{tikzpicture}

%% file: figs/Optimization/Opt_sig_1e-05_tau_1e-05.tex
\begin{tikzpicture}

\definecolor{darkgray176}{RGB}{176,176,176}
\definecolor{darkorange}{RGB}{255,140,0}
\definecolor{darkslateblue}{RGB}{72,61,139}
\definecolor{lightgray204}{RGB}{204,204,204}
\definecolor{maroon}{RGB}{128,0,0}
\definecolor{palevioletred204121167}{RGB}{204,121,167}
\definecolor{slategray}{RGB}{112,128,144}

\begin{axis}[
legend cell align={left},
legend style={
  fill opacity=0.8,
  draw opacity=1,
  text opacity=1,
  at={(0.97,0.03)},
  anchor=south east,
  draw=lightgray204
},
tick align=outside,
tick pos=left,
x grid style={darkgray176},
xlabel={Number of Simplex Gradients},
xmajorgrids,
xmin=-10, xmax=210,
xminorgrids,
xtick style={color=black},
y grid style={darkgray176},
ylabel={Fraction of Problems},
ymajorgrids,
ymin=-0.0424, ymax=0.8904,
yminorgrids,
ytick style={color=black}
]
\addplot [very thick, darkorange, dashed]
table {%
0 0
1 0
2 0
3 0
4 0
5 0
6 0
7 0
8 0
9 0
10 0
11 0
12 0
13 0.002
14 0.002
15 0.004
16 0.004
17 0.006
18 0.008
19 0.01
20 0.01
21 0.014
22 0.034
23 0.068
24 0.082
25 0.102
26 0.134
27 0.14
28 0.148
29 0.16
30 0.178
31 0.2
32 0.218
33 0.224
34 0.274
35 0.298
36 0.304
37 0.324
38 0.334
39 0.338
40 0.348
41 0.358
42 0.358
43 0.358
44 0.36
45 0.364
46 0.366
47 0.38
48 0.38
49 0.382
50 0.392
51 0.416
52 0.418
53 0.438
54 0.438
55 0.44
56 0.442
57 0.448
58 0.466
59 0.47
60 0.472
61 0.486
62 0.496
63 0.518
64 0.518
65 0.518
66 0.52
67 0.524
68 0.544
69 0.544
70 0.548
71 0.548
72 0.548
73 0.55
74 0.552
75 0.564
76 0.572
77 0.578
78 0.578
79 0.578
80 0.578
81 0.58
82 0.58
83 0.582
84 0.582
85 0.584
86 0.584
87 0.584
88 0.584
89 0.584
90 0.584
91 0.584
92 0.586
93 0.586
94 0.588
95 0.602
96 0.606
97 0.612
98 0.612
99 0.612
100 0.612
101 0.614
102 0.616
103 0.618
104 0.62
105 0.622
106 0.628
107 0.628
108 0.63
109 0.63
110 0.63
111 0.63
112 0.63
113 0.63
114 0.63
115 0.63
116 0.63
117 0.63
118 0.63
119 0.63
120 0.63
121 0.63
122 0.632
123 0.632
124 0.632
125 0.632
126 0.632
127 0.632
128 0.632
129 0.632
130 0.632
131 0.632
132 0.632
133 0.632
134 0.632
135 0.632
136 0.632
137 0.632
138 0.632
139 0.632
140 0.632
141 0.632
142 0.632
143 0.632
144 0.632
145 0.632
146 0.632
147 0.632
148 0.632
149 0.632
150 0.632
151 0.632
152 0.632
153 0.632
154 0.632
155 0.632
156 0.632
157 0.632
158 0.632
159 0.632
160 0.632
161 0.632
162 0.632
163 0.632
164 0.632
165 0.632
166 0.632
167 0.632
168 0.632
169 0.632
170 0.632
171 0.632
172 0.634
173 0.634
174 0.634
175 0.636
176 0.636
177 0.636
178 0.636
179 0.636
180 0.636
181 0.636
182 0.636
183 0.636
184 0.636
185 0.636
186 0.636
187 0.638
188 0.638
189 0.638
190 0.638
191 0.638
192 0.638
193 0.638
194 0.638
195 0.638
196 0.638
197 0.638
198 0.638
199 0.638
200 0.638
};
\addlegendentry{CASG RBF H}
\addplot [very thick, darkslateblue]
table {%
0 0
1 0
2 0
3 0
4 0
5 0.022
6 0.022
7 0.024
8 0.038
9 0.038
10 0.058
11 0.062
12 0.062
13 0.064
14 0.064
15 0.064
16 0.064
17 0.084
18 0.086
19 0.126
20 0.128
21 0.15
22 0.17
23 0.254
24 0.286
25 0.338
26 0.344
27 0.35
28 0.368
29 0.41
30 0.418
31 0.418
32 0.422
33 0.422
34 0.426
35 0.45
36 0.454
37 0.474
38 0.49
39 0.544
40 0.546
41 0.594
42 0.598
43 0.598
44 0.602
45 0.602
46 0.606
47 0.608
48 0.612
49 0.616
50 0.622
51 0.626
52 0.626
53 0.63
54 0.636
55 0.64
56 0.64
57 0.642
58 0.644
59 0.644
60 0.644
61 0.644
62 0.648
63 0.648
64 0.664
65 0.666
66 0.67
67 0.674
68 0.676
69 0.678
70 0.678
71 0.68
72 0.682
73 0.688
74 0.698
75 0.7
76 0.702
77 0.71
78 0.71
79 0.716
80 0.716
81 0.718
82 0.722
83 0.724
84 0.724
85 0.73
86 0.732
87 0.732
88 0.732
89 0.732
90 0.734
91 0.734
92 0.734
93 0.736
94 0.736
95 0.736
96 0.736
97 0.738
98 0.74
99 0.74
100 0.742
101 0.746
102 0.748
103 0.75
104 0.754
105 0.754
106 0.756
107 0.756
108 0.756
109 0.756
110 0.756
111 0.758
112 0.758
113 0.77
114 0.77
115 0.778
116 0.778
117 0.778
118 0.778
119 0.778
120 0.778
121 0.78
122 0.78
123 0.78
124 0.78
125 0.782
126 0.784
127 0.786
128 0.786
129 0.786
130 0.786
131 0.79
132 0.794
133 0.794
134 0.794
135 0.794
136 0.794
137 0.796
138 0.798
139 0.798
140 0.798
141 0.804
142 0.806
143 0.808
144 0.812
145 0.816
146 0.818
147 0.818
148 0.818
149 0.818
150 0.818
151 0.818
152 0.818
153 0.82
154 0.82
155 0.82
156 0.82
157 0.82
158 0.822
159 0.822
160 0.822
161 0.822
162 0.824
163 0.824
164 0.824
165 0.826
166 0.826
167 0.826
168 0.826
169 0.828
170 0.828
171 0.828
172 0.828
173 0.828
174 0.828
175 0.828
176 0.828
177 0.83
178 0.83
179 0.832
180 0.836
181 0.836
182 0.836
183 0.836
184 0.836
185 0.838
186 0.838
187 0.84
188 0.84
189 0.842
190 0.842
191 0.846
192 0.846
193 0.846
194 0.846
195 0.848
196 0.848
197 0.848
198 0.848
199 0.848
200 0.848
};
\addlegendentry{CD}
\addplot [very thick, darkorange]
table {%
0 0
1 0
2 0
3 0.0160320641282565
4 0.0160320641282565
5 0.0220440881763527
6 0.0420841683366733
7 0.0440881763527054
8 0.0440881763527054
9 0.0681362725450902
10 0.0921843687374749
11 0.114228456913828
12 0.134268537074148
13 0.18436873747495
14 0.216432865731463
15 0.236472945891784
16 0.276553106212425
17 0.280561122244489
18 0.328657314629259
19 0.348697394789579
20 0.412825651302605
21 0.434869739478958
22 0.450901803607214
23 0.460921843687375
24 0.474949899799599
25 0.482965931863727
26 0.486973947895792
27 0.492985971943888
28 0.513026052104208
29 0.513026052104208
30 0.517034068136273
31 0.543086172344689
32 0.545090180360721
33 0.545090180360721
34 0.549098196392786
35 0.561122244488978
36 0.569138276553106
37 0.57314629258517
38 0.575150300601202
39 0.581162324649299
40 0.583166332665331
41 0.591182364729459
42 0.597194388777555
43 0.607214428857715
44 0.609218436873748
45 0.613226452905812
46 0.62124248496994
47 0.627254509018036
48 0.637274549098196
49 0.643286573146293
50 0.651302605210421
51 0.653306613226453
52 0.653306613226453
53 0.659318637274549
54 0.665330661322645
55 0.667334669338677
56 0.671342685370741
57 0.673346693386774
58 0.675350701402806
59 0.677354709418838
60 0.67935871743487
61 0.685370741482966
62 0.691382765531062
63 0.693386773547094
64 0.693386773547094
65 0.693386773547094
66 0.693386773547094
67 0.693386773547094
68 0.693386773547094
69 0.695390781563126
70 0.695390781563126
71 0.695390781563126
72 0.695390781563126
73 0.69939879759519
74 0.69939879759519
75 0.69939879759519
76 0.701402805611222
77 0.703406813627255
78 0.705410821643287
79 0.709418837675351
80 0.711422845691383
81 0.717434869739479
82 0.723446893787575
83 0.725450901803607
84 0.725450901803607
85 0.725450901803607
86 0.725450901803607
87 0.725450901803607
88 0.725450901803607
89 0.725450901803607
90 0.725450901803607
91 0.727454909819639
92 0.727454909819639
93 0.727454909819639
94 0.727454909819639
95 0.727454909819639
96 0.727454909819639
97 0.729458917835671
98 0.729458917835671
99 0.729458917835671
100 0.729458917835671
101 0.729458917835671
102 0.729458917835671
103 0.729458917835671
104 0.729458917835671
105 0.729458917835671
106 0.729458917835671
107 0.729458917835671
108 0.729458917835671
109 0.729458917835671
110 0.729458917835671
111 0.729458917835671
112 0.731462925851703
113 0.731462925851703
114 0.731462925851703
115 0.731462925851703
116 0.731462925851703
117 0.731462925851703
118 0.731462925851703
119 0.731462925851703
120 0.731462925851703
121 0.731462925851703
122 0.731462925851703
123 0.731462925851703
124 0.731462925851703
125 0.731462925851703
126 0.731462925851703
127 0.731462925851703
128 0.731462925851703
129 0.731462925851703
130 0.731462925851703
131 0.731462925851703
132 0.731462925851703
133 0.731462925851703
134 0.731462925851703
135 0.731462925851703
136 0.731462925851703
137 0.731462925851703
138 0.731462925851703
139 0.731462925851703
140 0.731462925851703
141 0.731462925851703
142 0.731462925851703
143 0.731462925851703
144 0.731462925851703
145 0.731462925851703
146 0.731462925851703
147 0.731462925851703
148 0.731462925851703
149 0.731462925851703
150 0.731462925851703
151 0.731462925851703
152 0.731462925851703
153 0.731462925851703
154 0.731462925851703
155 0.731462925851703
156 0.731462925851703
157 0.731462925851703
158 0.731462925851703
159 0.731462925851703
160 0.731462925851703
161 0.731462925851703
162 0.731462925851703
163 0.731462925851703
164 0.731462925851703
165 0.731462925851703
166 0.731462925851703
167 0.731462925851703
168 0.731462925851703
169 0.731462925851703
170 0.731462925851703
171 0.731462925851703
172 0.731462925851703
173 0.731462925851703
174 0.731462925851703
175 0.731462925851703
176 0.731462925851703
177 0.731462925851703
178 0.731462925851703
179 0.731462925851703
180 0.731462925851703
181 0.733466933867735
182 0.735470941883767
183 0.735470941883767
184 0.735470941883767
185 0.7374749498998
186 0.7374749498998
187 0.739478957915832
188 0.739478957915832
189 0.739478957915832
190 0.739478957915832
191 0.739478957915832
192 0.739478957915832
193 0.739478957915832
194 0.741482965931864
195 0.741482965931864
196 0.745490981963928
197 0.745490981963928
198 0.745490981963928
199 0.745490981963928
200 0.745490981963928
};
\addlegendentry{CASG exact H}
\addplot [very thick, maroon]
table {%
0 0
1 0
2 0
3 0
4 0
5 0
6 0
7 0
8 0
9 0
10 0
11 0
12 0
13 0
14 0
15 0
16 0
17 0
18 0
19 0
20 0
21 0.00206185567010309
22 0.00206185567010309
23 0.00206185567010309
24 0.00412371134020619
25 0.00618556701030928
26 0.00618556701030928
27 0.00618556701030928
28 0.00618556701030928
29 0.00824742268041237
30 0.00824742268041237
31 0.00824742268041237
32 0.00824742268041237
33 0.00824742268041237
34 0.00824742268041237
35 0.00824742268041237
36 0.00824742268041237
37 0.00824742268041237
38 0.00824742268041237
39 0.00824742268041237
40 0.00824742268041237
41 0.00824742268041237
42 0.00824742268041237
43 0.00824742268041237
44 0.00824742268041237
45 0.00824742268041237
46 0.00824742268041237
47 0.00824742268041237
48 0.00824742268041237
49 0.0103092783505155
50 0.0103092783505155
51 0.0103092783505155
52 0.0103092783505155
53 0.0103092783505155
54 0.0103092783505155
55 0.0103092783505155
56 0.0103092783505155
57 0.0103092783505155
58 0.0103092783505155
59 0.0103092783505155
60 0.0103092783505155
61 0.0103092783505155
62 0.0103092783505155
63 0.0103092783505155
64 0.0103092783505155
65 0.0103092783505155
66 0.0103092783505155
67 0.0103092783505155
68 0.0103092783505155
69 0.0103092783505155
70 0.0103092783505155
71 0.0103092783505155
72 0.0123711340206186
73 0.0123711340206186
74 0.0144329896907217
75 0.0144329896907217
76 0.0144329896907217
77 0.0144329896907217
78 0.0144329896907217
79 0.0144329896907217
80 0.0144329896907217
81 0.0144329896907217
82 0.0144329896907217
83 0.0144329896907217
84 0.0144329896907217
85 0.0144329896907217
86 0.0144329896907217
87 0.0164948453608247
88 0.0164948453608247
89 0.0185567010309278
90 0.0185567010309278
91 0.0206185567010309
92 0.022680412371134
93 0.022680412371134
94 0.022680412371134
95 0.022680412371134
96 0.022680412371134
97 0.022680412371134
98 0.022680412371134
99 0.022680412371134
100 0.022680412371134
101 0.022680412371134
102 0.022680412371134
103 0.022680412371134
104 0.022680412371134
105 0.022680412371134
106 0.0247422680412371
107 0.0247422680412371
108 0.0247422680412371
109 0.0247422680412371
110 0.0247422680412371
111 0.0247422680412371
112 0.0247422680412371
113 0.0247422680412371
114 0.0247422680412371
115 0.0268041237113402
116 0.0288659793814433
117 0.0288659793814433
118 0.0309278350515464
119 0.0350515463917526
120 0.0350515463917526
121 0.0350515463917526
122 0.0371134020618557
123 0.0371134020618557
124 0.0371134020618557
125 0.0391752577319588
126 0.0391752577319588
127 0.0391752577319588
128 0.0391752577319588
129 0.0391752577319588
130 0.0391752577319588
131 0.0391752577319588
132 0.0391752577319588
133 0.0412371134020619
134 0.0412371134020619
135 0.0412371134020619
136 0.0412371134020619
137 0.0412371134020619
138 0.0412371134020619
139 0.0412371134020619
140 0.0412371134020619
141 0.0412371134020619
142 0.0412371134020619
143 0.0412371134020619
144 0.0412371134020619
145 0.0412371134020619
146 0.0412371134020619
147 0.0412371134020619
148 0.0412371134020619
149 0.0412371134020619
150 0.0412371134020619
151 0.0412371134020619
152 0.0412371134020619
153 0.0412371134020619
154 0.0412371134020619
155 0.0412371134020619
156 0.0412371134020619
157 0.0412371134020619
158 0.0412371134020619
159 0.0412371134020619
160 0.0412371134020619
161 0.0412371134020619
162 0.0412371134020619
163 0.0412371134020619
164 0.0412371134020619
165 0.0412371134020619
166 0.0412371134020619
167 0.0412371134020619
168 0.0412371134020619
169 0.0412371134020619
170 0.0412371134020619
171 0.0412371134020619
172 0.0412371134020619
173 0.0412371134020619
174 0.0412371134020619
175 0.0412371134020619
176 0.0412371134020619
177 0.0412371134020619
178 0.0412371134020619
179 0.0412371134020619
180 0.0412371134020619
181 0.0412371134020619
182 0.0412371134020619
183 0.0412371134020619
184 0.0412371134020619
185 0.0412371134020619
186 0.0412371134020619
187 0.0412371134020619
188 0.0412371134020619
189 0.0412371134020619
190 0.0412371134020619
191 0.0412371134020619
192 0.0412371134020619
193 0.0412371134020619
194 0.0412371134020619
195 0.0412371134020619
196 0.0412371134020619
197 0.0412371134020619
198 0.0412371134020619
199 0.0412371134020619
200 0.0412371134020619
};
\addlegendentry{RBF}
\addplot [very thick, palevioletred204121167]
table {%
0 0
1 0.002
2 0.002
3 0.002
4 0.002
5 0.002
6 0.002
7 0.002
8 0.002
9 0.002
10 0.004
11 0.004
12 0.008
13 0.018
14 0.022
15 0.024
16 0.028
17 0.03
18 0.034
19 0.036
20 0.038
21 0.044
22 0.048
23 0.052
24 0.054
25 0.056
26 0.056
27 0.056
28 0.056
29 0.056
30 0.056
31 0.056
32 0.056
33 0.056
34 0.056
35 0.056
36 0.056
37 0.056
38 0.056
39 0.06
40 0.06
41 0.062
42 0.064
43 0.07
44 0.074
45 0.082
46 0.088
47 0.104
48 0.118
49 0.138
50 0.154
51 0.168
52 0.18
53 0.204
54 0.212
55 0.22
56 0.234
57 0.246
58 0.25
59 0.258
60 0.26
61 0.264
62 0.27
63 0.27
64 0.274
65 0.28
66 0.28
67 0.28
68 0.286
69 0.286
70 0.292
71 0.302
72 0.308
73 0.308
74 0.31
75 0.318
76 0.318
77 0.318
78 0.318
79 0.32
80 0.322
81 0.326
82 0.33
83 0.332
84 0.332
85 0.334
86 0.34
87 0.342
88 0.344
89 0.348
90 0.354
91 0.36
92 0.36
93 0.362
94 0.364
95 0.364
96 0.364
97 0.368
98 0.37
99 0.372
100 0.372
101 0.372
102 0.376
103 0.378
104 0.378
105 0.38
106 0.386
107 0.39
108 0.398
109 0.4
110 0.404
111 0.41
112 0.41
113 0.42
114 0.426
115 0.428
116 0.434
117 0.436
118 0.436
119 0.44
120 0.444
121 0.448
122 0.45
123 0.456
124 0.456
125 0.456
126 0.456
127 0.458
128 0.458
129 0.458
130 0.458
131 0.458
132 0.458
133 0.458
134 0.46
135 0.46
136 0.46
137 0.46
138 0.462
139 0.462
140 0.462
141 0.462
142 0.462
143 0.464
144 0.464
145 0.464
146 0.466
147 0.466
148 0.468
149 0.468
150 0.468
151 0.472
152 0.478
153 0.48
154 0.486
155 0.488
156 0.488
157 0.488
158 0.49
159 0.49
160 0.492
161 0.492
162 0.494
163 0.496
164 0.498
165 0.5
166 0.5
167 0.502
168 0.506
169 0.51
170 0.51
171 0.51
172 0.512
173 0.514
174 0.518
175 0.52
176 0.52
177 0.522
178 0.528
179 0.53
180 0.53
181 0.53
182 0.53
183 0.53
184 0.53
185 0.534
186 0.534
187 0.536
188 0.536
189 0.538
190 0.538
191 0.54
192 0.54
193 0.542
194 0.542
195 0.542
196 0.542
197 0.544
198 0.544
199 0.544
200 0.546
};
\addlegendentry{AFD}
\addplot [very thick, slategray]
table {%
0 0
1 0
2 0
3 0.008
4 0.01
5 0.022
6 0.042
7 0.042
8 0.044
9 0.066
10 0.074
11 0.104
12 0.13
13 0.19
14 0.214
15 0.248
16 0.27
17 0.284
18 0.322
19 0.336
20 0.374
21 0.406
22 0.436
23 0.448
24 0.45
25 0.458
26 0.464
27 0.468
28 0.492
29 0.494
30 0.496
31 0.52
32 0.522
33 0.522
34 0.524
35 0.54
36 0.548
37 0.548
38 0.554
39 0.56
40 0.562
41 0.566
42 0.572
43 0.576
44 0.576
45 0.578
46 0.58
47 0.58
48 0.582
49 0.584
50 0.586
51 0.588
52 0.59
53 0.59
54 0.59
55 0.592
56 0.596
57 0.598
58 0.598
59 0.6
60 0.602
61 0.602
62 0.604
63 0.606
64 0.608
65 0.608
66 0.608
67 0.61
68 0.616
69 0.616
70 0.616
71 0.616
72 0.616
73 0.618
74 0.618
75 0.62
76 0.62
77 0.62
78 0.622
79 0.622
80 0.622
81 0.622
82 0.622
83 0.624
84 0.624
85 0.624
86 0.624
87 0.628
88 0.63
89 0.63
90 0.63
91 0.632
92 0.632
93 0.632
94 0.632
95 0.632
96 0.632
97 0.634
98 0.636
99 0.636
100 0.636
101 0.636
102 0.636
103 0.636
104 0.636
105 0.636
106 0.636
107 0.636
108 0.636
109 0.636
110 0.636
111 0.638
112 0.638
113 0.638
114 0.638
115 0.638
116 0.638
117 0.638
118 0.638
119 0.638
120 0.638
121 0.638
122 0.638
123 0.638
124 0.638
125 0.638
126 0.638
127 0.638
128 0.638
129 0.638
130 0.638
131 0.638
132 0.638
133 0.638
134 0.638
135 0.638
136 0.638
137 0.638
138 0.638
139 0.638
140 0.638
141 0.638
142 0.638
143 0.638
144 0.638
145 0.638
146 0.638
147 0.638
148 0.638
149 0.638
150 0.638
151 0.638
152 0.638
153 0.638
154 0.638
155 0.638
156 0.638
157 0.638
158 0.638
159 0.638
160 0.638
161 0.638
162 0.638
163 0.638
164 0.638
165 0.638
166 0.638
167 0.638
168 0.638
169 0.638
170 0.638
171 0.64
172 0.64
173 0.64
174 0.64
175 0.64
176 0.64
177 0.64
178 0.64
179 0.64
180 0.64
181 0.64
182 0.64
183 0.64
184 0.64
185 0.64
186 0.64
187 0.64
188 0.64
189 0.642
190 0.642
191 0.642
192 0.642
193 0.642
194 0.642
195 0.642
196 0.642
197 0.644
198 0.644
199 0.644
200 0.644
};
\addlegendentry{FD exact H}
\addplot [very thick, slategray, dashed]
table {%
0 0
1 0
2 0
3 0
4 0
5 0
6 0
7 0
8 0
9 0
10 0
11 0
12 0
13 0
14 0
15 0
16 0
17 0
18 0
19 0
20 0
21 0.00400801603206413
22 0.0260521042084168
23 0.0400801603206413
24 0.0581162324649299
25 0.064128256513026
26 0.106212424849699
27 0.122244488977956
28 0.12625250501002
29 0.142284569138277
30 0.156312625250501
31 0.178356713426854
32 0.196392785571142
33 0.212424849699399
34 0.236472945891784
35 0.258517034068136
36 0.280561122244489
37 0.30060120240481
38 0.306613226452906
39 0.312625250501002
40 0.32064128256513
41 0.330661322645291
42 0.336673346693387
43 0.336673346693387
44 0.336673346693387
45 0.338677354709419
46 0.340681362725451
47 0.348697394789579
48 0.354709418837675
49 0.360721442885772
50 0.360721442885772
51 0.382765531062124
52 0.382765531062124
53 0.402805611222445
54 0.404809619238477
55 0.404809619238477
56 0.406813627254509
57 0.416833667334669
58 0.418837675350701
59 0.430861723446894
60 0.440881763527054
61 0.448897795591182
62 0.448897795591182
63 0.468937875751503
64 0.468937875751503
65 0.468937875751503
66 0.468937875751503
67 0.468937875751503
68 0.490981963927856
69 0.490981963927856
70 0.490981963927856
71 0.492985971943888
72 0.492985971943888
73 0.492985971943888
74 0.492985971943888
75 0.498997995991984
76 0.507014028056112
77 0.511022044088176
78 0.51503006012024
79 0.51503006012024
80 0.51503006012024
81 0.51503006012024
82 0.51503006012024
83 0.51503006012024
84 0.517034068136273
85 0.517034068136273
86 0.517034068136273
87 0.517034068136273
88 0.517034068136273
89 0.517034068136273
90 0.517034068136273
91 0.517034068136273
92 0.517034068136273
93 0.517034068136273
94 0.517034068136273
95 0.527054108216433
96 0.529058116232465
97 0.537074148296593
98 0.537074148296593
99 0.537074148296593
100 0.537074148296593
101 0.537074148296593
102 0.537074148296593
103 0.537074148296593
104 0.537074148296593
105 0.537074148296593
106 0.537074148296593
107 0.539078156312625
108 0.539078156312625
109 0.541082164328657
110 0.543086172344689
111 0.543086172344689
112 0.543086172344689
113 0.543086172344689
114 0.543086172344689
115 0.543086172344689
116 0.543086172344689
117 0.543086172344689
118 0.543086172344689
119 0.543086172344689
120 0.543086172344689
121 0.543086172344689
122 0.543086172344689
123 0.543086172344689
124 0.543086172344689
125 0.543086172344689
126 0.543086172344689
127 0.543086172344689
128 0.543086172344689
129 0.543086172344689
130 0.543086172344689
131 0.543086172344689
132 0.543086172344689
133 0.543086172344689
134 0.543086172344689
135 0.543086172344689
136 0.543086172344689
137 0.543086172344689
138 0.543086172344689
139 0.543086172344689
140 0.543086172344689
141 0.543086172344689
142 0.543086172344689
143 0.543086172344689
144 0.543086172344689
145 0.543086172344689
146 0.543086172344689
147 0.543086172344689
148 0.543086172344689
149 0.543086172344689
150 0.543086172344689
151 0.543086172344689
152 0.543086172344689
153 0.543086172344689
154 0.543086172344689
155 0.545090180360721
156 0.545090180360721
157 0.545090180360721
158 0.545090180360721
159 0.545090180360721
160 0.545090180360721
161 0.545090180360721
162 0.545090180360721
163 0.545090180360721
164 0.545090180360721
165 0.545090180360721
166 0.545090180360721
167 0.549098196392786
168 0.549098196392786
169 0.549098196392786
170 0.549098196392786
171 0.549098196392786
172 0.549098196392786
173 0.549098196392786
174 0.549098196392786
175 0.549098196392786
176 0.549098196392786
177 0.551102204408818
178 0.551102204408818
179 0.551102204408818
180 0.551102204408818
181 0.551102204408818
182 0.551102204408818
183 0.551102204408818
184 0.551102204408818
185 0.551102204408818
186 0.551102204408818
187 0.551102204408818
188 0.551102204408818
189 0.551102204408818
190 0.551102204408818
191 0.551102204408818
192 0.551102204408818
193 0.551102204408818
194 0.551102204408818
195 0.551102204408818
196 0.551102204408818
197 0.551102204408818
198 0.551102204408818
199 0.55310621242485
200 0.55310621242485
};
\addlegendentry{FD RBF H}
\end{axis}

\end{tikzpicture}